\documentclass[10pt,reqno]{amsart}
\usepackage{amssymb,mathrsfs,graphicx,subfigure, enumerate}
\usepackage{amsmath,amsfonts,amssymb,amscd,amsthm,bbm}
\usepackage{extpfeil}
\usepackage{xcolor}
\usepackage{soul}

\usepackage{graphicx,colortbl}

\usepackage{hyperref}

\usepackage{float}
\usepackage{epsfig}
\usepackage{caption}
\usepackage{bm}
\graphicspath{{Figure/}}
\allowdisplaybreaks
\usepackage{comment}
\usepackage[margin=0.853in]{geometry}
\usepackage{textcase}

 \title[Hydrodynamic limit of the Boltzmann equation]
{Hydrodynamic Limit of the Boltzmann Equation toward Generic Riemann Solutions with Shocks\\[2.1em]
{\normalfont\mdseries\itshape\fontsize{10}{10}\selectfont
\NoCaseChange{In honor of Professor Costas Dafermos's 85th birthday}}}

\author[Choe]{Mingi Choe}
\address[Mingi Choe]{\newline Department of Mathematical Sciences \newline Korea Advanced Institute of Science and Technology, Daejeon  34141, Republic of Korea}
\email{mgchoe@kaist.ac.kr}

\author[Kang]{Moon-jin Kang}
\address[Moon-jin Kang]{\newline Department of Mathematical Sciences \newline Korea Advanced Institute of Science and Technology, Daejeon  34141, Republic of Korea}
\email[Moonjin Kang]{moonjinkang@kaist.ac.kr}

\author[Kim]{Chanwoo Kim}
\address[Chanwoo Kim]{\newline Department of Mathematics, \newline University of Wisconsin-Madison, Madison, WI, 53717, USA}
\email[Chanwoo Kim]{chanwookim.math@gmail.com; chanwoo.kim@wisc.edu}

\newtheorem{theorem}{Theorem}[section]
\newtheorem{lemma}{Lemma}[section]
\newtheorem{corollary}{Corollary}[section]
\newtheorem{proposition}{Proposition}[section]
\newtheorem{remark}{Remark}[section]

\let\hide\iffalse
\let\unhide\fi

\newcommand{\R}{\mathbb{R}}

\newcommand{\norm}[1]{\bigl\lVert#1\bigr\rVert}

\newcommand{\beq}{\begin{equation}}
\newcommand{\eeq}{\end{equation}}

\newcommand{\eps}{\varepsilon }

\newcommand{\myparas}[1]{\left(#1\right)}
\newcommand{\myparam}[1]{\left\{#1\right\}}
\newcommand{\myparab}[1]{\left[#1\right]}

\newcommand{\abs}[1]{\left| #1\right|}

\newcommand{\wtilde}[1]{\widetilde{#1}}
\newcommand{\shock}[2]{\bigl((#1^{S_{#2}})^{-X_{#2}}\bigr)}
\newcommand{\shockw}[2]{\bigl(#1^{S_{#2}}\bigr)^{-X_{#2}}}
\newcommand{\shockp}[2]{#1^{S_{#2},-X_{#2}}}

\newcommand{\y}[1]{\overline{#1}}

\sethlcolor{yellow}

\begin{document}

\date{\today}
\subjclass{76P05, 35Q20, 82B40} \keywords{Boltzmann equation, Hydrodynamic limit, Compressible Euler system, Riemann problem, Shock}

\thanks{\textbf{Acknowledgment.} MC and MJK were partially supported by the National Research Foundation of Korea  (RS-2024-00361663  and NRF-2019R1A5A1028324). 
CK is partially supported by NSF-CAREER 2047681. This material is partly based upon work supported by the National Science Foundation under Grant No. DMS-2424139, while one of the authors (C.K.) was in residence at the Simons Laufer Mathematical Sciences Institute in Berkeley, California, during the Fall 2025 semester.}

\begin{abstract}
We establish the hydrodynamic limit of the one-dimensional Boltzmann equation with hard-sphere collisions toward Riemann solutions of the compressible Euler system. The Riemann solutions covered by our result include generic superpositions of elementary waves: either two shock waves and a contact discontinuity, or a rarefaction wave, a contact discontinuity, and a shock wave. For suitably well-prepared initial data and sufficiently small wave strength, we prove that the corresponding Boltzmann solution exists globally in time and converges, as the Knudsen number vanishes, to the local Maxwellian associated with the Riemann solution in $L^2([0,T]\times\mathbb R_x\times\mathbb R^3_\xi)$ for any $T>0$. The proof combines the macro--micro decomposition with a kinetic adaptation of the $a$-contraction method, in which the propagation speeds of Boltzmann shocks are determined by Rakine-Hugoniot speed with dynamical modulation parameters. The resulting coercive control of the shock translation modes, together with layer analysis and the uniform bound of the Shifts, allows us to pass to the Knudsen limit in the full space-time domain. To the best of our knowledge, this is the first rigorous result on the hydrodynamic limit  towards generic Riemann solutions containing shocks: either the shock--contact--shock case or  the rarefaction--contact--shock case, in a global space-time energy norm without removing neighborhoods of either the initial time or the shock layers. In the special case of a single shock, the argument further yields a sharp quantitative description of the kinetic shock layer, up to the dynamically selected Shift.
\end{abstract}

\maketitle

\tableofcontents


\section{Introduction}
\setcounter{equation}{0}

\subsection{The problem and the main theorem}

The derivation of fluid equations from kinetic theory is one of the classical
problems originating in the works of Maxwell and Boltzmann and in Hilbert's
sixth problem.  In the kinetic description of a rarefied gas, the state of the
gas is described by a distribution function \(f(t,x,\xi)\), depending on
position and molecular velocity.  By contrast, the macroscopic equations of
fluid mechanics involve only finitely many fields, such as density, momentum,
and temperature.  Thus the passage from kinetic theory to fluid dynamics is,
at its core, a closure problem: the Boltzmann equation contains an infinite
hierarchy of velocity moments, whereas the fluid equations retain only the
moments associated with the collision invariants.

The small parameter governing this passage is the Knudsen number
\[
        \kappa=\frac{\lambda_{\mathrm{mfp}}}{L},
\]
the ratio between the molecular mean free path \(\lambda_{\mathrm{mfp}}\) and
the macroscopic length scale \(L\).  In the regime \(\kappa\ll1\), collisions
occur on a much shorter scale than the macroscopic variation of the gas.  The
collision operator is therefore dominant, and the distribution is expected to
relax rapidly toward a local Maxwellian.  Since local Maxwellians are
parametrized only by density, velocity, and temperature, this relaxation
provides the formal closure mechanism leading to the compressible Euler
equations.

The smooth compressible Euler limit of the Boltzmann equation has a long
history.  Nishida obtained an early analytic result \cite{Nishida1978} which
may be viewed as a rigorous realization of the Maxwellian closure mechanism.
Indeed, multiplying the Boltzmann equation by the collision invariants gives
exact balance laws for mass, momentum, and energy.  These balance laws are not
closed in general, because the stress tensor and the heat flux involve higher
velocity moments of the distribution function.  If, however, the distribution is
a local Maxwellian, the non-equilibrium stress and heat flux vanish, and the
macroscopic balances close to the compressible Euler system.  Nishida justified
this closure in the zero-mean-free-path limit for small analytic perturbations
of an absolute Maxwellian, using the spectral theory of the linearized
Boltzmann equation and an abstract Cauchy--Kowalewski theorem in a scale of
analytic Banach spaces.

It is useful to distinguish this moment-closure viewpoint from the
asymptotic-expansion viewpoint.  In the former, the main issue is to prove that
the kinetic solution is confined near the local Maxwellian manifold and that the
higher-moment closure defects vanish in the limit.  In the latter, one seeks an
order-by-order representation
\[
        F^\kappa
        =
        M[U]
        +
        \kappa F_1
        +
        \kappa^2 F_2
        +\cdots,
\]
where the leading term \(M[U]\) is a local Maxwellian and the higher-order
terms are obtained by solving linearized collision equations subject to
solvability conditions.  Hilbert and Chapman--Enskog expansions give detailed
higher-order information, including Navier--Stokes corrections, initial layers,
boundary layers, and shock-layer corrections.  At the same time, such expansions
are tied to the regularity of the leading macroscopic profile and become
delicate, or break down, when the limiting Euler solution develops
discontinuities.

Caflisch \cite{Caflisch1980} and Ukai--Asano \cite{UkaiAsano1983} developed
Hilbert-expansion approaches to the smooth compressible Euler limit, including
the treatment of initial layers.  More recent works refined these arguments in
\(L^2\)-\(L^\infty\) frameworks, notably the work of Guo--Jang--Jiang
\cite{GuoJangJiang2009}, which revisited Caflisch's compressible Euler limit
and proved the validity of the Hilbert expansion before shock formation.

These smooth-limit results justify the kinetic-to-fluid transition only as long
as the limiting Euler flow remains smooth.  In this sense, they describe the
pre-shock classical regime of compressible gas dynamics.  The restriction is
not merely technical; it reflects the structure of the known methods.  The
usual expansion and stability arguments are organized around a smooth local
Maxwellian field and require uniform control of the macroscopic derivatives.
When shocks form, the limiting Maxwellian becomes discontinuous, higher-gradient
corrections and kinetic layers enter the asymptotics, and the smooth-limit
framework no longer applies directly.

A notable contrast is provided by the Euler--Poisson setting.  Guo--Jang
\cite{GuoJang2010} obtained a global Hilbert expansion for the
Vlasov--Poisson--Boltzmann system toward the compressible Euler--Poisson system,
precisely because the limiting Euler--Poisson flow remains globally smooth.  The
self-consistent electric field suppresses shock formation for small
irrotational perturbations through the dispersive \emph{Klein--Gordon effect}
first exploited by Guo \cite{Guo1998EP}.  

For the pure compressible Euler equations, by contrast, shock formation is an
intrinsic feature.  As a hyperbolic system of conservation laws, the Euler
system develops shocks even from smooth initial data.  Thus the
Boltzmann--Euler limit must ultimately be understood beyond the lifespan of
classical Euler solutions, in regimes where the limiting fluid profile contains
shock waves and contact discontinuities.

A complete kinetic theory of shock formation would have to connect the
pre-shock smooth Euler-limit regime with the post-shock kinetic layer regime.
Such a theory would require a simultaneous description of the emerging Euler
discontinuity, the inner Boltzmann shock layer, and its dynamically selected
location.  For the fluid equations themselves, shock formation from smooth data
has been studied in depth, from the classical one-dimensional theory to modern
multidimensional works such as those of Christodoulou \cite{Christodoulou2007},
Luk--Speck \cite{LukSpeck2018}, and Buckmaster--Shkoller--Vicol
\cite{BuckmasterShkollerVicol2022,BuckmasterShkollerVicol2023}, among others.
There has also been recent progress on the interaction between shock formation
and small viscous regularization, for instance in the works of
Chaturvedi--Graham \cite{ChaturvediGraham2022} and
Anderson--Chaturvedi--Graham \cite{AndersonChaturvediGraham2025}.  At the
kinetic level, however, a corresponding result that passes through the shock
formation time and connects the smooth Euler limit to the emerging Boltzmann
shock layer appears to remain largely open.

The limitation of smooth-limit theories is also consistent with a more
critical viewpoint, advocated by Slemrod, according to which the
Boltzmann-to-Euler passage beyond the smooth regime may face a genuine
viscosity--capillarity obstruction
\cite{Slemrod2013Hilbert,Slemrod2013Admissibility,Slemrod2018Failure}.\footnote{%
This viewpoint should be regarded as stronger than the standard smooth-limit
theory. The author argues, by a shock-scaling consideration, that in nonsmooth hydrodynamics
``there is no concept of `small' or `negligible' higher derivative terms''
\cite[p.~1498]{Slemrod2013Hilbert}.  In a later formulation, he states that
``Hilbert's program will fail because of the appearance of van der
Waals--Korteweg capillarity terms''
\cite[p.~1]{Slemrod2018Failure}.  We cite this as a critical viewpoint,
rather than as a settled consensus on the Boltzmann--Euler limit.
}
Building on the exact summation of the Chapman--Enskog expansion for
Grad-type moment systems by Gorban--Karlin
\cite{GorbanKarlin2014BAMS}, Slemrod emphasized that the passage from the
Boltzmann equation to the compressible Euler equations beyond the smooth regime
is not a straightforward continuation of the pre-shock theory.  In particular,
once shocks or other nonsmooth structures are present, higher-gradient
corrections cannot simply be regarded as negligible: the limiting process may
retain a viscosity--capillarity mechanism of Korteweg type, and the capillarity
contributions need not vanish in the sense of distributions.

 This viewpoint should, in the authors' view, be distinguished from a
negative statement about kinetic-to-Euler limits themselves.  In a different
incompressible scaling, the previous work {of Bae--Kim} \cite{BaeKim2026QCA} justifies the
Boltzmann--Euler closure beyond the class of smooth classical Euler flows,
without using either a Hilbert expansion or a Chapman--Enskog expansion.  The
argument relies instead on the exact macroscopic balance laws, together with
quantitative estimates on the kinetic closure defects.  These estimates show
that the non-hydrodynamic moments vanish in the limit, and hence that the
limiting dynamics is governed by the incompressible Euler equations even in
rough solution classes.  Thus, for the purposes of the present discussion, the
obstruction emphasized by Slemrod should be interpreted chiefly as a limitation
of formal expansion methods in nonsmooth regimes, rather than as a general
obstruction to kinetic-to-Euler limits.

 The present paper develops a related but distinct non-expansion mechanism for
the compressible post-shock regime.  We do not study the formation of shocks
from smooth Euler flows.  Rather, the limiting Euler profile is already a
Riemann wave: a discontinuous composite pattern containing shocks and a contact
discontinuity.  The problem is therefore to justify the Boltzmann--Euler limit
after the inviscid dynamics has entered the nonsmooth regime.

This is precisely the regime in which the formal smooth-limit theories cease
to be adequate.  The difficulty is not only to show that the kinetic closure
defects vanish.  Across each shock, the Boltzmann solution resolves the Euler
jump through an \(O(\kappa)\)-scale kinetic layer.  This inner layer carries a
translation mode: shifting the shock profile costs essentially no leading-order
energy.  Consequently, even a small imbalance in the conserved quantities can
produce a non-negligible displacement of the shock location.  The shock
position is therefore not a parameter that can be fixed in advance; it must be
selected dynamically as part of the limiting process.

We consider the one-dimensional Boltzmann equation
\begin{equation}\label{eq:bee0}
    f_t+\xi_1 f_x = \frac{1}{\kappa}\mathcal N(f,f),
\end{equation}
where \(\xi=(\xi_1,\xi_2,\xi_3)\in\mathbb R^3\), \(x\in\mathbb R\), and
\(\mathcal N\) denotes the hard-sphere collision operator.  After transforming to Lagrangian mass coordinates, which we still denote by $x$, one obtains
\begin{equation}\label{eq:bslk}
    f_t+\frac{\xi_1-u_1}{v}f_x = \frac{1}{\kappa}\mathcal N(f,f)
\end{equation}
where the specific volume $v=\rho^{-1}$. 
As the Knudsen number \(\kappa\to0\), the above equation is expected to converge
to the one-dimensional compressible Euler system in Lagrangian coordinates:
\begin{equation}\label{eq:cesL}
\begin{aligned}
    & v_t-u_{1x}=0,\\
    & u_{1t}+p_x=0,\\
    & u_{it}=0,\qquad i=2,3,\\
    & \left(\theta+\frac{|u|^2}{2}\right)_t+(pu_1)_x=0.
\end{aligned}
\end{equation}
In this setting, the natural limiting objects are Riemann solutions to
\eqref{eq:cesL}.
 The analysis of already-formed shocks in kinetic equations rests first on the
existence and structure theory of kinetic shock profiles.  For the Boltzmann
equation, Caflisch and Nicolaenko constructed small-amplitude shock profile
solutions \cite{CaflischNicolaenko1982}, writing the profile as a singular
perturbative expansion in the shock strength whose leading terms agree with
the corresponding Navier--Stokes shock profile.  A fundamental issue in using
such profiles as genuine distribution functions is their nonnegativity; this
was addressed by Liu and Yu \cite{LiuYu2004CMP} through a macro--micro
decomposition and energy method for Boltzmann shock layers.

 The kinetic shock-profile theory was further developed through the
invariant-manifold approach of Liu--Yu and Pogan--Zumbrun for steady Boltzmann
flows \cite{LiuYu2013ARMA,PoganZumbrun2019}.  This viewpoint treats the steady
Boltzmann equation as an infinite-dimensional dynamical system in the spatial
variable and provides a structural framework for kinetic layers connecting
Maxwellian end states.  More recently, Albritton--Bedrossian--Novack
constructed weak shock profiles for the Landau equation
\cite{AlbrittonBedrossianNovack2024}, extending the kinetic shock-profile
theory beyond the hard-sphere and angular-cutoff Boltzmann setting to the
plasma collision model.  These works provide the kinetic shock layers that
underlie hydrodynamic-limit results with already-formed shocks and also serve
as the profile building blocks in the present paper.

Building on the kinetic shock-profile theory, Yu studied the
zero-Knudsen-number limit toward piecewise smooth Euler flows containing shock
waves \cite{Yu2005}.  The formulation there is made, for simplicity, in a
periodic spatial setting and on a fixed finite time interval, with a finite
number of isolated noninteracting shock curves.  The proof is based on a highly
delicate matched-asymptotic construction: a generalized Hilbert expansion is
constructed away from the shock curves, exact Boltzmann shock profiles are
inserted in the shock layers, and the resulting approximate solution is
corrected through conservation-law constraints and a macro--micro stability
analysis.  This method requires successive matching and cancellation of the
residual macroscopic fluxes generated by the shock-layer patching, together
with a careful bookkeeping of outgoing diffusion waves and shock-location
corrections.

The complexity of this strategy reflects the intrinsic difficulty of the
compressible shock regime.  As emphasized by Slemrod, the passage from
Boltzmann to compressible Euler beyond the smooth regime is not a
straightforward continuation of the pre-shock theory: in the presence of
discontinuities, higher-gradient corrections, kinetic layers, and the selection
of the shock location become part of the limiting process.  Therefore a
shock-layer-corrected Hilbert expansion must justify, at the same time, the
outer Euler expansion, the inner Boltzmann shock layer, the matching between
them, the residual conservation defects produced by the patching, and the
stability of the final corrected approximation.  This makes the approach
powerful but necessarily rather elaborate.

Another result particularly close to the Riemann-problem setting of the present
paper is due to Huang, Wang, Wang, and Yang \cite{HuangWangWangYang2013}, who
justified the Boltzmann--Euler limit for Riemann solutions containing a generic
superposition of shock waves, rarefaction waves, and a contact discontinuity.
Their proof is based on a carefully corrected approximate-wave construction:
auxiliary hyperbolic waves with different backgrounds are introduced to capture
the extra masses generated by the hyperbolic approximation of rarefaction waves
and the diffusion approximation of contact discontinuities.  The perturbation is
then controlled by an energy method around this corrected approximate Riemann
pattern.
A limitation of this approach, however, is that convergence is formulated only away from the shock and initial layers. {More precisely, neighborhoods of both the shock front and the initial time must be excluded, and the convergence estimate depends on the size of these excluded neighborhoods.}  This reflects the structure of the method and the topology for convergence.  The corrected
approximate-wave construction is designed to control the macroscopic defects
of the Riemann solution on the Euler scale, while the shock layer itself lives on
the kinetic scale 
\[
        \frac{x-\sigma t}{\kappa}.
\]
Inside this layer, derivatives of the shock profile are of size
\(O(\kappa^{-1})\), and the leading difficulty is the translation mode of the
kinetic shock.  
The anti-derivative method, without time modulation of the shock location, does not provide direct control of the superposition of the shock and rarefaction waves inside the shock and initial layers.

The present paper develops a different approach. We do not rely on a shock-layer-corrected Hilbert expansion around a prescribed discontinuous Euler flow, nor do we use the anti-derivative method. Instead, our argument is based on a
kinetic adaptation of the \(a\)-contraction method for shocks, together with BV compactness of the shifts and control of the shock layers and approximate waves in the limit. The $a$-contraction method was first developed in \cite{KV-ARMA16,Vasseur-16} for the stability of Riemann shocks, and extended to the stability of large perturbations of viscous shocks for viscous conservation laws as in \cite{EEK-BNSF,Kang-JMPA21,KO,KV-Poincare17,KV-JEMS21,KV-Inven21,KVW-CMP21}. 
{For extensions of the method to the time-asymptotic stability of small perturbations of Riemann solutions for the Navier-Stokes system, we refer to} \cite{han2023large,KangLee,KVW23,kang2025time}.
Recently, in \cite{wang2025}, {the method was applied to the Boltzmann equation} \eqref{eq:bslk} {to study its long-time behavior near a composite wave consisting of a viscous shock, a rarefaction wave, and a viscous contact wave.}     
{The key idea of the $a$-contraction method is to combine the relative entropy method with a suitable weight and shift to quantify the orbital stability of shock waves.} The weight encodes the jump strength and characteristic speed of a shock, and {the shift is defined through a modulation equation designed to control the translational mode of the shock in the energy estimate.} 

On the other hand, the relative entropy method alone is sufficient to establish the stability of regular solutions other than shocks.    
The relative entropy method was first introduced by Dafermos \cite{Dafermos1} and DiPerna \cite{DiPerna} to prove the stability of Lipschitz solutions to the hyperbolic conservation laws endowed with a convex entropy.  
It was also used for stochastic particle systems \cite{Y1,OVY,Va} and later adapted to the kinetic framework in \cite{Go,Saint}.  In the kinetic setting,
this method is naturally compatible with Lions' notion of
\emph{dissipative solutions} and is tailored to weak--strong stability. 
For other applications of the method to studies of asymptotic limits, we refer to \cite{B-G-L-2,B-G-L-1,B-V,B-T-V,G-J-V-2,G-S,KV-M3AS,L-M,M-S,M-V,S}.
 
In our setting, the \(a\)-contraction method is combined with the
macro--micro structure of the Boltzmann equation.  The Boltzmann shock profiles
are retained as part of the target profile, but their locations are not fixed
by an order-by-order correction of a matched asymptotic expansion.  This is a
key difference from the shock-layer construction of Yu, where the insertion of
inner Boltzmann shock profiles into an outer Hilbert expansion creates residual
macroscopic fluxes, and the shock locations must be corrected along the
construction together with outgoing diffusion waves in order to restore the
conservation constraints.

 Here, by contrast, the shock locations are encoded by dynamically determined Shifts.
The Shifts are selected through modulation equations to produce the coercive damping needed to control the translation modes of the shocks. Thus the shock shift
in the present paper is not a bookkeeping correction in a formal expansion, but
a coercive unknown of the stability estimate.  The remaining kinetic errors are
then controlled by the microscopic dissipation of the linearized Boltzmann
operator.

Thus, our result should not be viewed as an implementation of the
generalized Hilbert-expansion construction for another wave pattern, nor as an energy method based on anti-derivatives.

In this paper, we justify the hydrodynamic limit from \eqref{eq:bslk} to
\eqref{eq:cesL} for generic Riemann solutions containing shocks, more precisely for composite waves consisting either of a \(1\)-shock,
a \(2\)-contact discontinuity, and a \(3\)-shock, or of a \(1\)-rarefaction wave,
a \(2\)-contact discontinuity, and a \(3\)-shock.  Our result shows that, for
suitably well-prepared initial data and sufficiently small wave strength, the
corresponding Boltzmann solution exists globally and converges in $L^2$, as
\(\kappa\to0\), to the local Maxwellian associated with the Riemann solution.  
{
\begin{theorem}[Informal statements of Theorem~\ref{thm:main} and Theorem~\ref{cor:rarefaction-contact-shock}]
Consider a Riemann solution of the compressible Euler system \eqref{eq:cesL}
consisting of a \(1\)-shock, a \(2\)-contact discontinuity, and a
\(3\)-shock or a \(1\)-rarefaction, a \(2\)-contact discontinuity, and a \(3\)-shock, with sufficiently small total wave strength. Then, for suitably
well-prepared initial data and sufficiently small Knudsen number \(\kappa>0\),
the corresponding solution to the Boltzmann equation \eqref{eq:bslk} exists
globally and converges, as \(\kappa\to0\), to the local {Maxwellian} associated with the Riemann solution, more precisely, for any $T>0$,
\[
\int_0^T \iint_{\mathbb R\times\mathbb R^3}
\bigl|f^\kappa-M[v^E,u^E,\theta^E]\bigr|^2
\,d\xi\,dx\,dt
\longrightarrow 0
\qquad\text{as }\kappa\to0.
\] 
\end{theorem}
}
As a consequence, when the Riemann solution is a single shock, our general estimate yields a
\textit{global} quantitative estimate for the shock, not merely a \textit{local}
away-from-the-shock convergence theorem.   
{{The dependence on the distance from the modulated shock location is explicit. More precisely,} $y=\sigma \tau+X^\kappa(\tau)$, {where} $X^\kappa(\tau)$ {is a dynamical shift for the shock component, with an exponential tail at the kinetic scale} \(\kappa\){, as follows.}}

To formulate this consequence, let \(M^{E,\kappa}(\tau,y,\xi)\) be the Maxwellian associated with the Riemann shock whose location is
modulated by {{a shift}} \(X^\kappa(\tau)\):
\[
M_{E,\kappa}(\tau,y,\xi)
:=
\begin{cases}
M_-(\xi), & y\le \sigma \tau+X^\kappa(\tau),\\[1mm]
M_+(\xi), & y> \sigma \tau+X^\kappa(\tau).
\end{cases}
\]
\begin{corollary}[Informal statement of Theorem~\ref{thm:single-shock-away}]
For a sufficiently small single shock and suitably well-prepared initial data,
the corresponding Boltzmann solution converges to the Maxwellian
associated with the Riemann shock, with the shock location modulated by a suitable
time-dependent shift \(X^\kappa(\tau)\). More precisely, for every finite time interval \([0,T]\),
\begin{align*}
\left\|
f^\kappa(\tau,y,\cdot)-M_{E,\kappa}(\tau,y,\cdot)
\right\|_{M_\#}
\le
C\,\kappa 
+
C e^{-c|y-\sigma \tau-X^\kappa(\tau)|/\kappa},
\qquad
(\tau,y)\in[0,T]\times\mathbb R.
\end{align*}
In particular, the convergence is uniform on every region staying a fixed
positive distance from the modulated shock wave.
\end{corollary}

\subsection{Main difficulties and strategy}
We now describe the main difficulties and the main ideas of the proof. The first
difficulty is the presence of shock waves in the limiting Euler profile.  Across
a shock, the limiting Maxwellian is discontinuous, and therefore the Boltzmann
solution cannot converge uniformly through the shock layer.  Moreover, the shock
location is not fixed at the level of the kinetic perturbation.  Even a small
amount of mass, momentum, or energy imbalance can produce a displacement of the
shock wave.  For this reason the limiting profile must be modulated by
time-dependent shifts.

We first normalize the Boltzmann equation with $\kappa=1$ as in \eqref{eq:bee}.
Since the Boltzmann equation contains microscopic degrees of freedom that
are invisible at the Euler level,  we use the macro--micro decomposition around
the local Maxwellian to write
\[
        f=M+G,
\]
and further decompose the microscopic part $G$ relative to the shock
profiles. Here, the fluid state {associated with} the Maxwellian $M$ satisfies the Navier-Stokes-Fourier-type system \eqref{eq:NSF1} coupled with microscopic variables. 
The  macro--micro decomposition is also used to write the Boltzmann shock $F^{S_i}=M^{S_i}+ G^{S_i}$. 
The coercivity of the linearized collision operator supplies
the basic microscopic dissipation, while the higher-order estimates close the
terms involving derivatives of the microscopic component.

The key idea for controlling the macroscopic part of the perturbation is based on the a-contraction method. We consider a shock-adapted weight \(a(t,x)\) as in \eqref{eq:wefun} and study the weighted
relative entropy
\[
        \int_{\mathbb R} a(t,x)\eta(U|\bar U)(t,x)\,dx,
\]
where \(\eta(U|\bar U)\) is the relative entropy between the fluid state
\(U=(v,u,\theta)\) and the composite approximate wave
\(\bar U=(\bar v,\bar u,\bar\theta)\) as in \eqref{eq:rel-entropy}.  
As in \eqref{eq:weighted-entropy}, the evolution of the weighted relative entropy produces the modulation part 
\[
\sum_{i\in\{1,3\}}  \dot X_i(t)Y_i(t),
\]
the bad terms $\mathcal{J}^{\mbox {bad}}$, the good terms $\mathcal{J}^{\mbox {good}}$, and the kinetic part $\mathcal{J}^{\mbox {kinetic}}$. 
Based on the $a$-contraction method, the terms localized by derivatives of shock or weight in $\mathcal{J}^{\mbox {bad}}$  are controlled by the modulation part with suitable shift and the diffusion from $\mathcal{J}^{\mbox {good}}$. To control the contact discontinuity, we consider the associated viscous contact wave constructed in \cite{huang2010hydrodynamic}. 
The contact component contributes Gaussian space-time
weights of the form
\[
        (1+t)^{-1}
        \exp\left(-\frac{c|x|^2}{1+t}\right),
\]
which cannot simply be absorbed by the shock dissipation. These terms are
handled by a separate weighted estimate and by the decay properties of the
viscous contact wave.
On the other hand, the interaction of the waves is not easy to control, since {the 1-shock and the 3-shock are shifted independently.}
This requires that the locations of the two shocks are well-separated as in \eqref{eq:shiftsp1}. {Using this separation, we localize the diffusion term with the cutoff introduced in} \eqref{eq:intske}{, so that the resulting localized bad terms can be controlled by the corresponding localized dissipation through the $a$-contraction method.} 

Since the fluid system \eqref{eq:NSF1} is coupled with the microscopic part, the above relative entropy estimates for the macroscopic variables should be coupled with microscopic energy-dissipation estimates for the kinetic remainder. 
{These estimates are derived} under the a priori smallness assumption on the perturbation, as in Proposition \ref{prop:priest}. The a priori estimates provide stability of the viscous waves, as in \cite{wang2025}. 
{However, {to obtain the desired singular limit, we need a careful layer analysis and estimates for the shifts} after scaling back to the $\kappa$-scaled equation \eqref{eq:bslk}.} 

The proof is organized as follows. Section 2 states the main theorem and the well-preparedness assumptions. Section 3 collects the necessary preliminaries on the macro--micro decomposition, viscous contact waves, and Boltzmann shock profiles. Section 4 constructs the composite approximate wave and introduces the perturbation system, the dynamical shifts, and the main a priori proposition. Section 5 proves the main theorem assuming the a priori proposition, including the global continuation argument, the rescaling to the original Knudsen number, and the convergence to the shifted Riemann solution. Section 6 establishes the low-order weighted relative-entropy estimates. Section 7 proves the higher-order macroscopic and microscopic estimates. Section 8 completes the proof of the main a priori proposition. Section 9 records the single-shock consequence and the rarefaction--contact--shock extension.

\section{Statement of the main result}
\setcounter{equation}{0}

\subsection{Composite Euler waves and shifted Maxwellians}
In this subsection, we describe the composite Euler waves that arise as the limiting profiles and the corresponding shifted local Maxwellians.

Fix a constant right end state
\[
U_+:=(v_+,u_{+},\theta_+)\in \mathbb{R}_+\times\mathbb{R}^3\times\mathbb{R}_+, \qquad u_+ := (u_{1+},0,0).
\]
For sufficiently small wave strength $\delta>0$, let
\[
U^E=(v^E,u^E,\theta^E)(t,x)
\]
be a Riemann solution to the one-dimensional compressible Euler system \eqref{eq:cesL}, connecting the left state
\[
U_-:=(v_-,u_{-},\theta_-), \qquad u_- := (u_{1-},0,0)
\]
to $U_+$ through two intermediate states $U_*$ and $U^*$, with
\[
\delta:=|\theta_+-\theta_-|.
\]
We shall consider the following two wave patterns:
\begin{itemize}
    \item[(a)] a superposition of a $1$-rarefaction wave, a $2$-contact discontinuity, and a $3$-shock wave;
    \item[(b)] a superposition of a $1$-shock wave, a $2$-contact discontinuity, and a $3$-shock wave.
\end{itemize}

We denote by
\[
U_{R_1}^E,\qquad U_C^E,\qquad U_{S_1}^E,\qquad U_{S_3}^E
\]
the corresponding elementary Euler wave patterns, whenever they are present. Thus, in case (a),
\[
U^E = U_{R_1}^E + U_C^E + U_{S_3}^E - U_* - U^*,
\]
whereas in case (b),
\[
U^E = U_{S_1}^E + U_C^E + U_{S_3}^E - U_* - U^*,
\]
where
\[
U_* = (v_*,u_*,\theta_*), \quad  U^* = (v^*,u^*,\theta^*), \quad u_* = (u_{1*},0,0), \quad u^* = (u_1^*,0,0). 
\]

{Let $M[U]$ denote the local Maxwellian associated with the fluid state $U=(v,u,\theta)$ as below.
\begin{align*}
M[U] := \frac{1}{v(2\pi\theta)^{3/2}}e^{-\frac{|\xi-u|^2}{2\theta}}
\end{align*}}
In the presence of shock waves, the limiting profile is described up to suitable dynamical shifts along the shock components only. 
To denote a function $f$ shifted by $h$, we use the notation:
\[
f^{-h}(x)= f(x-h).
\]
{For simplicity, for a shift function} $X_3=X_3(t)${, we set}
\begin{equation}\label{eq:def-shifted-profile-rs}
U^{X_3}
:=
U_{R_1}^E + U_C^E + \bigl(U_{S_3}^E\bigr)^{-X_3} - U_* - U^*,
\end{equation}
and
\begin{equation}\label{eq:def-shifted-maxwellian-rs}
M_{X_3}[U^E]
:=
M[U^{X_3}],
\end{equation}
for the rarefaction--contact--shock case.

Similarly, for shift functions $X_1=X_1(t)$, $X_3=X_3(t)$ and $X=(X_1,X_3)$, we denote
\begin{equation}\label{eq:def-shifted-profile-ss}
U^{X}
:=
\bigl(U_{S_1}^E\bigr)^{-X_1}
+
U_C^E
+
\bigl(U_{S_3}^E\bigr)^{-X_3}
-
U_*
-
U^*,
\end{equation}
and
\begin{equation}\label{eq:def-shifted-maxwellian-ss}
M_{X}[U^E]
:=
M[U^{X}],
\end{equation}
for the shock--contact--shock case.

In addition, let $f^{S_3}$ denote the $3$-Boltzmann shock profile connecting the two Maxwellians $M[U^*]$ and $M[U_+]$. In the shock--contact--shock case, we also denote by $f^{S_1}$ the $1$-Boltzmann shock profile connecting $M[U_-]$ and $M[U_*]$.

\subsection{Well-prepared initial data}
We next specify the class of admissible initial data for the Boltzmann equation \eqref{eq:bslk}. Let
\[
U_0^E(x):=U^E(0,x)
\]
be the initial Riemann profile. We shall consider families of nonnegative smooth initial data
\[
\{f_0^\kappa\}_{\kappa>0}
\]
on $\mathbb{R}_x\times\mathbb{R}_\xi^3$ such that for a sufficiently small constant $\varepsilon_1>0$,
\begin{align}
& \iint_{\mathbb{R}\times\mathbb{R}^3}
\frac{|f^\kappa_0-M[U_0^E]|^2}{M_\#}\,d\xi\,dx
+
\kappa^2
\iint_{\mathbb{R}\times\mathbb{R}^3}
\frac{|f^\kappa_{0t}|^2+|f^\kappa_{0x}|^2}{M_\#}\,d\xi\,dx
\notag\\
&\qquad
+
\kappa^4
\iint_{\mathbb{R}\times\mathbb{R}^3}
\frac{|f^\kappa_{0xx}|^2+|f^\kappa_{0tx}|^2}{M_\#}\,d\xi\,dx
\le
\kappa \varepsilon_1^2.
\label{eq:initic2-main}
\end{align}
Here $M_\#:=M[U_\#]$ is a fixed global Maxwellian. 
We refer to \eqref{eq:initic2-main} as the well-preparedness condition. The appearance of the time derivatives $f^\kappa_{0t}$ and $f^\kappa_{0tx}$ reflects the higher-order compatibility built into the perturbative framework.

{
\begin{remark}

In Theorem \ref{thm:main}, {the upper bound for} $\eps_1$ {is chosen sufficiently small to ensure the global existence of the solution} $f^\kappa$ {for each} $\kappa$.  
More precisely, after normalizing \eqref{eq:initic2-main}, $f_0(x,\xi):=f^\kappa(\kappa t, \kappa x, \xi)|_{t=0}$ satisfies
 \begin{align} \label{eq:initda}
\mathcal E_{\mathrm{ini}}^2 := & \iint_{\mathbb{R}\times\mathbb{R}^3}
\frac{|f_0-M[U_0^E]|^2}{M_\#}\,d\xi\,dx
+
\iint_{\mathbb{R}\times\mathbb{R}^3}
\frac{|f_{0t}|^2+|f_{0x}|^2}{M_\#}\,d\xi\,dx
\nonumber\\
&\qquad
+
\iint_{\mathbb{R}\times\mathbb{R}^3}
\frac{|f_{0xx}|^2+|f_{0tx}|^2}{M_\#}\,d\xi\,dx
\le
 \varepsilon_1^2.
\end{align}
{The} $\kappa${-uniform smallness of} $\mathcal{E}_{\rm ini}$ {therefore yields the required smallness of the initial energy} $\mathcal{E}(0)$ {in Proposition} \ref{prop:priest}{, up to the small profile contribution in} \eqref{eq:pxt}{. This motivates the well-preparedness condition} \eqref{eq:initic2-main}.
 \end{remark}
}

\subsection{Statement of the main theorem}
We now state the precise form of the hydrodynamic limit theorem. 
{
\begin{theorem}[Hydrodynamic limit near a shock--contact--shock profile]\label{thm:main}
Fix a constant end state
\[
U_+:=(v_+,u_{+},\theta_+)\in \mathbb{R}_+\times\mathbb{R}^3\times\mathbb{R}_+, \qquad u_+ := (u_{1+},0,0).
\]
Then there exist positive constants \(\varepsilon_0\), \(\delta_0\), and \(\kappa_0\), together with a global Maxwellian $M_\#$, 
such that the following statement holds.

Let \(0<\delta\le \delta_0\), and let
\[
U^E=(v^E,u^E,\theta^E)(t,x)
\]
be a Riemann solution to the one-dimensional compressible Euler system \eqref{eq:cesL}, connecting
\[
U_-:=(v_-,u_{-},\theta_-), \, \quad u_- :=(u_{1-},0,0)
\quad \text{to} \quad
U_+,
\]
through two intermediate states \(U_*\) and \(U^*\), with total wave strength
\[
\delta=\delta_1+\delta_C+\delta_3,
\]
{where} \(\delta_C\) {is the strength of the contact discontinuity,} \(\delta_1:=|v_--v_*|\) {the strength of the 1-shock wave, and} \(\delta_3:=|v^*-v_+|\) {the strength of the 3-shock wave.}
Assume that \(U^E\) is a superposition of a \(1\)-shock wave, a \(2\)-contact discontinuity, and a \(3\)-shock wave. 
For each \(0<\varepsilon_1\le \varepsilon_0\) and \(0<\kappa\le \kappa_0\), let \(f_0^\kappa\) be a nonnegative smooth initial datum satisfying the well-preparedness condition \eqref{eq:initic2-main}.

 {Then there exists} a unique solution $f^\kappa$
 to the Boltzmann equation \eqref{eq:bslk} on a time interval \([0,T]\), where \(T>0\) is arbitrary.
Moreover, 
\begin{align}\label{eq:hlmt1-main}
&\int_0^T \iint_{\mathbb{R}\times\mathbb{R}^3}
\left|f^\kappa-M[U^E]\right|^2
\,d\xi\,dx\,dt\nonumber\\
&\qquad\le
C\kappa^{1/2}\Bigl[\kappa^{1/2}\bigl((\varepsilon_1+\delta_0)^2+\delta_0^{1/2}\bigr)T
+
\delta_C\,T^{3/2}
+
\delta_0^2T^{3/2}\Bigr].
\end{align}
Here, the positive constant \(C\) is independent of \(\kappa\) and \(T\).
\end{theorem}
}
\section{Preliminaries on the Boltzmann equation and wave profiles}

\setcounter{equation}{0}

\subsection{Macro-micro decomposition and fluid-type formulation}
We now turn to the normalized problem obtained by scaling out the Knudsen number. More precisely, although the original Boltzmann equation was introduced in \eqref{eq:bee0} and \eqref{eq:bslk} with Knudsen number $\kappa$, throughout the a priori analysis we work with the rescaled equation in which $\kappa=1$. The dependence on the original Knudsen number will be recovered at the end of the proof by a scaling argument.

Thus, in this section we consider the one-dimensional Boltzmann equation in Eulerian coordinates
\begin{equation}\label{eq:bee}
    f_t+\xi_1 f_x = \mathcal N(f,f),
\end{equation}
where $\xi=(\xi_1,\xi_2,\xi_3)\in\mathbb{R}^3$, $x\in\mathbb{R}$, and $\mathcal N$ is the hard-sphere collision operator:
\begin{equation}\label{eq:Qdef}
\mathcal N(g,h)(\xi)
=
\iint_{\mathbb{R}^3\times \mathbb{S}_+^2}
|(\xi-\xi_*)\cdot \Omega|
\Bigl(g(\xi')h(\xi_*')-g(\xi)h(\xi_*)\Bigr)
\,d\Omega\,d\xi_*,
\end{equation}
where
\begin{equation}\label{eq:postcoll}
\xi'+\xi_*'=\xi+\xi_*,
\qquad
|\xi'|^2+|\xi_*'|^2=|\xi|^2+|\xi_*|^2,
\end{equation}
and
\[
\mathbb{S}_+^2
:=
\left\{
\Omega\in\mathbb{S}^2:(\xi-\xi_*)\cdot\Omega\ge0
\right\}.
\]

Passing from Eulerian to Lagrangian coordinates, {we continue to denote the spatial coordinate by} $x${. Then} \eqref{eq:bee} {becomes}
\begin{equation}\label{eq:bel}
    f_t+\frac{\xi_1-u_1}{v}f_x = \mathcal N(f,f),
\end{equation}
where $u_1$ is the first component of the fluid velocity and $v=\rho^{-1}$ is the specific volume.

For a solution $f(t,x,\xi)$ of \eqref{eq:bel}, we employ the macro--micro decomposition relative to the local Maxwellian; see \cite{liu2004energy}. The macroscopic part corresponds to the fluid variables, namely the density $\rho$, the momentum $\rho u$, and the total energy $\rho\bigl(e+\frac{|u|^2}{2}\bigr)$, where $u=(u_1,u_2,u_3)$ denotes the fluid velocity and $e$ the internal energy.

More precisely, let $M$ be the normalized local Maxwellian, and define
\begin{equation}\label{eq:mami}
    P_0^{[M]} g := \sum_{i=0}^4 \langle g,\chi^{[M]}_i\rangle \chi^{[M]}_i,
    \qquad
    P_1^{[M]} g := g-P_0^{[M]} g,
\end{equation}
where
\begin{equation}\label{eq:chi}
\chi_0=\frac{M}{\sqrt{\rho}},
\qquad
\chi_i=\frac{\xi_i-u_i}{\sqrt{R\theta\,\rho}}M
\quad (i=1,2,3),
\qquad
\chi_4=
\frac{1}{\sqrt{6\rho}}
\left(
\frac{|\xi-u|^2}{R\theta}-3
\right)M,
\end{equation}
and
\begin{equation}\label{eq:inner}
    \langle g,h\rangle_{M} := \int_{\mathbb{R}^3}\frac{g(\xi)h(\xi)}{M(\xi)}\,d\xi.
\end{equation}

Here, $P_0^{[M]}$ and $P_1^{[M]}$ denote the macroscopic and microscopic projections, respectively. We write
\[
f=M+G,
\qquad
M:=P_0^{[M]}f,
\qquad
G:=P_1^{[M]}f.
\]

In addition, we also define useful norms by

\begin{align}
\norm{g}_{M_\#}^2 := \int_{\mathbb R^3} \frac{|g|^2}{M_\#} d\xi, \qquad \norm{g}_{\nu,M_\#}^2 := \int_{\mathbb R^3} \frac{(1+|\xi|)|g|^2}{M_\#} d\xi
\end{align}

where $M_\#$ is a given global Maxwellian. We sometimes set \(\norm{\cdot}_{L^2_{\xi}(M_\#)}:=\norm{\cdot}_{M_\#}\) to emphasize the function-space norm.

Integrating \eqref{eq:bel} against the collisional invariants
\[
1,\quad \xi_1,\quad \xi_2,\quad \xi_3,\quad \frac{|\xi|^2}{2},
\]
we obtain the macroscopic fluid-type system
\begin{equation}\label{eq:NSl}
\begin{aligned}
& v_t-u_{1x}=0,\\
& u_{1t}+p_x=-\int \xi_1^2 G_x\,d\xi,\\
& u_{it}=-\int \xi_1\xi_i G_x\,d\xi,\qquad i=2,3,\\
& \left(\theta+\frac{|u|^2}{2}\right)_t+(pu_1)_x
=
-\int \frac{1}{2}\xi_1|\xi|^2 G_x\,d\xi.
\end{aligned}
\end{equation}

On the other hand, substituting $f=M+G$ into \eqref{eq:bel} and applying the microscopic projection, we obtain
\begin{equation}\label{eq:mile}
\begin{aligned}
G_t
-\frac{u_1}{v}G_x
+\frac{1}{v}P_1^{[M]}(\xi_1 M_x)
+\frac{1}{v}P_1^{[M]}(\xi_1 G_x)
=
L_MG+\mathcal N(G,G),
\end{aligned}
\end{equation}
where
\[
L_MG:=\mathcal N(M,G)+\mathcal N(G,M)
\]
is the linearized collision operator around $M$.

Applying $L_M^{-1}$ to \eqref{eq:mile}, we decompose the microscopic part into a diffusion term and a remainder term:
\begin{equation}\label{eq:G}
    G
    =
    \frac{1}{v}L_M^{-1}\Bigl(P_1^{[M]}(\xi_1M_x)\Bigr)
    +\Pi_1,
\end{equation}
with
\begin{equation}\label{eq:Ger}
    \Pi_1
    :=
    L_M^{-1}
    \left(
    G_t
    -\frac{u_1}{v}G_x
    +\frac{1}{v}P_1^{[M]}(\xi_1G_x)
    -\mathcal N(G,G)
    \right).
\end{equation}

Substituting \eqref{eq:G} into \eqref{eq:NSl}, we arrive at the fluid-type system
\begin{equation}\label{eq:NSF1}
\begin{aligned}
& v_t-u_{1x}=0,\\
& u_{1t}+p_x
=
\frac{4}{3}\left(\frac{\mu(\theta)u_{1x}}{v}\right)_x
-\int \xi_1^2 \Pi_{1x}\,d\xi,\\
& u_{it}
=
\left(\frac{\mu(\theta)u_{ix}}{v}\right)_x
-\int \xi_1\xi_i \Pi_{1x}\,d\xi,
\qquad i=2,3,\\
& \left(\theta+\frac{|u|^2}{2}\right)_t+(pu_1)_x
=
\left(\frac{\alpha_{\rm{th}}(\theta)\theta_x}{v}\right)_x
+\frac{4}{3}\left(\frac{\mu(\theta)u_1u_{1x}}{v}\right)_x
\\
&\qquad \qquad \qquad \qquad \qquad +\sum_{i=2}^3 \left(\frac{\mu(\theta)u_i u_{ix}}{v}\right)_x
-\int \frac{1}{2}\xi_1|\xi|^2 \Pi_{1x}\,d\xi.
\end{aligned}
\end{equation}
Here $\mu(\theta)$ and $\alpha_{\rm{th}}(\theta)$ denote the viscosity and heat conductivity, respectively. {{For hard-sphere collisions,} $\frac{\alpha_{\rm{th}}(\theta)}{\mu(\theta)}$ {is constant; see} \cite{kawashima1979fluid}{.} so we will use the following notation $\gamma:=\frac{\alpha_{\rm{th}}(\theta)}{\mu(\theta)}$.} 

\subsection{Viscous contact wave}
For the viscous contact wave, following \cite{huang2008contact} and \cite{huang2010hydrodynamic}, we consider the self-similar solution $\Theta^{\mathrm{sim}}$ to the nonlinear diffusion equation
\begin{align}
\begin{aligned}
& \Theta_t^{\mathrm{sim}}
=
\frac{9p_*}{10}
\left(
\frac{\alpha_{\rm{th}}(\Theta^{\mathrm{sim}})\Theta_x^{\mathrm{sim}}}{\Theta^{\mathrm{sim}}}
\right)_x,\\
& \Theta^{\mathrm{sim}}(t,-\infty)=\theta_*,
\qquad
\Theta^{\mathrm{sim}}(t,\infty)=\theta^*
\end{aligned}
\end{align}
where $U_*=(v_*,u_*,\theta_*)$ is the left state and $U^*=(v^*,u^*,\theta^*)$ is the right state. 
Using this self-similar profile, we define the viscous contact wave $(v^C,u^C,\theta^C)(t,x)$ by
\begin{align}
\begin{aligned}\label{eq:ctident0}
& v^C(t,x):=\frac{2\Theta^{\mathrm{sim}}(t,x)}{3p_*},\\
& u_1^C(t,x):=u_{1*}+\frac{3\alpha_{\rm{th}}(\Theta^{\mathrm{sim}})\Theta_x^{\mathrm{sim}}}{5\Theta^{\mathrm{sim}}},\\
& u_i^C(t,x)\equiv 0,\qquad i=2,3,\\
& \theta^C(t,x):=\Theta^{\mathrm{sim}}(t,x).
\end{aligned}
\end{align}
Then $(v^C,u^C,\theta^C)$ satisfies
\begin{align}
\begin{aligned}\label{eq:ctident}
& v_t^C-u_{1x}^C=0,\\
& u_{1t}^C+p_x^C=\frac{4}{3}\left(\frac{\mu(\theta^C)u_{1x}^C}{v^C}\right)_x+Q_1^C,\\
& u_i^C=0,\qquad i=2,3,\\
& \theta_t^C+p^C u_{1x}^C
=
\left(\frac{\alpha_{\rm{th}}(\theta^C)\theta_x^C}{v^C}\right)_x
+\frac{4}{3}\mu(\theta^C)\frac{(u_{1x}^C)^2}{v^C}
+Q_2^C,
\end{aligned}
\end{align}
where
\begin{align}
\begin{aligned}
Q_1^C
 :=
u_{1t}^C-\frac{4}{3}\left(\frac{\mu(\theta^C)u_{1x}^C}{v^C}\right)_x,\ \ \   
Q_2^C
 :=
-\frac{4}{3}\mu(\theta^C)\frac{(u_{1x}^C)^2}{v^C}.
\end{aligned}
\end{align}
{Moreover,} 
\begin{align}
\label{eq:contact-Q-bounds}
Q_1^C=O(1)\delta_C(1+t)^{-3/2}e^{-c_0x^2/(1+t)},
\qquad
Q_2^C=O(1)\delta_C(1+t)^{-2}e^{-2c_0x^2/(1+t)}.
\end{align}

  The construction above is standard; see \cite{huang2010hydrodynamic}. For the convenience of the reader, we record the Gaussian bounds needed below.

\begin{lemma}\label{lem:contact-gaussian}
Let
\[
\delta_C:=|\theta^*-\theta_*|.
\]
Then there exists a positive constant $c_0>0$ such that
\begin{align}
\label{eq:Theta-sim-x}
|\Theta_x^{\mathrm{sim}}(t,x)|
\le
C\delta_C(1+t)^{-1/2}e^{-c_0x^2/(1+t)},
\end{align}
and consequently we have \eqref{eq:contact-Q-bounds}:
\begin{align} \notag
Q_1^C=O(1)\delta_C(1+t)^{-3/2}e^{-c_0x^2/(1+t)},
\qquad
Q_2^C=O(1)\delta_C(1+t)^{-2}e^{-2c_0x^2/(1+t)}.
\end{align}
\end{lemma}

\begin{lemma}[\cite{huang2008contact}]\label{lem:contact-wave-estimate0}
    Let $\delta_C$ denote the {strength} of the contact discontinuity as 
    \[
        \delta_C := |v^*-v_*|
    \]
    The viscous contact discontinuity $(v^C,u_1^C,\theta^C)(t,x)$ defined in \eqref{eq:ctident0} satisfies
    \begin{align} \label{eq:ctident1}
        (v^C-v_*,u_1^C - u_{1*}, \theta^C -\theta_*) & = O(1) \delta_C e^{-\frac{c_0x^2}{1+t}}, && \forall x<0, \notag\\
        (v^C-v^*,u_1^C - u_1^*, \theta^C - \theta^*) & = O(1) \delta_C e^{-\frac{c_0x^2}{1+t}}, && \forall x>0, \notag\\
        (\partial_x^n v^C, \partial_x^n \theta^C)(t,x) &= O(1)\delta_C (1+t)^{-\frac{n}{2}}e^{-\frac{c_0x^2}{1+t}}, && \forall x\in \R, \quad n=1,2,\ldots, \notag\\
        \partial_x^n u_1^C(t,x) &= O(1) \delta_C(1+t)^{-\frac{1+n}{2}}e^{-\frac{c_0x^2}{1+t}}, && \forall x\in \R, \quad n=1,2, \ldots, 
    \end{align}
    where $c_0>0$ is a constant.
\end{lemma}

It follows in particular from Lemma~\ref{lem:contact-wave-estimate0} that 
\[
    (v^C,u_1^C,\theta^C)(t,x) \to (v_*,u_{1*},\theta_*) \quad \text{as }x\to -\infty,
\]

and 
\[
    (v^C,u_1^C,\theta^C)(t,x) \to (v^*,u_1^*,\theta^*) \quad \text{as }x\to +\infty.
\]
{Furthermore, restoring the dependence on the heat-conductivity coefficient in terms of the Knudsen number \(\kappa\), the viscous contact wave $(v^{C,\kappa},u_1^{C,\kappa},\theta^{C,\kappa})$ with respect to $\kappa>0$ satisfies, for any \(1\le p<\infty\), 
\begin{equation}\label{eq:contma}
    \Bigl\| (v^{C,\kappa},u_1^{C,\kappa},\theta^{C,\kappa})(t,\cdot)-(v^c,u_1^c,\theta^c)(t,\cdot) \Bigr\|_{L^p(\R)}\le C\delta_C\kappa^{\frac{1}{2p}}(1+t)^{\frac{1}{2p}}.
\end{equation}
Consequently, for each fixed \(T>0\), 
\begin{align}\label{eq:hydrolimit-contact}
    (v^{C,\kappa},u_1^{C,\kappa},\theta^{C,\kappa})\longrightarrow (v^c,u_1^c,\theta^c)\quad \text{in }L^\infty(0,T;L^p(\R)) \quad \text{as }\kappa \to 0.
\end{align}
Here $(v^c,u_1^c,\theta^c)$ is the inviscid contact discontinuity. This convergence should be understood on a fixed finite time interval.}

\subsection{Boltzmann shock profile}
We briefly recall the mechanism behind these standard results. The steady Boltzmann
profile equation may be viewed as an infinite-dimensional dynamical system near an end
Maxwellian. For small-amplitude shocks, one constructs a local invariant/center manifold
around equilibrium and reduces the profile equation to a finite-dimensional slow dynamics
whose leading part is the corresponding compressible Navier--Stokes shock ODE. The
Boltzmann shock profile is then obtained as a heteroclinic orbit on this reduced
manifold connecting the two end states. In this framework, the exponential decay,
monotonicity, and higher-derivative bounds for the macroscopic profile follow from the
reduced shock dynamics, while the bounds for the microscopic component \(G^{S_i}\) and
the remainder \(\Pi_1^{S_i}\) are recovered through the macro--micro decomposition and
the coercive invertibility of the linearized collision operator on the microscopic
subspace; see \cite{LiuYu2013ARMA,pogan2018center}.

{
For \(i=1,3\), set
\[
z_i:=x-\sigma_i t,
\]
where \(\sigma_i\) is the \(i\)-th shock speed. We denote by
\[
F^{S_i,\kappa}(t,x,\xi):=F^{S_i,\kappa}(z_i,\xi)
=
F^{S_i,\kappa}(x-\sigma_i t,\xi)
\]
the traveling profile of the steady \(i\)-th Boltzmann shock profile in
Lagrangian coordinates. Hence, when acting on \(F^{S_i}\), we have
\[
\partial_t=-\sigma_i\partial_{z_i},
\qquad
\partial_x=\partial_{z_i}.
\]
Then \(F^{S_i,\kappa}\) satisfies
\begin{equation}\label{eq:bse}
-\sigma_i (F^{S_i,\kappa})_{z_i}
-\frac{u_1^{S_i,\kappa}-\xi_1}{v^{S_i,\kappa}}(F^{S_i,\kappa})_{z_i}
=
\frac{1}{\kappa}\mathcal N(F^{S_i},F^{S_i}).
\end{equation}
Here, the end states are 
\begin{align*}
& F^{S_1,\kappa}(-\infty,\xi)=M[U_-],
\qquad
F^{S_1,\kappa}(+\infty,\xi)=M[U_*], 
\qquad u_*=(u_{1*},0,0),\\
& F^{S_3,\kappa}(-\infty,\xi)=M[U^*],
\qquad
F^{S_3,\kappa}(+\infty,\xi)=M[U_+], \qquad u^*=(u^{1*},0,0).
\end{align*}

As in \eqref{eq:mami}, we decompose the shock profile into its macroscopic and microscopic parts:
\begin{align*}
& M^{S_i,\kappa}(z,\xi):=M_{[v^{S_i,\kappa},u^{S_i,\kappa},\theta^{S_i,\kappa}]}(z,\xi),\\
& F^{S_i,\kappa}=M^{S_i,\kappa}+G^{S_i,\kappa},
\qquad
M^{S_i,\kappa}=P_0^{[M^{S_i,\kappa}]}F^{S_i,\kappa},
\qquad
G^{S_i,\kappa}:=P_1^{[M^{S_i,\kappa}]}F^{S_i,\kappa},
\end{align*}
where
\begin{align*}
& P_0^{[M^{S_i,\kappa}]}g:=\sum_{k=0}^4\langle g,\chi_k^{S_i^0,\kappa}\rangle_{M^{S_i}}\chi_k^{S_i^0,\kappa},
\qquad
P_1^{[M^{S_i,\kappa}]}g:=g-P_0^{[M^{S_i,\kappa}]}g, \\
& \chi_0^{S_i^0,\kappa}=\frac{M^{S_i,\kappa}}{\sqrt{\rho^{S_i,\kappa}}},
\qquad
\chi_j^{S_i^0,\kappa}=\frac{\xi_j-u_j^{S_i,\kappa}}{\sqrt{R\theta^{S_i,\kappa}\rho^{S_i,\kappa}}}M^{S_i,\kappa},
\quad j=1,2,3,\\
& \chi_4^{S_i^0,\kappa}
=
\frac{1}{\sqrt{6\rho^{S_i,\kappa}}}
\left(
\frac{|\xi-u^{S_i,\kappa}|^2}{R\theta^{S_i,\kappa}}-3
\right)M^{S_i,\kappa}.
\end{align*}

Set $P_j:=P_j^{[M]},\,P_j^{S_i^0,\kappa}:=P_j^{[M^{S_i,\kappa}]}$ for $j=0,1$. By the macro--micro decomposition, the macroscopic part satisfies the fluid-type system
\begin{equation}\label{eq:NSls}
\begin{aligned}
& -\sigma_i (v^{S_i,\kappa})_{z_i}-(u_1^{S_i,\kappa})_{z_i}=0,\\
& -\sigma_i (u_1^{S_i,\kappa})_{z_i}+(p^{S_i,\kappa})_{z_i}
=
\frac{4}{3}\kappa\left(\frac{\mu(\theta^{S_i,\kappa})u_{1z_i}^{S_i,\kappa}}{v^{S_i,\kappa}}\right)_{z_i}
-\int \xi_1^2(\Pi_1^{S_i,\kappa})_{z_i}\,d\xi,\\
& u_j^{S_i,\kappa}\equiv0,\qquad j=2,3,\\
& -\sigma_i (\theta^{S_i,\kappa})_{z_i}+p^{S_i,\kappa}(u_1^{S_i,\kappa})_{z_i}
=
\kappa\left(\frac{\alpha_{\rm{th}}(\theta^{S_i,\kappa})\theta_{z_i}^{S_i,\kappa}}{v^{S_i,\kappa}}\right)_{z_i}
+\frac{4}{3}\kappa\mu(\theta^{S_i,\kappa})\frac{(u_{1z_i}^{S_i,\kappa})^2}{v^{S_i}}
\\
&\qquad \qquad \qquad \qquad \qquad \qquad \qquad -\int \xi_1\left(\frac{|\xi|^2}{2}-u_1^{S_i,\kappa}\xi_1\right)(\Pi_1^{S_i,\kappa})_{z_i}\,d\xi,
\end{aligned}
\end{equation}
while the microscopic part satisfies
\begin{equation}\label{eq:Gers}
\begin{aligned}
& -\sigma_i G_{z_i}^{S_i,\kappa}
-\frac{u_1^{S_i,\kappa}}{v^{S_i,\kappa}}G_{z_i}^{S_i,\kappa}
+\frac{1}{v^{S_i,\kappa}}P_1^{S_i^0,\kappa}(\xi_1 G_{z_i}^{S_i,\kappa})
+\frac{1}{v^{S_i,\kappa}}P_1^{S_i^0,\kappa}(\xi_1 M_{z_i}^{S_i,\kappa})
=
\frac{1}{\kappa}L_{M^{S_i,\kappa}}G^{S_i,\kappa}+\frac{1}{\kappa}\mathcal N(G^{S_i,\kappa},G^{S_i,\kappa}),\\
& G^{S_i,\kappa}
=
\kappa\frac{1}{v^{S_i,\kappa}}(L_{M^{S_i,\kappa}})^{-1}\!\Bigl(P_1^{S_i^0,\kappa}(\xi_1M_{z_i}^{S_i,\kappa})\Bigr)
+\Pi_1^{S_i,\kappa},\\
& \Pi_1^{S_i,\kappa}
:=
\kappa(L_{M^{S_i,\kappa}})^{-1}
\left(
-\sigma_i G_{z_i}^{S_i,\kappa}
-\frac{u_1^{S_i,\kappa}}{v^{S_i,\kappa}}G_{z_i}^{S_i,\kappa}
+\frac{1}{v^{S_i,\kappa}}P_1^{S_i^0,\kappa}(\xi_1G_{z_i}^{S_i,\kappa})
-\frac{1}{\kappa}\mathcal N(G^{S_i,\kappa},G^{S_i,\kappa})
\right).
\end{aligned}
\end{equation}

The following properties of the shock profile will be used later.

\begin{lemma}[\cite{HuangWangWangYang2013},\cite{LiuYu2013ARMA}]\label{lem:bs}
For a given right end state $(v_R,u_{1R},0,0,\theta_R)$, there exists a constant $C>0$ such that the following holds.
For any left end state $(v_L,u_{1L},0,0,\theta_L)$ sufficiently close to $(v_R,u_{1R},0,0,\theta_R)$ and connected to it by an $i$-shock curve, $i=1,3$, with shock strength
\[
\delta_i:=|v_R-v_L|
\sim |u_{1R}-u_{1L}|
\sim |\theta_R-\theta_L|,
\]
there exists a unique shock profile
\[
(v^{S_i,\kappa}(z),u^{S_i,\kappa}(z),\theta^{S_i,\kappa}(z))
\]
associated with \eqref{eq:bse}. Moreover,
\begin{align*}
& (v^{S_1,\kappa})'(z)<0,\qquad (u_1^{S_1,\kappa})'(z)<0,\qquad (\theta^{S_1,\kappa})'(z)>0,\\
& (v^{S_3,\kappa})'(z)>0,\qquad (u_1^{S_3,\kappa})'(z)<0,\qquad (\theta^{S_3,\kappa})'(z)<0,
\qquad \forall z\in\mathbb{R},\\
& |(v^{S_i,\kappa}(z)-v_L,\ u_1^{S_i,\kappa}(z)-u_{1L},\ \theta^{S_i,\kappa}(z)-\theta_L)|
\le
C\delta_i e^{-C\frac{\delta_i}{\kappa}|z|},
\qquad z<0,\\
& |(v^{S_i,\kappa}(z)-v_R,\ u_1^{S_i,\kappa}(z)-u_{1R},\ \theta^{S_i,\kappa}(z)-\theta_R)|
\le
C\delta_i e^{-C\frac{\delta_i}{\kappa}|z|},
\qquad z>0,\\
& |((v^{S_i,\kappa})'(z),(u_1^{S_i,\kappa})'(z),(\theta^{S_i,\kappa})'(z))|
\le
C\frac{\delta_i^2}{\kappa} e^{-C\frac{\delta_i}{\kappa}|z|},
\qquad \forall z\in\mathbb{R},\\
& |\partial_z^k ((v^{S_i,\kappa})(z),(u_1^{S_i,\kappa})(z),(\theta^{S_i,\kappa})(z))|
\le
C\frac{\delta_i^{k-1}}{\kappa^{k-1}}|((v^{S_i,\kappa})'(z),(u_1^{S_i,\kappa})'(z),(\theta^{S_i,\kappa})'(z))|, \quad k\ge 2.
\end{align*}
\end{lemma}

\begin{lemma} [\cite{HuangWangWangYang2013},\cite{LiuYu2013ARMA}]\label{lem:bs1}
Under the assumptions of Lemma~\ref{lem:bs}, there exists a unique Boltzmann shock profile $F^{S_i,\kappa}(z,\xi)$ solving \eqref{eq:bse}. Moreover,
\begin{align*}
& \left(\int \frac{(1+|\xi|)|G^{S_i,\kappa}|^2}{M_\#}\,d\xi\right)^{1/2}
\le
C\delta_i^2 e^{-C\frac{\delta_i}{\kappa}|z|},\\
& \left(\int \frac{(1+|\xi|)|\partial_z^k G^{S_i,\kappa}|^2}{M_\#}\,d\xi\right)^{1/2}
\le
C\frac{\delta_i^k}{\kappa^{k}}
\left(\int \frac{(1+|\xi|)|G^{S_i,\kappa}|^2}{M_\#}\,d\xi\right)^{1/2}, \quad k\ge1.
\end{align*}
where \(M_\#\) is a fixed global Maxwellian close to the shock profile. In addition, {after setting} $\kappa=1$ {for the a priori estimates,} we have the following estimate:
\begin{align*}
& \left|\int \xi_1\varphi_\ell(\xi)\Pi_1^{S_i}\,d\xi\right|
\le
C\delta_i
\left(\int \frac{(1+|\xi|)|G^{S_i}|^2}{M_\#}\,d\xi\right)^{1/2},
\qquad \ell=1,4,\\
& \left|\int \xi_1\varphi_\ell(\xi)(\Pi_1^{S_i})_x\,d\xi\right|
\le
C\delta_i^2
\left(\int \frac{(1+|\xi|)|G^{S_i}|^2}{M_\#}\,d\xi\right)^{1/2},
\qquad \ell=1,4,\\
& \int \xi_1\varphi_\ell(\xi)\Pi_1^{S_i}\,d\xi=0,
\qquad \ell=2,3,
\end{align*}
where \(M_\#\) is a fixed global Maxwellian close to the shock profile, and
\[
\varphi_\ell(\xi)=\xi_\ell,\quad \ell=1,2,3,
\qquad
\varphi_4(\xi)=\frac{|\xi|^2}{2}.
\]
{Here and below, we suppress the subscript} $\kappa${.}
\end{lemma}

\begin{lemma}[\cite{wang2025}]\label{lem:bs2}
Under the assumptions of Lemmas~\ref{lem:bs},\ref{lem:bs1} and after normalizing for a priori estimates, the following hold:
\begin{align*}
& |(u_1^{S_1})'-\sigma_*(v^{S_1})'|
\le C\delta_1 |(v^{S_1})'|,
\qquad
|(u_1^{S_3})'+\sigma^*(v^{S_3})'|
\le C\delta_3 |(v^{S_3})'|,\\
& |(\theta^{S_1})'+p_*(v^{S_1})'|
\le C\delta_1 |(v^{S_1})'|,
\qquad
|(\theta^{S_3})'+p^*(v^{S_3})'|
\le C\delta_3 |(v^{S_3})'|,
\end{align*}
and
\begin{align*}
& \left|
\frac{p^{S_1}-p_-}{v^{S_1}-v_-}
-
\frac{p^{S_1}-p_*}{v^{S_1}-v_*}
-
\frac{5p_*}{9v_*^2}
\left(
\frac{10\mu(\theta_*)-9\alpha_{\rm{th}}(\theta_*)}{10\mu(\theta_*)+3\alpha_{\rm{th}}(\theta_*)}+3
\right)
\right|
\le C\delta_1^2,\\
& \left|
\frac{p^{S_3}-p_+}{v^{S_3}-v_+}
-
\frac{p^{S_3}-p^*}{v^{S_3}-v^*}
-
\frac{5p^*}{9(v^*)^2}
\left(
\frac{10\mu(\theta^*)-9\alpha_{\rm{th}}(\theta^*)}{10\mu(\theta^*)+3\alpha_{\rm{th}}(\theta^*)}+3
\right)
\right|
\le C\delta_3^2
\end{align*}
where $p_*=\frac{2\theta_*}{3v_*}$, $p^*=\frac{2\theta^*}{3v^*}$, $\sigma_*=\sqrt{\frac{5p_*}{3v_*}}$, and $\sigma^*=\sqrt{\frac{5p^*}{3v^*}}$.
\end{lemma}
}
\section{Main proposition : a priori estimates}
\setcounter{equation}{0}
\subsection{Composite approximate wave}
We now introduce the composite approximate wave, consisting of the viscous contact wave and the two viscous shock waves shifted by the dynamical shifts $X_i(t)$, $i=1,3$, to be determined later. For convenience, for any macroscopic profile $h^{S_i}$ we write
\[
\bigl(h^{S_i}\bigr)^{-X_i}(t,x)
:=
h^{S_i}(x-\sigma_i t-X_i(t)), \qquad i=1,3
\]
and
the microscopic profile $g^{S_i}$ we write
\[
\bigl(g^{S_i}\bigr)^{-X_i}(t,x,\xi)
:=
g^{S_i}(x-\sigma_i t-X_i(t),\xi), \qquad i=1,3
\]
We define
\begin{equation}\label{eq:unvar}
\begin{aligned}
\bar v(t,x)
&:=
\bigl(v^{S_1}\bigr)^{-X_1}(t,x)
+
v^C(t,x)
+
\bigl(v^{S_3}\bigr)^{-X_3}(t,x)
-v_*-v^*,\\
\bar u(t,x) &:= (\bar u_1, 0, 0)(t,x), \\
\bar u_1(t,x)
&:=
\bigl(u_1^{S_1}\bigr)^{-X_1}(t,x)
+
u_1^C(t,x)
+
\bigl(u_1^{S_3}\bigr)^{-X_3}(t,x)
-u_{1*}-u_1^*,\\
\bar\theta(t,x)
&:=
\bigl(\theta^{S_1}\bigr)^{-X_1}(t,x)
+
\theta^C(t,x)
+
\bigl(\theta^{S_3}\bigr)^{-X_3}(t,x)
-\theta_*-\theta^*,\\
\bar U &:= \bigl(\bar v, \bar u_1,0,0,\bar \theta \bigr),\\
\bar p (t,x)&:= p(\bar v(t,x), \bar \theta(t,x)),\\
\shockp{p}{i} (t,x)&:=\shockw{p}{i}(t,x),\qquad i=1,3,\\
\bar M(t,x,\xi)
&:=M[\bar U](\xi),\\
\bar f(t,x,\xi)
&:=
\bar M(t,x,\xi)
+
\bigl(G^{S_1}\bigr)^{-X_1}(t,x,\xi)
+
\bigl(G^{S_3}\bigr)^{-X_3}(t,x,\xi).
\end{aligned}
\end{equation}

We record the elementary identity
\begin{equation}\label{eq:rbsws}
\begin{aligned}
\partial_t\left( \bigl(v^{S_i}\bigr)^{-X_i}\right)
&=
-(\sigma_i+\dot X_i(t))\,\partial_x \left( \bigl(v^{S_i}\bigr)^{-X_i} \right)\\
&=
-\dot X_i(t)\,\partial_x \left(\bigl(v^{S_i}\bigr)^{-X_i}\right)
+\partial_x \left(\bigl(u_1^{S_i}\bigr)^{-X_i}\right),
\qquad i=1,3,
\end{aligned}
\end{equation}
where in the second equality we used the profile equation
\[
-\sigma_i (v^{S_i})'-(u_1^{S_i})'=0.
\]
By combining the equations for the viscous contact wave and the two shifted shock profiles {(see} \eqref{eq:ctident} {and} \eqref{eq:NSls}{)}, we obtain that the composite ansatz $(\bar v,\bar u,\bar\theta)$ satisfies

\begin{align}\label{eq:ansasy}
&\bar v_t
+\dot X_1(t)\partial_x \shockw{v}{1}
+\dot X_3(t)\partial_x \shockw{v}{3}
-\bar u_{1x}
=0,\nonumber\\
&\bar u_{1t}
+\dot X_1(t)\partial_x \shockw{u_1}{1}
+\dot X_3(t)\partial_x \shockw{u_1}{3}
+\bar p_x\nonumber\\
&\qquad \qquad =
\frac{4}{3}\left(\frac{\mu(\bar\theta)\bar u_{1x}}{\bar v}\right)_x
+Q_1
-\sum_{i\in\{1,3\}}\int \xi_1^2\partial_x\shockw{\Pi_1}{i}\,d\xi,\nonumber\\
&\bar u_i=0,\qquad i=2,3,\nonumber\\
&\bar\theta_t
+\dot X_1(t)\partial_x \shockw{\theta}{1}
+\dot X_3(t)\partial_x \shockw{\theta}{3}
+\bar p\,\bar u_{1x}\nonumber\\
&\,=
\left(\frac{\alpha_{\rm{th}}(\bar\theta)\bar\theta_x}{\bar v}\right)_x
+\frac{4}{3}\mu(\bar\theta)\frac{\bar u_{1x}^2}{\bar v}
+Q_2 -\sum_{i\in\{1,3\}}\int \xi_1\left(\frac{|\xi|^2}{2}-\bigl(u_1^{S_i}\bigr)^{-X_i}\xi_1\right)\partial_x \shockw{\Pi_1}{i}\,d\xi,
\end{align}
where
\begin{equation}\label{eq:Q1Q2}
\begin{aligned}
Q_1
&:=
Q_1^C
+\Bigl(
\bar p
-\shockp{p}{1}
-\shockp{p}{3}
-p^C
\Bigr)_x\\
&\qquad
+\frac{4}{3}\left(
\frac{\mu\bigl(\shockw{\theta}{1}\bigr)\bigl(\partial_x\shockw{u_{1}}{1}\bigr)}{\shockw{v}{1}}
+
\frac{\mu\bigl(\shockw{\theta}{3}\bigr)\bigl(\partial_x\shockw{u_1}{3}\bigr)}{\shockw{v}{3}}
+
\frac{\mu(\theta^C)u_{1x}^C}{v^C}
-
\frac{\mu(\bar\theta)\bar u_{1x}}{\bar v}
\right)_x\\
&=: Q_1^C + Q_1^I,
\end{aligned}
\end{equation}
and
\begin{equation}\label{eq:Q2Q2}
\begin{aligned}
Q_2
&:=
Q_2^C
+
\Bigl(
\bar p\,\bar u_{1x}
-\shockp{p}{1}\bigl(\partial_x\shockw{u_1}{1}\bigr)
-\shockp{p}{3}\bigl(\partial_x\shockw{u_1}{3}\bigr)
-p^Cu_{1x}^C
\Bigr)\\
&\quad
+\left(
\frac{\alpha_{\rm{th}}\bigl(\shockw{\theta}{1}\bigr)\bigl(\partial_x\shockw{\theta}{1}\bigr)}{\shockw{v}{1}}
+
\frac{\alpha_{\rm{th}}\bigl(\shockw{\theta}{3}\bigr)\bigl(\partial_x\shockw{\theta}{3}\bigr)}{\shockw{v}{3}}
+
\frac{\alpha_{\rm{th}}(\theta^C)\theta_x^C}{v^C}
-
\frac{\alpha_{\rm{th}}(\bar\theta)\bar\theta_x}{\bar v}
\right)_x\\
&\quad
+\frac{4}{3}\mu(\theta^C)\frac{(u_{1x}^C)^2}{v^C}
+\frac{4}{3}\mu\bigl(\shockw{\theta}{1}\bigr)\frac{\bigl(\partial_x\shockw{u_1}{1}\bigr)^2}{\shockw{v}{1}}
+\frac{4}{3}\mu\bigl(\shockw{\theta}{3}\bigr)\frac{\bigl(\partial_x\shockw{u_1}{3}\bigr)^2}{\shockw{v}{3}}
\\
&\quad -\frac{4}{3}\mu(\bar\theta)\frac{\bar u_{1x}^2}{\bar v}\\
&=: Q_2^C + Q_2^I
\end{aligned}
\end{equation}

\subsection{Perturbation variables}
To implement the relative entropy method, we introduce the perturbation variables around the composite approximate wave \eqref{eq:unvar}. Define
\begin{equation}\label{eq:pevar}
\begin{aligned}
\phi(t,x)
&:=v(t,x)-\bar v(t,x),\\
\psi (t,x) &= (\psi_1,\psi_2,\psi_3)(t,x):=u(t,x)-\bar u(t,x),\\
\zeta (t,x) &:= \theta(t,x)-\bar \theta(t,x), \\
\widetilde f(t,x,\xi)
    &:=f(t,x,\xi)
    -
    \bigl(F^{S_1}\bigr)^{-X_1}(t,x,\xi)
    -
    \bigl(F^{S_3}\bigr)^{-X_3}(t,x,\xi),\\
\widetilde G(t,x,\xi)
&:=
G(t,x,\xi)
-
\bigl(G^{S_1}\bigr)^{-X_1}(t,x,\xi)
-
\bigl(G^{S_3}\bigr)^{-X_3}(t,x,\xi).
\end{aligned}
\end{equation}
We summarize the notation used in the perturbation formulation. 
The function \(f\) denotes the dynamic solution to the Boltzmann equation, 
whereas \(\bar f\) denotes the composite approximate distribution. 
We denote by \(\bar M\) the local Maxwellian associated with the composite 
macroscopic profile. After subtracting the shock contributions, we write 
\(\widetilde f\) for the shock-subtracted distribution and 
\(\widetilde G\) for the corresponding shock-subtracted microscopic perturbation.

Subtracting the composite ansatz system \eqref{eq:ansasy} from the full system, we obtain
\begin{equation}\label{eq:pesy1}
\begin{aligned}
& \phi_t
-\dot X_1(t)\left(\bigl(v^{S_1}\bigr)^{-X_1}\right)_x
-\dot X_3(t)\left(\bigl(v^{S_3}\bigr)^{-X_3}\right)_x
-\psi_{1x}=0,\\
& \psi_{1t}
-\dot X_1(t)\left(\bigl(u_1^{S_1}\bigr)^{-X_1}\right)_x
-\dot X_3(t)\left(\bigl(u_1^{S_3}\bigr)^{-X_3}\right)_x
+(p-\bar p)_x\\
& \qquad=
\frac{4}{3}\left(
\frac{\mu(\theta)u_{1x}}{v}
-
\frac{\mu(\bar\theta)\bar u_{1x}}{\bar v}
\right)_x
-Q_1
-\int \xi_1^2\Bigl(\Pi_1-\bigl(\Pi_1^{S_1}\bigr)^{-X_1}-\bigl(\Pi_1^{S_3}\bigr)^{-X_3}\Bigr)_x\,d\xi,
\end{aligned}
\end{equation}
\begin{equation}\label{eq:pesy2}
\begin{aligned}
& \psi_{it}
=
\left(\frac{\mu(\theta)\psi_{ix}}{v}\right)_x
-\int \xi_1\xi_i \Pi_{1x}\,d\xi,
\qquad i=2,3,\\
& \zeta_t
-\dot X_1(t)\left(\bigl(\theta^{S_1}\bigr)^{-X_1}\right)_x
-\dot X_3(t)\left(\bigl(\theta^{S_3}\bigr)^{-X_3}\right)_x
+\bigl(pu_{1x}-\bar p\,\bar u_{1x}\bigr)\\
& \qquad=
\left(
\frac{\alpha_{\rm{th}}(\theta)\theta_x}{v}
-
\frac{\alpha_{\rm{th}}(\bar\theta)\bar\theta_x}{\bar v}
\right)_x
+\frac{4}{3}\left(
\frac{\mu(\theta)u_{1x}^2}{v}
-
\frac{\mu(\bar\theta)\bar u_{1x}^2}{\bar v}
\right)
+\frac{\mu(\theta)(\psi_{2x}^2+\psi_{3x}^2)}{v}
-Q_2\\
& \qquad\quad
-\int \xi_1\frac{|\xi|^2}{2}\Bigl(\Pi_1-\bigl(\Pi_1^{S_1}\bigr)^{-X_1}-\bigl(\Pi_1^{S_3}\bigr)^{-X_3}\Bigr)_x\,d\xi
+\sum_{j=2}^3 \psi_j\int \xi_1\xi_j\Pi_{1x}\,d\xi\\
& \qquad\quad
+\left(
u_1\int \xi_1^2\Pi_{1x}\,d\xi
-\sum_{i\in\{1,3\}}\bigl(u_1^{S_i}\bigr)^{-X_i}\int \xi_1^2\left(\bigl(\Pi_1^{S_i}\bigr)^{-X_i}\right)_x\,d\xi
\right),
\end{aligned}
\end{equation}
and
\begin{equation}\label{eq:pesy3}
\begin{aligned}
\widetilde G_t-L_M\widetilde G
&=
\dot X_1(t)\left(\bigl(G^{S_1}\bigr)^{-X_1}\right)_x
+\dot X_3(t)\left(\bigl(G^{S_3}\bigr)^{-X_3}\right)_x
+\frac{u_1}{v}\widetilde G_x
-\frac{1}{v}P_1(\xi_1\widetilde G_x)\\
&\quad
+\sum_{i\in\{1,3\}}\left(\frac{u_1}{v}-\frac{(u_1^{S_i})^{-X_i}}{(v^{S_i})^{-X_i}}\right)\left(\bigl(G^{S_i}\bigr)^{-X_i}\right)_x\\
&\quad
-\sum_{i\in\{1,3\}}\left(
\frac{1}{v}P_1\bigl(\xi_1\partial_x\shockw{G}{i}\bigr)
-\frac{1}{(v^{S_i})^{-X_i}}P_1^{S_i}\bigl(\xi_1\partial_x\shockw{G}{i}\bigr)
\right)\\
&\quad
-\frac{1}{v}P_1(\xi_1 M_x)
+\sum_{i\in\{1,3\}}\frac{1}{(v^{S_i})^{-X_i}}P_1^{S_i}\bigl(\xi_1\partial_x\shockw{M}{i}\bigr)
\\
&\quad
+\mathcal N(\widetilde G,\widetilde G) +\mathcal N\bigl((G^{S_1})^{-X_1},\widetilde G\bigr) +\mathcal N\bigl(\widetilde G,(G^{S_1})^{-X_1}\bigr)
\\
&\quad +\mathcal N\bigl((G^{S_3})^{-X_3},\widetilde G\bigr)
+\mathcal N\bigl(\widetilde G,(G^{S_3})^{-X_3}\bigr)\\
&\quad +\mathcal N\bigl((G^{S_1})^{-X_1},(G^{S_3})^{-X_3}\bigr) +\mathcal N\bigl((G^{S_3})^{-X_3},(G^{S_1})^{-X_1}\bigr)
\\
&\quad +\sum_{i=1,3}\bigl(L_M-L_i^S\bigr)\bigl(G^{S_i}\bigr)^{-X_i}
\end{aligned}
\end{equation}

where

\begin{align*}
L_i^{S_i}:=L_{\left(M^{S_i}\right)^{-X_i}},\quad i=1,3, \quad P_j^{S_i} := P_j^{[\shockw{M}{i}]}, \quad i=1,3,\,j=0,1.
\end{align*}

\subsection{Dynamical shifts and Weight function}
To handle the two shock waves, we introduce weight functions associated with the $1$-shock and the $3$-shock; see \cite{huang2026timeasymptoticstabilitycompositewave,han2023large}. Define
\begin{equation}\label{eq:aweight}
\begin{aligned}
a_1(z)&:=1+\frac{1}{\sqrt{\delta_1}}\bigl(v^{S_1}(z)-v_-\bigr),\\
a_3(z)&:=1+\frac{1}{\sqrt{\delta_3}}\bigl(v^{S_3}(z)-v^*\bigr),
\end{aligned}
\end{equation}
where
\[
\delta_1:=|v_--v_*|,
\qquad
\delta_3:=|v_+-v^*|
\]
denote the strengths of the $1$-shock and the $3$-shock, respectively.

For the shifted shock profiles, we set
\begin{equation}\label{eq:wefun}
a(t,x)
:=
a_1(x-\sigma_1 t-X_1(t))
+
a_3(x-\sigma_3 t-X_3(t))
-1.
\end{equation}
Before defining the shift coefficients, we recall the constants
\(p_*, p^*, \sigma_*,\sigma^*\), which are fixed throughout the argument.
\begin{align*}
p_* = \frac{2\theta_*}{3v_*},\quad p^*=\frac{2\theta^*}{3v^*},\quad \sigma_* = \sqrt{\frac{5p_*}{3v_*}},\quad \sigma^*=\sqrt{\frac{5p^*}{3v^*}}.
\end{align*}
These constants are chosen according to the admissible range of the shock
strengths and the associated shock speeds, and they will be used in the
definition of the coefficients \(\mathfrak m_1\) and \(\mathfrak m_3\) below.
We now define the dynamical shifts $(X_1,X_3)$ by the system of ordinary differential equations
\begin{equation}\label{eq:shift}
\begin{aligned}
\dot X_1(t)
&=
-\frac{\mathfrak m_1}{\delta_1}
\int_{\mathbb R}
a
\left(
\left((u_1^{S_1})^{-X_1}\right)_x\psi_1
+
\frac{\left((v^{S_1})^{-X_1}\right)_x\,\y{p}}{\y{v}}\phi
+
\frac{\left((\theta^{S_1})^{-X_1}\right)_x}{\y{\theta}}\zeta
\right)\,dx,\\
\dot X_3(t)
&=
-\frac{\mathfrak m_3}{\delta_3}
\int_{\mathbb R}
a
\left(
\left((u_1^{S_3})^{-X_3}\right)_x\psi_1
+
\frac{\left((v^{S_3})^{-X_3}\right)_x\,\y{p}}{\y{v}}\phi
+
\frac{\left((\theta^{S_3})^{-X_3}\right)_x}{\y{\theta}}\zeta
\right)\,dx,\\
X_1(0)&=X_3(0)=0,
\end{aligned}
\end{equation}
where the constants $\mathfrak m_i$, $i=1,3$ are {defined as follows:}

\begin{align}
\mathfrak m_1 := \frac{20}{3}\frac{p_*}{\sigma_*^3v_*^2}\frac{5+3\gamma}{10+3\gamma}, \qquad \mathfrak m_3:=\frac{20}{3}\frac{p^*}{(\sigma^*)^3(v^*)^2}\frac{5+3\gamma}{10+3\gamma}.
\end{align}

To separate the two shock regions, we introduce the cutoff functions $\varphi_1$ and $\varphi_3$ by
\begin{equation}\label{eq:intske}
\varphi_1(t,x)
:=
\begin{cases}
1, & x\le \dfrac{X_1(t)+\sigma_1 t}{2},\\[2mm]
0, & x\ge \dfrac{X_3(t)+\sigma_3 t}{2},\\[2mm]
\text{linearly decreasing from $1$ to $0$}, & \dfrac{X_1(t)+\sigma_1 t}{2}<x<\dfrac{X_3(t)+\sigma_3 t}{2},
\end{cases}
\end{equation}

and 

\begin{equation}\label{eq:intske1}
\varphi_3(t,x):=1-\varphi_1(t,x).
\end{equation}

\subsection{Main proposition}

To prove the global existence on the time interval $[0,T]$, we shall close the following a priori estimate. Define

\begin{equation}\label{eq:mic2}
    \begin{aligned}
    & \widetilde G_{\text{rem}}
    :=
    \widetilde G-\widetilde G_C,
    \qquad
    \widetilde G_C
    :=
    \frac{3}{2v\theta}
    L_M^{-1}P_1
    \left[
    \xi_1M
    \left(
    \xi_1u_{1x}^C+\frac{|\xi-u|^2}{2\theta}\theta_x^C
    \right)
    \right].
    \end{aligned}
\end{equation}
and 
\begin{equation}\label{eq:priass}
\begin{aligned}
\mathcal{E}(T)^2
:=
\sup_{t\in[0,T]}
\Biggl\{
&\|(\phi,\psi,\zeta)(t)\|_{H^1_x}^2
+\int_{\mathbb R}\norm{\wtilde G_{\text{rem}}}_{M_\#}^2\,+\norm{\wtilde G_x}_{M_\#}^2\,dx\\
&\qquad+\int_{\mathbb R}\norm{\wtilde G_t}_{M_\#}^2\,+\norm{\wtilde f_{xx}}_{M_\#}^2\,+\norm{\wtilde f_{tx}}_{M_\#}^2\,dx\Biggr\}.
\end{aligned}
\end{equation}

The main a priori estimate is stated as follows.

\begin{proposition}\label{prop:priest}
Let $U_+:=(v_+,u_{1+},0,0,\theta_+)\in\mathbb R_+\times\mathbb R^3\times\mathbb R_+$. Then there exist positive constants $\delta_0$, $\varepsilon$, and $C$, together with a global Maxwellian $M_\#:=M[U_\#]$, independent of $T$, such that the following holds.

Suppose that $(U,f)=(v,u,\theta,f)$ solves \eqref{eq:NSF1} and \eqref{eq:bel} on $[0,T]$ for some $T>0$. Set $\delta:=\delta_1+\delta_C+\delta_3$. For any $0<\delta<\delta_0$, let $(\bar v,\bar u,\bar\theta)$ be the composite wave defined in \eqref{eq:unvar}, where the dynamical shifts $X_1$ and $X_3$ are the absolutely continuous solutions to \eqref{eq:shift} associated with the weight function \eqref{eq:wefun}. Assume that
\begin{align} \label{eq:prioriass}
\mathcal{E}(T)^2\le \varepsilon^2.
\end{align}
Then
\begin{equation}\label{eq:pries}
\begin{aligned}
\mathcal{E}(T)^2
&+\sum_{i=1,3}\int_0^T \delta_i |\dot X_i(t)|^2\,dt
+\sum_{i=1,3}\int_0^T \mathcal G_i^S(t)\,dt\\
&+\sum_{|\beta|=1}\int_0^T \|\partial^\beta(\phi,\psi,\zeta)(t)\|_{L^2_x}^2\,dt+\int_0^T
\|(\phi_{xx},\psi_{xx},\zeta_{xx},\phi_{xt},\psi_{xt},\zeta_{xt})(t)\|_{L^2_x}^2\,dt\\
&+\int_0^T\int_{\mathbb R} \norm{\widetilde G_{\textup{rem}}}_{\nu,M_\#}^2 + \norm{\widetilde G_t}_{\nu,M_\#}^2 + \norm{\widetilde G_x}_{\nu,M_\#}^2 + \norm{\widetilde G_{tx}}_{\nu,M_\#}^2 + \norm{\widetilde G_{xx}}_{\nu,M_\#}^2 \,dx\,dt\\
&\le
C\bigl(\mathcal{E}(0)^2+\delta_0^{1/2}\bigr),
\end{aligned}
\end{equation}
where
\begin{equation}\label{eq:skGt}
\mathcal G_i^S(t)
:=
\int_{\mathbb R}
\bigl|\partial_x(v^{S_i})^{-X_i}(t,x)\bigr|
\,|\varphi_i(t,x)|^2
\,|(\phi,\psi,\zeta)(t,x)|^2\,dx,
\qquad i=1,3,
\end{equation}
and $\varphi_i$ is defined in \eqref{eq:intske}.
\end{proposition}

The remainder of the paper is devoted to the proof of Proposition~\ref{prop:priest}. The global existence argument is then completed by combining Proposition~\ref{prop:priest} with the local-in-time existence result proved in Appendix.

\subsection{Wave interaction estimates}
First, by the a priori assumption \eqref{eq:prioriass} together with the Sobolev embedding, we have
\[
    \bigl\|(\phi,\psi,\zeta)\bigr\|_{L^\infty((0,T)\times \R)} \le C\varepsilon
\]
In particular, \(v,u,\theta \in L^\infty((0,T)\times\R)\).

Next, using the ODE \eqref{eq:shift} and Lemma~\ref{lem:bs}, we obtain
\begin{align}
    |\dot{X}_1(t)|+|\dot{X}_3(t)| &\le \sum_{i\in\{1,3\}} \frac{C}{\delta_i} \bigl\|(\phi,\psi,\zeta)\bigr\|_{L^\infty(\R)}\int_\R |\partial_x (\shockw vi)|\,dx \notag\\
    &\le C\bigl\|(\phi,\psi,\zeta)\bigr\|_{L^\infty(\R)} \le C \varepsilon . \label{eq:shift-speed-bound}
\end{align}

Hence, by taking \(\varepsilon>0\) sufficiently small, we may assume that 
\[
    |\dot{X}_1(t)|+|\dot{X}_3(t)| \le \varepsilon
\]
with \(\varepsilon\) small compared to the shock speeds. Since \(X_1(0)=X_3(0)=0\), it follows from \eqref{eq:shift-speed-bound} that 
\[
    X_1(t)\le -\frac{\sigma_1}{2}t, \qquad X_3(t)\ge -\frac{\sigma_3}{2}t, \qquad t>0.
\]
Equivalently, 
\begin{align}\label{eq:shiftsp1}
    X_1(t)+\sigma_1 t \le \frac{\sigma_1}{2} t,\qquad
    X_3(t)+\sigma_3 t \ge \frac{\sigma_3}{2} t, \qquad t>0.
\end{align}
This proves that the shifts are well-separated.

The following lemma quantifies the weak interaction between the two separated shock waves.
\begin{lemma}[\cite{huang2026timeasymptoticstabilitycompositewave}]\label{lem:intsk}
Assume $X_1(t)\leq -\frac{\sigma_1}{2}t$ and $X_3(t)\geq-\frac{\sigma_3}{2}t$. There exist positive constants $\delta_0$ and $C$ such that if
\[
\delta_1,\delta_C,\delta_3\in(0,\delta_0),
\]
then for all $t>0$ and $x\in\mathbb R$,
\begin{equation}\label{eq:intsk1}
\begin{aligned}
&\varphi_3(t,x)\,\bigl|\partial_x(v^{S_1})^{-X_1}(t,x)\bigr|
\le
C\delta_1^2 e^{-C\delta_1 t},
\quad
\varphi_1(t,x)\,\bigl|\partial_x(v^{S_3})^{-X_3}(t,x)\bigr|
\le
C\delta_3^2 e^{-C\delta_3 t},\\
&\int_{\mathbb R}\varphi_3(t,x)\,\bigl|\partial_x(v^{S_1})^{-X_1}(t,x)\bigr|\,dx
\le
C\delta_1 e^{-C\delta_1 t},
\quad
\int_{\mathbb R}\varphi_1(t,x)\,\bigl|\partial_x(v^{S_3})^{-X_3}(t,x)\bigr|\,dx
\le
C\delta_3 e^{-C\delta_3 t}.
\end{aligned}
\end{equation}
\end{lemma}

\begin{lemma}[\cite{huang2026timeasymptoticstabilitycompositewave}]\label{lem:wave-interaction-L2}
Assume $X_1(t)\leq -\frac{\sigma_1}{2}t$ and $X_3(t)\geq-\frac{\sigma_3}{2}t$. There exists a positive constant \(C\), independent of \(T\), \(\delta_1\), \(\delta_3\), and \(\delta_C\), such that for all \(t\in[0,T]\),
\begin{align}
\label{eq:wave-interaction-L2}
\left\|
|\partial_x(v^{S_1})^{-X_1}|
\,|(v^C-v_*,\theta^C-\theta_*)|
\right\|_{L^2_x}
&\le
C\delta_1^{3/2}\delta_C e^{-C \delta_1 t},
\\
\label{eq:wave-interaction-L2b}
\left\|
|\partial_x(v^{S_1})^{-X_1}|
\,|(v^{S_3}-v^*,\theta^{S_3}-\theta^*)^{-X_3}|
\right\|_{L^2_x}
&\le
C\delta_1^{3/2}\delta_3
\bigl(e^{-C\delta_1 t}+e^{-C\delta_3 t}\bigr).
\end{align}
Similarly,
\begin{align}
\label{eq:wave-interaction-L2-3}
\left\|
|\partial_x(v^{S_3})^{-X_3}|
\,|(v^C-v^*,\theta^C-\theta^*)|
\right\|_{L^2_x}
&\le
C\delta_3^{3/2}\delta_C e^{-C\delta_3 t},
\\
\label{eq:wave-interaction-L2-3b}
\left\|
|\partial_x(v^{S_3})^{-X_3}|
\,|(v^{S_1}-v_*,\theta^{S_1}-\theta_*)^{-X_1}|
\right\|_{L^2_x}
&\le
C\delta_3^{3/2}\delta_1
\bigl(e^{-C\delta_1 t}+e^{-C\delta_3 t}\bigr).
\end{align}
\end{lemma}

\vspace{0.5cm}

 \begin{lemma}[Derivative wave-interaction estimate]\label{lem:wave-interaction-derivative}
For \(|\alpha|=1,2\), there exists a positive constant \(C\), independent of \(T\), \(\delta_1\), \(\delta_3\), and \(\delta_C\), such that for all \(t\in[0,T]\),
\begin{align}
\label{eq:wave-interaction-derivative-1}
&\int_{\mathbb R}
\bigl|\partial_x^\alpha(v^C-v_*)\bigr|^2 \bigl|(v^{S_1})_x^{-X_1}\bigr|^2\,dx
+
\int_{\mathbb R}
\bigl|\partial_x^\alpha\bigl((v^{S_3}-v^*)^{-X_3}\bigr)\bigr|^2
\bigl|(v^{S_1})_x^{-X_1}\bigr|^2\,dx
\nonumber\\
&\qquad\le
C\delta_1^{3}
\delta_C^2 e^{-C\delta_1 t}
+
C\delta_1^3\delta_3^{2(|\alpha|+1)}\bigl(e^{-C\delta_3 t}+e^{-C\delta_1 t}\bigr).
\end{align}
Similarly,
\begin{align}
\label{eq:wave-interaction-derivative-3}
&\int_{\mathbb R}
\bigl|\partial_x^\alpha(v^C-v^*)\bigr|^2 \bigl|(v^{S_3})_x^{-X_3}\bigr|^2\,dx
+
\int_{\mathbb R}
\bigl|\partial_x^\alpha\bigl((v^{S_1}-v_*)^{-X_1}\bigr)\bigr|^2
\bigl|(v^{S_3})_x^{-X_3}\bigr|^2\,dx
\nonumber\\
&\qquad\le
C\delta_3^{3}
\delta_C^2 e^{-C\delta_3 t}
+
C\delta_3^3\delta_1^{2(|\alpha|+1)}\bigl(e^{-C\delta_3 t}+e^{-C\delta_1 t}\bigr).
\end{align}
\end{lemma}

\begin{proof}
  {As in Lemma}~\ref{lem:wave-interaction-L2}, we split the real line and compute the wave interaction. Using Lemma \eqref{lem:bs} and \eqref{eq:Theta-sim-x}, we can obtain \eqref{eq:wave-interaction-derivative-1} and \eqref{eq:wave-interaction-derivative-3}.
\end{proof}

\vspace{0.5cm}

\begin{lemma}[Wave-interaction estimate, \cite{kang2025time}]\label{lem:QI-term} Let $X_i$ ($i=1,\,3$) be the shift defined by \eqref{eq:shift} and $Q_i^I$($i=1,\,2$) be defined in \eqref{eq:Q1Q2} and \eqref{eq:Q2Q2}. There exists a positive constant $C$, independent of $T$, $\delta_1$, $\delta_3$, and $\delta_C$, such that for all $t\in[0,T]$,
\begin{equation}\label{eq:intskfull}
    \norm{Q_i^I}_{L^2} \leq C\delta_1^{\frac{3}{2}}(\delta_3+\delta_C)e^{-C\delta_1t}+C\delta_3^{\frac{3}{2}}(\delta_1+\delta_C)e^{-C\delta_3t}
\end{equation}
\end{lemma}

\vspace{0.5 cm}

\section{Proof of the main result}
\setcounter{equation}{0}

We here complete the proof of Theorem \ref{thm:main} by combining the a priori estimates with the local existence.  The proof of Proposition \ref{prop:priest} will be handled in the subsequent sections.
\subsection{Global continuation}
To complete the continuation argument, we establish a local-in-time existence result around the unshifted composite approximate wave. Define
\begin{equation}\label{eq:lceans}
\begin{bmatrix}
\hat v\\
\hat u_1\\
\hat u_2\\
\hat u_3\\
\hat \theta
\end{bmatrix}(t,x)
:=
\begin{bmatrix}
v^{S_1}(x-\sigma_1 t)+v^C(t,x)+v^{S_3}(x-\sigma_3 t)-v_*-v^*\\
u_1^{S_1}(x-\sigma_1 t)+u_1^C(t,x)+u_1^{S_3}(x-\sigma_3 t)-u_{1*}-u_1^*\\
0\\
0\\
\theta^{S_1}(x-\sigma_1 t)+\theta^C(t,x)+\theta^{S_3}(x-\sigma_3 t)-\theta_*-\theta^*
\end{bmatrix},
\end{equation}
where two intermediate states
\[
U_* = (v_*,u_*,\theta_*), \quad U^* = (v^*,u^*,\theta^*), \quad u_*=(u_{1*},0,0), \quad u^*=(u^*_1,0,0)
\]
and let
\begin{equation}\label{eq:hatM}
\hat M:=M[\hat U],\qquad \hat U :=(\hat v,\hat u,\hat\theta).
\end{equation}

We shall use the following local existence result.

\begin{proposition}\label{prop:lcext}
For any sufficiently small constant $R>0$, there exists a positive time $T_{\mathrm{loc}}=T_{\mathrm{loc}}(R)$ such that the following holds. Define $h:= f-\hat M$. 
Assume that
\begin{equation}\label{eq:lcext-ass}
\sum_{|\alpha|+|\beta|\le2,\ |\beta|\le1} \Bigl\|\norm{\partial_x^\alpha\partial_t^\beta(f_0-\hat M)}_{M_\#} \Bigr\|_{L^2_x} \le \frac{R}{2}
\end{equation}
for some global Maxwellian $M_\#$. Then the Cauchy problem for the normalized Lagrangian Boltzmann equation \eqref{eq:bel} admits a unique nonnegative solution $f$ on $[0,T_*]$ such that
\begin{equation}\label{eq:lcext-conc}
\sup_{t\in[0,T_{\mathrm{loc}}]}
\sum_{|\alpha|+|\beta|\le2,\ |\beta|\le1} \Bigl\|\norm{\partial_x^\alpha\partial_t^\beta \,h}_{M_\#} \Bigr\|_{L^2_x} 
\le R.
\end{equation}
Here, the time derivatives of $f_0$  are understood through the Boltzmann equation evaluated at $t=0$. 
\end{proposition}

By the shift estimate \eqref{eq:shift-speed-bound},
\begin{align*}
\norm{v^{S_i}-\shockw{v}{i}}_{H^1}^2 \le& C \delta_i + \int_{-|X_i|}^0 (v^{S_i}-v_R)^2 dx + \int_0^{|X_i|}(v^{S_i}-v_L)^2 dx \\
\le& C\delta_i(1+t)
\end{align*}
where $v_R$ is the right state of {the }$i${-shock} and $v_L$ is the left state of {the }$i${-shock} for $i=1,3$. Therefore, in the continuation argument, all perturbation norms are understood with respect to the shifted composite profile $\bar{U}$, not the unshifted profile $\hat U$. This is legitimate because the local well-posedness and coercivity estimates are invariant under spatial translations of the shock profiles.

We now verify the smallness of the initial energy.  Applying
\eqref{eq:initda} at \(t=0\), gives
\begin{align}\label{eq:pxt}
\begin{aligned} 
\mathcal E(0)^2 
&\le
C\mathcal E_{\mathrm{ini}}^2
+
C(\delta_1+\delta_C+\delta_3)^2 .
\end{aligned}
\end{align}
{
We first justify that the microscopic part can be controlled by \(f-\hat M\). Set \(U=(v,u,\theta)\), \(\hat U=(\hat v,\hat u, \hat \theta) \), and \(h= f_0-\hat M_0\) where $\hat{M}_0 := M[\hat{U}|_{t=0}]$. Since \(G=f-M[U]\) is microscopic part, 
\begin{align*}
    \int_{\R^3} G\psi(\xi)\,d\xi =0,\qquad \psi=(1,\xi,|\xi|^2).
\end{align*} 
Hence 
\begin{align*}
    \mathcal A(U) -\mathcal A(\hat U) = \int_{\R^3} h\,\psi(\xi)\,d\xi \qquad \mathcal A(U):=\int_{\R^3}M[U]\psi(\xi)\,d\xi = \Bigl(
    \frac1{v} ,\, \frac{u}{v} ,\, \frac{1}{v}\Bigl(\frac{|u|^2}{2}+\frac32 \theta\Bigr)
    \Bigr).
\end{align*}
The moment map \(\mathcal A\) is invertible on the compact set under consideration;
\begin{align*}
|U-\hat U| \le C\Bigl|\int_{\R^3} h\, \psi(\xi)\, d\xi \Bigr| \le C \bigl\|h\bigr\|_{M_\#}.
\end{align*}
Consequently,
\begin{align*}
    \norm{U-\hat U}_{L_x^2}^2 \le C \norm{f-\hat M}_{L^2_x(L_\xi^2(M_\#))}^2.
\end{align*}
Moreover, since 
\begin{align}\label{eq:GfMuM}
    G = (f-\hat M) - (M[U]-\hat M),
\end{align}
the smooth dependence of the Maxwellian on \(U\) gives 
\begin{align*}
    \int \norm{G}_{M_\#}^2 dx \le C \norm{f-\hat M}_{L_x^2(L_\xi(M_\#))}^2 \qquad (\text{by }\eqref{eq:Maxwellian-Lipschitz-single}). 
\end{align*}
So, 
\begin{align*}
    \int \norm{\widetilde G_{\mathrm{rem}}}_{M_\#}^2 dx \le C \norm{f-\hat M}_{L_x^2(L_\xi(M_\#))}^2 + C(\delta_1+\delta_C+\delta_3)^2.
\end{align*}
Note that 
\[
\begin{aligned}
\left\|f_0-\hat M_0\right\|_{L_x^2(L_\xi^2(M_\#))}^2 &\le \left\|f_0-M[U^E_0]+M[U^E_0]-\hat M_0\right\|_{L_x^2(L_\xi^2(M_\#))}^2\\
&\le \left\|f_0-M[U^E_0]\right\|_{L_x^2(L_\xi^2(M_\#))}^2 +\left\|M[U^E_0]-\hat M_0\right\|_{L_x^2(L_\xi^2(M_\#))}^2 \\
&\le \left\|f_0-M[U^E_0]\right\|_{L_x^2(L_\xi^2(M_\#))}^2 +\left\|U_0^E-\hat U|_{t=0}\right\|_{L_x^2}^2 \qquad (\text{by }\eqref{eq:Maxwellian-Lipschitz-single})\\
&\le \mathcal E_{\mathrm{ini}}^2 + C(\delta_1+\delta_C+\delta_3)^2
\end{aligned}
\]
at time $t=0$. 
Similarly, we can apply for $\partial^{\alpha}(U-\hat{U})$ where $\partial^\alpha = \partial^{\alpha_0}_t\partial^{\alpha_1}_x$ and $|\alpha|=|\alpha_0|+|\alpha_1|\le 2,\, |\alpha_0|\le 1$.

When $|\alpha|=1$, 

\begin{align*}
    \Bigl|\int_{\mathbb \R^3} \partial^\alpha h \psi(\xi) d\xi \Bigr|&=\Bigl| D\mathcal A(U)\partial^\alpha U - D\mathcal A(\hat U) \partial^\alpha \hat U \Bigr| \\
    &= \Bigl|D \mathcal A(U) \partial^\alpha(U-\hat U) + \bigl(D\mathcal A(U)-D\mathcal A(\hat U)\bigr)\partial^\alpha \hat U\Bigr|
\end{align*}
Since $D\mathcal A$ is invertible, 
\begin{align*}
     |\partial^\alpha(U-\hat U)| &\le C\Bigl|\int_{\mathbb \R^3} \partial^\alpha h \psi(\xi) d\xi \Bigr| + |\partial^\alpha \hat U||D\mathcal A(U)-D\mathcal A(\hat U)| \\
    &\le  C\Bigl|\int_{\mathbb \R^3} \partial^\alpha h \psi(\xi) d\xi \Bigr| + C(\delta_1+\delta_C+\delta_3)|U-\hat U|.
\end{align*}
Therefore, $|\alpha|=1$, 
\begin{align*}
    \|\partial^\alpha(U-\hat U)\|_{L_x^2}^2
    \le
    C\left\|\partial^\alpha(f-\hat M)|_{t=0}\right\|_{L_x^2(L_\xi^2(M_\#))}^2 + C(\delta_1+\delta_C+\delta_3)^2 \left\|(f-\hat M)|_{t=0}\right\|_{L_x^2(L_\xi^2(M_\#))}^2
\end{align*}

\begin{align*}
    \left\|\partial^\alpha(f-\hat M)|_{t=0}\right\|_{L_x^2(L_\xi^2(M_\#))}^2 &\le  \left\|\partial^\alpha f_0\right\|_{L_x^2(L_\xi^2(M_\#))}^2 +\left\|\partial^\alpha \hat M_0\right\|_{L_x^2(L_\xi^2(M_\#))}^2 \\
    &\le \mathcal E_{\mathrm{ini}}^2 + C(\delta_1+\delta_C+\delta_3)^2
\end{align*}
and
\begin{align*}
\left\|\partial^\alpha \widetilde G|_{t=0} \right\|^2_{L^2_x(L^2_\xi(M_\#))} \le& \left\|\partial^\alpha (f-\hat M)|_{t=0} \right\|^2_{L^2_x(L^2_\xi(M_\#))} + \left\|\partial^\alpha (M[U]-\hat M)|_{t=0} \right\|^2_{L^2_x(L^2_\xi(M_\#))} \qquad (\text{by }\eqref{eq:GfMuM})\\
\le& \mathcal E_{\mathrm{ini}}^2 + C(\delta_1+\delta_C+\delta_3)^2.
\end{align*}
at time $t=0$.}
When $|\alpha|=2$ and $|\alpha_0|\le1$, these are controlled by $f$ itself from the definition of $\mathcal E(t)$, \eqref{eq:priass}. Combining the above estimates with the initial data \eqref{eq:initda} yields
\begin{align}\label{eq:priass-2}
    \mathcal E(0)^2
    \le C\varepsilon_1^2
    + C(\delta_1+\delta_C+\delta_3)^2 .
\end{align}

Now, we prove the global existence of the normalized solution. By the local-in-time existence result, Proposition~\ref{prop:lcext},
there exists a unique nonnegative solution on a short time interval.
On the other hand, Proposition~\ref{prop:priest} yields an a priori estimate
which improves the a priori bound in \eqref{eq:prioriass}, provided that the initial perturbation
and the wave strengths are sufficiently small.
Therefore, a standard continuity argument implies that the normalized solution can be continued
globally in time on any interval \([0,T]\), and the estimate \eqref{eq:pries} holds uniformly on
\([0,T]\).

In particular, the shifts \(X_1\) and \(X_3\) are absolutely continuous and satisfying \eqref{eq:shiftsp1} on every finite interval,
and all estimates in Proposition~\ref{prop:priest} are available for arbitrary \(T>0\).

\subsection{Rescaling to the Knudsen number}

Let \(f\) denote the global normalized solution furnished by the global continuation argument above.
For \(0<\kappa\le \kappa_0\), define the rescaled solution by
\begin{align}
f^\kappa(\tau,y,\xi):=
f\!\left(\frac{\tau}{\kappa},\frac{y}{\kappa},\xi\right),
\qquad
\bar f^{\,\kappa}(\tau,y,\xi):=
\bar f\!\left(\frac{\tau}{\kappa},\frac{y}{\kappa},\xi\right).
\label{eq:rescaled-solution}
\end{align}
Then \(f^\kappa\) solves the Boltzmann equation \eqref{eq:bslk} with Knudsen number
\(\kappa\), and \(\bar f^{\,\kappa}\) is the corresponding rescaled composite
approximate wave.

For the shifts, define
\begin{align}
X_i^\kappa(\tau):=
\kappa X_i\!\left(\frac{\tau}{\kappa}\right),
\qquad i=1,3.
\label{eq:rescaled-shifts}
\end{align}

\begin{lemma}\label{lem:rescaled-perturbation}
For every \(T>0\),
\begin{align}
\int_0^T\iint_{\mathbb R\times\mathbb R^3}
\frac{|f^\kappa-\bar f^{\,\kappa}|^2}{M_\#}\,d\xi\,dy\,d\tau
\le
C\kappa T\,\bigl(\mathcal{E}(0)^2+\delta_0^{1/2}\bigr).
\label{eq:rescaled-perturbation}
\end{align}
In particular, using the well-preparedness condition
\eqref{eq:initic2-main}, one has
\begin{align}
\int_0^T\iint_{\mathbb R\times\mathbb R^3}
\frac{|f^\kappa-\bar f^{\,\kappa}|^2}{M_\#}\,d\xi\,dy\,d\tau
\le
C\kappa\bigl((\varepsilon_1+\delta_0)^2+\delta_0^{1/2}\bigr)T .
\label{eq:rescaled-perturbation-sharp}
\end{align}
\end{lemma}

\begin{proof}
Let \(t=\tau/\kappa\) and \(x=y/\kappa\). Then
\[
d\tau\,dy = \kappa^2\,dt\,dx,
\]
and therefore
\begin{align*}
& \int_0^T\int_{\mathbb R}
\norm{f^\kappa-\bar f^{\,\kappa}}_{M_\#}^2 \,dy\,d\tau \\
& \qquad =
\kappa^2
\int_0^{T/\kappa}\int_{\mathbb R}
\norm{f-\bar f}_{M_\#}^2\,dx\,dt
\\
& \qquad \le
\kappa T
\sup_{0\le t\le T/\kappa}
\int_{\mathbb R}
\norm{f-\bar f}_{M_\#}\,dx 
\\
&\qquad \le \kappa T
\sup_{0\le t\le T/\kappa}  \Biggl[\int_{\R} \norm{M-\bar M}_{M_\#}^2 \, dx+ \int_{\R} \norm{\wtilde G}_{M_\#}^2 \, dx \Biggr]\, \qquad \qquad \qquad (\text{by }\eqref{eq:unvar}, \eqref{eq:pevar}) \\
&\qquad \le \kappa T
\sup_{0\le t\le T/\kappa} \Biggl[ \norm{U-\bar U}_{L^2_x}^2 + \int_{\R} \norm{\wtilde G_C}_{M_\#}^2 \, dx + \int_{\R} \norm{\wtilde G_{\text{rem}}}_{M_\#}^2 \, dx \Biggr] \quad (\text{by } \eqref{eq:mic2}, \eqref{eq:Maxwellian-Lipschitz-single})\\
&\qquad \le C\kappa T\,\bigl(\mathcal{E}(0)^2+\delta_0^{1/2}\bigr),\end{align*}
by \eqref{eq:GCe} and \eqref{eq:pries}. The sharpened bound \eqref{eq:rescaled-perturbation-sharp}
follows from the well-preparedness hypothesis.
\end{proof}
{
\begin{lemma}\label{lem:shift-compactness}
For each fixed \(T>0\), the families \(\{X_i^\kappa\}_{0<\kappa\le\kappa_0}\),
\(i=1,3\), are bounded in \(BV(0,T)\). More precisely, 
\begin{align}
\bigl|X_i^\kappa(t)\bigr| \le C T^{1/2} \kappa^{1/2},
\qquad t\in[0,T] \text{ and }i=1,3.
\label{eq:shift-compactness}
\end{align}
Here, $C$ is a positive constant that is independent of $\kappa$ and $T$. 
In the rarefaction--contact--shock case, the same conclusion holds for the single family
\(\{X_3^\kappa\}\).
\end{lemma}

\begin{proof}
By \eqref{eq:rescaled-shifts},
\[
\frac{d}{d\tau}X_i^\kappa(\tau)=\dot X_i\!\left(\frac{\tau}{\kappa}\right),
\]
which together with $X_i^{\kappa}(0) = 0$ implies that for each $t\in[0,T]$, 
\begin{align*}
\left|X_i^\kappa(t)\right|\le \int_0^T \left| \dot{X}_i^\kappa (\tau) \right| \, d\tau
&=
\int_0^T \left|\dot X_i\!\left(\frac{\tau}{\kappa}\right)\right|\,d\tau
=
\kappa\int_0^{T/\kappa} |\dot X_i(t)|\,dt \\
&\leq \kappa^{1/2}T^{1/2} \myparas{\int_0^{T/\kappa}\abs{\dot{X}_i}^2}^{1/2}\leq \widetilde C T^{1/2} \delta_i^{-1/2}\kappa^{1/2},
\end{align*}
where for the last inequality, we have used 
\(\int_0^{T/\kappa}\delta_i|\dot X_i|^2 dt\le C \) by \eqref{eq:pries}.
\end{proof}
}
\subsection{Convergence to the shifted Riemann solution}
{
\begin{proof}[\textbf{Proof of Theorem~\ref{thm:main}}]

Let \(f\) be the global normalized solution obtained above, and let \(f^\kappa\) be the
rescaled solution \eqref{eq:rescaled-solution}. Then \(f^\kappa\) exists uniquely on \([0,T]\)
for every \(T>0\), since \(f\) exists globally on \([0,\infty)\).

By Lemma~\ref{lem:shift-compactness}, after extraction of a subsequence if necessary, we may assume
\[
X_i^\kappa \to X_i^0
\qquad\text{in }L^1(0,T),
\qquad i=1,3.
\]

We first compare \(f^\kappa\) with the rescaled composite approximate wave.
By \eqref{eq:rescaled-perturbation-sharp},
\begin{align}
\int_0^T\iint_{\mathbb R\times\mathbb R^3}
\frac{|f^\kappa-\bar f^{\,\kappa}|^2}{M_\#}\,d\xi\,dy\,d\tau
\le
C\kappa\bigl((\varepsilon_1+\delta_0)^2+\delta_0^{1/2}\bigr)T .
\label{eq:main-step-1}
\end{align}

Next, we compare \(\bar f^{\,\kappa}\) with the shifted Euler Maxwellian profile.
Using Lemma~\ref{lem:bs}, \eqref{eq:contma} and \eqref{eq:Maxwellian-Lipschitz-single}, we obtain
\begin{align}
\int_0^T\iint_{\mathbb R\times\mathbb R^3}
\frac{|\bar f^{\,\kappa}-M_{X^\kappa}[U^E]|^2}{M_\#}
\,d\xi\,dy\,d\tau
\le
C\delta_C\,\kappa^{1/2}T^{3/2}+ C(\delta_1+\delta_3)\kappa T.
\label{eq:main-step-2}
\end{align}

Finally, by observing that the height of the difference between the two
shifted Euler shocks is of order the shock strength, while its width in
\(x\) is the difference of the shifts for each shock wave,
for each \(\tau\in[0,T]\),
\begin{align*}
\iint_{\mathbb R\times\mathbb R^3}
\frac{|M_{X^\kappa}[U^E]
-
M[U^E]|^2}{M_\#}\,d\xi\,dy
\le C \int \left|(U^E)^{X^\kappa}-(U^E)\right|^2 dy \le
C\delta_0^2\sum_{i=1,3}|X_i^\kappa(\tau)|.
\end{align*}
Integrating in \(\tau\) yields
\begin{align}
\int_0^T\iint_{\mathbb R\times\mathbb R^3}
\frac{|M_{X^\kappa}[U^E]
-
M[U^E]|^2}{M_\#}\,d\xi\,dy\,d\tau \le &
C\delta_0^2\sum_{i=1,3}\|X_i^\kappa\|_{L^1(0,T)} \nonumber \\
\le & C\delta_0^2T^{3/2}\kappa^{1/2} \qquad \text{by }\eqref{eq:shift-compactness}
\label{eq:main-step-3}
\end{align}

Combining \eqref{eq:main-step-1}, \eqref{eq:main-step-2}, and \eqref{eq:main-step-3},
and using that \(M_\#\) is uniformly bounded from above by positive constants,
we obtain
\begin{align*}
&\int_0^T \iint_{\mathbb R\times\mathbb R^3}
\left|f^\kappa-M[U^E]\right|^2
\,d\xi\,dy\,d\tau\\
&\quad \le
C\kappa\bigl((\varepsilon_1+\delta_0)^2+\delta_0^{1/2}\bigr)T
+
C\delta_C\,\kappa^{1/2}T^{3/2} + C(\delta_1+\delta_2)\kappa T +
C\delta_0^2T^{3/2}\kappa^{1/2}.
\end{align*}
This is exactly \eqref{eq:hlmt1-main}. 
\end{proof}
}

\section{Low order estimates}

\setcounter{equation}{0}

\subsection{Relative Entropy Method}
For simplicity, we introduce the following notation for the left and right end states of the two shock profiles:
\begin{equation}\label{eq:ridx}
\begin{aligned}
&v_1:=v_*,\qquad \theta_1:=\theta_*,\qquad p_1:=p_*,\qquad \sigma_{1}^{\text{m}}:=\sigma_*,\\
&v_3:=v^*,\qquad \theta_3:=\theta^*,\qquad p_3:=p^*,\qquad \sigma_{3}^{\text{m}}:=\sigma^*.
\end{aligned}
\end{equation}

We define
\begin{equation}\label{eq:repy}
\Phi(z):=z-1-\ln z,
\end{equation}
and the relative entropy density by
\begin{equation}\label{eq:rel-entropy}
\eta(U \mid \bar U)
:=
\frac{2}{3}\bar\theta\,\Phi\!\left(\frac{v}{\bar v}\right)
+\bar\theta\,\Phi\!\left(\frac{\theta}{\bar\theta}\right)
+\sum_{k=1}^3\frac{\psi_k^2}{2},
\end{equation}
where
\[
U=(v,u,\theta),\qquad \bar U=(\bar v,\bar u,\bar\theta), \qquad u = (u_1,u_2,u_3), \qquad \bar u = (\bar u_1,0,0).
\]

Differentiating \eqref{eq:rel-entropy} with respect to $t$, we obtain
\begin{equation}\label{eq:etpcom}
\begin{aligned}
\partial_t\eta(U \mid \bar U)
=
\bar\theta_t
\left(
\frac{2}{3}\Phi\!\left(\frac{v}{\bar v}\right)
+
\Phi\!\left(\frac{\theta}{\bar\theta}\right)
\right)
+\sum_{i=1}^3 \psi_i\psi_{it}
+\frac{2}{3}\bar\theta \partial_t\Phi\!\left(\frac{v}{\bar v}\right)
+\bar\theta \partial_t\Phi\!\left(\frac{\theta}{\bar\theta}\right).
\end{aligned}
\end{equation}

Substituting the perturbation system \eqref{eq:pesy1}--\eqref{eq:pesy3} into \eqref{eq:etpcom}, integrating over $\mathbb R_x$ with the weight $a(t,x)$, and using the identity
\[
\frac{\bar p\,\phi}{v}\psi_{1x}+\psi_{1x}(p-\bar p)-\frac{\zeta}{\theta}(pu_{1x}-\bar p\,\bar u_{1x})
=
\frac{\bar p\,\phi}{\bar v}\psi_{1x}
-\frac{\zeta}{\theta}(p-\bar p)\bar u_{1x},
\]
together with
\[
\frac{\bar v\bar p}{v}=\frac{2\bar\theta}{3v}=\frac{\bar\theta p}{\theta},
\]
we arrive at the weighted relative entropy identity
\begin{equation}\label{eq:weighted-entropy}
\frac{d}{dt}\int_{\mathbb R} a(t,x)\eta(U \mid \bar U)\,dx
=
\sum_{i\in\{1,3\}}\dot X_i(t)\,\mathcal Y_i(U)
+\mathcal J^{\mathrm{bad}}(U)
+\mathcal J^{\mathrm{kinetic}}(U)
-\mathcal J^{\mathrm{good}}(U).
\end{equation}

Here the modulation functionals are defined by
\begin{equation}\label{eq:Yi}
\begin{aligned}
\mathcal Y_i(U)
&:=
-\int_{\mathbb R}\partial_x\bigl(a_i^{-X_i}\bigr)\,\eta(U \mid \bar U)\,dx
+\int_{\mathbb R}a
\left(
\frac{2}{3}\Phi\!\left(\frac{v}{\bar v}\right)
+
\Phi\!\left(\frac{\theta}{\bar\theta}\right)
\right)
\partial_x \shockw{\theta}{i}\,dx\\
&\quad
+\int_{\mathbb R}a
\left(
\frac{\partial_x \shockw{v}{i}\,\bar p}{\bar v}\phi
+\frac{\partial_x \shockw{\theta}{i}}{\bar\theta}\zeta
+\psi_1 \partial_x \shockw{u_1}{i}
\right)\,dx,
\qquad i=1,3.
\end{aligned}
\end{equation}

We further decompose
\begin{equation}\label{eq:Yi2}
\mathcal Y_i(U)=\sum_{j=1}^6 \mathcal Y_{ij}, \qquad i=1,3
\end{equation}
where
\begin{equation}\label{eq:Yij}
\begin{aligned}
&\mathcal Y_{i1}:=\int_{\mathbb R} a \partial_x(u_1^{S_i})^{-X_i} \psi_1\,dx,
\qquad
&&\mathcal Y_{i2}:=\int_{\mathbb R} a \frac{\partial_x\bigl((v^{S_i})^{-X_i}\bigr)\bar p}{\bar v}\phi\,dx,
\\
&\mathcal Y_{i3}:=\int_{\mathbb R} a \frac{\partial_x(\theta^{S_i})^{-X_i}}{\bar\theta}\zeta\,dx, \qquad 
&&\mathcal Y_{i4}:=\frac{2}{3}\int_{\mathbb R} a \partial_x \bigl((\theta^{S_i})^{-X_i}\bigr) \Phi\!\left(\frac{v}{\bar v}\right)\,dx,
\\
&\mathcal Y_{i5}:=\int_{\mathbb R} a \partial_x\bigl((\theta^{S_i})^{-X_i}\bigr) \Phi\!\left(\frac{\theta}{\bar\theta}\right)\,dx, \qquad 
&&\mathcal Y_{i6}:=-\int_{\mathbb R}\partial_x\bigl(a_i^{-X_i}\bigr)\,\eta(U \mid \bar U)\,dx.
\end{aligned}
\end{equation}

The good, bad, and kinetic parts in \eqref{eq:weighted-entropy} are given by
\begin{align}\label{eq:good1}
\mathcal J^{\mathrm{good}}(U)
:=&
\sum_{i\in\{1,3\}}\sigma_i\int_{\mathbb R}\partial_x \bigl(a_i^{-X_i}\bigr)\,\eta(U \mid \bar U)\,dx\nonumber\\
&\quad +\int_{\mathbb R} a
\left(
\frac{\alpha_{\rm{th}}(\bar\theta)}{v\theta}\zeta_x^2
+\frac{4}{3}\frac{\mu(\bar\theta)}{v}\psi_{1x}^2
+\frac{\mu(\theta)}{v}\sum_{k=2}^3\psi_{kx}^2
\right)\,dx =:\mathfrak G(U)+\mathcal D_{\mathrm{mac}}(U)
\end{align}

\begin{equation}\label{eq:bad1}
\mathcal J^{\mathrm{bad}}(U)
:=
\sum_{\ell=1}^7 \mathcal{B}_\ell(U)+\mathcal S_1(U)+\mathcal S_2(U),
\end{equation}
and
\begin{equation}\label{eq:kint1}
\mathcal J^{\mathrm{kinetic}}(U)
:=
\sum_{\ell=1}^6 K_\ell(U).
\end{equation}

More explicitly, the macroscopic remainder terms are defined by
\begin{align*}
&\mathcal{B}_1:=\int_{\mathbb R} a\,\mathcal J_1\,dx,\\
&\mathcal{J}_1:= \myparab{\frac{2}{3}\myparas{-\sigma_1\partial_x\shockw{\theta}{1}+\theta_t^C-\sigma_3\partial_x\shockw{\theta}{3}}\Phi\myparas{\frac{v}{\y{v}}}-\frac{\y{p}\y{u}_{1x}}{v\y{v}}\phi^2} \\
&\qquad+\myparab{\myparas{-\sigma_1\partial_x\shockw{\theta}{1}+\theta_t^C-\sigma_3\partial_x\shockw{\theta}{3}}\Phi\myparas{\frac{\theta}{\y{\theta}}}-\frac{\y{u}_{1x}}{\theta}\zeta\myparas{p-\y{p}}+\y{p}\y{u}_{1x}\frac{\zeta^2}{\theta\y{\theta}}}\\
&\qquad =:\mathcal{J}_2+\mathcal{J}_3,\\
&\mathcal{B}_2:=\int_{\mathbb R} a_x\,\psi_1(p-\bar p)\,dx,\\
&\mathcal{B}_3:=-\int_{\mathbb R}a_x\left[
\frac{\zeta}{\theta}\left(\alpha_{\rm{th}}(\theta)\frac{\theta_x}{v}-\alpha_{\rm{th}}(\bar\theta)\frac{\bar\theta_x}{\bar v}\right) +\frac{4}{3}\psi_1\left(\mu(\theta)\frac{u_{1x}}{v}-\mu(\bar\theta)\frac{\bar u_{1x}}{\bar v}\right)\right.\\
&\left.\qquad \qquad \qquad \qquad +\frac{\mu(\theta)}{v}\sum_{i=2}^3\psi_i\psi_{ix}\right]dx,\\
&\mathcal{B}_4:=\int_{\mathbb R} a\,\mathcal R_4^{\text{mac}}\,dx,\qquad
\mathcal{B}_5:=\int_{\mathbb R} a\,\mathcal R_5^{\text{mac}}\,dx,\qquad
\mathcal{B}_6:=\int_{\mathbb R} a\,\mathcal R_6^{\text{mac}}\,dx,\\
& \mathcal S_1:=-\int_{\mathbb R} a\,\psi_1 Q_1\,dx,\qquad
\mathcal S_2:=-\int_{\mathbb R} a\,\frac{\zeta}{\bar\theta}Q_2\,dx,
\end{align*}
where
\begin{align*}
    & \mathcal{R}_4^{\text{mac}} := \biggl[\frac{4}{3}\frac{\mu\bigl(\y{\theta}\bigr)}{v\y{v}}\y{u}_{1x}\psi_{1x}\phi+\frac{\alpha_{\rm{th}}\bigl(\y{\theta}\bigr)}{\theta^2}\zeta\theta_x\myparas{\frac{\theta_x}{v}-\frac{\y{\theta}_x}{\y{v}}}+\frac{\alpha_{\rm{th}}\bigl(\y{\theta}\bigr)}{v\y{v}\y{\theta}}\y{\theta}_x\zeta_x\phi\biggr.\\
    &\biggl.\qquad \qquad \qquad \qquad +\frac{4}{3}\frac{\zeta}{\theta}\myparas{\mu(\theta)\frac{u_{1x}^2}{v}-\mu\bigl(\y{\theta}\bigr)\frac{\y{u}_{1x}^2}{\y{v}}}+\frac{\mu(\theta)}{v}\sum_{i=2}^3 \psi_{ix}^2\biggr], \\
    &\mathcal{R}_5^{\text{mac}} := \myparab{-\frac{\zeta^2}{\theta\y{\theta}}\myparam{\myparas{\alpha_{\rm{th}}\bigl(\y{\theta}\bigr)\frac{\y{\theta}_x}{\y{v}}}_x+\frac{4}{3}\mu\myparas{\y{\theta}}\frac{\y{u}_{1x}^2}{v}}}, \,\\
    &\mathcal{R}_6^{\text{mac}} := \myparab{-\myparas{\frac{\zeta}{\theta}}_x\frac{\theta_x}{v}\myparas{\alpha_{\rm{th}}(\theta)-\alpha_{\rm{th}}\bigl(\y{\theta}\bigr)}-\frac{4}{3}\frac{u_{1x}\psi_{1x}}{v}\myparas{\mu(\theta)-\mu\bigl(\y{\theta}\bigr)}}
\end{align*}

\noindent while the kinetic terms are
\begin{align*}
K_1&:=-\iint a\frac{\zeta}{\theta}\left(\xi_1\frac{|\xi|^2}{2}\right)\Bigl(\Pi_1-(\Pi_1^{S_1})^{-X_1}-(\Pi_1^{S_3})^{-X_3}\Bigr)_x\,d\xi\,dx,\\
K_2&:=\iint a\frac{\zeta}{\theta}
\left[
u_1\xi_1^2\Pi_{1x}
-\sum_{i\in\{1,3\}}(u_1^{S_i})^{-X_i}\xi_1^2\partial_x(\Pi_1^{S_i})^{-X_i}
\right]d\xi\,dx,\\
K_3&:=\iint a\frac{\zeta}{\theta}\sum_{i=2}^3\psi_i\,\xi_1\xi_i\,\Pi_{1x}\,d\xi\,dx,\\
K_4&:=-\iint a\psi_1\,\xi_1^2\Bigl(\Pi_1-(\Pi_1^{S_1})^{-X_1}-(\Pi_1^{S_3})^{-X_3}\Bigr)_x\,d\xi\,dx,\\
K_5&:=-\iint a\sum_{i=2}^3\psi_i\,\xi_1\xi_i\,\Pi_{1x}\,d\xi\,dx,\\
K_6&:=\iint a\frac{\zeta^2}{\theta\bar\theta}
\sum_{i\in\{1,3\}}
\xi_1\left(\frac{|\xi|^2}{2}-(u_1^{S_i})^{-X_i}\xi_1\right)(\partial_x\Pi_1^{S_i})^{-X_i}\,d\xi\,dx.
\end{align*}

The detailed analysis of the leading shock terms, the macroscopic remainders, the kinetic contributions, and the coercive part of $\mathcal J^{\mathrm{good}}$ will be carried out in the following subsections.

\subsection{Control of the leading shock contributions}
We now isolate the leading shock contributions contained in the bad part of the weighted relative entropy identity $\mathcal{B}_1$. First, we focus on $\mathcal{J}_2$. {Recall that}  
\begin{equation}\label{eq:J2-def}
\mathcal J_2
=
\frac{2}{3}\Bigl(-\sigma_1\partial_x(\theta^{S_1})^{-X_1}+\theta_t^C-\sigma_3\partial_x(\theta^{S_3})^{-X_3}\Bigr)
\Phi\!\left(\frac{v}{\bar v}\right)
-\frac{\bar p\,\bar u_{1x}}{v\bar v}\phi^2,
\end{equation}
where
\[
\bar u_{1x}
=
\partial_x(u_1^{S_1})^{-X_1}+u_{1x}^C+\partial_x(u_1^{S_3})^{-X_3}.
\]
Accordingly, we split
\begin{equation}\label{eq:J2-split}
\mathcal J_2=\mathfrak j_{21}+\mathfrak j_{22}+\mathfrak j_{23},
\end{equation}
with
\begin{align}
\mathfrak j_{21}
&:=
-\frac{2}{3}\sigma_1\partial_x(\theta^{S_1})^{-X_1}
\Phi\!\left(\frac{v}{\bar v}\right)
-\frac{\bar p\,\partial_x(u_1^{S_1})^{-X_1}}{v\bar v}\phi^2,
\label{eq:J21}
\\
\mathfrak j_{22}
&:=
\frac{2}{3}\theta_t^C
\Phi\!\left(\frac{v}{\bar v}\right)
-\frac{\bar p\,u_{1x}^C}{v\bar v}\phi^2,
\label{eq:J22}
\\
\mathfrak j_{23}
&:=
-\frac{2}{3}\sigma_3\partial_x(\theta^{S_3})^{-X_3}
\Phi\!\left(\frac{v}{\bar v}\right)
-\frac{\bar p\,\partial_x(u_1^{S_3})^{-X_3}}{v\bar v}\phi^2.
\label{eq:J23}
\end{align}

Similarly, we can investigate the term $\mathcal{J}_3$ as $\mathcal{J}_2$. Here, 
\begin{equation}\label{eq:J3-def}
\mathcal J_3
=
\Bigl(-\sigma_1\partial_x(\theta^{S_1})^{-X_1}+\theta_t^C-\sigma_3\partial_x(\theta^{S_3})^{-X_3}\Bigr)
\Phi\!\left(\frac{\theta}{\bar\theta}\right)
-\frac{\bar u_{1x}}{\theta}\zeta(p-\bar p)
+\bar p\,\bar u_{1x}\frac{\zeta^2}{\theta\bar\theta}.
\end{equation}
We also split
\begin{equation}\label{eq:J3-split}
\mathcal J_3=\mathfrak j_{31}+\mathfrak j_{32}+\mathfrak j_{33},
\end{equation}
where
\begin{align}
\mathfrak j_{31}
&:=
-\sigma_1\partial_x(\theta^{S_1})^{-X_1}
\Phi\!\left(\frac{\theta}{\bar\theta}\right)
-\frac{(p-\bar p)\zeta \partial_x (u_1^{S_1})^{-X_1}}{\theta}
+\bar p\,\partial_x (u_1^{S_1})^{-X_1}\frac{\zeta^2}{\theta\bar\theta},
\label{eq:J31}
\\
\mathfrak j_{32}
&:=
\theta_t^C\Phi\!\left(\frac{\theta}{\bar\theta}\right)
-\frac{u_{1x}^C}{\theta}\zeta(p-\bar p)
+\bar p\,u_{1x}^C\frac{\zeta^2}{\theta\bar\theta},
\label{eq:J32}
\\
\mathfrak j_{33}
&:=
-\sigma_3\partial_x(\theta^{S_3})^{-X_3}
\Phi\!\left(\frac{\theta}{\bar\theta}\right)
-\frac{(p-\bar p)\zeta \partial_x(u_1^{S_3})^{-X_3}}{\theta}
+\bar p\,\partial_x (u_1^{S_3})^{-X_3}\frac{\zeta^2}{\theta\bar\theta}.
\label{eq:J33}
\end{align}

We use Taylor expansions
\begin{equation}\label{eq:taylor-Phi}
\Phi\!\left(\frac{v}{\bar v}\right)
=
\frac{\phi^2}{2\bar v^2}+O(|\phi|^3),
\qquad
\Phi\!\left(\frac{\theta}{\bar\theta}\right)
=
\frac{\zeta^2}{2\bar\theta^2}+O(|\zeta|^3),
\end{equation}
together with \eqref{eq:ctident} and Lemma \ref{lem:contact-gaussian}
\begin{equation}\label{eq:thetaC-t}
\theta_t^C
=
-p^C u_{1x}^C
+\left(\frac{\alpha_{\rm{th}}(\theta^C)\theta_x^C}{v^C}\right)_x
+\frac{4}{3}\mu(\theta^C)\frac{(u_{1x}^C)^2}{v^C}
+Q_2^C,
\end{equation}
where
\begin{equation}\label{eq:Q2C-bound}
Q_2^C
=
O(1)\delta_C(1+t)^{-2}e^{-2c_0x^2/(1+t)}.
\end{equation}
Moreover, see (4.35) in \cite{kang2025time}, by Lemma~\ref{lem:bs},
\begin{equation}\label{eq:pbar-endstate}
|\bar p- p_1|+|\bar p- p_3|\le C\delta_0.
\end{equation}

The next estimates follow from \eqref{eq:taylor-Phi}--\eqref{eq:pbar-endstate}, the exponential localization of the shock profiles, and the Gaussian estimate for the viscous contact wave. See (4.33) with $\gamma=\frac{5}{3}$ in \cite{kang2025time}:
\begin{align}
\mathfrak j_{21}
&\le
-\frac{4 p_1\,\sigma_1^{\text{m}}}{3 v_1^2}\partial_x(v^{S_1})^{-X_1}\phi^2
+
C(\delta_0+\varepsilon)|\partial_x(v^{S_1})^{-X_1}|\phi^2,
\label{eq:J21-est}
\\
\mathfrak j_{22}
&\le
C\delta_C(1+t)^{-1}e^{-C_1|x|^2/(1+t)}\phi^2,
\label{eq:J22-est}
\\
\mathfrak j_{23}
&\le
\frac{4p_3\,\sigma_3^{\text{m}}}{3 v_3^2}\partial_x(v^{S_3})^{-X_3}\phi^2
+
C(\delta_0+\varepsilon)|\partial_x(v^{S_3})^{-X_3}|\phi^2.
\label{eq:J23-est}
\end{align}
Similarly,
\begin{align}\label{eq:J3-est}
\mathcal J_3
\le& \sum_{i\in\{1,3\}} \frac{\sigma_i^{\text{m}} |\partial_x(v^{S_i})^{-X_i}|}{2v_i\theta_i} \left(\frac{2}{3}\zeta-2p_i\phi \right) \zeta + C\delta_C(1+t)^{-1}e^{-C_1|x|^2/(1+t)}|(\phi,\zeta)|^2 \nonumber \\
&\qquad+ C(\delta_0+\varepsilon)\bigl(|\partial_x(v^{S_1})^{-X_1}|+|\partial_x(v^{S_3})^{-X_3}|\bigr)|(\phi,\zeta)|^2.
\end{align}

Consequently, we obtain the following estimate for the leading shock contribution:
\begin{equation}\label{eq:nbad1}
\begin{aligned}
\int_{\mathbb R}a(\mathcal J_2+\mathcal J_3)\,dx
\le\;&
\sum_{i=1,3}\int_{\mathbb R}a\,|\partial_x(v^{S_i})^{-X_i}|
\left(
\frac{4p_i\sigma_i^{\text{m}}}{3v_i^2}\phi^2
+\frac{\sigma_i^{\text{m}}}{3v_i\theta_i}\zeta^2
-\frac{p_i\sigma_i^{\text{m}}}{v_i\theta_i}\phi\zeta
\right)\,dx\\
&+
C\delta_C(1+t)^{-1}\int_{\mathbb R}ae^{-C_1|x|^2/(1+t)}|(\phi,\zeta)|^2\,dx\\
&+
C(\delta_0+\varepsilon)\int_{\mathbb R}
\bigl(|\partial_x(v^{S_1})^{-X_1}|+|\partial_x(v^{S_3})^{-X_3}|\bigr)|(\phi,\zeta)|^2\,dx.
\end{aligned}
\end{equation}
The remaining dominant terms are estimated in the following lemma by using the a-contraction method.

\subsection{Decomposition of the weighted entropy dissipation}
We now rewrite the weighted relative entropy identity \eqref{eq:weighted-entropy} by using \eqref{eq:nbad1}. More precisely, we decompose the right-hand side into modulation terms, macroscopic error terms, microscopic error terms, and coercive contributions:
\begin{equation}\label{eq:energy2}
\begin{aligned}
&\frac{d}{dt}\int_{\mathbb R}a\,\eta(U\mid\bar U)\,dx \\
&\quad \le
\sum_{i\in\{1,3\}}\dot X_i(t)\mathcal Y_i(U)
+\mathcal B^{\text{sh}} + \sum_{l=2}^6 \mathcal B_l + \mathcal B^{\text{res}}
+\mathcal S_1+\mathcal S_2
+\sum_{l=1}^6 K_l
-\mathfrak G(U)-\mathcal D_{\mathrm{mac}}(U).
\end{aligned}
\end{equation}
Here, the reformulated error terms \(\mathcal B^{\text{sh}}, \, \mathcal B^{\text{res}}\) are defined by
\begin{align*}
& \mathcal B^{\text{sh}}:=\sum_{i\in\{1,3\}}\int_{\mathbb R}a|\partial_x(v^{S_i})^{-X_i}|
\left(
\frac{4p_i\sigma_{i}^{\text{m}}}{3v_i^2}\phi^2
+\frac{\sigma_{i}^{\text{m}}}{3v_i\theta_i}\zeta^2
-\frac{p_i\sigma_{i}^{\text{m}}}{v_i\theta_i}\phi\zeta
\right)\,dx, \\
& \mathcal B^{\text{res}}:=C\delta_C\frac{1}{1+t}\int_{\mathbb R}ae^{-C_1|x|^2/(1+t)}|(\phi,\zeta)|^2\,dx
\\
&\qquad \qquad+C(\delta_0+\varepsilon)\int_{\mathbb R}\bigl(|\partial_x(v^{S_1})^{-X_1}|+|\partial_x(v^{S_3})^{-X_3}|\bigr)|(\phi,\zeta)|^2\,dx.
\end{align*}

\begin{lemma}\label{lem:actr1}
For $t>0$, there exist positive constants \(C\) and \(\alpha\) such that
\begin{align}\label{eq:macro1}
& -\sum_{i=1,3}\frac{\delta_i}{2M_i}|\dot X_i|^2 + \mathcal B^{\textup{sh}} + \mathcal B_2 - \mathfrak G - \frac34 \mathcal D_{\mathrm{mac}} \\
& \quad \le\; -\alpha\sum_{i=1,3}\mathcal{G}_i^S+C\delta_1^{\frac{4}{3}}\bigl(\delta_3^{\frac{4}{3}}+\delta_C^{\frac{4}{3}}\bigr)e^{-C\delta_1 t} + C\delta_3^{\frac{4}{3}}\bigl(\delta_1^{\frac{4}{3}}+\delta_C^{\frac{4}{3}}\bigr)e^{-C\delta_3 t}\nonumber\\
& \qquad + C\left(\sum_{i=1,3}\delta_i^2e^{-C\delta_i t}+\frac{1}{\delta_*t^2}\right)\int_{\mathbb R}\eta(U \mid \bar U)\,dx - \frac34\int_{\mathbb R}a\frac{\mu(\theta)}{v}\sum_{k=2}^3\psi_{kx}^2\,dx.
\end{align}
Recall that
\begin{equation} \label{eq:maccoreciv}
    \mathcal{G}_i^S := \int_{\mathbb R}\Bigl|\partial_x\shockw{v}{i}\Bigr|\,|\varphi_i(\phi,\psi,\zeta)|^2\,dx.
\end{equation}
\end{lemma}

\begin{proof}
We next combine the modulation terms, the principal shock contribution \(\mathcal B^{\text{sh}}\), and the dissipative part \(\mathcal D_{\mathrm{mac}}(U)\) in order to recover the coercive shock functional \(\mathcal G_i^S\).

\subsubsection*{Estimate of \(\mathcal B_2-\mathfrak G\)}
See (4.46) with $R=\frac{2}{3}$ in \cite{kang2025time}. By a similar argument, we obtain the following estimate.
\begin{equation}\label{eq:nbad2}
\begin{aligned}
\mathcal B_2(U)-\mathfrak G(U)
\le\;&
-\sum_{i\in\{1,3\}}\mathfrak G_{1i}(U)
-\sum_{i\in\{1,3\}}\mathfrak G_{2i}(U)
-\sum_{i\in\{1,3\}}\mathfrak G_{3i}(U)
+B_{\mathrm{new}}(U)\\
&+
C\sum_{i\in\{1,3\}}\delta_i\sqrt{\delta_i}e^{-C\delta_i t}
\int_{\mathbb R}\eta(U \mid \bar U)\,dx,
\end{aligned}
\end{equation}
where
\begin{align}
\begin{aligned}
& \mathfrak G_{1i}(U):=\frac{\theta_i\sigma_{i}^{\text{m}}}{3v_i^2}\int_{\mathbb R}|\partial_x(a_i^{-X_i})|\varphi_i\left(\phi+\frac{\psi_1}{\sigma_i^{\text{m}}}\right)^2dx,\\
& \mathfrak G_{2i}(U):=\frac{\sigma_{i}^{\text{m}}}{2\theta_i}\int_{\mathbb R}|\partial_x(a_i^{-X_i})|\varphi_i\left(\zeta-\frac{2\theta_i}{3\sigma_i^{\text{m}} v_i}\psi_1\right)^2dx,\\
& \mathfrak G_{3i}(U):=\frac{\sigma_{i}}{2}\sum_{j=2}^3\int_{\mathbb R}\partial_x(a_i^{-X_i})\varphi_i \psi_j^2dx=\frac{1}{2\sqrt{\delta_i}}\sum_{j=2}^3\int_\mathbb{R}\sigma_i\partial_x\shockw{v}{i}\varphi_i\psi_j^2 dx,
\end{aligned}
\end{align}
and
\begin{align}
\begin{aligned}
B_{\mathrm{new}}(U):=\;&
C\sum_{i\in\{1,3\}}\delta_i\int_{\mathbb R}|\partial_x(a_i^{-X_i})||(\phi,\psi,\zeta)|^2dx
+
C\sum_{i\in\{1,3\}}\int_{\mathbb R}|\partial_x(a_i^{-X_i})||(\phi,\psi,\zeta)|^3dx\\
&+
C\int_{\mathbb R}|\partial_x(a_1^{-X_1})|
\Bigl(|(v^{S_3}-v^*,\theta^{S_3}-\theta^*)|+|(v^C-v_*,\theta^C-\theta_*)|\Bigr)
|(\phi,\psi_1,\zeta)|^2dx\\
&+
C\int_{\mathbb R}|\partial_x(a_3^{-X_3})|
\Bigl(|(v^{S_1}-v_-,\theta^{S_1}-\theta_-)|+|(v^C-v_*,\theta^C-\theta_*)|\Bigr)
|(\phi,\psi_1,\zeta)|^2dx.
\end{aligned}
\end{align}
Indeed, 
\begin{align*}
\mathcal B_2(U)-\mathfrak G(U)
&=
\int_{\mathbb R} a_x \psi_1(p-\bar p)\,dx
-\sum_{i\in\{1,3\}}\sigma_i \int_{\mathbb R} \partial_x(a_i^{-X_i}) \eta(U\mid \bar U)\,dx\\
&=
\sum_{i\in\{1,3\}}
\int_{\mathbb R}\partial_x(a_i^{-X_i})
\left[
\psi_1(p-\bar p)
-\sigma_i\left(
\frac23\bar\theta\Phi\!\left(\frac{v}{\bar v}\right)
+\bar\theta\Phi\!\left(\frac{\theta}{\bar\theta}\right)
+\frac{|\psi|^2}{2}
\right)
\right]dx.
\end{align*}
Using
\begin{align*}
\Phi(z)=\frac{(z-1)^2}{2}+O(|z-1|^3),
\quad
p-\bar p=\frac{2}{3v}\zeta-\frac{\bar p}{v}\phi,
\end{align*}
we complete squares separately for the \(1\)-shock and \(3\)-shock contributions and obtain \eqref{eq:nbad2}. Here, {the good terms} $\mathfrak G_{3i}$ {provide coercivity for the transverse velocity components included in} $\mathcal{G}_i^S${.}

It remains to estimate \(B_{\mathrm{new}}\). Using \(\varphi_1+\varphi_3=1\), Lemmas~\ref{lem:bs}--\ref{lem:intsk}, the Sobolev inequality, and Young's inequality, we obtain

\begin{align}
\begin{aligned}
B_{\mathrm{new}}(U)\le\;&
C(\varepsilon+\sqrt{\delta_0})
\left(
\sum_{i\in\{1,3\}}(\mathfrak G_{1i}+\mathfrak G_{2i})+\mathcal D_{\mathrm{mac}}(U)+\sum_{i=1,3}\mathcal G_i^S
\right)\\
&+
C\sum_{i\in\{1,3\}}\varepsilon\sqrt{\delta_i}\delta_i e^{-C\delta_i t}
\int_{\mathbb R}\eta(U \mid \bar U)\,dx\\
&+
C\delta_1^{\frac{4}{3}}\bigl(\delta_C^{\frac{4}{3}}+\delta_3^{\frac{4}{3}}\bigr)e^{-C\delta_1t}+ C\delta_3^{\frac{4}{3}}\bigl(\delta_C^{\frac{4}{3}}+\delta_1^{\frac{4}{3}}\bigr)e^{-C\delta_3t}.
\end{aligned}
\end{align}

\subsubsection*{Estimate of the shift terms}

We now estimate the modulation term
\[
-\sum_{i=1,3}\frac{\delta_i}{2\mathfrak m_i}|\dot X_i|^2.
\]
Set
\begin{equation}\label{eq:newvar}
\begin{aligned}
& w_1:=\varphi_1\psi_1,\qquad w_3:=\varphi_3\psi_1,\\
& y_1:= -\frac{\left(v^{S_1}\right)^{-X_1}-v_-}{\delta_1},\qquad
  y_3:= \frac{\left(v^{S_3}\right)^{-X_3}-v^*}{\delta_3},\\
& z_1:=x-\sigma_1 t-X_1(t),\qquad z_3:=x-\sigma_3 t-X_3(t),
\end{aligned}
\end{equation}
so that
\[
\frac{dy_1}{dz_1}=-\frac{\partial_x\shockw{v}{1}}{\delta_1}>0,
\qquad
\frac{dy_3}{dz_3}=\frac{\partial_x\shockw{v}{3}}{\delta_3}>0.
\]

Using the decomposition \(\mathcal Y_i=\sum_{j=1}^6\mathcal Y_{ij}\), we write the shift ODE in the form
\begin{equation}\label{eq:dsaf}
\dot X_i(t)=-\frac{\mathfrak m_i}{\delta_i}(\mathcal Y_{i1}+\mathcal Y_{i2}+\mathcal Y_{i3}),
\qquad i=1,3.
\end{equation}
For the \(1\)-shock, the terms \(\mathcal Y_{11},\mathcal Y_{12},\mathcal Y_{13}\) can be compared with \(\int_0^1 w_1\,dy_1\). More precisely,
\begin{align}
\left|\mathcal Y_{11}+\delta_1\sigma_{1}^{\text{m}}\int_0^1w_1\,dy_1\right|
&\le
C\delta_1(\sqrt{\delta_0}+\delta_0)\int_0^1|w_1|\,dy_1 \nonumber \\
&\quad +
C\int_{\mathbb R}|\partial_x\shockw{v}{1}|\varphi_3|\psi_1|\,dx,
\\
\left|\mathcal Y_{12}+\frac{p_1\delta_1}{\sigma_{1}^{\text{m}}v_1}\int_0^1w_1\,dy_1\right|
&\le
C\delta_1(\delta_0+\sqrt{\delta_0})\int_0^1|w_1|\,dy_1
\nonumber\\
&\quad +
C\sqrt{\delta_1}\int_{\mathbb R}|\partial_x(a_1^{-X_1})|\varphi_1\left|\phi-\frac{\psi_1}{\sigma_1}\right|dx \nonumber \\
&\quad +
C\int_{\mathbb R}\left|\partial_x\shockw{v}{1}\right|\varphi_3|\phi|\,dx,
\\
\left|\mathcal Y_{13}+\frac{2}{3}\frac{p_1\delta_1}{\sigma_{1}^{\text{m}}v_1}\int_0^1w_1\,dy_1\right|
&\le
C\delta_1(\delta_0+\sqrt{\delta_0}+\varepsilon)\int_0^1|w_1|\,dy_1
\nonumber\\
&\quad
+
C\sqrt{\delta_1}\int_{\mathbb R}|\partial_x(a_1^{-X_1})|\varphi_1\left|\zeta-\frac{2\theta_-}{3\sigma_1v_-}\psi_1\right|dx\nonumber\\
&\quad +
C\int_{\mathbb R}\varphi_3\left|\partial_x\shockw{\theta}{1}\right||\zeta|\,dx.
\end{align}
Hence,
\begin{align}
\left|\dot X_1-2\sigma_1^{\text{m}}\mathfrak m_1\int_0^1w_1\,dy_1\right|^2
\le\;&
C(\delta_0+\sqrt{\delta_0}+\varepsilon)^2\int_0^1 w_1^2\,dy_1
+\frac{C}{\delta_1}(\mathfrak G_{11}+\mathfrak G_{21})
\nonumber\\
&+C\delta_1e^{-C\delta_1 t}\int_{\mathbb R}\eta(U \mid \bar U)\,dx.
\end{align}
Therefore,
\begin{equation}
\begin{aligned}
-\frac{\delta_1}{2\mathfrak m_1}|\dot X_1|^2
\le\;&
-(\sigma_1^{\text{m}})^2\mathfrak m_1\delta_1\left(\int_0^1 w_1\,dy_1\right)^2
+
C\delta_1(\delta_0+\sqrt{\delta_0}+\varepsilon)^2\int_0^1w_1^2\,dy_1
\\
&+
C\sqrt{\delta_1}(\mathfrak G_{11}+\mathfrak G_{21})
+
C\delta_1^2e^{-C\delta_1 t}\int_{\mathbb R}\eta(U\mid\bar U)\,dx.
\end{aligned}
\end{equation}
The \(3\)-shock case is treated in the same way.

\subsubsection*{Estimate of \(\mathcal B^{\text{sh}}\)}

Recall that
\[
\mathcal B^{\text{sh}}(U)
=
\sum_{i\in\{1,3\}}
\int_{\mathbb R}
a\Bigl|\partial_x\shockw{v}{i}\Bigr|
\left(
\frac{4p_i\sigma_{i}^{\text{m}}}{3v_i^2}\phi^2
+\frac{\sigma_{i}^{\text{m}}}{3v_i\theta_i}\zeta^2
-\frac{\sigma_{i}^{\text{m}} p_i}{v_i\theta_i}\phi\zeta
\right)\,dx.
\]
We decompose
\[
\mathcal B_1=\sum_{i\in\{1,3\}}(B_{11}^i+B_{12}^i+B_{13}^i)
\]
according to the three terms inside the bracket. By splitting \(\phi\) and \(\zeta\) into the favorable part on \(\varphi_i\) and the exponentially small interaction part on \(1-\varphi_i^2\), and using Lemma~\ref{lem:intsk}, we obtain
\begin{align}
B_{11}^1
&\le
\frac{4p_1}{3v_1^2\sigma_{1}^{\text{m}}}(1+\sqrt{\delta_0}+\delta_1^{1/4})\delta_1\int_0^1|w_1|^2\,dy_1
\nonumber\\
&\quad +
C\delta_1^{1/4}\mathfrak G_{11}
+
C\delta_1^2e^{-C\delta_1 t}\int_{\mathbb R}\eta(U\mid\bar U)\,dx,
\\
B_{12}^1
&\le
\frac{p_1}{3\sigma_{1}^{\text{m}}v_1^2}(1+\sqrt{\delta_0}+\delta_1^{1/4})\delta_1\int_0^1|w_1|^2\,dy_1
\nonumber\\
&\quad +
C\delta_1^{1/4}\mathfrak G_{21}
+
C\delta_1^2e^{-C\delta_1 t}\int_{\mathbb R}\eta(U\mid\bar U)\,dx,
\\
B_{13}^1
&\le
\frac{\sigma_{1}^{\text{m}}p_1}{v_1\theta_1}(1+\sqrt{\delta_0}+\delta_1^{1/4})\delta_1\int_0^1|w_1|^2\,dy_1
\nonumber\\
&\quad +
C\delta_1^{1/4}(\mathfrak G_{11}+\mathfrak G_{21})
+
C\delta_1^2e^{-C\delta_1 t}\int_{\mathbb R}\eta(U\mid\bar U)\,dx.
\end{align}
The same estimate holds for the \(3\)-shock contribution. Consequently,
\begin{equation}
\begin{aligned}
\mathcal B^{\text{sh}}
\le\;&
\sum_{i\in\{1,3\}}
\frac{20p_i}{9v_i^2\sigma_{i}^{\text{m}}}(1+\sqrt{\delta_0}+\delta_i^{1/4})
\delta_i\int_0^1|w_i|^2\,dy_i
\\
&+
C\sum_{i\in\{1,3\}}\delta_i^{1/4}(\mathfrak G_{1i}+\mathfrak G_{2i})
+
C\sum_{i\in\{1,3\}}\delta_i^2e^{-C\delta_i t}
\int_{\mathbb R}\eta(U \mid \bar U)\,dx.
\end{aligned}
\end{equation}

\subsubsection*{Estimate of \(\mathcal D_{\mathrm{mac}}\)}

Recall that
\[
\mathcal D_{\mathrm{mac}}(U)
=
\int_{\mathbb R}a
\left(
\frac{\alpha_{\rm{th}}(\bar\theta)}{v\theta}\zeta_x^2
+
\frac{4}{3}\frac{\mu(\bar\theta)}{v}\psi_{1x}^2
+
\frac{\mu(\theta)}{v}\sum_{k=2}^3\psi_{kx}^2
\right)dx.
\]
Using the cutoff decomposition and the change of variable \(x\mapsto y_i\), together with Lemmas~\ref{lem:bs1} and \ref{lem:bs2}, we obtain
\begin{equation}
\begin{aligned}\label{eq:diffusion}
-\mathcal D_{\mathrm{mac}}(U)
\le\;&
-\sum_{i\in\{1,3\}}2\alpha_i\bigl(1-C(\sqrt{\delta_0}+\varepsilon+\delta_i+\delta_*)-\delta_i^{1/4}\bigr)\delta_i\int_0^1 w_i^2\,dy_i
\\
&+
\sum_{i\in\{1,3\}}\frac{20+12\gamma}{10+3\gamma}\alpha_i\delta_i\bigl( w_i^{\mathrm{avg}}\bigr)^2
+
\sum_{i\in\{1,3\}}C\delta_i^{1/4}\mathfrak G_{2i}
\\
&+
\frac{C}{\delta_*t^2}\int_{\mathbb R}\eta(U|\bar U)\,dx
-
\int_{\mathbb R}a\frac{\mu(\theta)}{v}\sum_{k=2}^3\psi_{kx}^2\,dx,
\end{aligned}
\end{equation}
where
\[
\alpha_i=\frac{20}{9}\frac{p_i}{\sigma_{i}^{\text{m}} v_i^2},
\qquad
 w_i^{\mathrm{avg}}:=\int_0^1 w_i\,dy_i,
\qquad
\gamma=\frac{\alpha_{\rm{th}}(\bar\theta)}{\mu(\bar\theta)}.
\]

More precisely, to investigate the diffusion term, we decompose as follows.
\begin{align}
\begin{aligned}
\wtilde{D}_{u_1}(U)=\int_\mathbb{R} a \myparas{\frac{4}{3}\frac{\mu(\y{\theta})}{v}\psi_{1x}^2}dx,\quad \wtilde{D}_{\theta}(U)=\int_\mathbb{R} a \myparas{\frac{\alpha_{\rm{th}}(\y{\theta})}{v\theta}\zeta_x^2}dx.
\end{aligned}
\end{align}

We focus on $\wtilde{D}_{u_1}$. Observe that

\begin{align}
\begin{aligned}
\wtilde{D}_{u_1}(U)\geq \sum_{i\in\{1,3\}}\int_\mathbb{R} a \varphi_i^2 \frac{4}{3}\frac{\mu(\y{\theta})}{v}\psi_{1x}^2dx.
\end{aligned}
\end{align}

Using the Young inequality, we have the following.

\begin{align}
\begin{aligned}
\int_\mathbb{R} a \frac{\mu(\y{\theta})}{v}\abs{\myparas{\varphi_i
\psi_1}_x}^2dx\leq (1+\delta_*)\int_\mathbb{R} a\frac{\mu(\y{\theta})}{v}\varphi_i^2\psi_{1x}^2dx+\frac{C}{\delta_*}\int_\mathbb{R} a \frac{\mu(\y{\theta})}{v}\abs{\varphi_{ix}}^2\abs{\psi_1}^2dx
\end{aligned}
\end{align}

for any sufficiently small $\delta_*>0$ that is chosen later.

Since $\varphi_1^2+\varphi_3^2\leq 1$, we get the following estimate, 

\begin{align}
\begin{aligned}
-\wtilde{D}_{u_1}(U)\leq -\frac{1}{1+\delta_*}\frac{4}{3}\sum_{i\in\{1,3\}}\int_\mathbb{R} a \frac{\mu(\y{\theta})}{v}\abs{\myparas{\varphi_i\psi_1}_x}^2dx+\frac{4}{3}\frac{C}{\delta_*}\int_\mathbb{R} a\frac{\mu(\y{\theta})}{v}\myparas{\varphi_{ix}}^2\abs{\psi_1}^2dx.
\end{aligned}
\end{align}

For clarity, each term is defined as follows.

\begin{align}
\begin{aligned}
&J_1 := -\frac{1}{1+\delta_*}\sum_{i\in\{1,3\}}\int_\mathbb{R} a \frac{\mu(\y{\theta})}{v}\abs{\myparas{\varphi_i\psi_1}_x}^2dx,\\
&J_2:= \frac{C}{\delta_*}\int_\mathbb{R} a\frac{\mu(\y{\theta})}{v}\myparas{\varphi_{ix}}^2\abs{\psi_1}^2dx.
\end{aligned}
\end{align}

By the change of variable(cf. \eqref{eq:newvar}), we have

\begin{align}
\begin{aligned}
J_1 = -\frac{1}{1+\delta_*}\sum_{i\in\{1,3\}}\int_0^1 a \frac{\mu(\y{\theta})}{v}\frac{dy_i}{dx}\abs{(w_i)_{y_i}}^2dy_i.
\end{aligned}
\end{align}

From Lemmas \ref{lem:bs1} and \ref{lem:bs2}, we have the following. 
\begin{align}
\begin{aligned}
\frac{dy_1}{dz_1}\geq \Biggl(\frac{1}{y_1(1-y_1)}\frac{4}{3}\frac{\mu\Bigl(\shockw{\theta}{1}\Bigr)}{\shockw{v}{1}}\Biggr)^{-1}(A\alpha_1\delta_1-C\delta_1^2)>0
\end{aligned}
\end{align}
where $A := \frac{10}{10+\gamma}$.

Therefore, we can get the following estimate,
\begin{align}
\begin{aligned}
J_1&\leq -\frac{1}{1+\delta_*}\sum_{i\in\{1,3\}}\int_0^1 a \frac{\shockw{v}{i}}{v}\frac{\mu(\y{\theta})}{\mu\myparas{\shockw{\theta}{i}}}\frac{3}{4}y_i(1-y_i)(A\alpha_i\delta_i-C\delta_i^2)\abs{\myparas{w_i}_{y_i}}^2dy_i\\
&\leq -\frac{1}{1+\delta_*}\sum_{i\in\{1,3\}}\int_0^1 \frac{3}{4}(1-2\sqrt{\delta_0})(1-C\varepsilon)(A\alpha_i-C\delta_i)\delta_iy_i(1-y_i)\abs{\myparas{w_i}_{y_i}}^2dy_i.
\end{aligned}
\end{align}
Since $\frac{1}{1+\delta_*}\geq 1-\delta_*$ and
\begin{align}
\begin{aligned}
(1-2\sqrt{\delta_0})(1-C\varepsilon)(A\alpha_i-C\delta_i)(1-\delta_*)\geq A\alpha_i(1-C(\sqrt{\delta_0}+\varepsilon+\delta_i+\delta_*)),
\end{aligned}
\end{align}

we organize the above estimate as follows.

\begin{align}
\begin{aligned}
J_1\leq -\sum_{i\in\{1,3\}}\int_0^1 \frac{3}{4}A\alpha_i(1-C(\sqrt{\delta_0}+\varepsilon+\delta_i+\delta_*))\delta_iy_i(1-y_i)\abs{\myparas{w_i}_{y_i}}^2dy_i.
\end{aligned}
\end{align}

Now, we consider another term which is singular at time $t=0$.
By \eqref{eq:shiftsp1},
\begin{equation*}
    \frac{1}{2}\myparas{X_3(t)+\sigma_3t-X_1(t)-\sigma_1t}\geq \frac{\sigma_3-\sigma_1}{4}t>0, \quad T\geq t>0
\end{equation*}
we have 
\begin{align}
\begin{aligned}\label{eq:cutof1}
&\abs{\varphi_{ix}}\leq \frac{4}{\sigma_3-\sigma_1}\frac{1}{t},\quad \abs{\varphi_{ix}}^2\leq \frac{16}{(\sigma_3-\sigma_1)^2}\frac{1}{t^2}, \quad T\geq t>0, \quad \forall x\in \mathbb{R}.
\end{aligned}
\end{align}
From this, we can compute 
\begin{align}
\begin{aligned}
&J_2=\frac{C}{\delta_*}\int_\mathbb{R} a \frac{\mu(\y{\theta})}{v}\abs{\varphi_{ix}}^2\abs{\psi_1}^2dx\leq \frac{C}{\delta_*t^2}\int_\mathbb{R} \eta(U\mid\bar{U})dx.
\end{aligned}
\end{align}
Thus, we have the estimate of $\wtilde{D}_{u_1}$ as follows:
\begin{align}
\begin{aligned}
&-\wtilde{D}_{u_1}\leq-\sum_{i\in\{1,3\}}\int_0^1\frac{3}{4}A\alpha_i\myparas{1-C\myparas{\sqrt{\delta_0}+\varepsilon+\delta_i+\delta_*}}\delta_iy_i(1-y_i)\abs{\myparas{w_i}_{y_i}}^2dy_i\\
&\qquad \qquad \qquad +\frac{C}{\delta_*t^2}\int_\mathbb{R}\eta(U\mid\bar{U})dx.
\end{aligned}
\end{align}

{We treat} $\wtilde{D}_{\theta_1}(U)$ {in the same manner.} Take $\mathfrak z_i = \varphi_i \zeta$ for $i=1,3$. Then, we have the following estimate.

\begin{align}
\begin{aligned}\label{eq:diffusionx}
-\mathcal D_{\mathrm{mac}}(U)\leq&-\sum_{i\in\{1,3\}}\int_0^1A\alpha_i\myparas{1-C\myparas{\sqrt{\delta_0}+\varepsilon+\delta_i+\delta_*}}\delta_iy_i(1-y_i)\abs{\myparas{w_i}_{y_i}}^2dy_i\\
&-\sum_{i\in\{1,3\}}\int_0^1\frac{3}{4}A\alpha_i\frac{\gamma}{\theta_i}\myparas{1-C\myparas{\sqrt{\delta_0}+\varepsilon+\delta_i+\delta_*}}\delta_iy_i(1-y_i)\abs{\myparas{\mathfrak z_i}_{y_i}}^2dy_i\\
& +\frac{C}{\delta_*t^2}\int_\mathbb{R}\eta(U\mid\bar{U})dx-\int_\mathbb{R} a \myparas{\frac{\mu(\theta)}{v}\sum_{k=2}^3\psi_{kx}^2}dx.
\end{aligned}
\end{align}

Now, to replace $\mathfrak z_i$ by $w_i$, we compute these as follows.

\begin{align}
\begin{aligned}
&\int_0^1 \abs{\mathfrak z_i}^2 dy_i \geq \myparas{\frac{2\theta_i}{3\sigma_i^{\text{m}}v_i}}^2\myparas{1-\delta_i^{\frac{1}{4}}}\int_0^1 \abs{w_i}^2 dy_i-\frac{C}{\delta_i^{\frac{1}{4}}}\int_0^1 \myparab{\varphi_i\myparas{\zeta-\frac{2\theta_i}{3\sigma_i^{\text{m}}v_i}\psi_{1}}}^2dy_i,\\
&\myparas{\int_0^1 \mathfrak z_i dy_i}^2\leq 2\myparab{\int_0^1 \frac{2\theta_i}{3\sigma_i^{\text{m}}v_i}\varphi_i\psi_{1}dy_i}^2+2\myparab{\int_0^1 \varphi_i\myparas{\zeta-\frac{2\theta_i}{3\sigma_i^{\text{m}}v_i}\varphi_i\psi_{1}}dy_i}^2.
\end{aligned}
\end{align}
Now, we introduce the following useful lemma for estimating \eqref{eq:diffusionx}. 
\begin{lemma}[\cite{KV-JEMS21}] \label{lem:poincare}
    For any $f:[0,1]\rightarrow \mathbb{R}$ satisfying $\int_0^1 y(1-y)\abs{f'}^2 dy<\infty$,
    \begin{equation}\label{eq:poincare}
        \int_0^1 \abs{f-\int_0^1 f dy}^2 dy \leq \frac{1}{2} \int_0^1 y(1-y)\abs{f'}^2 dy.
    \end{equation}
\end{lemma}

By the Poincare inequality \eqref{eq:poincare}, we have 
\begin{align}
\begin{aligned}
&-\mathcal D_{\mathrm{mac}}(U)\\
&\quad \le -\sum_{i\in\{1,3\}}2A\alpha_i\myparas{1-C\myparas{\sqrt{\delta_0}+\varepsilon+\delta_i+\delta_*}}\delta_i\myparab{\int_0^1 w_i^2 dy_i-\bigl(w_i^{\mathrm{avg}}\bigr)^2}\\
&\quad \quad-\sum_{i\in\{1,3\}}\frac{3}{4}2A\alpha_i\frac{\gamma}{\theta_i}\myparas{1-C\myparas{\sqrt{\delta_0}+\varepsilon+\delta_i+\delta_*}}\delta_i\myparab{\int_0^1z_i^2 dy_i-\bigl(z_i^{\mathrm{avg}}\bigr)^2}\\
&\quad \quad+\frac{C}{\delta_*t^2}\int_\mathbb{R}\eta(U\mid\bar{U})dx-\int_\mathbb{R} a \myparas{\frac{\mu(\theta)}{v}\sum_{k=2}^3\psi_{kx}^2}dx \\
&\quad \le  -\sum_{i\in\{1,3\}}2A\alpha_i\myparas{1-C\myparas{\sqrt{\delta_0}+\varepsilon+\delta_*}}\delta_i\myparam{1+\frac{3}{4}\frac{\gamma}{\theta_i}\myparas{\frac{2\theta_i}{3\sigma_i^{\text{m}}v_i}}^2\myparas{1-\delta_i^{\frac{1}{4}}}}\int_0^1 w_i^2 dy_i\\
&\quad \quad +\sum_{i\in\{1,3\}}2A\alpha_i\myparas{1-C\myparas{\sqrt{\delta_0}+\varepsilon+\delta_*}}\delta_i\myparam{1+\frac{3}{4}2\frac{\gamma}{\theta_i}\myparas{\frac{2\theta_i}{3\sigma_i^{\text{m}}v_i}}^2}\bigl(w_i^{\mathrm{avg}}\bigr)^2\\
&\quad \quad +\sum_{i\in\{1,3\}}\frac{3}{2}A\alpha_i\frac{\gamma}{\theta_i}\myparas{1-C\myparas{\sqrt{\delta_0}+\varepsilon+\delta_*}}\delta_i\myparas{\frac{C}{\delta_i^{\frac{1}{4}}}-2}\int_0^1 \myparab{\varphi_i\myparas{\zeta-\frac{2\theta_i}{3\sigma_i^{\text{m}}v_i}\psi_{1i}}}^2 dy_i\\
&\quad \quad +\frac{C}{\delta_*t^2}\int_\mathbb{R}\eta(U\mid\bar{U})dx-\int_\mathbb{R} a \myparas{\frac{\mu(\theta)}{v}\sum_{k=2}^3\psi_{kx}^2}dx.
\end{aligned}
\end{align}
Also, we control {the third term on the right-hand side of the above inequality as follows:}
\begin{align}
\begin{aligned}
& \sum_{i\in\{1,3\}}\frac{3}{2}A\alpha_i\frac{\gamma}{\theta_i}\myparas{1-C\myparas{\sqrt{\delta_0}+\varepsilon+\delta_*}}\delta_i\myparas{\frac{C}{\delta_i^{\frac{1}{4}}}-2}\int_0^1 \myparab{\varphi_i\myparas{\zeta-\frac{2\theta_i}{3\sigma_i^{\text{m}}v_i}\psi_{1}}}^2 dy_i \\
& \le \sum_{i\in\{1,3\}}\frac{3}{2}CA\alpha_i\frac{\gamma}{\theta_i}\myparas{1-C\myparas{\sqrt{\delta_0}+\varepsilon+\delta_*}}\delta_i^{\frac{3}{4}}\int_0^1\varphi_i \myparab{\myparas{\zeta-\frac{2\theta_i}{3\sigma_i^{\text{m}}v_i}\psi_{1}}}^2 dy_i\\
& \le \sum_{i\in\{1,3\}}C\delta_i^{\frac{1}{4}}\mathfrak G_{2i}.
\end{aligned}
\end{align}
We organize {the first term on the right-hand side of the above inequality as follows:}
\begin{align}
\begin{aligned}
&-\sum_{i\in\{1,3\}}2A\alpha_i\myparas{1-C\myparas{\sqrt{\delta_0}+\varepsilon+\delta_*}-\delta_i^{\frac{1}{4}}}\delta_i\myparam{1+\frac{3}{4}\frac{\gamma}{\theta_i}\myparas{\frac{2\theta_i}{3\sigma_i^{\text{m}}v_i}}^2}\int_0^1 w_i^2 dy_i\\
\le&-\sum_{i\in\{1,3\}}2\alpha_i\myparas{1-C\myparas{\sqrt{\delta_0}+\varepsilon+\delta_*}-\delta_i^{\frac{1}{4}}}\delta_i\int_0^1 w_i^2 dy_i.
\end{aligned}
\end{align}

Therefore, we have
\begin{align}
\begin{aligned}
& \sum_{i\in\{1,3\}}2A\alpha_i\myparas{1-C\myparas{\sqrt{\delta}+\varepsilon+\delta_*}}\delta_i\myparam{1+\frac{3}{2}\frac{\gamma}{\theta_i}\myparas{\frac{2\theta_i}{3\sigma_iv_i}}^2}\bigl(w_i^{\mathrm{avg}}\bigr)^2\\
\leq & \sum_{i\in\{1,3\}}\frac{20+12\gamma}{10+3\gamma} \alpha_i\delta_i\bigl(w_i^{\mathrm{avg}}\bigr)^2.
\end{aligned}
\end{align}

{Thus, we obtain} \eqref{eq:diffusion}{.}

\subsubsection*{Conclusion of the dominant estimate}
Choose $\delta_*=O(1)>0$ such that 
\begin{align*}
    1-C\myparas{\sqrt{\delta_0}+\varepsilon+\delta_*}-\delta_i^{\frac{1}{4}}>\frac{1}{3} , \quad i=1,3. 
\end{align*}

Collecting the estimates above, we find
\begin{align*}
&-\sum_{i\in\{1,3\}}\frac{\delta_i}{2\mathfrak m_i}|\dot X_i|^2
+\mathcal B^{\text{sh}}+\mathcal B_2-\mathfrak G-\frac34D\\
&\quad\le\;
-\alpha\sum_{i\in\{1,3\}}\mathcal G_i^S
+C\delta_1^{\frac{4}{3}}\bigl(\delta_C^{\frac{4}{3}}+\delta_3^{\frac{4}{3}}\bigr)e^{-C\delta_1t}+ C\delta_3^{\frac{4}{3}}\bigl(\delta_C^{\frac{4}{3}}+\delta_1^{\frac{4}{3}}\bigr)e^{-C\delta_3t}
\\
&\qquad +\left(\sum_{i\in\{1,3\}}C\delta_i^2e^{-C\delta_i t}+\frac{C}{\delta_*t^2}\right)\int_{\mathbb R}\eta(U|\bar U)\,dx
-\frac34\int_{\mathbb R}a\frac{\mu(\theta)}{v}\sum_{k=2}^3\psi_{kx}^2\,dx,
\end{align*}
for some positive constant \(\alpha=O(1)\). This proves Lemma~\ref{lem:actr1}.\end{proof}

\subsection{Estimate of the macroscopic error terms}
In this subsection, we estimate the remaining macroscopic error terms
\[
\mathcal B_3,\dots,\mathcal B_6,\, \mathcal B^{\text{res}}, \mathcal S_1,\, \mathcal S_2.
\]
Here the term \(\mathcal B_2\) has already been incorporated into the dominant combination treated in the previous subsection. The terms considered below are all of lower order and will be controlled either by the dissipation \(\mathcal D_{\mathrm{mac}}(U)\), or by the smallness of the wave strengths and the bootstrap bound \(\mathcal{E}(T)^2\le \varepsilon^2\) in \eqref{eq:priass}.

We begin with the term involving the derivative of the weight function.

\begin{lemma}\label{lem:B3}
For any sufficiently small constant \(\lambda>0\), there exists a constant \(C_\lambda>0\) such that
\begin{align}\label{eq:B3-est}
&|\mathcal B_3|
\le
\lambda \mathcal D_{\mathrm{mac}}(U)
+
C_\lambda
\int_{\mathbb R}\bigl(\delta_1|\partial_x(v^{S_1})^{-X_1}|+\delta_3|\partial_x(v^{S_3})^{-X_3}|\bigr)|(\phi,\psi,\zeta)|^2\,dx \\
&\qquad + C_\lambda
\int_{\mathbb R}
\bigl(|\partial_x(v^{S_1})^{-X_1}|^2+|\partial_x(v^{S_3})^{-X_3}|^2+|u_{1x}^C|^2+|\theta_x^C|^2\bigr)
|(\phi,\psi,\zeta)|^2\,dx.
\end{align}
\end{lemma}

\begin{proof}
By the definition of the weight \(a\), we have
\[
|a_x|
\le
\frac{C}{\sqrt{\delta_1}}|\partial_x(v^{S_1})^{-X_1}|
+
\frac{C}{\sqrt{\delta_3}}|\partial_x(v^{S_3})^{-X_3}|.
\]
Since \(\delta_1,\delta_3\) are sufficiently small and the shock derivatives are exponentially localized, the factor \(a_x\) is supported only in the shock region and is controlled by the shock profiles.

Using the smoothness of \(\alpha_{\rm{th}}\) and \(\mu\), together with
\[
\alpha_{\rm{th}}(\theta)-\alpha_{\rm{th}}(\bar\theta)=O(|\zeta|),
\qquad
\mu(\theta)-\mu(\bar\theta)=O(|\zeta|),
\]
and Young's inequality, we obtain
\begin{align*}
&|\mathcal B_3|
\le
\lambda \mathcal D_{\mathrm{mac}}(U)
+
C_\lambda
\int_{\mathbb R}\bigl(\delta_1|\partial_x(v^{S_1})^{-X_1}|+\delta_3|\partial_x(v^{S_3})^{-X_3}|\bigr)|(\phi,\psi,\zeta)|^2\,dx \\
&\qquad + C_\lambda
\int_{\mathbb R}
\bigl(|\partial_x(v^{S_1})^{-X_1}|^2+|\partial_x(v^{S_3})^{-X_3}|^2+|u_{1x}^C|^2+|\theta_x^C|^2\bigr)
|(\phi,\psi,\zeta)|^2\,dx.
\end{align*}
This proves \eqref{eq:B3-est}.
\end{proof}

We next estimate the quadratic remainder terms.

\begin{lemma}\label{lem:B4B5B6}
For any sufficiently small constant \(\lambda>0\), there exists a constant \(C_\lambda>0\) such that
\begin{align}\label{eq:B4B5B6-est}
&|\mathcal B_4|+|\mathcal B_5|+|\mathcal B_6|\nonumber\\
&\qquad \le
\lambda \mathcal D_{\mathrm{mac}}(U) +C_\lambda
\int_{\mathbb R}
\bigl(|\partial_x(v^{S_1})^{-X_1}|^2+|\partial_x(v^{S_3})^{-X_3}|^2+|u_{1x}^C|^2+|\theta_x^C|^2\bigr)
|(\phi,\psi,\zeta)|^2\,dx.
\end{align}
\end{lemma}

\begin{proof}
All the terms in \(\mathcal B_4\), \(\mathcal B_5\), and \(\mathcal B_6\) are at least quadratic in the perturbation variables. Moreover, each coefficient depends smoothly on the background state \(\bar U\), hence is uniformly bounded under the bootstrap bound \(\mathcal{E}(T)^2\le \varepsilon^2\) in \eqref{eq:priass} and the smallness of \(\delta_0\).

For example,
\[
\left|
\frac{\mu(\bar\theta)}{v\bar v}\bar u_{1x}\psi_{1x}\phi
\right|
\le
\lambda \frac{\mu(\bar\theta)}{v}\psi_{1x}^2
+
C_\lambda |\bar u_{1x}|^2\,\phi^2,
\]
and similarly
\[
\left|
\frac{\alpha_{\rm{th}}(\bar\theta)}{v\bar v\theta}\bar\theta_x\zeta_x\phi
\right|
\le
\lambda \frac{\alpha_{\rm{th}}(\bar\theta)}{v\theta}\zeta_x^2
+
C_\lambda |\bar\theta_x|^2\,\phi^2.
\]
All remaining terms are treated in the same way, using
\[
|\mu(\theta)-\mu(\bar\theta)|+|\alpha_{\rm{th}}(\theta)-\alpha_{\rm{th}}(\bar\theta)|
\le C|\zeta|
\]
and the fact that
\[
|\bar u_{1x}|^2+|\bar\theta_x|^2
\le
C\bigl(|\partial_x(v^{S_1})^{-X_1}|^2+|\partial_x(v^{S_3})^{-X_3}|^2+|u_{1x}^C|^2+|\theta_x^C|^2\bigr).
\]
Therefore \eqref{eq:B4B5B6-est} follows after summing all contributions.
\end{proof}

The contact-wave error term is estimated as follows.

\begin{lemma}\label{lem:B7}
There exists a positive constant \(C\) such that
\begin{align}\label{eq:B7-est}
&|\mathcal B^{\textup{res}}|
\nonumber\\
&\le
\frac{C\delta_C}{1+t}
\int_{\mathbb R}e^{-\frac{C_1|x|^2}{1+t}}|(\phi,\zeta)|^2\,dx
+
C(\delta_0+\varepsilon)
\int_{\mathbb R}
\bigl(|\partial_x(v^{S_1})^{-X_1}|+|\partial_x(v^{S_3})^{-X_3}|\bigr)|(\phi,\zeta)|^2\,dx.
\end{align}
\end{lemma}

\begin{proof}
This is immediate from the definition of \(\mathcal B^{\text{res}}\).
\end{proof}

We now estimate the profile-error terms \(\mathcal S_1\) and \(\mathcal S_2\).

\begin{lemma}\label{lem:S1S2}
There exists a constant \(C>0\) such that
\begin{equation}\label{eq:S1S2-est}
|\mathcal S_1|+|\mathcal S_2|
\le
C(\varepsilon+\delta_0)\myparam{\frac{\delta_C}{\myparas{1+t}^{\frac{5}{4}}}+\delta_1\myparas{\delta_3+\delta_C}e^{-C\delta_1t}+\delta_3\myparas{\delta_1+\delta_C}e^{-C\delta_3t}}
\end{equation}
\end{lemma}

\begin{proof}
Recall that
\[
\mathcal S_1=-\int_{\mathbb R}a\,\psi_1Q_1\,dx,
\qquad
\mathcal S_2=-\int_{\mathbb R}a\,\frac{\zeta}{\bar\theta}Q_2\,dx.
\]
Using the decomposition of \(Q_1\) and \(Q_2\) in \eqref{eq:Q1Q2}--\eqref{eq:Q2Q2}, together with the Gaussian bound for the viscous contact wave and the exponential localization of the shock profiles, we obtain
\begin{align*}
    \abs{\mathcal S_1}+\abs{\mathcal S_2} &\leq \myparas{\norm{Q_1^I}_{L^2}+\norm{Q_1^C}_{L^2}+\norm{Q_2^I}_{L^2}+\norm{Q_2^C}_{L^2}}\norm{\myparas{\psi,\zeta}}_{L^2}\\
    &\leq C(\varepsilon+\delta_0)\myparam{\frac{\delta_C}{\myparas{1+t}^{\frac{5}{4}}}+\delta_1\myparas{\delta_3+\delta_C}e^{-C\delta_1t}+\delta_3\myparas{\delta_1+\delta_C}e^{-C\delta_3t}}
\end{align*}
The desired estimate then follows from Young's inequality.
\end{proof}

Collecting the previous lemmas, we obtain the following macroscopic estimate.

\begin{proposition}\label{prop:macro-error}
For any sufficiently small constant \(\lambda>0\), there exist constants \(C>0\) and \(C_\lambda>0\) such that
\begin{equation}\label{eq:macro-error-final}
\begin{split}
&\sum_{i=3}^6 |\mathcal B_i| + |\mathcal B^{\textup{res}}| + |\mathcal S_1| + |\mathcal S_2|\\&
\le
\lambda \mathcal D_{\mathrm{mac}}(U)
+
C_\lambda \delta_C\frac{1}{1+t}
\int_{\mathbb R}e^{-C_1|x|^2/(1+t)}|(\phi,\psi,\zeta)|^2\,dx\\
&
+
C_\lambda(\delta_0+\varepsilon)
\int_{\mathbb R}
\bigl(|\partial_x(v^{S_1})^{-X_1}|+|\partial_x(v^{S_3})^{-X_3}|\bigr)
|(\phi,\psi,\zeta)|^2\,dx\\
& + C(\varepsilon+\delta_0)\myparam{\frac{\delta_C}{\myparas{1+t}^{\frac{5}{4}}}+\delta_1\myparas{\delta_3+\delta_C}e^{-C\delta_1t}+\delta_3\myparas{\delta_1+\delta_C}e^{-C\delta_3t}}.
\end{split}
\end{equation}
\end{proposition}

\begin{proof}
This follows immediately from Lemmas~\ref{lem:B3}--\ref{lem:S1S2}.
\end{proof}

We also need the contribution of the remaining modulation terms
\[
\sum_{i=1,3}\dot X_i\sum_{j=4}^6\mathcal Y_{ij},
\]
which are of lower order.

\begin{lemma}\label{lem:Y456}
There exists a constant \(C>0\) such that
\begin{equation}\label{eq:Y456-est}
\left|
\sum_{i=1,3}\dot X_i\sum_{j=4}^6 \mathcal Y_{ij}
\right|
\le
\sum_{i=1,3}\frac{\delta_i}{8\mathfrak m_i}|\dot X_i|^2
+ C(\delta_0+\varepsilon)^2\sum_{i=1,3}\mathcal G_i^S +
C (\delta_0+\varepsilon)^2
\sum_{i=1,3}\delta_i^2e^{-C\delta_i t}.
\end{equation}
\end{lemma}

\begin{proof}
By the definitions of \(Y_{i4}\), \(Y_{i5}\), and \(Y_{i6}\), together with
\[
\Phi\!\left(\frac{v}{\bar v}\right)\sim \phi^2,
\qquad
\Phi\!\left(\frac{\theta}{\bar\theta}\right)\sim \zeta^2,
\]
and the bootstrap bound \(\mathcal{E}(T)^2\le \varepsilon^2\) in \eqref{eq:priass}, we obtain

\begin{align*}
\abs{\mathcal Y_{i4}} \le & C \int \abs{\partial_x\shockw{\theta}{i}} \abs{\phi}^2 dx \\
\le & C \int \abs{\varphi_i\partial_x\shockw{\theta}{i}} \abs{\phi}^2 dx + C \int \abs{(1-\varphi_i^2)\partial_x\shockw{\theta}{i}} \abs{\phi}^2 dx \\
\le & C \mathcal{G}_i^S + C \int \abs{(1-\varphi_i)\partial_x\shockw{\theta}{i}}\abs{\phi}^2 dx  \\
\le & C \mathcal{G}_i^S + C(\varepsilon+\delta_0)^2\delta_i^2e^{-C\delta_it},
\end{align*}

\begin{align*}
\frac{C}{\delta_i}\abs{\mathcal Y_{i4}}^2 \le& \frac{C\abs{\mathcal Y_{i4}}}{\delta_i}\abs{\int a\partial_x\shockw{\theta}{i} \Phi\left(\frac{v}{\bar v}\right) dx}\\
\le& C\delta_i(\varepsilon+\delta_0)^2\left(\mathcal{G}_i^S +(\varepsilon+\delta_0)^2\delta_i^2e^{-C\delta_it}\right).
\end{align*}

{The same argument applied to} $Y_{i5}$ {gives}

\begin{equation}\label{eq:Y45-est}
\frac{C}{\delta_i}\bigl(|\mathcal Y_{i4}|^2+|\mathcal Y_{i5}|^2\bigr)
\le
C\delta_i(\delta_0+\varepsilon)^2
\left(\mathcal G_i^S+(\delta_0+\varepsilon)^2\delta_i^2e^{-C\delta_i t}\right),
\qquad i=1,3.
\end{equation}
Similarly, observe that
\begin{align}\label{eq:Y6-est}
\frac{C}{\delta_i}|\mathcal Y_{i6}|^2 \le & \frac{C}{\delta_i^2}\myparas{\int \abs{\partial_x\shockw{v}{i}}\abs{(\phi,\psi,\zeta)}^2 dx}^2 \nonumber\\
\le & C(\delta_0+\varepsilon)^2\myparas{\mathcal G_i^S+(\delta_0+\varepsilon)^2\delta_i^2e^{-C\delta_i t}}, \qquad i=1,3.
\end{align}

Therefore, by Young's inequality,
\begin{align*}
\left|
\sum_{i=1,3}\dot X_i\sum_{j=4}^6\mathcal Y_{ij}
\right|
&\le
\sum_{i=1,3}\left(
\frac{\delta_i}{8\mathfrak m_i}|\dot X_i|^2
+
\frac{C}{\delta_i}\sum_{j=4}^6|\mathcal Y_{ij}|^2
\right)\\
&\le
\sum_{i=1,3}\frac{\delta_i}{8\mathfrak m_i}|\dot X_i|^2
+
C(\delta_0+\varepsilon)^2\sum_{i=1,3}\mathcal G_i^S \\
&\qquad +
C (\delta_0+\varepsilon)^2
\sum_{i=1,3}\delta_i^2e^{-C\delta_i t},
\end{align*}
which proves \eqref{eq:Y456-est}.
\end{proof}

As a consequence, combining Proposition~\ref{prop:macro-error} with Lemma~\ref{lem:Y456}, we obtain
\begin{equation}\label{eq:macro-error-plus-shift}
\begin{aligned}
&\sum_{l=3}^6 |\mathcal B_l| + \mathcal B^{\text{res}} + |\mathcal S_1| + |\mathcal S_2|
+\left|
\sum_{i=1,3}\dot X_i\sum_{j=4}^6 \mathcal Y_{ij}
\right|\\
&\le
\sum_{i=1,3}\frac{\delta_i}{8\mathfrak M_i}|\dot X_i|^2
+
\frac{1}{32} \mathcal D_{\mathrm{mac}}(U)
+
C \delta_C\frac{1}{1+t}
\int_{\mathbb R}e^{-\frac{C_1|x|^2}{1+t}}|(\phi,\psi,\zeta)|^2\,dx\\
&\quad
+
C(\delta_0+\varepsilon)\sum_{i=1,3}\mathcal G_i^S + C(\varepsilon+\delta_0)\myparam{\frac{\delta_C}{\myparas{1+t}^{\frac{5}{4}}}+\delta_1\delta_0e^{-C\delta_1t}+\delta_3\delta_0e^{-C\delta_3t}}.
\end{aligned}
\end{equation}

\subsection{Estimate of the microscopic error terms}
In this subsection, we estimate the kinetic error terms
\[
K_1,\dots,K_6
\]
appearing in the weighted entropy identity \eqref{eq:energy2}. These terms arise from the microscopic part of the perturbation and will be controlled by the microscopic dissipation together with the smallness of the macroscopic perturbation and the exponential localization of the composite wave.

For convenience, we introduce the microscopic dissipation functional
\begin{equation}\label{eq:mic-diss}
\mathcal D_{\mathrm{mic}}(t)
:=
\int_{\mathbb R}\norm{\widetilde{G}_{\text{rem}}}_{\nu,M_\#}^2\,dx.
\end{equation}

We first estimate the terms involving the difference
\begin{align}\label{eq:pi1w}
\widetilde{\Pi}_1:=\,\Pi_1-(\Pi_1^{S_1})^{-X_1}-(\Pi_1^{S_3})^{-X_3}.
\end{align}
 
\begin{lemma}\label{lem:K1245}
There exists a positive constant \(C\) such that
\begin{align}
\begin{aligned}
&|K_1|+|K_2|+|K_3|+|K_4|+|K_5|\\
&\quad \le\; C(\delta_0+\varepsilon)\sum_{i\in\{1,3\}}\delta_i|\dot X_i|^2 + C\delta_0\sum_{i\in\{1,3\}}\mathcal G_i^S + \frac1{20}\mathcal D_{\mathrm{mac}}(U) + C(\delta_0+\varepsilon)\mathcal D_{\mathrm{mic}}(t) \\
&\qquad + C\delta_0(\delta_0+\varepsilon)\bigl(\delta_1e^{-C\delta_1 t}+\delta_3e^{-C\delta_3 t}\bigr) + C\frac{\delta_C}{1+t}\int_{\mathbb R}e^{-\frac{2c_0x^2}{1+t}}|(\phi,\psi,\zeta)|^2\,dx\\
&\qquad + \frac{C\delta_C}{(1+t)^{\frac{5}{4}}} + C\int_{\mathbb R}\norm{\widetilde{G}_t}_{\nu,M_\#}^2+\norm{\widetilde{G}_x}_{\nu,M_\#}^2\,dx .
\end{aligned}
\label{eq:K1245-est}
\end{align}
\end{lemma}

\begin{proof}

\noindent\textit{Control of \(K_4\).}
By definition,
\[
K_4
=
-\iint a\psi_1\,\xi_1^2
\widetilde{\Pi}_{1x}\,d\xi\,dx.
\]
Integrating by parts in \(x\), we obtain
\begin{equation}\label{eq:K4-ibp}
K_4
=
\iint (a\psi_1)_x\,\xi_1^2
\widetilde{\Pi}_1\,d\xi\,dx.
\end{equation}
Hence, for arbitrary \(\alpha,\beta>0\),
\begin{align}
|K_4|
\le\;&
\frac{\alpha}{2}\int_{\mathbb R}(a_x\psi_1)^2\,dx
+\frac{\alpha^{-1}}{2}\int_{\mathbb R}
\left(
\int_{\mathbb R^3}\xi_1^2
\widetilde{\Pi}_1\,d\xi
\right)^2dx
\nonumber\\
&+
\frac{\beta}{2}\int_{\mathbb R}(a\psi_{1x})^2\,dx
+\frac{\beta^{-1}}{2}\int_{\mathbb R}
\left(
\int_{\mathbb R^3}\xi_1^2
\widetilde{\Pi}_1\,d\xi
\right)^2dx.
\label{eq:K4-Young}
\end{align}

We now decompose the perturbative microscopic flux. Recall that
\begin{align}
\mathcal N(\widetilde G,\widetilde G)
=&-\mathcal N\bigl((G^{S_1})^{-X_1},\widetilde G\bigr)-\mathcal N\bigl(\widetilde G,(G^{S_1})^{-X_1}\bigr)-\mathcal N\bigl((G^{S_1})^{-X_1},(G^{S_1})^{-X_1}\bigr)+\mathcal N(G,G)\nonumber\\
&-\mathcal N\bigl(\widetilde G,(G^{S_3})^{-X_3}\bigr)-\mathcal N\bigl((G^{S_1})^{-X_1},(G^{S_3})^{-X_3}\bigr)-\mathcal N\bigl((G^{S_3})^{-X_3},\widetilde G\bigr)\nonumber\\
&-\mathcal N\bigl((G^{S_3})^{-X_3},(G^{S_3})^{-X_3}\bigr)-\mathcal N\bigl((G^{S_3})^{-X_3},(G^{S_1})^{-X_1}\bigr).
\label{eq:Q-tilde-decomp}
\end{align}
Accordingly,
\begin{align}
& \widetilde{\Pi}_1
=
L_M^{-1}
\left(
\widetilde G_t-\frac{u_1}{v}\widetilde G_x+\frac1vP_1(\xi_1\widetilde G_x)-\mathcal N(\widetilde G,\widetilde G)
\right)
\nonumber\\
&\qquad
-L_M^{-1}
\left[
\mathcal N\bigl(\widetilde G,(G^{S_1})^{-X_1}+(G^{S_3})^{-X_3}\bigr)
+
\mathcal N\bigl((G^{S_1})^{-X_1}+(G^{S_3})^{-X_3},\widetilde G\bigr)
\right]
\nonumber\\
&\qquad
-L_M^{-1}
\left[
\mathcal N\bigl((G^{S_1})^{-X_1},(G^{S_3})^{-X_3}\bigr)
+
\mathcal N\bigl((G^{S_3})^{-X_3},(G^{S_1})^{-X_1}\bigr)
\right]
\nonumber\\
&\qquad
-\dot X_1(L_{1}^{S})^{-1}\partial_x(G^{S_1})^{-X_1}
-\dot X_3(L_{3}^{S})^{-1}\partial_x(G^{S_3})^{-X_3}
+J,
\label{eq:Pi-perturb}
\end{align}
where
\begin{align}
J:=\;&
\sum_{i\in\{1,3\}}
\bigl(L_M^{-1}-(L_{i}^{S})^{-1}\bigr)
\left((G^{S_i})_t^{-X_i}-\mathcal N\bigl((G^{S_i})^{-X_i},(G^{S_i})^{-X_i}\bigr)\right)
\nonumber\\
&+
\sum_{i\in\{1,3\}}
\left(
\frac{(u_1^{S_i})^{-X_i}}{(v^{S_i})^{-X_i}}(L_{i}^{S})^{-1}
-\frac{u_1}{v}L_M^{-1}
\right)\partial_x(G^{S_i})^{-X_i}
\nonumber\\
&+
\sum_{i\in\{1,3\}}
\left(
\frac1vL_M^{-1}P_1\xi_1
-\frac1{(v^{S_i})^{-X_i}}(L_{i}^{S})^{-1}P_1^{S_i}\xi_1
\right)\partial_x(G^{S_i})^{-X_i}.
\label{eq:J-term}
\end{align}

Using the weighted inverse estimate for \(L_M^{-1}\) and the collision estimates from Appendix~A, we have
\begin{align}
\int_{\mathbb R}
\left(
\int_{\mathbb R^3}\xi_1^2L_M^{-1}\widetilde G_t\,d\xi
\right)^2dx
&\le
C\iint\frac{(1+|\xi|)^{-1}|\widetilde G_t|^2}{M_\#}\,d\xi\,dx,
\label{eq:K4-Gt}
\\
\int_{\mathbb R}
\left(
\int_{\mathbb R^3}\xi_1^2L_M^{-1}\left(\frac{u_1}{v}\widetilde G_x\right)\,d\xi
\right)^2dx
&\le
C\iint\frac{(1+|\xi|)^{-1}|\widetilde G_x|^2}{M_\#}\,d\xi\,dx,
\label{eq:K4-Gx}
\\
\int_{\mathbb R}
\left(
\int_{\mathbb R^3}\xi_1^2L_M^{-1}\left(\frac1vP_1(\xi_1\widetilde G_x)\right)\,d\xi
\right)^2dx
&\le
C\iint\frac{(1+|\xi|)^{-1}|\widetilde G_x|^2}{M_\#}\,d\xi\,dx.
\label{eq:K4-P1Gx}
\end{align}
Moreover,
\begin{align}
& \int_{\mathbb R}
\left(
\int_{\mathbb R^3}\xi_1^2L_M^{-1}\mathcal N(\widetilde G,\widetilde G)\,d\xi
\right)^2dx \nonumber \\
&\qquad \le
C(\delta_0+\varepsilon)\mathcal{D}_{\text{mic}}
+
C\delta_C^4
\left(
\frac1{(1+t)^4}+\frac1{(1+t)^2}
\right)
\int_{\mathbb R}e^{-c_0x^2/(1+t)}\,dx.
\label{eq:K4-QGG}
\end{align}

Here we used
\begin{align} \label{eq:GCe}
\norm{\widetilde{G}_C}_{\nu,M_\#}^2
\approx
|(u_{1x}^C,\theta_x^C)|^2.
\end{align}

Similarly,
\begin{align}
&\int_{\mathbb R}
\left(
\int_{\mathbb R^3}\xi_1^2L_M^{-1}\mathcal N\bigl(\widetilde G,(G^{S_1})^{-X_1}\bigr)\,d\xi
\right)^2dx \nonumber\\
&\qquad \le
C\delta_1^2\mathcal{D}_{\text{mic}}
+
C\int |(u_{1x}^C,\theta_x^C)|^2 |\partial_x\bigl((v^{S_1})^{-X_1}\bigr)|^2\,dx,
\label{eq:K4-mixed}
\end{align}
and the same estimate holds for the remaining mixed and interaction terms involving \((G^{S_3})^{-X_3}\).

Next, we estimate the term \(J\). Observe that
\[
L_M\bigl(L_M^{-1}-(L_{1}^{S})^{-1}\bigr)h
=
\mathcal N(M_1-M,(L_{1}^{S})^{-1}h)+\mathcal N((L_{1}^{S})^{-1}h,M_1-M)
\]
and
\begin{align*}
    &L_M (L_{1}^{S})^{-1} \shock{G}{1}_x \\
    &\quad = \mathcal N\Bigl(M-\shockw{M}{1},(L_{1}^{S})^{-1} \partial_x\shockw{G}{1}\Bigr) + \mathcal N\Bigl(\shockw{M}{1},(L_{1}^{S})^{-1} \partial_x\shockw{G}{1}\Bigr)\\
    &\qquad + \mathcal N\Bigl((L_{1}^{S})^{-1} \partial_x\shockw{G}{1},M-\shockw{M}{1}\Bigr)+\mathcal N\Bigl((L_{1}^{S})^{-1} \partial_x\shockw{G}{1},\shockw{M}{1}\Bigr)\\
    &\quad =\mathcal N\Bigl(M-
    \shockw{M}{1},(L_{1}^{S})^{-1} \partial_x\shockw{G}{1}\Bigr)\\
    &\qquad +\mathcal N\Bigl((L_{1}^{S})^{-1} \partial_x\shockw{G}{1},M-\shock{M}{1}\Bigr) +\partial_x\shockw{G}{1}.
\end{align*}
Using the identity
\[
M(a)-M(a')
=
\int_0^1 D_aM(a_s)\,(a-a')\,ds,
\qquad a_s:=a'+s(a-a'),
\]
with \(a=(v,u,\theta)\) and \(a'=(v',u',\theta')\), we infer
\[
\norm{M-\shockw{M}{1}}_{\nu,M_\#}^2
\le
C\Bigl|\bigl(v-\shockw{v}{1},u-\shockw{u}{1},\theta-\shockw{\theta}{1}\bigr)\Bigr|^2.
\]

Therefore,
\begin{align}
&\int_{\mathbb R}
\left(
\int_{\mathbb R^3}\xi_1^2\bigl(L_M^{-1}-(L_{1}^{S})^{-1}\bigr)\partial_x\shockw{G}{1}\,d\xi
\right)^2dx\nonumber\\
&\qquad \le
C\delta_1
\int_{\mathbb R}\left|\partial_x\shockw{v}{1}\right|^2
\left|\left(v-\shockw{v}{1},u-\shockw{u}{1},\theta-\shockw{\theta}{1}\right)\right|^2\,dx.
\label{eq:J-est1}
\end{align}
Using the decomposition
\begin{align}
&\left(v-\shockw{v}{1},u-\shockw{u}{1} ,\theta-\shockw{\theta}{1}\right)
\nonumber\\
&\qquad =\bigl(v-\y{v},u-\y{u},\theta-\y{\theta}\bigr)+
\left(\shockw{v}{3}-v^*,\shockw{u}{3}-u^*,\shockw{\theta}{3}-\theta^*\right)\nonumber\\
&\qquad \quad +\bigl(v^C-v_*,u^C-u_*,\theta^C-\theta_*\bigr)
\end{align}
the first part is controlled by \(\mathcal G_1^S\) of \eqref{eq:skGt}, while the second part is estimated by Lemma~\ref{lem:wave-interaction-L2}. Hence,
\begin{align}
&\int_{\mathbb R}\left|\partial_x\shockw{v}{1}\right|^2 \left|\left(v-\shockw{v}{1},u-\shockw{u}{1},\theta-\shockw{\theta}{1}\right)\right|^2\,dx\nonumber\\
&\qquad \le\;
C\delta_1\mathcal G_1^S
+
C\delta_1^3e^{-C\delta_1 t}\int_{\mathbb R}\eta(U\mid\bar U)\,dx +C(\delta_1^2\delta_3^2+\delta_1^2\delta_C^2)\delta_1e^{-C\delta_1t}\nonumber \\
&\qquad \qquad +C\delta_1^3\delta_3(\delta_3e^{-C\delta_3t})\nonumber \\
&\qquad \le\;C\delta_1\mathcal G_1^S + C(\delta_0+\varepsilon)^2\delta_1^3e^{-C\delta_1 t}+C\delta_0^4\bigl(\delta_1e^{-C\delta_1 t}+\delta_3e^{-C\delta_3t}\bigr).
\label{eq:interaction-est}
\end{align}

The same estimate holds for the \(3\)-shock.

Collecting \eqref{eq:K4-Young}--\eqref{eq:interaction-est} and choosing \(\alpha,\beta\) sufficiently small, we conclude that
\begin{align}
|K_4|
\le\;&
C\delta_0^4\sum_{i\in\{1,3\}}\delta_i|\dot X_i|^2
+
C\delta_0\sum_{i\in\{1,3\}}\mathcal G_i^S
+
C\delta_0(\varepsilon+\delta_0)^2\bigl(\delta_1e^{-C\delta_1 t}+\delta_3e^{-C\delta_3 t}\bigr)
\nonumber\\
&+\frac{C\delta_C}{(1+t)^{\frac{5}{4}}}
+\frac{C\delta_C^2}{1+t}\int_{\mathbb R}e^{-2c_0x^2/(1+t)}|(\phi,\zeta)|^2\,dx+\frac1{200}\mathcal D_{\mathrm{mac}}(U)\nonumber\\
&+C(\delta_0+\varepsilon)\mathcal D_{\mathrm{mic}}(t)+
C\int\norm{\wtilde{G}_t}_{\nu,M_\#}^2+\norm{\wtilde{G}_x}_{\nu,M_\#}^2\,dx.
\label{eq:K4-final}
\end{align}

\medskip

\noindent\textit{Control of \(K_1\).}
By definition,
\[
K_1
=
-\iint a\frac{\zeta}{\theta}\left(\xi_1\frac{|\xi|^2}{2}\right)
\widetilde{\Pi}_{1x}\,d\xi\,dx.
\]
Integrating by parts in \(x\), we obtain
\[
K_1
=
\iint
\left(a\frac{\zeta}{\theta}\right)_x
\left(\xi_1\frac{|\xi|^2}{2}\right)
\widetilde{\Pi}_{1}\,d\xi\,dx.
\]
Hence, by Young's inequality,

\begin{align}
|K_1|
\le\;&
\underbrace{\frac{\alpha_1}{2}\int_{\mathbb R}\left(a_x\frac{\zeta}{\theta}\right)^2\,dx}_{K_{11}}
+\frac{\alpha_1^{-1}}{2}\int_{\mathbb R}
\left(
\int_{\mathbb R^3}\xi_1\frac{\abs{\xi}^2}{2}
\widetilde{\Pi}_{1}\,d\xi
\right)^2dx
\nonumber\\
&+
\underbrace{\frac{\alpha_2}{2}\int_{\mathbb R}\left(a\frac{\zeta_x}{\theta}\right)^2\,dx}_{K_{12}}
+\frac{\alpha_2^{-1}}{2}\int_{\mathbb R}
\left(
\int_{\mathbb R^3}\xi_1\frac{\abs{\xi}^2}{2}
\widetilde{\Pi}_{1}\,d\xi
\right)^2dx \nonumber\\
&+
\underbrace{\frac{\alpha_3}{2}\int_{\mathbb R}\left(a\frac{\zeta\y{\theta}_x}{\theta^2}\right)^2\,dx}_{K_{13}}
+\frac{\alpha_3^{-1}}{2}\int_{\mathbb R}
\left(
\int_{\mathbb R^3}\xi_1\frac{\abs{\xi}^2}{2}
\widetilde{\Pi}_{1}\,d\xi
\right)^2dx \nonumber\\
&+
\underbrace{\frac{\alpha_4}{2}\int_{\mathbb R}\left(a\frac{\zeta\zeta_x}{\theta^2}\right)^2\,dx}_{K_{14}}
+\frac{\alpha_4^{-1}}{2}\int_{\mathbb R}
\left(
\int_{\mathbb R^3}\xi_1\frac{\abs{\xi}^2}{2}
\widetilde{\Pi}_{1}\,d\xi
\right)^2dx.
\label{eq:K1-Young}
\end{align}

The terms $K_{11}$ and $K_{13}$ are controlled by $\mathcal{G}_i^S$, and the terms $K_{12}$ and $K_{14}$ are controlled by the diffusion term $\mathcal D_{\mathrm{mac}}(U)$.

{Arguing as for} $K_4${,}

\begin{align}
& |K_1|\le\;
C\delta_0^4\sum_{i\in\{1,3\}}\delta_i|\dot X_i|^2
+
C\delta_0\sum_{i\in\{1,3\}}\mathcal G_i^S
+
C\delta_0(\varepsilon+\delta_0)^2\bigl(\delta_1e^{-C\delta_1 t}+\delta_3e^{-C\delta_3 t}\bigr)
\nonumber\\
& \qquad \quad+\frac{C\delta_C}{(1+t)^{\frac{5}{4}}}
+\frac{C\delta_C^2}{1+t}\int_{\mathbb R}e^{-\frac{2c_0x^2}{1+t}}|(\phi,\zeta)|^2\,dx+\frac1{200}\mathcal D_{\mathrm{mac}}(U)\nonumber\\
& \qquad \quad
+C(\delta_0+\varepsilon)\mathcal D_{\mathrm{mic}}(t)+
C\int\norm{\wtilde{G}_t}_{\nu,M_\#}^2+\norm{\wtilde{G}_x}_{\nu,M_\#}^2\,dx.
\label{eq:K1-final}
\end{align}

\medskip

\noindent\textit{Control of \(K_2\).}
We write
\begin{align}
K_2
&=
\iint a\frac{\zeta}{\theta}
\bigg[
u_1\xi_1^2\Pi_{1x}
-\sum_{i\in\{1,3\}}(u_1^{S_i})^{-X_i}\xi_1^2\partial_x(\Pi_1^{S_i})^{-X_i}
\bigg]d\xi\,dx
\nonumber\\
&=
\iint a\frac{\zeta}{\theta}u_1\xi_1^2
\widetilde{\Pi}_{1x}\,d\xi\,dx
\nonumber\\
&\quad
+
\sum_{i\in\{1,3\}}
\iint a\frac{\zeta}{\theta}\bigl(u_1-(u_1^{S_i})^{-X_i}\bigr)\xi_1^2\partial_x\bigl((\Pi_1^{S_i})^{-X_i}\bigr)\,d\xi\,dx
\nonumber\\
&=:J^{\text{mic}}_1+J^{\text{mic}}_2.
\label{eq:K2-split}
\end{align}

The term \(J_1^{\text{mic}}\) is treated exactly as \(K_4\). For \(J_2^{\text{mic}}\), using the exponential decay of \(\bigl((\Pi_1^{S_i})^{-X_i}\bigr)_x\) (See Lemma \ref{lem:bs1}), we obtain
\begin{align}
|J_2^{\text{mic}}|
\le\;&
C\delta_1^2\mathcal G_1^S
+
C\delta_1^2(\delta_1e^{-C\delta_1 t}+\delta_3e^{-C\delta_3 t})
\nonumber\\
&+
C\delta_3^2\mathcal G_3^S
+
C\delta_3^2(\delta_1e^{-C\delta_1 t}+\delta_3e^{-C\delta_3 t}).
\label{eq:J2-est}
\end{align}
Hence,
\begin{align}
|K_2|
\le\;&
C\delta_0^4\sum_{i\in\{1,3\}}\delta_i|\dot X_i|^2
+
C\delta_0\sum_{i\in\{1,3\}}\mathcal G_i^S
+
\frac1{200}\mathcal D_{\mathrm{mac}}(U)
+
C(\delta_0+\varepsilon)\mathcal D_{\mathrm{mic}}(t)
\nonumber\\
&+C\delta_0(\varepsilon+\delta_0)^2\bigl(\delta_1e^{-C\delta_1 t}+\delta_3e^{-C\delta_3 t}\bigr)+\frac{C\delta_C}{1+t}\int_{\mathbb R}e^{-\frac{2c_0x^2}{1+t}}|(\phi,\zeta)|^2\,dx\nonumber\\
&+\frac{C\delta_C^4}{(1+t)^{\frac{5}{4}}}+
C\int\norm{\wtilde{G}_t}_{\nu,M_\#}^2+\norm{\wtilde{G}_x}_{\nu,M_\#}^2\,dx.
\label{eq:K2-final}
\end{align}

Finally, the mixed transverse kinetic terms \(K_3\) and \(K_5\) are controlled the same way as \(K_4\), since
\[
\int_{\mathbb R^3}\xi_1\xi_j\bigl((\Pi_1^{S_1})^{-X_1}+(\Pi_1^{S_3})^{-X_3}\bigr)\,d\xi=0,
\qquad j=2,3.
\]
Therefore, \(K_3\) and \(K_5\) have the same bound as \(K_4\). Summing the above estimates yields \eqref{eq:K1245-est}.
\end{proof}

Finally, we estimate the quadratic microscopic interaction term.

\begin{lemma}\label{lem:K6}
There exists a positive constant \(C\) such that
\begin{equation}\label{eq:K6-est}
|K_6|
\le
C\delta_0\sum_{i=1,3}\mathcal G_i^S
+
\frac{C\delta_C}{1+t}\int_{\mathbb R}e^{-\frac{2c_0x^2}{1+t}}|(\phi,\zeta)|^2\,dx.
\end{equation}
\end{lemma}

\begin{proof}
Using the exponential decay of \((\Pi_1^{S_i})^{-X_i}\) (See Lemma \ref{lem:bs1}) and Young's inequality, we obtain
\begin{align*}
|K_6|
&\le
C
\sum_{i\in\{1,3\}}
\iint
a\frac{\zeta^2}{\theta\bar\theta}
\left|
\xi_1\left(\frac{|\xi|^2}{2}-(u_1^{S_i})^{-X_i}\xi_1\right)
\partial_x(\Pi_1^{S_i})^{-X_i}
\right|\,d\xi\,dx\\
&\le C\sum_{i\in\{1,3\}} \int \abs{\zeta}^2 \left\{\abs{\int \xi_1\frac{\abs{\xi}^2}{2}\partial_x\shockw{\Pi_1}{i} d\xi}+\abs{\int \xi_1^2\partial_x\shockw{\Pi_1}{i} d\xi}\right\}dx\\
&\le
C
\sum_{i\in\{1,3\}}
\delta_i^2 \int_{\mathbb R}|\partial_x(v^{S_i})^{-X_i}|\,\zeta^2\,dx.
\end{align*}
The right-hand side is bounded by
\[
C\delta_0\sum_{i=1,3}\mathcal G_i^S
+
\frac{C\delta_C}{1+t}\int_{\mathbb R}e^{-\frac{2c_0x^2}{1+t}}|(\phi,\zeta)|^2\,dx + C\delta_0(\varepsilon+\delta_0)^2\bigl(\delta_1e^{-C\delta_1 t}+\delta_3e^{-C\delta_3 t}\bigr),
\]
which proves \eqref{eq:K6-est}.
\end{proof}

Collecting the above estimates, we obtain the following proposition.

\begin{proposition}\label{prop:micro-error}
There exists a positive constant \(C\) such that
\begin{align}
\begin{aligned}
\sum_{j=1}^6 |K_j| \le&\;
C\delta_0\sum_{i\in\{1,3\}}\delta_i|\dot X_i|^2+C\delta_0\sum_{i\in\{1,3\}}\mathcal G_i^S+\frac{1}{20}\mathcal D_{\mathrm{mac}}(U)+C(\delta_0+\varepsilon)\mathcal D_{\mathrm{mic}}(t)\\
& +C\delta_0(\delta_0+\varepsilon)\bigl(\delta_1e^{-C\delta_1 t}+\delta_3e^{-C\delta_3 t}\bigr)+\frac{C\delta_C}{1+t}\int_{\mathbb R}e^{-\frac{2c_0x^2}{1+t}}|(\phi,\zeta)|^2\,dx\\
& +C\frac{\delta_C}{(1+t)^{\frac{5}{4}}} +
C\int\norm{\wtilde{G}_t}_{\nu,M_\#}^2+\norm{\wtilde{G}_x}_{\nu,M_\#}^2\,dx.
\end{aligned}
\label{eq:micro-error-final}
\end{align}
\end{proposition}

\begin{proof}
Lemma~\ref{lem:K1245} already gives the stated bound for \(K_1\), \(K_2\),\(K_3\), \(K_4\), and \(K_5\). 
Lemma~ \ref{lem:K6} provides the corresponding bounds for the remaining term \(K_6\). 
Adding these inequalities yields \eqref{eq:micro-error-final}.
\end{proof}

Combining Lemma~\ref{lem:actr1}, Proposition~\ref{prop:macro-error}, Lemma~\ref{lem:Y456}, and Proposition~\ref{prop:micro-error}, we arrive at the following preliminary zeroth-order inequality.

\begin{lemma}\label{lem:zero-prelim}
There exist positive constants \(\alpha_0\) and \(C\) such that
\begin{align}
\begin{aligned}
&\frac{d}{dt}\int_{\mathbb R}a\,\eta(U\mid\bar U)\,dx + \alpha_0\sum_{i=1,3}\mathcal{G}_i^S+\frac{1}{40}\mathcal D_{\mathrm{mac}}(U)+\sum_{i=1,3}\frac{\delta_i}{8\mathfrak m_i}|\dot X_i|^2\\
&\,\le
\left(
\sum_{i=1,3}C\delta_i^2e^{-C\delta_i t}
+\frac{C}{\delta_*t^2}
\right)
\int_{\mathbb R}\eta(U\mid\bar U)\,dx +
C\delta_0\bigl(\delta_1e^{-C\delta_1 t}+\delta_3e^{-C\delta_3 t}\bigr)\\
&\quad +C\frac{\delta_C}{(1+t)^{\frac{5}{4}}}
+C\frac{\delta_C}{1+t}\int_{\mathbb R}e^{-\frac{2c_0|x|^2}{1+t}}|(\phi,\zeta)|^2\,dx+C(\delta_0+\varepsilon)\mathcal D_{\mathrm{mic}}(t)\\
&\quad+
C\int\norm{\wtilde{G}_t}_{\nu,M_\#}^2+\norm{\wtilde{G}_x}_{\nu,M_\#}^2\,dx.
\end{aligned}
\label{eq:zero-prelim}
\end{align}
\end{lemma}

\begin{proof}
See Lemma~\ref{lem:actr1}, Lemma~\ref{lem:Y456}, Proposition~\ref{prop:macro-error} and Proposition~\ref{prop:micro-error}.
Combining \eqref{eq:macro1}, \eqref{eq:macro-error-final}, and \eqref{eq:micro-error-final} and choosing \(\delta_0+\varepsilon\) sufficiently small, we absorb the lower-order shock terms into the principal coercive term \(\sum_i\mathcal G_i^S\), and the dissipation contributions into \(\mathcal D_{\mathrm{mac}}(U)\). This proves \eqref{eq:zero-prelim}.
\end{proof}

\subsection{Short-time estimates}

\begin{lemma}[Short-time estimates]\label{lem:short-time-estimate}
For $t\le 1$, there exist constants $C>0$ and $\wtilde{c}_0>0$ such that
\begin{align}
\begin{aligned}\label{eq:final-short-estimate}
    &\left.\int \eta(U\mid\bar{U}) dx\right|_{t=1} +\wtilde{c}_0\int_0^1 \left[
\sum_{i=1,3}\mathcal G_i^S(s)
+
\mathcal D_{\mathrm{mac}}(U)(s)
+
\sum_{i=1,3}\delta_i|\dot X_i(s)|^2\right] ds \\
&\le \left.\int \eta(U\mid\bar{U}) dx\right|_{t=0} + C(\delta_0+\varepsilon) \int_0^1 \int \norm{\wtilde{G}_{\textup{rem}}}_{\nu,M_\#} \,dx\,ds \\
&\qquad +C\int_0^1
\int \norm{\wtilde{G}_t}_{\nu,M_\#}^2 + \norm{\wtilde{G}_x}_{\nu,M_\#}^2\,dx\,ds + C\delta_0.
\end{aligned}
\end{align}
\end{lemma}

\begin{proof}

Recall that \eqref{eq:energy2},

\begin{align}\label{eq:energy2-1}
\frac{d}{dt}\int a\eta(U\mid\bar{U}) dx \le \sum_{i\in\{1,3\}}\dot{X}_i\mathcal Y_i(U)+\sum_{l=1}^6 \mathcal B_l + \mathcal S_1+\mathcal S_2+ \sum_{l=1}^6 K_l -\mathfrak G(U)-\mathcal D_{\mathrm{mac}}(U).
\end{align}

For any $\lambda>0$, rearrange \eqref{eq:energy2-1},

\begin{align}\label{eq:energy2-re}
&\frac{d}{dt}\int a\eta(U\mid\bar{U}) dx + \sum_{i\in\{1,3\}}\frac{\delta_i}{4\mathfrak m_i}\abs{\dot{X}_i}^2 +\frac{1}{2}\mathcal D_{\mathrm{mac}}(U) +\lambda \sum_{i\in\{1,3\}}\mathcal{G}_i^S \nonumber\\
&\le \sum_{i\in\{1,3\}} \frac{C}{\delta_i}\sum_{j=4}^6 |\mathcal Y_{ij}|^2 + \sum_{l=1}^6 \mathcal B_l + \mathcal S_1 + \mathcal S_2 + \sum_{l=1}^6 K_l -\mathfrak G(U) - \frac{1}{2}\mathcal D_{\mathrm{mac}}(U) +\lambda \sum_{i\in\{1,3\}}\mathcal{G}_i^S.
\end{align}

For $t\le 1$, we obtain the following rough estimate using \eqref{eq:macro-error-plus-shift}, \eqref{eq:micro-error-final} and the bootstrap assumptions. 

\begin{align}\label{eq:coarsebd}
    \begin{aligned}
        &\sum_{i=1,3}\myparas{\frac{C}{\delta_i}\sum_{j=4}^6\abs{\mathcal Y_{ij}}^2}+\sum_{l=1}^6 \abs{\mathcal B_l} + \abs{\mathcal S_1} + \abs{\mathcal S_2} + \sum_{l=1}^6 \abs{K_l}  \\
        &\qquad - \frac{1}{2}\mathcal D_{\mathrm{mac}}(U) - \mathfrak G(U) + \lambda \sum_{i\in\{1,3\}}\mathcal{G}_i^S\\
        &\qquad \le C_\lambda \delta_0 + C(\delta_0+\varepsilon)  \int \norm{\wtilde{G}_{\text{rem}}}_{\nu,M_\#}^2 \,dx + C\int\norm{\wtilde{G}_t}_{\nu,M_\#}^2+\norm{\wtilde{G}_x}_{\nu,M_\#}^2\,dx.
    \end{aligned}
\end{align}
Using \eqref{eq:coarsebd}, we have \eqref{eq:final-short-estimate}.

\end{proof}

\subsection{Estimate of the microscopic dissipation}
In this subsection, we estimate the microscopic energy associated with \(\widetilde G_{\text{rem}}\) and recover the microscopic dissipation needed to close the zeroth-order energy estimate. 

Recall that
\[
\widetilde G=\widetilde G_{C}+\widetilde G_{\text{rem}},
\qquad
\widetilde G_C
=
\frac{3}{2v\theta}L_M^{-1}P_1
\left[
\xi_1M\left(\xi_1u_{1x}^C+\frac{|\xi-u|^2}{2\theta}\theta_x^C\right)
\right].
\]
We work in this subsection with the microscopic energy
\begin{equation}\label{eq:Emic}
\mathcal E_{\mathrm{mic}}(t)
:=
\int_{\mathbb R}\norm{\wtilde G_{\text{rem}}}_{M_\#}^2\,dx.
\end{equation}

Subtracting the equation satisfied by \(\widetilde G_C\) from the microscopic perturbation equation, we obtain
\begin{align}
\left(\widetilde G_{\text{rem}}\right)_t-L_M\widetilde G_{\text{rem}}
=\;&
\sum_{i=1,3}\dot X_i \shock{G}{i}_x
+\frac{u_1}{v}\widetilde G_x
-\frac1vP_1(\xi_1\widetilde G_x)\nonumber\\
&+\sum_{i=1,3}\left(\frac{u_1}{v}-\frac{(u_1^{S_i})^{-X_i}}{(v^{S_i})^{-X_i}}\right)\partial_x\shockw{G}{i}\nonumber\\
&-\sum_{i=1,3}
\left[
\frac1vP_1\bigl(\xi_1\partial_x\shockw{G}{i}\bigr)
-
\frac1{(v^{S_i})^{-X_i}}P_1^{S_i}\bigl(\xi_1\partial_x\shockw{G}{i}\bigr)
\right]
\nonumber\\
&+\mathcal N(\widetilde G,\widetilde G) +
2\sum_{i=1,3}\mathcal N\bigl((G^{S_i})^{-X_i},\widetilde G\bigr)
+
2\mathcal N\bigl((G^{S_1})^{-X_1},(G^{S_3})^{-X_3}\bigr)\nonumber\\
&+\sum_{i=1,3}\bigl(L_M-L_i^{S}\bigr)(G^{S_i})^{-X_i}
\nonumber\\
&-
\widetilde G_{Ct}
-\frac{3}{2v\theta}P_1
\left[
\xi_1M\left(\xi\cdot\psi_x+\frac{|\xi-u|^2}{2\theta}\zeta_x\right)
\right]
+\mathcal R_{\mathrm{prof}},
\label{eq:G1-detailed}
\end{align}
where
\begin{align}
\mathcal R_{\mathrm{prof}}
:=
\sum_{i=1,3}
\left[
\frac1{(v^{S_i})^{-X_i}}P_1^{S_i}\bigl(\xi_1\partial_x\bigl((M^{S_i})^{-X_i}\bigr)\bigr)
-
\frac{3}{2v\theta}P_1
\left(
\xi_1M\left(\partial_x\bigl((u_1^{S_i})^{-X_i}\bigr)+\frac{|\xi-u|^2}{2\theta}\partial_x\bigl((\theta^{S_i})^{-X_i}\bigr)\right)
\right)
\right].
\label{eq:Rprof}
\end{align}

We now estimate the right-hand side term by term.

\begin{proposition}\label{prop:mic-diss}
There exists a positive constant \(C\) such that
\begin{align}
\begin{aligned}
& \frac{d}{dt}\mathcal E_{\mathrm{mic}}(t)
+
\int_{\mathbb R} \norm{\wtilde{G}_{\textup{rem}}}_{\nu,M_\#}\,dx \\
&\qquad \le\;
C\delta_0\sum_{i=1,3}\mathcal G_i^S + C\delta_0(\delta_0+\varepsilon)\sum_{i=1,3}\delta_i|\dot X_i|^2 + C\|(\psi_x,\zeta_x)\|_{L^2_x}^2 \\
&\quad \qquad + C\delta_0\|(\phi_t,\psi_t,\zeta_t)\|_{L^2_x}^2
+ C\int_{\mathbb R} \norm{\wtilde{G}_x}_{\nu,M_\#} \,dx \\
&\quad \qquad + C\frac{\delta_C}{(1+t)^{\frac{5}{4}}} + C(\delta_0+\varepsilon)\delta_1^2e^{-C\delta_1 t}
+ C(\delta_0+\varepsilon)\delta_3^2e^{-C\delta_3 t}.
\end{aligned}
\label{eq:mic-diss-final}
\end{align}
\end{proposition}

\begin{proof}
Multiply \eqref{eq:G1-detailed} by \(\widetilde G_{\text{rem}}/M_\#\) and integrate over \(\mathbb R_x\times\mathbb R^3_\xi\). Since \(\widetilde G_{\text{rem}}\in\mathfrak Z_M^\perp\), the coercivity estimate for the linearized collision operator yields
\[
-\iint \frac{\widetilde G_{\text{rem}}L_M\widetilde G_{\text{rem}}}{M_\#}\,d\xi\,dx
\ge
\lambda_{\mathrm{mic}}
\int \norm{\widetilde{G}_{\text{rem}}}_{\nu,M_\#}^2\,dx
\]
for some \(\lambda_{\mathrm{mic}}>0\).

We estimate the source terms one by one.

For the shift terms, Young's inequality gives
\begin{align}
\iint \dot X_i (G^{S_i}_x)^{-X_i}\frac{\widetilde G_{\text{rem}}}{M_\#}\,d\xi\,dx
\le\;&
\frac{\lambda_{\mathrm{mic}}}{32}
\int \norm{\widetilde{G}_{\text{rem}}}_{\nu,M_\#}^2\,dx
+
C\delta_i^5|\dot X_i|^2 .
\label{eq:mic-shift-est}
\end{align}
Similarly,
\begin{align}
\iint
& \left(\frac{u_1}{v}-\frac{(u_1^{S_i})^{-X_i}}{(v^{S_i})^{-X_i}}\right)
\bigl((G^{S_i})^{-X_i}\bigr)_x\frac{\widetilde G_{\text{rem}}}{M_\#}\,d\xi\,dx \nonumber\\
& \, \le\;\frac{\lambda_{\mathrm{mic}}}{32}\int \norm{\widetilde{G}_{\text{rem}}}_{\nu,M_\#}^2\,dx
+C\delta_i^2\int_{\mathbb R}|\partial_x\bigl((v^{S_i})^{-X_i}\bigr)|^2|(u_1-(u_1^{S_i})^{-X_i},\,v-(v^{S_i})^{-X_i})|^2\,dx .
\label{eq:mic-mismatch-est}
\end{align}
Using the decomposition
\[
U-(U^{S_1})^{-X_1}
=
(U-\bar U)+(\bar U-(U^{S_1})^{-X_1}),
\]
the first part is controlled by \(\mathcal G_1^S\) of \eqref{eq:skGt}, while the second part is estimated by Lemma~\ref{lem:wave-interaction-L2}. Thus the last integral is bounded by
\[
C\delta_0\mathcal G_i^S
+
C\delta_0^3\bigl(\delta_1e^{-C\delta_1 t}+\delta_3e^{-C\delta_3 t}\bigr)
+
C\delta_i^3e^{-C\delta_i t}\int_{\mathbb R}\eta(U\mid \bar U)\,dx .
\]

Next, note that
\[
P_1(\xi_1h)=\xi_1h-\sum_{j=0}^4\langle \xi_1h,\chi_j\rangle\chi_j.
\]
Hence
\begin{align}
-\iint \frac1vP_1(\xi_1\widetilde G_x)\frac{\widetilde G_{\text{rem}}}{M_\#}\,d\xi\,dx
\le\;&
C\int \norm{\widetilde{G}_x}_{\nu,M_\#}^2\,dx
+
\frac{\lambda_{\mathrm{mic}}}{32}
\int \norm{\widetilde{G}_{\text{rem}}}_{\nu,M_\#}^2\,dx .
\label{eq:mic-P1Gx-est}
\end{align}
Also,
\begin{align}
& -\iint
\frac{3}{2v\theta}
P_1\left[
\xi_1M\left(\xi\cdot\psi_x+\frac{|\xi-u|^2}{2\theta}\zeta_x\right)
\right]
\frac{\widetilde G_{\text{rem}}}{M_\#}\,d\xi\,dx \nonumber \\
& \qquad \le\;
C\|(\psi_x,\zeta_x)\|_{L^2_x}^2
+
\frac{\lambda_{\mathrm{mic}}}{64}
\int \norm{\widetilde G_{\text{rem}}}_{\nu,M_\#}^2 \,dx .
\label{eq:mic-source-est}
\end{align}

We next treat the profile-difference terms. A representative estimate is
\begin{align}
&\iint
\left[
-\frac{3\bigl((u_1^{S_1})^{-X_1}\bigr)_x}{2v\theta}P_1(\xi_1^2M)
+
\frac{3\bigl((u_1^{S_1})^{-X_1}\bigr)_x}{2(v^{S_1})^{-X_1}(\theta^{S_1})^{-X_1}}
P_1^{S_1}\xi_1^2(M^{S_1})^{-X_1}
\right]
\frac{\widetilde G_{\text{rem}}}{M_\#}\,d\xi\,dx
\nonumber\\
&\qquad\le
C\delta_1^2\mathcal G_1^S
+
C\delta_0^3\bigl(\delta_1e^{-C\delta_1 t}+\delta_3e^{-C\delta_3 t}\bigr)
+
\frac{\lambda_{\mathrm{mic}}}{64}
\int \norm{\widetilde G_{\text{rem}}}_{\nu,M_\#}^2\,dx .
\label{eq:mic-profile-est}
\end{align}
The same estimate holds for the \(3\)-shock contribution. Likewise, the terms
\[
-\sum_{i=1,3}\left[
\frac1vP_1\bigl(\xi_1\bigl((G^{S_i})^{-X_i}\bigr)_x\bigr)
-
\frac1{(v^{S_i})^{-X_i}}P_1^{S_i}\bigl(\xi_1\bigl((G^{S_i})^{-X_i}
\bigr)_x\bigr)
\right]
\]
and
\[
\sum_{i=1,3}(L_M-L_1^S)(G^{S_i})^{-X_i}
\]
are profile-difference terms of the same type and are estimated in the same way.

For the nonlinear term, Lemmas~\ref{lem:collision-gain-loss}, \ref{lem:collision-coercivity} in Appendix give
\begin{align}
\iint \mathcal N(\widetilde G,\widetilde G)\frac{\widetilde G_{\text{rem}}}{M_\#}\,d\xi\,dx
\le\;&
C(\delta_0+\varepsilon)\int \norm{\widetilde{G}_{\text{rem}}}_{\nu,M_\#}^2\,dx
\nonumber\\
&+\frac{\lambda_{\mathrm{mic}}}{64}\int \norm{\widetilde{G}_{\text{rem}}}_{\nu,M_\#}^2\,dx +\frac{C\delta_C^4}{(1+t)^{\frac{5}{4}}}
\label{eq:mic-Q-est}
\end{align}
Here we used
\[
\norm{\widetilde{G}_C}_{\nu,M_\#}^2
\approx |(u_{1x}^C,\theta_x^C)|^2 .
\]
The mixed terms \(\mathcal N((G^{S_i})^{-X_i},\widetilde G)\) and \(\mathcal N((G^{S_1})^{-X_1},(G^{S_3})^{-X_3})\) are controlled in the same way and contribute only lower-order shock terms and exponentially decaying remainders.

Finally, differentiating \(\widetilde G_C\) in time yields
\[
\iint \widetilde G_{Ct}\frac{\widetilde G_{\text{rem}}}{M_\#}\,d\xi\,dx
\le
C\delta_0\|(\phi_t,\psi_t,\zeta_t)\|_{L^2_x}^2
+
\frac{\lambda_{\mathrm{mic}}}{64}
\int \norm{\widetilde{G}_{\text{rem}}}_{\nu,M_\#}^2\,dx .
\]

Collecting \eqref{eq:mic-shift-est}--\eqref{eq:mic-Q-est}, choosing \(\delta_0+\varepsilon\) sufficiently small, and absorbing the microscopic coercive terms into the left-hand side, we obtain \eqref{eq:mic-diss-final}.
\end{proof}

\subsection{Closure of the zeroth-order estimate}
We now combine the preliminary zeroth-order inequality obtained in Lemma~\ref{lem:zero-prelim} with the microscopic energy estimate from Proposition~\ref{prop:mic-diss}. This yields the zeroth-order bound in the form needed for the full a priori estimate.

We begin with the equivalence between the relative entropy and the square of the perturbation variables.

\begin{lemma}\label{lem:rel-entropy-equiv}
Under the bootstrap bound \(\mathcal{E}(T)^2\le \varepsilon^2\) in \eqref{eq:priass}, there exist positive constants \(c\) and \(C\) such that
\begin{equation}\label{eq:rel-entropy-equiv}
c\,|(\phi,\psi,\zeta)|^2
\le
\eta(U\mid\bar U)
\le
C\,|(\phi,\psi,\zeta)|^2.
\end{equation}
Consequently,
\begin{equation}\label{eq:weighted-rel-entropy-equiv}
c\int_{\mathbb R}a\,|(\phi,\psi,\zeta)|^2\,dx
\le
\int_{\mathbb R}a\,\eta(U\mid\bar U)\,dx
\le
C\int_{\mathbb R}a\,|(\phi,\psi,\zeta)|^2\,dx.
\end{equation}
\end{lemma}

\begin{proof}
Since \(v,\theta,\bar v,\bar\theta\) remain uniformly away from zero and infinity under \eqref{eq:priass}, the standard convexity property of
\[
\Phi(z)=z-1-\ln z
\]
implies
\[
\Phi\!\left(\frac{v}{\bar v}\right)\sim \left|\frac{\phi}{\bar v}\right|^2,
\qquad
\Phi\!\left(\frac{\theta}{\bar\theta}\right)\sim \left|\frac{\zeta}{\bar\theta}\right|^2.
\]
Together with the definition of \(\eta(U\mid\bar U)\), this yields \eqref{eq:rel-entropy-equiv}. Since the weight \(a\) is uniformly positive and bounded for sufficiently small \(\delta_0\), \eqref{eq:weighted-rel-entropy-equiv} follows.
\end{proof}

\begin{lemma}\label{lem:time-away-lem}
For $t\in[1,T]$, there exist constants $c_0>0$ and $C>0$ such that

\begin{align}
\eta(U\mid\bar{U})(t)&+c_0\int_1^t\left[\sum_{i=1,3}\mathcal G_i^S(s)+\mathcal D_{\mathrm{mac}}(U)(s)+\sum_{i=1,3}\delta_i|\dot X_i(s)|^2\right]ds \nonumber\\
&\le C\eta(U\mid\bar{U})(1)+C\int_1^t \frac{\delta_C}{1+s}\int_{\mathbb R}e^{-\frac{2c_0|x|^2}{1+s}}|(\phi,\zeta)|^2\,dx\,ds\label{eq:zero-order-inter}\\
&\quad+C\delta_0\int_1^t \|(\phi_t,\psi_t,\zeta_t)\|_{L^2_x}^2\,ds +C(\delta_0+\varepsilon)\int_1^t
\mathcal{D}_{\textup{mic}}\,ds \nonumber\\
&\quad+ C\int_1^t
\int_{\mathbb R}
\norm{\widetilde{G}_t}_{\nu,M_\#}^2+\norm{\widetilde{G}_x}_{\nu,M_\#}^2\,dx\,ds + C\delta_0^{1/2}\nonumber
\end{align}
\end{lemma}

\begin{proof}
The right-hand side of \eqref{eq:zero-prelim} is integrable away from time 0, and hence {contributes} only \(O(\delta_0^{1/2})\) after time integration over $[1,t]$.
\end{proof}

Next, define the total zeroth-order energy
\begin{equation}\label{eq:E-zero}
\mathcal{E}_0(t)
:=
\int_{\mathbb R}a(t,x)\eta(U\mid\bar U)(t,x)\,dx
+
\kappa_1
\mathcal E_{\mathrm{mic}}(t),
\end{equation}
where \(\kappa_1>0\) is a sufficiently small constant to be chosen later.

By Lemma~\ref{lem:rel-entropy-equiv}, the positivity and boundedness of \(a\), and the definition of \(\mathcal{E}_0(t)\), there exist positive constants \(c\) and \(C\) such that
\begin{align}\label{eq:E0-equiv}
&c\left(
\|(\phi,\psi,\zeta)(t)\|_{L^2_x}^2
+
\mathcal E_{\mathrm{mic}}(t)
\right)
\le
\mathcal{E}_0(t)\nonumber\\
&\qquad \qquad \le
C\left(
\|(\phi,\psi,\zeta)(t)\|_{L^2_x}^2
+
\mathcal E_{\mathrm{mic}}(t)
\right).
\end{align}

We now add Lemma~\ref{lem:zero-prelim} and \(\kappa_0\) times Proposition~\ref{prop:mic-diss}. Since \(a(t,x)\sim 1\) for sufficiently small \(\delta_0\), the weighted microscopic dissipation \(\mathcal D_{\mathrm{mic}}(t)\) is equivalent to the unweighted quantity
\[
\int_{\mathbb R}
\norm{\widetilde{G}_{\text{rem}}}_{\nu,M_\#}^2\,dx.
\]
Hence, choosing \(\kappa_1>0\) sufficiently small and then taking \(\delta_0+\varepsilon\) sufficiently small depending on \(\kappa_1\), we obtain the following proposition.

\begin{proposition}\label{prop:zero-order}
Under the assumptions of Proposition~\ref{prop:priest}, there exists a positive constant \(C\) such that for all \(t\in[0,T]\),
\begin{align}
\begin{aligned}
&\|(\phi,\psi,\zeta)(t)\|_{L^2_x}^2
+
\mathcal E_{\mathrm{mic}}(t)
\\
&\quad
+\int_0^t
\left[
\sum_{i=1,3}\mathcal G_i^S(s)
+
\mathcal D_{\mathrm{mac}}(U)(s)
+
\sum_{i=1,3}\delta_i|\dot X_i(s)|^2
+
\int_{\mathbb R} \norm{\widetilde{G}_{\textup{rem}}}_{\nu,M_\#}^2 \,dx
\right]ds
\\
&\le
C\left(
\|(\phi,\psi,\zeta)(0)\|_{L^2_x}^2
+
\mathcal E_{\mathrm{mic}}(0)
+\delta_0^{1/2}
\right)
+
C\delta_C\int_0^t \frac{1}{1+s}\int_{\mathbb R}e^{-\frac{2c_0|x|^2}{1+s}}|(\phi,\zeta)|^2\,dx\,ds
\\
&\quad
+
C\int_0^t
\int_{\mathbb R}
\norm{\widetilde{G}_t}_{\nu,M_\#}^2+\norm{\widetilde{G}_x}_{\nu,M_\#}^2\,dx\,ds
+
C\delta_0\int_0^t \|(\phi_t,\psi_t,\zeta_t)\|_{L^2_x}^2\,ds .
\end{aligned}
\label{eq:zero-order-prop}
\end{align}
\end{proposition}

\begin{proof}
See Lemma~\ref{lem:short-time-estimate}, Proposition~\ref{prop:mic-diss}, and Lemma~\ref{lem:time-away-lem}. Integral \eqref{eq:mic-diss-final} over $[0,T]$. Collecting all of these, we have the following estimate.

There exist constants $c_0>0$ and $C>0$ such that
\begin{align}
&\mathcal{E}_0(t)+
c_0\int_0^t
\left[
\sum_{i=1,3}\mathcal G_i^S(s)
+
\mathcal D_{\mathrm{mac}}(U)(s)
+
\sum_{i=1,3}\delta_i|\dot X_i(s)|^2\right] ds\nonumber\\
&\qquad +c_0 \int_0^t \int_{\mathbb R}\norm{\widetilde{G}_{\textup{rem}}}_{\nu,M_\#}^2\,dx ds\nonumber\\
&\le C\mathcal{E}_0(0)+
C\int_0^t \frac{\delta_C}{1+s}\int_{\mathbb R}e^{-\frac{2c_0|x|^2}{1+s}}|(\phi,\zeta)|^2\,dx\,ds+
C\delta_0\int_0^t \|(\phi_t,\psi_t,\zeta_t)\|_{L^2_x}^2\,ds\nonumber\\
&\qquad+C\int_0^t\int_{\mathbb R}\norm{\widetilde G_t}^2_{\nu,M_\#}+\norm{\widetilde G_x}^2_{\nu,M_\#}\,dx\,ds
+C\delta_0^{1/2}.
\label{eq:zero-order-final}
\end{align}

The estimate \eqref{eq:zero-order-prop} follows immediately from \eqref{eq:Emic}, \eqref{eq:E0-equiv} and \eqref{eq:zero-order-final}.
\end{proof}

\section{Higher order estimates}
\setcounter{equation}{0}

\subsection{High-order macroscopic estimates}

The high-order macroscopic estimates are derived by the standard differentiated energy method for the fluid part. Since the argument closely follows the established framework in \cite{wang2025}, \cite{kang2025time}, and \cite{huang2026timeasymptoticstabilitycompositewave}, we record only the main differentiated identities and the resulting estimates, emphasizing the terms that interact with the microscopic component and the dynamical shifts.

\subsubsection{Time-derivative estimates}

The proof of the energy estimate of \(\norm{(\phi_t,\psi_t,\zeta_t)}_{L^2_x}^2\) follows from a standard energy estimate. For the reader's {convenience}, we provide the proof in the Appendix.

\begin{lemma}[Time-derivative estimate]
\label{lem:macro-time-derivative}
    There exists $C>0$ such that 
\begin{align}
\begin{aligned}\label{eq:time-derivative-est}
&\norm{\myparas{\phi_t,\psi_t,\zeta_t}}_{L^2}^2 \\
&\leq C(\delta_0+\varepsilon)\sum_{i=1,3}\delta_i\abs{\dot{X}_i}^2 + C\delta_0(\mathcal{G}_1^S+\mathcal{G}_3^S) +  \frac{C\delta_C}{1+t} \int e^{-\frac{2c\abs{x}^2}{1+t}}\abs{\myparas{\phi,\psi,\zeta}}^2 dx \\
&\qquad + C\norm{\myparas{\phi_x,\psi_x,\zeta_x}}_{L^2}^2 + \frac{C\delta_C}{(1+t)^{\frac{5}{4}}} + C\delta_0(\delta_0+\varepsilon)(\delta_1e^{-C\delta_1t}+\delta_3e^{-C\delta_3t})\\
&\qquad + C \int \norm{\widetilde{G}_x}_{\nu,M_\#}^2 \,dx
\end{aligned}
\end{align}
\end{lemma}

\subsubsection{Coercive estimate for \(\phi_x\)}

The previous lemmas controls the full first-order spatial energy, but it does not provide a sufficiently direct coercive estimate for \(\phi_x\). To recover this missing control, we derive a compensating estimate by coupling the differentiated mass equation with the differentiated first momentum equation.

At this stage, it is important to avoid using the Chapman--Enskog expansion at the differentiated level, since such an expansion leads to a loss of regularity and obstructs the bootstrap closure. For this reason, all microscopic contributions will be estimated directly in terms of \(\widetilde G\), rather than through further Chapman--Enskog substitutions.

\begin{lemma}[Compensating estimate for \(\phi_x\)]
    There exists \(C>0\) such that
\begin{align}\label{eq:low-order-differential}
&\frac{d}{dt}\int \myparas{\frac{2}{3}\mu(\y{\theta})\phi_x^2-v\psi_1\phi_x} dx + \int \frac{2\theta}{3v}\phi_x^2 dx \nonumber\\
&\leq C(\delta_0+\varepsilon)\sum_{i=1,3} \mathcal{G}_i^S + C\delta_0(\delta_0+\varepsilon)(\delta_1e^{-C\delta_1t}+\delta_3e^{-C\delta_3t})+ C\frac{\delta_C}{(1+t)^{\frac{5}{4}}} \nonumber \\
&\quad+ C(\delta_0+\varepsilon) \sum_{i=1,3}\delta_i\abs{\dot{X}_i}^2 +  C\norm{(\psi_x,\zeta_x)}_{L^2}^2 + C \norm{\phi_t}_{L^2}^2 +C\varepsilon^2 \norm{\psi_{1xx}}_{L^2}^2 \nonumber \\
& \quad+ C \frac{\delta_C}{1+t} \int e^{-\frac{2c\abs{x}^2}{1+t}}\abs{\myparas{\phi,\psi,\zeta}}^2\, dx +C\int \norm{\widetilde G_{xx}}_{\nu,M_\#}^2 + \norm{\widetilde G_{xt}}_{\nu,M_\#}^2 \, dx \nonumber\\
&\quad+ C(\delta_0+\varepsilon)\int \norm{\widetilde G_{\textup{rem}}}_{\nu,M_\#}^2 + \norm{\widetilde G_t}_{\nu,M_\#}^2 + \norm{\widetilde G_{x}}_{\nu,M_\#}^2 \, dx.
\end{align}
\end{lemma}
 
\begin{proof}
    We have 
    \begin{align*}
        &\myparas{\frac{2}{3}\mu(\y{\theta})\phi_x^2-v\psi_1\phi_x}_t + v_t \psi_1 \phi_x + (v\psi_1\phi_t)_x-(v_x\psi_1+v\psi_{1x})\phi_t + \frac{2}{3}\frac{\theta}{v}\phi_x^2 \\
        &+ \sum_{i=1,3}\dot{X}_i\phi_x\myparam{-\frac{4}{3}\mu(\y{\theta})\partial_{xx}\shockw{v}{i}+v\partial_x\shockw{u_1}{i}}\\
        &=\frac{2}{3}v\phi_x\myparas{-\frac{\theta}{v^2}+\frac{\y{\theta}}{\y{v}^2}}\y{v}_x+v\phi_xQ_1+v\phi_x \int \xi_1^2  \widetilde{\Pi}_{1x} \,d\xi \\
        &+\frac{2}{3}\phi_x\zeta_x+\frac{2}{3}v\phi_x\y{\theta}_x\myparas{\frac{1}{v}-\frac{1}{\y{v}}}+\frac{2}{3}\mu'(\y{\theta})\y{\theta}_t\phi_x^2-\frac{4}{3}\mu(\y{\theta})v\phi_x\y{u}_{1xx}\myparas{\frac{1}{v}-\frac{1}{\y{v}}}\\
        &-\frac{4}{3}\phi_x\myparas{\mu(\theta)-\mu(\y{\theta})}_xu_{1x}-\frac{4}{3}\phi_x\mu'(\y{\theta})\y{\theta}_x\myparas{\psi_{1x}-\frac{\phi}{\y{v}}\y{u}_{1x}}-\frac{4}{3}\mu(\y{\theta})v\phi_x\y{u}_{1x}\myparas{\frac{\y{v}_x}{\y{v}^2}-\frac{v_x}{v^2}}\\
        &+\frac{4}{3}\phi_xv\myparas{\mu(\theta)-\mu(\y{\theta})}\myparas{\frac{u_{1x}}{v}}_x+\frac{4}{3}\frac{\mu(\y{\theta})}{v}\phi_x\psi_{1x}v_x.
    \end{align*}

 We single out the term involving \(\phi_t\), since it requires a direct use of the continuity equation. The remaining terms are estimated in the same way as the corresponding macroscopic error terms in the first-order spatial energy estimate.

\begin{align*}
-\int (v_x\psi_1+v\psi_{1x})\phi_t\,dx
\le\;&
\int \abs{\phi_x\psi_1\phi_t} + \abs{\bar v_x\psi_1\phi_t} + \abs{\phi\psi_{1x}\phi_t} + \abs{\bar v\psi_{1x}\phi_t}\,dx \\
\le\;&
\Bigl(\frac{1}{128}+C(\varepsilon+\delta_0)\Bigr)\int \abs{\phi_x}^2\,dx
+ C\|\phi_t\|_{L^2}^2
+ C\|\psi_{1x}\|_{L^2}^2.
\end{align*}
Here we use the continuity equation
\[
\phi_t=\psi_{1x}+\sum_{i=1,3}\dot X_i \partial_x(v^{S_i})^{-X_i},
\]
which yields
\[
\|\phi_t\|_{L^2}^2
\le
C\|\psi_{1x}\|_{L^2}^2
+
C\sum_{i=1,3}|\dot X_i|^2\|\partial_x(v^{S_i})^{-X_i}\|_{L^2}^2
\le
C\|\psi_{1x}\|_{L^2}^2
+
C\sum_{i=1,3}\delta_i^3|\dot X_i|^2.
\]
Therefore,
\begin{align*}
-\int (v_x\psi_1+v\psi_{1x})\phi_t\,dx
\le\;&
\Bigl(\frac{1}{128}+C(\varepsilon+\delta_0)\Bigr)\int \abs{\phi_x}^2\,dx
+ C\|\psi_{1x}\|_{L^2}^2
+ C\sum_{i=1,3}\delta_i^3|\dot X_i|^2 .
\end{align*}

The remaining terms on the right-hand side are lower-order macroscopic error terms or microscopic coupling terms. They are estimated exactly as in the first-order spatial energy estimate, using the previously established bounds for \(Q_1\), the profile derivatives of \(\bar U\), and the microscopic moments involving \(\widetilde{\Pi}_1\). Collecting all these estimates, we obtain the desired inequality.\end{proof}

\subsubsection{First-order spatial estimates}

We begin with the first-order spatial estimates for the macroscopic perturbation \((\phi,\psi,\zeta)\).
The first lemma provides the basic differentiated \(H^1_x\)-energy inequality for \((\phi_x,\psi_x,\zeta_x)\), together with the corresponding dissipation on \(\psi_{xx}\) and \(\zeta_{xx}\). 
The remaining terms are lower-order error terms generated by the composite profile, the dynamical shifts, and the coupling with the microscopic component. For the reader's {convenience}, we provide the proof in the Appendix.

\begin{lemma}[First-order spatial energy estimate] \label{lem:fosee}
    There exists \(C>0\) such that
    \begin{align}
        \begin{aligned}\label{eq:eoppz1}
            &\frac{d}{dt}\int \myparas{\frac{\y{p}\theta}{2v}\phi_x^2+\frac{3\y{p}v}{4}\psi_{1x}^2+\sum_{k=2}^3 \frac{\psi_{kx}^2}{2}+\frac{\zeta_x^2}{2}} dx \\
            &\quad + \int 2\mu(\y{\theta})\y{p}\psi_{1xx}^2 + \sum_{k=2}^3 \frac{\mu(\theta)}{v} \psi_{kxx}^2 + \frac{\alpha_{\rm{th}}(\y{\theta})}{v}\zeta_{xx}^2 dx \\
            &\qquad \le C(\delta_0+\varepsilon) \sum_{i=1,3} \delta_i\abs{\dot{X}_i}^2 + C \delta_0 (\mathcal{G}_1^S+\mathcal{G}_3^S) + C\frac{\delta_C}{1+t}\int e^{\frac{-2c\abs{x}^2}{1+t}}\abs{(\phi,\zeta)}^2 dx +C\delta_0\norm{\phi_{xx}}_{L^2}^2\\
            &\, \qquad + C(\delta_0+\varepsilon)\sum_{\abs{\beta}=1}\norm{\partial^\beta(\phi,\psi,\zeta)}_{L^2}^2 +  C\frac{\delta_C}{(1+t)^{5/4}} +C(\varepsilon+\delta_0)\delta_0\myparas{\delta_3 e^{-C\delta_3t} + \delta_1 e^{-C\delta_1t}}\\
            &\, \qquad + C(\delta_0+\varepsilon)\int \norm{\widetilde G_{\text{rem}}}_{\nu,M_\#}^2 + \norm{\widetilde G_t}_{\nu, M_\#}^2 + \norm{\widetilde G_x}_{\nu, M_\#}^2 dx +C\int \norm{\widetilde G_{tx}}_{\nu,M_\#}^2 + \norm{\widetilde G_{xx}}_{\nu,M_\#}^2 \,dx.
        \end{aligned}
    \end{align}
\end{lemma}

\subsubsection{Second-order spatial estimates}

\begin{lemma}[Second-order spatial estimate for \(\phi\)]
    There exists \(C>0\) such that
\begin{align}
\begin{aligned}
&-\frac{d}{dt} \int \phi_{xx}\psi_{1x}\,dx + \int \frac{p}{2v} \phi_{xx}^2\,dx \\
&\leq C\delta_0 (\mathcal{G}^S_1+\mathcal{G}^S_3)+ C(\delta_0+\varepsilon)\norm{(\phi_x,\psi_x,\zeta_x)}_{L^2}^2 + C \norm{(\psi_{1xx},\zeta_{xx})}_{L^2}^2 \\
&\quad +  C\delta_C \frac{1}{1+t} \int e^{-\frac{2c_0\abs{x}^2}{1+t}}\abs{\myparas{\phi,\zeta}}^2 dx +C(\delta_0+\varepsilon) \sum_{i=1,3}\delta_i\abs{\dot{X}_i}^2\\
&\quad + C\delta_0(\delta_0+\varepsilon)(\delta_1e^{-C\delta_1t}+\delta_3e^{-C\delta_3t})+ C\frac{\delta_C}{(1+t)^{\frac{5}{4}}} +C \int \norm{\widetilde G_{xx}}_{\nu,M_\#}^2\, dx. 
\end{aligned}
\label{eq:phi-xx-est}
\end{align}
\end{lemma}

\begin{proof}
Differentiating the first momentum equation in the perturbed system with respect to \(x\), we obtain
\begin{align*}
\begin{aligned}
&\psi_{1tx}+\myparas{p-\y{p}}_{xx}+\myparas{\bar{p}-\shockp{p}{1}-p^C-\shockp{p}{3}}_{xx}-\sum_{i=1,3}\dot{X}_i \partial_{xx}\shockw{u_1}{i} \\
&= -Q_{1x}^C-\frac{4}{3}\myparas{\frac{\mu\myparas{\theta^C}u_{1x}^C}{v^C}}_{xx}-\int \xi_1^2 \wtilde{G}_{xx} \,d\xi .
\end{aligned}
\end{align*}
Multiplying this equation by \(-\phi_{xx}\) and integrating over \(x\in\mathbb R\), we start from
\[
-\int \phi_{xx}\psi_{1tx}\,dx.
\]
By integration by parts in time,
\begin{align*}
-\int \phi_{xx}\psi_{1tx}\,dx
&=
-\frac{d}{dt}\int \phi_{xx}\psi_{1x}\,dx
+\int \phi_{xxt}\psi_{1x}\,dx \\
&=
-\frac{d}{dt}\int \phi_{xx}\psi_{1x}\,dx
-\int \phi_{xt}\psi_{1xx}\,dx .
\end{align*}
Hence,
\begin{align*}
-\int \phi_{xx}\psi_{1tx}\,dx
\le
-\frac{d}{dt}\int \phi_{xx}\psi_{1x}\,dx
+\frac{1}{128}\|\phi_{xt}\|_{L^2}^2
+
C\|\psi_{1xx}\|_{L^2}^2.
\end{align*}
Now differentiate the continuity equation
\[
\phi_t-\psi_{1x}-\sum_{i=1,3}\dot X_i \partial_x\shockw{v}{i}=0
\]
with respect to \(x\). Then
\[
\phi_{xt}
=
\psi_{1xx}
+
\sum_{i=1,3}\dot X_i \partial_{xx}\shockw{v}{i},
\]
so that
\begin{align*}
\|\phi_{xt}\|_{L^2}^2
&\le
C\|\psi_{1xx}\|_{L^2}^2
+
C\sum_{i=1,3} |\dot X_i|^2 \|\partial_{xx}\shockw{v}{i}\|_{L^2}^2 \\
&\le
C\|\psi_{1xx}\|_{L^2}^2
+
C\sum_{i=1,3}\delta_i^5 |\dot X_i|^2 \\
&\le
C\|\psi_{1xx}\|_{L^2}^2
+
C(\delta_0+\varepsilon)\sum_{i=1,3}\delta_i |\dot X_i|^2 .
\end{align*}
Therefore,
\begin{align}
\label{eq:phixt-control}
-\int \phi_{xx}\psi_{1tx}\,dx
\le
-\frac{d}{dt}\int \phi_{xx}\psi_{1x}\,dx
+
C\|\psi_{1xx}\|_{L^2}^2
+
C(\delta_0+\varepsilon)\sum_{i=1,3}\delta_i |\dot X_i|^2 .
\end{align}

Next, note that
\begin{align*}
\begin{aligned}
\myparas{p-\y{p}}_{xx}
=
-\frac{p}{v}\phi_{xx}
+\frac{2}{3v}\zeta_{xx}
-\frac{1}{v}\myparas{p-\y{p}}\y{v}_{xx}
-\frac{\phi}{v}\y{p}_{xx}
-\frac{2v_x}{v}\myparas{p-\y{p}}_x
-\frac{2\y{p}_x}{v}\phi_x .
\end{aligned}
\end{align*}
The leading term gives the desired coercivity:
\[
-\int \phi_{xx}\myparas{-\frac{p}{v}\phi_{xx}}\,dx
=
\int \frac{p}{v}\phi_{xx}^2\,dx .
\]
The remaining terms are lower-order error terms. For example,
\begin{align*}
-\int \frac{2}{3v}\phi_{xx}\zeta_{xx}\,dx
&\le
\frac{1}{128}\int \frac{p}{v}\phi_{xx}^2\,dx
+
C\|\zeta_{xx}\|_{L^2}^2, \\
\int \frac{1}{v}\myparas{p-\y{p}}\y{v}_{xx}\phi_{xx}\,dx
&\le
\frac{1}{128}\int \frac{p}{v}\phi_{xx}^2\,dx
+
C\int \abs{\y{v}_{xx}}^2\abs{(\phi,\zeta)}^2\,dx, \\
\int \frac{\phi\phi_{xx}}{v}\y{p}_{xx}\,dx
&\le
\frac{1}{128}\int \frac{p}{v}\phi_{xx}^2\,dx
+
C\int \abs{\y{p}_{xx}}^2\abs{\phi}^2\,dx, \\
\int \frac{2v_x\phi_{xx}}{v}\myparas{p-\y{p}}_x\,dx
&\le
\frac{1}{128}\int \frac{p}{v}\phi_{xx}^2\,dx
+
C(\delta_0+\varepsilon)^2\|(\phi_x,\zeta_x)\|_{L^2}^2 \\
&\qquad
+
C(\delta_0+\varepsilon)^2\int \abs{(\y{\theta}_x,\y{v}_x)}^2\abs{(\phi,\zeta)}^2\,dx, \\
\int \frac{2\y{p}_x}{v}\phi_{xx}\phi_x\,dx
&\le
\frac{1}{128}\int \frac{p}{v}\phi_{xx}^2\,dx
+
C(\delta_0+\varepsilon)^2\|\phi_x\|_{L^2}^2 .
\end{align*}

We next estimate the profile-error and microscopic terms:
\begin{align*}
-\int \phi_{xx} \myparas{\y{p}-\shockp{p}{1}-p^C-\shockp{p}{3}}_{xx}\,dx
&\le
\frac{1}{128}\int \frac{p}{v}\phi_{xx}^2\,dx
+
C \|Q_{1x}\|_{L^2}^2, \\
\sum_{i=1,3} \int \dot{X}_i \partial_{xx}\shockw{u_1}{i} \phi_{xx}\,dx
&\le
\frac{1}{128}\int \frac{p}{v}\phi_{xx}^2\,dx
+
C\sum_{i=1,3} \delta_i^3 |\dot{X}_i|^2, \\
\int \phi_{xx} \myparas{Q_{1x}^C + \frac{4}{3}\myparas{\frac{\mu\myparas{\theta^C}u^C_{1x}}{v^C}}_{xx}}\,dx
&\le
\frac{1}{128}\int \frac{p}{v}\phi_{xx}^2\,dx
+
C \|Q_{1x}^C\|_{L^2}^2, \\
\int \phi_{xx} \int \xi_1^2 \wtilde{G}_{xx}\,d\xi\,dx
&\le
\frac{1}{128}\int \frac{p}{v}\phi_{xx}^2\,dx
\\
&\qquad +
C \int \norm{\wtilde G_{xx}}_{\nu,M_\#}^2 dx .
\end{align*}

Combining these bounds with \eqref{eq:phixt-control}, we obtain
\begin{align*}
\begin{aligned}
& -\frac{d}{dt}\int \phi_{xx}\psi_{1x}\,dx +\int \frac{p}{2v}\phi_{xx}^2\,dx \\
&\quad \le
C \|(\psi_{1xx},\zeta_{xx})\|_{L^2}^2
+
C(\delta_0+\varepsilon)\sum_{i=1,3}\delta_i |\dot X_i|^2 +
C \int \abs{(\y{v}_{xx},\y{p}_{xx})}^2\abs{(\phi,\zeta)}^2\,dx
+
C \|Q_{1x}\|_{L^2}^2
\\
&\qquad 
+
C \|Q_{1x}^C\|_{L^2}^2 +
C(\delta_0+\varepsilon)^2 \|(\phi_x,\zeta_x)\|_{L^2}^2
+
C(\delta_0+\varepsilon)^2 \int \abs{(\y{\theta}_x,\y{v}_x)}^2\abs{(\phi,\zeta)}^2\,dx\\
&\qquad +
C \int \norm{\wtilde G_{xx}}_{\nu,M_\#}^2 dx .
\end{aligned}
\end{align*}

Finally, the terms involving \(Q_{1x}\), \(Q_{1x}^C\), \((\y{v}_{xx},\y{p}_{xx})\), and \((\y{\theta}_x,\y{v}_x)\) are estimated exactly as in the previous macroscopic energy bounds, using the profile decay estimates for the composite wave. In particular,
\[
\int \abs{(\y{v}_{xx},\y{p}_{xx})}^2\abs{(\phi,\zeta)}^2\,dx
+
\int \abs{(\y{\theta}_x,\y{v}_x)}^2\abs{(\phi,\zeta)}^2\,dx
\]
is bounded by
\[
C\delta_0 (\mathcal G_1^S+\mathcal G_3^S)
+
C\delta_C^2 \frac{1}{1+t}\int e^{-\frac{2c_0|x|^2}{1+t}}\abs{(\phi,\zeta)}^2\,dx
+
C\delta_0^3(\delta_1e^{-C\delta_1 t}+\delta_Ce^{-Ct}+\delta_3e^{-C\delta_3 t}),
\]
while
\[
\|Q_{1x}\|_{L^2}^2+\|Q_{1x}^C\|_{L^2}^2
\]
is bounded by
\[
C\delta_0^3(\delta_1e^{-C\delta_1 t}+\delta_3e^{-C\delta_3 t})+C\frac{\delta_C^4}{(1+t)^{5/4}}
+
C(\varepsilon+\delta_0)\delta_3^2 e^{-C\delta_3t}
+
C(\varepsilon+\delta_0)\delta_1^2 e^{-C\delta_1t}.
\]
This proves \eqref{eq:phi-xx-est}.
\end{proof}

\subsubsection{Mixed derivative estimates}

\begin{lemma}[Mixed derivative estimate]
    There exists \(C>0\) such that
\begin{align}
\begin{aligned}
    &\norm{\myparas{\phi_{xt},\psi_{xt},\zeta_{xt}}}_{L^2}^2 \\
&\leq C \norm{\myparas{\phi_{xx},\psi_{xx},\zeta_{xx}}}_{L^2}^2 + C(\delta_0+\varepsilon) \sum_{i=1,3}\delta_i\abs{\dot{X}_i}^2 + C \delta_0 (\mathcal{G}^{S}_1+\mathcal{G}^S_3)\\
&\quad + C\delta_C \frac{1}{1+t} \int e^{-\frac{2c_0\abs{x}^2}{1+t}}\abs{\myparas{\phi,\psi,\zeta}}^2 dx+ C(\delta_0+\varepsilon)\norm{\myparas{\phi_x,\psi_x,\zeta_x}}_{L^2}^2 \\
&\quad +C\int \norm{\widetilde G_{xx}}_{\nu,M_\#}^2 \,dx + C(\delta_0+\varepsilon)\int \norm{\widetilde G_{x}}_{\nu,M_\#}^2 \, dx\\
&\quad + C\delta_0(\delta_0+\varepsilon)(\delta_1e^{-C\delta_1t}+\delta_3e^{-C\delta_3t})+ C\frac{\delta_C}{(1+t)^{\frac{5}{4}}}.
\end{aligned}
\label{eq:mixed-derivative-est}
\end{align}
\end{lemma}

\begin{proof}
Differentiating the perturbed system with respect to \(x\), we obtain
\begin{align*}
\begin{aligned}
&\phi_{xt}-\psi_{1xx}-\sum_{i=1,3} \dot{X}_i \partial_{xx}\shockw{v}{i} = 0,\\
&\psi_{1tx} + \myparas{p-p^{S_1,-X_1}-p^C-p^{S_3,-X_3}}_{xx}-\sum_{i=1,3} \dot{X}_i \partial_{xx}\shockw{u_1}{i} \\
&\qquad \quad= -Q_{1x}^C-\frac{4}{3}\myparas{\frac{\mu\myparas{\theta^C}u_{1x}^C}{v^C}}_{xx}-\int \xi_1^2 \wtilde{G}_{xx} d\xi,\\
&\psi_{itx}=-\int \xi_1\xi_i \wtilde{G}_{xx} d\xi,\quad (i=2,3),\\
&\zeta_{tx} +\myparas{pu_{1x}-p^{S_1,-X_1}\partial_x\shockw{u_1}{1}-p^Cu_{1x}^C-p^{S_3,-X_3}\partial_x\shockw{u_1}{3}}_x-\sum_{i=1,3}\dot{X}_i\partial_{xx}\shockw{\theta}{i}\\
&\qquad = -\frac{1}{2}\int \xi_1 \abs{\xi}^2 \wtilde{G}_{xx} d\xi +u_{1x}\int \xi_1^2 \wtilde{G}_x d\xi +u_1\int \xi_1^2 \wtilde{G}_{xx} d\xi \\
&\qquad\quad +\sum_{i=1,3} \myparas{u_1-\shockw{u_1}{i}}_x\int \xi_1^2 \partial_x\shockw{G}{i} d\xi\\
&\qquad\quad +\sum_{i=1,3} \myparas{u_1-\shockw{u_1}{i}}\int \xi_1^2 \partial_{xx}\shockw{G}{i} d\xi\\
&\qquad\quad +\sum_{k=2}^3 u_{kx} \int \xi_1\xi_k \wtilde{G}_x d\xi+\sum_{k=2}^3 u_k \int \xi_1\xi_k \wtilde{G}_{xx} d\xi
\\
&\qquad\quad-\myparas{\frac{\alpha_{\rm{th}}\myparas{\theta^C}\theta_x^C}{v^C}}_{xx}-\myparam{\frac{4}{3}\mu\myparas{\theta^C}\frac{\myparas{u_{1x}^C}^2}{v^C}+Q_2^C}_x .
\end{aligned}
\end{align*}

We estimate \(\phi_{xt}\), \(\psi_{xt}\), and \(\zeta_{xt}\) separately.

\medskip
\noindent\textit{Step 1: estimate of \(\phi_{xt}\).}
Multiplying the first equation by \(\phi_{xt}\) and integrating over \(x\), we get
\begin{align*}
\|\phi_{xt}\|_{L^2}^2
&\le
\frac{1}{128}\|\phi_{xt}\|_{L^2}^2
+
C\|\psi_{1xx}\|_{L^2}^2
+
\frac{1}{128}\|\phi_{xt}\|_{L^2}^2
+
C\sum_{i=1,3}\delta_i^5|\dot X_i|^2.
\end{align*}
Hence,
\begin{align}
\label{eq:phixt-est}
\|\phi_{xt}\|_{L^2}^2
\le
C\|\psi_{1xx}\|_{L^2}^2
+
C(\delta_0+\varepsilon)\sum_{i=1,3}\delta_i|\dot X_i|^2.
\end{align}

\medskip
\noindent\textit{Step 2: estimate of \(\psi_{1tx}\).}
Multiplying the differentiated first momentum equation by \(\psi_{1tx}\) and integrating over \(x\), we obtain
\begin{align*}
\|\psi_{1tx}\|_{L^2}^2
\le\;&
\frac{1}{128}\|\psi_{1tx}\|_{L^2}^2
+
C\|(p-\bar p)_{xx}\|_{L^2}^2
+
C\|Q_{1x}\|_{L^2}^2
+
C\sum_{i=1,3}\delta_i^3|\dot X_i|^2 \\
&+
C\|Q_{1x}^C\|_{L^2}^2
+
C\int \norm{\wtilde G_{xx}}_{\nu,M_\#}^2\,dx .
\end{align*}
Since
\[
(p-\bar p)_{xx}
=
-\frac{p}{v}\phi_{xx}
+\frac{2}{3v}\zeta_{xx}
-\frac{1}{v}(p-\bar p)\bar v_{xx}
-\frac{\phi}{v}\bar p_{xx}
-\frac{2v_x}{v}(p-\bar p)_x
-\frac{2\bar p_x}{v}\phi_x,
\]
we have
\begin{align*}
\|(p-\bar p)_{xx}\|_{L^2}^2
\le\;&
C\|(\phi_{xx},\zeta_{xx})\|_{L^2}^2
+
C(\delta_0+\varepsilon)^2\|(\phi_x,\zeta_x)\|_{L^2}^2 \\
&+
C\int \abs{(\bar v_{xx},\bar p_{xx})}^2\abs{(\phi,\zeta)}^2\,dx
+
C(\delta_0+\varepsilon)^2\int \abs{(\bar v_x,\bar\theta_x)}^2\abs{(\phi,\zeta)}^2\,dx .
\end{align*}
Therefore,
\begin{align}
\label{eq:psi1tx-est}
\begin{aligned}
\|\psi_{1tx}\|_{L^2}^2
\le\;&
C\|(\phi_{xx},\zeta_{xx})\|_{L^2}^2
+
C(\delta_0+\varepsilon)\sum_{i=1,3}\delta_i|\dot X_i|^2
+
C(\delta_0+\varepsilon)^2\|(\phi_x,\zeta_x)\|_{L^2}^2 \\
&+
C\int \abs{(\bar v_{xx},\bar p_{xx})}^2\abs{(\phi,\zeta)}^2\,dx
+
C(\delta_0+\varepsilon)^2\int \abs{(\bar v_x,\bar\theta_x)}^2\abs{(\phi,\zeta)}^2\,dx \\
&+
C\|Q_{1x}\|_{L^2}^2
+
C\|Q_{1x}^C\|_{L^2}^2
+
C\int \norm{\wtilde G_{xx}}_{\nu,M_\#}^2\, dx .
\end{aligned}
\end{align}

\medskip
\noindent\textit{Step 3: estimate of \(\psi_{itx}\), \(i=2,3\).}
Multiplying
\[
\psi_{itx}=-\int \xi_1\xi_i \wtilde{G}_{xx}\,d\xi
\]
by \(\psi_{itx}\) and integrating in \(x\), we get
\begin{equation}
\label{eq:psiitx-est}
\|\psi_{itx}\|_{L^2}^2
\le
C\int \norm{\wtilde G_{xx}}_{\nu,M_\#}^2 \, dx,
\qquad i=2,3.
\end{equation}

\medskip
\noindent\textit{Step 4: estimate of \(\zeta_{tx}\).}
Multiplying the differentiated energy equation by \(\zeta_{tx}\) and integrating over \(x\), we obtain
\begin{align*}
& \|\zeta_{tx}\|_{L^2}^2 \\
& \, \le
C\Bigl\|\myparas{pu_{1x}-p^{S_1,-X_1}\partial_x\shockw{u_1}{1}-p^Cu_{1x}^C-p^{S_3,-X_3}\partial_x\shockw{u_1}{3}}_x\Bigr\|_{L^2}^2 \\
&\quad +
C\sum_{i=1,3}\delta_i^3|\dot X_i|^2
+
C\|Q_{2x}\|_{L^2}^2
+
C\|Q_{2x}^C\|_{L^2}^2
+
C\|\theta_{xxx}^C\|_{L^2}^2 \\
&\quad +
C\sum_{i=1,3}\int \abs{\bigl(u_1-\shockw{u_1}{i}\bigr)_x}^2 \abs{\partial_x\shockw{v}{i}}^2\,dx\\
&\quad +
C(\delta_0+\varepsilon)\sum_{i=1,3}\int \abs{u_1-\shockw{u_1}{i}}^2 \abs{\partial_x\shockw{v}{i}}^2\,dx \\
&\quad +
C\int \norm{\wtilde G_{xx}}_{\nu,M_\#}^2\, dx
+
C(\delta_0+\varepsilon)\int \norm{\wtilde G_x}_{\nu,M_\#}^2\, dx.
\end{align*}
Now
\begin{align*}
&\abs{\myparas{pu_{1x}-\bar p\,\bar u_{1x}}_x}\\
&\quad \le
C \abs{\psi_{1xx}}
+
C \abs{(p-\bar p)_x}\abs{(\bar u_{1x},\psi_{1x})}
+
C \abs{\bar p_x}\abs{\psi_{1x}}
+
C\abs{(\phi,\zeta)}\abs{\bar u_{1xx}},
\end{align*}
and therefore
\begin{align*}
&\Bigl\|\myparas{pu_{1x}-\bar p\,\bar u_{1x}}_x\Bigr\|_{L^2}^2\\
&\quad \le
C\|\psi_{1xx}\|_{L^2}^2
+
C(\delta_0+\varepsilon)^2\|(\zeta_x,\psi_x,\phi_x)\|_{L^2}^2 \\
&\qquad +
C(\delta_0+\varepsilon)^2\int \abs{(\bar\theta_x,\bar v_x)}^2\abs{(\phi,\zeta)}^2\,dx +
C\int \abs{\bar u_{1xx}}^2\abs{(\phi,\zeta)}^2\,dx .
\end{align*}
Hence
\begin{align}
\label{eq:zetatx-est}
\begin{aligned}
& \|\zeta_{tx}\|_{L^2}^2 \\
& \quad \le
C\|\psi_{1xx}\|_{L^2}^2
+
C(\delta_0+\varepsilon)^2\|(\zeta_x,\psi_x,\phi_x)\|_{L^2}^2
+
C(\delta_0+\varepsilon)\sum_{i=1,3}\delta_i|\dot X_i|^2 \\
&\qquad +
C\int \abs{(\bar\theta_x,\bar v_x)}^2\abs{(\phi,\zeta)}^2\,dx
+
C\int \abs{\bar u_{1xx}}^2\abs{(\phi,\zeta)}^2\,dx \\
&\qquad +
C\|Q_{2x}\|_{L^2}^2
+
C\|Q_{2x}^C\|_{L^2}^2
+
C\|\theta_{xxx}^C\|_{L^2}^2 \\
&\qquad +
C\sum_{i=1,3}\int \abs{\partial_x\bigl(u_1-\shockw{u_1}{i}\bigr)}^2 \abs{\partial_x\shockw{v}{i}}^2\,dx
\\
&\qquad +
C(\delta_0+\varepsilon)\sum_{i=1,3}\int \abs{u_1-\shockw{u_1}{i}}^2 \abs{\partial_x\shockw{v}{i}}^2\,dx \\
&\qquad +
C\int \norm{\wtilde G_{xx}}_{\nu,M_\#}^2 \, dx
+
C(\delta_0+\varepsilon)\int \norm{\wtilde G_x}_{\nu,M_\#}^2 \, dx.
\end{aligned}
\end{align}

The profile terms in \eqref{eq:psi1tx-est} and \eqref{eq:zetatx-est} are estimated exactly as in the previous macroscopic estimates. In particular,
\[
\int \abs{(\bar v_{xx},\bar p_{xx},\bar u_{1xx})}^2\abs{(\phi,\zeta)}^2\,dx
+
\int \abs{(\bar\theta_x,\bar v_x)}^2\abs{(\phi,\zeta)}^2\,dx
\]
is bounded by
\[
C\delta_0 (\mathcal G_1^S+\mathcal G_3^S)
+
C\delta_C^2 \frac{1}{1+t}\int e^{-\frac{2c_0|x|^2}{1+t}}\abs{(\phi,\psi,\zeta)}^2\,dx
+
C\delta_0^3(\delta_1e^{-C\delta_1t}+\delta_3e^{-C\delta_3t}),
\]
while
\[
\|Q_{1x}\|_{L^2}^2+\|Q_{1x}^C\|_{L^2}^2+\|Q_{2x}\|_{L^2}^2+\|Q_{2x}^C\|_{L^2}^2+\|\theta_{xxx}^C\|_{L^2}^2
\]
is bounded by
\[
C\delta_0^3(\delta_1e^{-C\delta_1t}+\delta_3e^{-C\delta_3t})
+
C\frac{\delta_C^4}{(1+t)^{\frac{5}{4}}}
+
C(\varepsilon+\delta_0)\delta_3^2 e^{-C\delta_3t}
+
C(\varepsilon+\delta_0)\delta_1^2 e^{-C\delta_1t}.
\]
Finally, the shock-interaction terms are controlled by the already established wave-interaction bounds, giving
\begin{align*}
   & \sum_{i=1,3}\int \abs{\bigl(u_1-\shockw{u_1}{i}\bigr)_x}^2 \abs{\partial_x\shockw{v}{i}}^2\,dx
+
(\delta_0+\varepsilon)\int \abs{u_1-\shockw{u_1}{i}}^2 \abs{\partial_x\shockw{v}{i}}^2\,dx\\
&
\le
C\delta_0(\mathcal G_1^S+\mathcal G_3^S)
+
C\delta_0^3(\delta_1e^{-C\delta_1t}+\delta_3e^{-C\delta_3t}).
\end{align*}

Collecting \eqref{eq:phixt-est}, \eqref{eq:psi1tx-est}, \eqref{eq:psiitx-est}, and \eqref{eq:zetatx-est}, we obtain \eqref{eq:mixed-derivative-est}.
\end{proof}

\subsection{High-order microscopic estimates}

We now turn to the high-order estimates for the microscopic component. In contrast to the macroscopic part, the main difficulty here is that spatial and temporal derivatives of the microscopic equation generate nontrivial commutator terms involving the linearized collision operator, the composite background profile, and the shock modulation. To close the high-order energy estimates, it is therefore necessary to isolate the basic microscopic identities first and then estimate the differentiated quantities \(\widetilde G_x\), \(\widetilde G_t\), and the second-order derivatives in a systematic way.

\subsubsection{Key microscopic identity}
 In this subsection, we collect the auxiliary microscopic estimates that will be used repeatedly in the proofs of the high-order bounds for \(\widetilde G_x\), \(\widetilde G_t\), and the second-order microscopic derivatives. The main purpose of these lemmas is to control the interaction between the differentiated microscopic terms and the composite background profile, especially near the shock layers where the microscopic part of the Boltzmann shock remains essential.

The first group of lemmas provides weighted interaction estimates between the perturbation of the macroscopic variables and generic microscopic quantities. These bounds allow us to convert factors such as \(v-\shockw{v}{i}\), \(v_x-\partial_x\shockw{v}{i}\), and their derivatives into the shock dissipation \(\mathcal G_i^S\), lower-order macroscopic norms, and exponentially decaying interaction remainders.

\begin{lemma}\label{lem:kc1}For $k\geq 0$ and $\delta_1>0$, assume that
\[
\norm{H}_{M_\#}
\le C\,\delta_1^{\,k}\,\bigl|\bigl((v^{S_1})^{-X_1}\bigr)_x\bigr|.
\]
Then
\begin{align*}
\iint_{\mathbb R\times\mathbb R^3}
\bigl|v-\bigl(v^{S_1}\bigr)^{-X_1}\bigr|
\left|H\frac{G}{M_\#}\right|\,d\xi\,dx
\le\;&
C\delta_1^{2k+2}\mathcal G_1^S
+
C\delta_1^{2k+4}e^{-C\delta_1 t}\int_{\mathbb R}\eta(U\mid\bar U)\,dx \\
&+
C\delta_1^{2k+3}\bigl(\delta_3^2e^{-C\delta_3 t}+\delta_0^2e^{-C\delta_1t}\bigr)
+
\frac{1}{\widetilde\alpha}
\int \norm{G}_{M_\#}^2\,dx,
\end{align*}
where \(\widetilde\alpha=O(1)\) and $\mathcal G_1^S$ has been defined in \eqref{eq:skGt}
\end{lemma}

\begin{proof}
For convenience, we omit the shift $X_1$. By Young's inequality, for any \(\widetilde\alpha>0\),
\begin{align*}
\iint_{\mathbb R\times\mathbb R^3}
\bigl|v-v^{S_1}\bigr|
\left|H\frac{G}{M_\#}\right|\,d\xi\,dx
\le\;&
\widetilde\alpha
\int_{\mathbb R}
|v-v^{S_1}|^2\norm{H}_{M_\#}^2\,dx \\
&+
\frac{1}{\widetilde\alpha}
\int_{\mathbb R}
\norm{G}_{M_\#}^2\,dx .
\end{align*}
Using the assumption on \(H\), we obtain
\begin{align*}
\iint_{\mathbb R\times\mathbb R^3}
\bigl|v-v^{S_1}\bigr|
\left|H\frac{G}{M_\#}\right|\,d\xi\,dx
\le\;&
C\widetilde\alpha\,\delta_1^{2k}
\int_{\mathbb R}
|v-v^{S_1}|^2\,|(v^{S_1})_x|^2\,dx \\
&+
\frac{1}{\widetilde\alpha}
\int_{\mathbb R}
\norm{G}_{M_\#}^2\,dx .
\end{align*}

We now decompose
\[
v-v^{S_1}
=
\phi +(v^C-v_*)+(v^{S_3}-v^*),
\]
and therefore
\begin{align*}
\int_{\mathbb R}|v-v^{S_1}|^2 |(v^{S_1})_x|^2\,dx
\le\;&
C\int_{\mathbb R}\phi^2 |(v^{S_1})_x|^2\,dx \\
&+
C\int_{\mathbb R}|v^C-v_*|^2 |(v^{S_1})_x|^2\,dx \\
&+
C\int_{\mathbb R}|v^{S_3}-v^*|^2 |(v^{S_1})_x|^2\,dx .
\end{align*}

For the perturbation term, since \(|(v^{S_1})_x|\lesssim \delta_1^2\) and
\[
\varphi_1+\varphi_3=1,
\]
we have
\begin{align*}
\int_{\mathbb R}\phi^2 |(v^{S_1})_x|^2\,dx
&\le
C\delta_1^2\int_{\mathbb R}\phi^2 |(v^{S_1})_x|\,dx \\
&\le
C\delta_1^2\int_{\mathbb R}|\varphi_1\phi|^2 |(v^{S_1})_x|\,dx
+
C\delta_1^2\int_{\mathbb R}|\varphi_3\phi|^2 |(v^{S_1})_x|\,dx .
\end{align*}
By the definition of \(\mathcal G_1^S\) in \eqref{eq:skGt} and Lemma~\ref{lem:intsk},
\begin{align*}
\int_{\mathbb R}\phi^2 |(v^{S_1})_x|^2\,dx
\le
C\delta_1^2\mathcal G_1^S
+
C\delta_1^4 e^{-C\delta_1 t}\int_{\mathbb R}\eta(U|\bar U)\,dx .
\end{align*}

Next, the interaction terms between the \(1\)-shock and the contact wave / \(3\)-shock are controlled by the previously established wave-interaction estimate (Lemma \ref{lem:wave-interaction-L2}): 
\begin{align*}
&\int_{\mathbb R}|v^C-v_*|^2 |(v^{S_1})_x|^2\,dx
+
\int_{\mathbb R}|v^{S_3}-v^*|^2 |(v^{S_1})_x|^2\,dx\\
&\qquad \qquad \qquad \qquad \le
C\delta_1^{3}\bigl(\delta_3^2e^{-C\delta_3 t}+(\delta_C^2+\delta_3^2)e^{-C\delta_1 t}\bigr).
\end{align*}

Combining the above bounds, we conclude that
\begin{align*}
& \iint_{\mathbb R\times\mathbb R^3}
\bigl|v-v^{S_1}\bigr|
\left|H\frac{G}{M_\#}\right|\,d\xi\,dx \\
& \quad \le 
C\wtilde \alpha\delta_1^{2k+2}\mathcal G_1^S
+
C\wtilde \alpha\delta_1^{2k+4}e^{-C\delta_1 t}\int_{\mathbb R}\eta(U|\bar U)\,dx \\
&\qquad +
C\wtilde \alpha\delta_1^{2k+3}\bigl(\delta_3^2e^{-C\delta_3 t}+(\delta_C^2+\delta_3^2)e^{-C\delta_1t}\bigr)
+
\frac{1}{\widetilde\alpha}
\int_{\mathbb R}
\norm{G}_{M_\#}^2\,dx .
\end{align*}
This proves the lemma.
\end{proof}

\begin{lemma}\label{lem:kc2}
Let \(|\alpha|=1,2\). Assume that
\[
\norm{H}_{\nu,M_\#}
\le C\,\delta_1^{\,k}\,\bigl|(v^{S_1})_x\bigr|.
\]
Then
\begin{align*}
\iint_{\mathbb R\times\mathbb R^3}
\bigl|\partial_x^\alpha (v-v^{S_1})\bigr|
\left|H\frac{G}{M_\#}\right|\,d\xi\,dx
\le\;&
C\delta_1^{2k+3}
\bigl(\delta_3^{2(|\alpha|+1)}e^{-C\delta_3 t}
+\delta_C^{2}e^{-Ct}\bigr) \\
&+
C\delta_1^{2k+2}\|\partial^\alpha\phi\|_{L^2_x}^2
+
\frac{1}{\widetilde\alpha}
\int_{\mathbb R}\norm{G}_{M_\#}^2\,dx,
\end{align*}
where \(\widetilde\alpha=O(1)\).
\end{lemma}

\begin{proof}
For convenience, we omit the shift $X_1$. By Young's inequality, for any \(\widetilde\alpha>0\),
\begin{align*}
\iint_{\mathbb R\times\mathbb R^3}
\bigl|\partial^\alpha_x (v-v^{S_1})\bigr|
\left|H\frac{G}{M_\#}\right|\,d\xi\,dx
\le\;&
\widetilde\alpha
\int_{\mathbb R}
|\partial^\alpha_x (v-v^{S_1})|^2\norm{H}_{M_\#}^2\,dx \\
&+
\frac{1}{\widetilde\alpha}
\int_{\mathbb R}
\norm{G}_{M_\#}^2\,dx .
\end{align*}
Using the assumption on \(H\), we infer
\begin{align*}
\iint_{\mathbb R\times\mathbb R^3}
\bigl|\partial^\alpha_x (v-v^{S_1})\bigr|
\left|H\frac{G}{M_\#}\right|\,d\xi\,dx
\le\;&
C\widetilde\alpha\,\delta_1^{2k}
\int_{\mathbb R}
|\partial^\alpha_x (v-v^{S_1})|^2\,|(v^{S_1})_x|^2\,dx \\
&+
\frac{1}{\widetilde\alpha}
\int_{\mathbb R}
\norm{G}_{M_\#}^2\,dx .
\end{align*}

We now decompose
\[
\partial_x^\alpha (v-v^{S_1})
=
\partial_x^\alpha\phi
+
\partial_x^\alpha(v^C-v_*)
+
\partial_x^\alpha(v^{S_3}-v^*),
\]
and therefore
\begin{align*}
\int_{\mathbb R}
|\partial_x^\alpha (v-v^{S_1})|^2\,|(v^{S_1})_x|^2\,dx
\le\;&
C\int_{\mathbb R}|\partial_x^\alpha\phi|^2\,|(v^{S_1})_x|^2\,dx \\
&+
C\int_{\mathbb R}|\partial_x^\alpha(v^C-v_*)|^2\,|(v^{S_1})_x|^2\,dx \\
&+
C\int_{\mathbb R}|\partial_x^\alpha(v^{S_3}-v^*)|^2\,|(v^{S_1})_x|^2\,dx .
\end{align*}

For the perturbation term, since \(\|(v^{S_1})_x\|_{L^\infty_x}\le C\delta_1^2\), we obtain
\begin{align*}
\int_{\mathbb R}|\partial_x^\alpha\phi|^2\,|(v^{S_1})_x|^2\,dx
\le
C\delta_1^2\|\partial_x^\alpha\phi\|_{L^2_x}^2.
\end{align*}

For the interaction terms involving the viscous contact wave and the \(3\)-shock profile, using the previously established derivative interaction bounds (Lemma~\ref{lem:wave-interaction-derivative}), we derive that 
\begin{align*}
\int_{\mathbb R}
|\partial_x^\alpha (v-v^{S_1})|^2\,|(v^{S_1})_x|^2\,dx
\le\; 
C\delta_1^2\|\partial_x^\alpha\phi\|_{L^2_x}^2  +
C\delta_1^{3}
\bigl(\delta_3^{2(|\alpha|+1)}e^{-C\delta_3 t}
+\delta_C^2e^{-C\delta_1t}\bigr).
\end{align*}

Substituting this estimate into the previous bound, we conclude that
\begin{align*}
\iint_{\mathbb R\times\mathbb R^3}
\bigl|\partial_x^\alpha (v-v^{S_1})\bigr|
\left|H\frac{G}{M_\#}\right|\,d\xi\,dx
\le\;&
C\delta_1^{2k+2}\|\partial^\alpha\phi\|_{L^2_x}^2 \\
&+
C\delta_1^{2k+3}
\bigl(\delta_3^{2(|\alpha|+1)}e^{-C\delta_3 t}
+\delta_C^{2}e^{-C\delta_1t}\bigr) \\
&+
\frac{1}{\widetilde\alpha}
\int_{\mathbb R}
\norm{G}_{M_\#}^2\,dx .
\end{align*}
This proves the lemma.
\end{proof}

\begin{lemma}\label{lem:kc3}
Let \(|\beta|=0,1\). Assume that for \(j=0,1\),
\[
\norm{\partial_x^j H}_{\nu,M_\#} 
\le
C\,\delta_1^{k+j}\,\bigl|(v^{S_1})_x\bigr|.
\]
Then
\begin{align*}
&\iint_{\mathbb R\times\mathbb R^3}
\bigl|\partial_x\bigl(v-(v^{S_1})^{-X_1}\bigr)\bigr| \,
\bigl|\partial_x^\beta P_1(\xi_1H)\bigr| \,
\Bigl|\frac{G}{M_\#}\Bigr|\,d\xi\,dx \\
&\le
C(\delta_0+\varepsilon)
\int_{\mathbb R}
\norm{G}_{\nu,M_\#}^2\,dx \\
&\quad
+
C(\delta_0+\varepsilon)^{|\beta|}\delta_1^{k}
\bigl(
\delta_3^{2}\delta_1e^{-C\delta_3 t}
+\delta_0^2\delta_1e^{-C\delta_1 t}
\bigr)
+
C(\delta_0+\varepsilon)^{|\beta|}\delta_1^{k+1}
\|\partial_x\phi\|_{L^2_x}^2 .
\end{align*}
\end{lemma}

\begin{proof}
For convenience, we omit the shift $X_1$. We divide the proof into the cases \(|\beta|=0\) and \(|\beta|=1\).

\medskip
\noindent\emph{Case 1: \(|\beta|=0\).}
Since
\[
P_1(\xi_1H)
=
\xi_1H-\sum_{j=0}^4\langle \xi_1H,\chi_j\rangle_M\chi_j,
\]
we write
\begin{align*}
&\iint
\bigl|\partial_x^\alpha(v-v^{S_1})\bigr| \,
\bigl|P_1(\xi_1H)\bigr| \,
\Bigl|\frac{G}{M_\#}\Bigr|\,d\xi\,dx \\
&\le
\iint
\bigl|\partial_x^\alpha(v-v^{S_1})\bigr|
\, \bigl|\xi_1H\bigr|
\, \Bigl|\frac{G}{M_\#}\Bigr|\,d\xi\,dx \\
&\qquad
+
C\iint
\bigl|\partial_x^\alpha(v-v^{S_1})\bigr|
\,\bigl|\langle \xi_1H,\chi_j\rangle_M\bigr|
\,\bigl|M\bigr|
\,\Bigl|\frac{G}{M_\#}\Bigr|\,d\xi\,dx
=:I_1+I_2 .
\end{align*}

For \(I_1\), Holder's inequality in \(\xi\) gives
\begin{align*}
I_1
\le&
\int
\bigl|\partial_x^\alpha(v-v^{S_1})\bigr|
\, \norm{H}_{\nu,M_\#} \, \norm{G}_{\nu,M_\#} 
\, dx .
\end{align*}
Using the assumption with \(j=0\), we obtain
\begin{align*}
I_1
\le&
C\delta_1^k
\int
\bigl|\partial_x^\alpha(v-v^{S_1})\bigr| \,
\bigl|(v^{S_1})_x\bigr| \,
\norm{G}_{\nu,M_\#} \,
dx \\
\le&
C\delta_1^{k-1}
\int
\bigl|\partial_x^\alpha(v-v^{S_1})\bigr|^2 \,
\bigl|(v^{S_1})_x\bigr|^2\,dx
+
C\delta_1
\int
\norm{G}_{\nu,M_\#}^2 \,dx .
\end{align*}
If \(|\alpha|=0\), we use the zeroth-order interaction estimate (Lemma \ref{lem:wave-interaction-L2}); if \(|\alpha|=1\), we use the derivative interaction estimate (Lemma \ref{lem:wave-interaction-derivative}). In this case,
\begin{align*}
I_1
\le&
C\delta_1
\int
\norm{G}_{\nu,M_\#}^2\,dx \\
&+
C\delta_1^{k+1}
\bigl(
\delta_1\delta_3^{4}e^{-C\delta_3 t}
+\delta_C^2\delta_1e^{-C\delta_1t}
\bigr)
+
C\delta_1^{k+1}
\|\partial_x\phi\|_{L^2_x}^2 .
\end{align*}
The term \(I_2\) is treated in exactly the same way, since
\[
|\langle \xi_1H,\chi_j\rangle_M|
\le
C\left(
\int \frac{(1+|\xi|)|H|^2}{M_\#}\,d\xi
\right)^{1/2},
\qquad
\left|\frac{M}{M_\#}\right|\le C.
\]
This proves the estimate for \(|\beta|=0\). Also, we can compute this for $|\alpha|=0$. That is,

\begin{align} \label{eq:kc3-1}
&\iint_{\mathbb R\times\mathbb R^3}
\bigl|v-(v^{S_1})^{-X_1}\bigr| \,
\bigl|\partial_x^\beta P_1(\xi_1H)\bigr| \,
\Bigl|\frac{G}{M_\#}\Bigr|\,d\xi\,dx \nonumber \\
&\le
C\delta_1
\int_{\mathbb R}
\norm{G}_{\nu,M_\#}^2\,dx 
+
C\delta_1^{k+1}
\bigl(
\delta_3^{4}\delta_1e^{-C\delta_3 t}
+\delta_0^2\delta_1e^{-C\delta_1 t}
\bigr) \nonumber \\
&\quad +C\delta_1^{k+1}\mathcal G_1^S
+
C\delta_1^{k+3} e^{-C\delta_1 t}\int_{\mathbb R}\eta(U|\bar U)\,dx .
\end{align}

\medskip
\noindent\emph{Case 2: \(|\beta|=1\).}
Differentiating the projection gives
\[
\partial_x P_1(\xi_1H)
=
\xi_1H_x
-
\sum_{j=0}^4
\Bigl(
\partial_x\langle \xi_1H,\chi_j\rangle_M\,\chi_j
+
\langle \xi_1H,\chi_j\rangle_M\,\partial_x\chi_j
\Bigr).
\]
Hence
\begin{align*}
\left|\partial_x P_1(\xi_1H)\right|
\le&
|\xi_1H_x|
+
C|(v_x,u_x,\theta_x)|\,|M|
\, \norm{H}_{\nu,M_\#} +
C|M|\,\norm{H_x}_{\nu,M_\#}.
\end{align*}
Accordingly,
\begin{align*}
&\iint
\bigl|\partial_x^\alpha(v-v^{S_1})\bigr| \,
\bigl|\partial_xP_1(\xi_1H)\bigr| \,
\Bigl|\frac{G}{M_\#}\Bigr|\,d\xi\,dx \\
&\le
\iint
\bigl|\partial_x^\alpha(v-v^{S_1})\bigr|
\,\bigl|\xi_1H_x\bigr| \,
\Bigl|\frac{G}{M_\#}\Bigr|\,d\xi\,dx \\
&\qquad
+
C\iint
\bigl|(v_x,u_x,\theta_x)\bigr|\,
\bigl|\partial_x^\alpha(v-v^{S_1})\bigr| \,
\norm{H}_{\nu,M_\#} \,
\Bigl|\frac{MG}{M_\#}\Bigr|\,d\xi\,dx \\
&\qquad
+
C\iint
\bigl|\partial_x^\alpha(v-v^{S_1})\bigr| \,
\norm{H_x}_{\nu,M_\#} \,
\Bigl|\frac{MG}{M_\#}\Bigr|\,d\xi\,dx
=:J_1+J_2+J_3 .
\end{align*}

Using the assumptions with \(j=0,1\), together with the a priori bound
\[
\|(v_x,u_x,\theta_x)\|_{L^\infty_x}\le C(\delta_0+\varepsilon),
\]
we estimate each \(J_\ell\) exactly as above and obtain
\begin{align*}
J_\ell
\le&
C(\delta_0+\varepsilon)
\int
\norm{G}_{\nu,M_\#}^2 \,dx \\
&+
C(\delta_0+\varepsilon)\delta_1^{k}
\bigl(
\delta_1\delta_3^{|\alpha|+2}e^{-C\delta_3 t}
+\delta_0^2\delta_1^{|\alpha|+1}e^{-C\delta_1 t}
\bigr)
\\
&\qquad +
C(\delta_0+\varepsilon)\delta_1^{k+2}
\|\partial^\alpha\phi\|_{L^2_x}^2 ,
\qquad \ell=1,2,3.
\end{align*}
Summing these bounds gives the desired estimate for \(|\beta|=1\), and the proof is complete.
\end{proof}

 The next lemma gives the differentiated source estimate needed in the high-order microscopic analysis.

\begin{lemma}\label{lem:kc4} 
  There exists a constant $C>0$ such that 
  \begin{align*} 
    & \iint_{\mathbb R\times\mathbb R^3} \partial_x \myparab{ -\frac{1}{v}P_1(\xi_1M_x) +\sum_{i=1,3}\frac{1}{\shockw{v}{i}}P_1^{S_i}\myparas{\xi_1\partial_x\shockw{M}{i}} } \frac{|G|}{M_\#}\,d\xi\,dx \\ 
    &\leq C\wtilde{\alpha}(\delta_0+\varepsilon)^2\norm{\myparas{\phi_x,\psi_x,\zeta_x}}_{L^2}^2 + C\wtilde{\alpha}\delta_0^3(\delta_1e^{-C\delta_1t}+\delta_3e^{-C\delta_3t}) + C\wtilde{\alpha}\frac{\delta_C^4}{(1+t)^{3/2}} \\ 
    &\quad + C\wtilde{\alpha}(\varepsilon+\delta_0)^2\delta_3^2 e^{-C\delta_3t} + C\wtilde{\alpha}(\varepsilon+\delta_0)^2\delta_1^2 e^{-C\delta_1t} + \frac{1}{\wtilde{\alpha}}\int_{\mathbb R}\norm{G}_{M_\#}^2 \,dx \\ 
    &\quad + C\wtilde{\alpha}\norm{\myparas{\psi_{xx},\zeta_{xx}}}_{L^2}^2 + C\wtilde{\alpha}(\delta_0+\varepsilon)^2 (\mathcal{G}_1^S+\mathcal{G}_3^S), 
  \end{align*} 
  where \(\wtilde{\alpha}=O(1)\). 
  \end{lemma} 
  \begin{proof} 
    Set \[ \mathfrak R := -\frac{1}{v}P_1(\xi_1M_x) +\sum_{i=1,3}\frac{1}{\shockw{v}{i}}P_1^{S_i}\Bigl(\xi_1\partial_x\shockw{M}{i} \Bigr) . \] 
    We only prove the case \(|\alpha|=1\). 
    We first decompose the Maxwellian part: 
    \begin{align*} 
      &-\frac{1}{v}\xi_1 M_x + \sum_{i=1,3}\frac{1}{\shockw{v}{i}}\xi_1 \partial_x \shockw{M}{i} \\
      &\quad = -\frac{1}{v}\xi_1 \myparas{M-\shockw{M}{1}-\shockw{M}{3}}_x +\sum_{i=1,3}\myparas{\frac{1}{\shockw{v}{i}}-\frac{1}{v}}\xi_1 \partial_x\shockw{M}{i}, 
    \end{align*} 
    and similarly for the projected part, 
    \begin{align*} 
      &-\frac{1}{v}\langle \xi_1 M_x, M\rangle_M\chi + \sum_{i=1,3} \frac{1}{\shockw{v}{i}}\bigl\langle \xi_1 \partial_x \shockw{M}{i}, \shockw{\chi}{i}\bigr\rangle_{S_i}\shockw{\chi}{i}\\ 
      &\qquad = -\frac{1}{v}\bigl\langle \xi_1 \myparas{M-\shockw{M}{1}-\shockw{M}{3}}_x , \chi \bigr\rangle_M \chi \\
      &\qquad \quad + \sum_{i=1,3} \myparas{\frac{1}{\shockw{v}{i}}-\frac{1}{v}}\bigl\langle \xi_1 \partial_x\shockw{M}{i}, \chi \bigr\rangle_M \chi\\
      &\qquad \quad + \sum_{i=1,3} \frac{1}{\shockw{v}{i}} \Biggl\{\bigl\langle \xi_1 \partial_x\shockw{M}{i}, \shockw{\chi}{i} \bigr\rangle_{S_i}\shockw{\chi}{i} \Biggr. \\ 
      &\qquad \qquad \qquad \qquad \qquad \qquad \qquad \qquad  \Biggl.-\bigl\langle \xi_1 \partial_x\shockw{M}{i}, \chi \bigr\rangle_M\chi\Biggr\}. 
    \end{align*} 
    Accordingly, we write \[ \mathfrak R=\sum_{j=1}^5 \mathcal K_j, \] where 
    \begin{align*} 
      \mathcal{K}_1 :=& \iint \myparam{-\frac{1}{v}\xi_1\myparas{M-\shockw{M}{1}-\shockw{M}{3}}_x}\frac{|G|}{M_\#} \,d\xi dx,\\ 
      \mathcal{K}_{2i} :=& \iint \myparas{\frac{1}{\shockw{v}{i}}-\frac{1}{v}}\xi_1 \partial_x\shockw{M}{i} \frac{|G|}{M_\#} \,d\xi dx,\qquad (i=1,3),\\ 
      \mathcal{K}_{3} :=& -\iint \frac{1}{v} \Bigl\langle \xi_1\myparas{M-\shockw{M}{1}-\shockw{M}{3}}_x ,\chi \Bigr\rangle_M \chi \frac{|G|}{M_\#} \,d\xi dx,\\
      \mathcal{K}_{4i} :=& \iint \myparas{\frac{1}{\shockw{v}{i}}-\frac{1}{v}}\Bigl\langle \xi_1 \partial_x\shockw{M}{i}, \chi \Bigr\rangle_M \chi \frac{|G|}{M_\#} \,d\xi dx,\qquad (i=1,3),\\ 
      \mathcal{K}_{5i} :=& \iint \frac{1}{\shockw{v}{i}} \Bigl\{\Bigl\langle \xi_1\partial_x\shockw{M}{i},\shockw{\chi}{i} \Bigr\rangle_{S_i}\shockw{\chi}{i} \\
      & \qquad \qquad \qquad \qquad  -\Bigl\langle \xi_1\partial_x\shockw{M}{i},\chi \Bigr\rangle_M \chi\Bigr\} \frac{|G|}{M_\#} \,d\xi dx,\qquad (i=1,3), 
    \end{align*} 
    and \(\mathcal K_j:=\mathcal K_{j1}+\mathcal K_{j3}\) for \(j=2,4,5\). 
    Using Young's inequality and the smooth dependence of \(M\) and \(\chi\) on the macroscopic variables, it suffices to estimate the differentiated forms of the representative terms \[ \mathcal H_1,\ \mathcal H_2,\ \mathcal H_{3i},\ \mathcal H_{4i},\ \mathcal H_{5i}, \] defined by 
    \begin{align*} 
      \mathcal{H}_1 := & \iint \Bigl|P_1\xi_1\partial_{xx} \Bigl(M-\shockw{M}{1}-\shockw{M}{3}\Bigr) \Bigr|\,\Bigl|\frac{G}{M_\#} \Bigr| \,d\xi\, dx, \\ 
      \mathcal{H}_2 := & \iint |v_x| \, \Bigl|\xi_1\Bigl(M-\shockw{M}{1}-\shockw{M}{3}\Bigr)_x\Bigr|\, \Bigl|\frac{G}{M_\#}\Bigr| \,d\xi \, dx, \\ 
      \mathcal{H}_{3i} := & \iint \Bigl|v-\shockw{v}{i}\Bigr|\, \Bigl|\xi_1\partial_{xx}\shockw{M}{i}\Bigr| \, \Bigl|\frac{G}{M_\#}\Bigr| \,d\xi \, dx,\qquad (i=1,3), \\ 
      \mathcal{H}_{4i} := & \iint \Bigl|\Bigl(v-\shockw{v}{i}\Bigr)_x\Bigr|\,\Bigl|\xi_1\partial_x\shockw{M}{i}\Bigr|\,\Bigl|\frac{G}{M_\#}\Bigr| \,d\xi \, dx,\qquad (i=1,3), \\ 
      \mathcal{H}_{5i} := & \iint |v_x| \, \Bigl|v-\shockw{v}{i}\Bigr|\, \Bigl|\xi_1\partial_x\shockw{M}{i}\Bigr| \, \Bigl|\frac{G}{M_\#}\Bigr| \,d\xi \, dx,\qquad (i=1,3). 
    \end{align*} 
    $\mathcal H_1$ comes from $\mathcal K_1+\mathcal K_3$. We only spell out the estimate of \(\mathcal H_2\), since the others are handled in the same way by Lemmas~\ref{lem:kc1}--\ref{lem:kc3}, the derivative wave-interaction estimate, and the standard decay bounds for the composite wave. By the cutoff decomposition, 
    \begin{align*} 
      &\Bigl|\Bigl(M-\shockw{M}{1}-\shockw{M}{3}\Bigr)_x\Bigr| \\
      &\quad \leq \varphi_1\Bigl|\Bigl(M-\shockw{M}{1}\Bigr)_x\Bigr| +\varphi_3\Bigl|\Bigl(M-\shockw{M}{3}\Bigr)_x\Bigr| +C\varphi_3 \Bigl|\shockw{M}{1}\,\partial_x\shockw{v}{1}\Bigr| \\
      &\qquad +C \varphi_1 \Bigl|\shockw{M}{3} \, \partial_x\shockw{v}{3}\Bigr|, 
    \end{align*} 
    and therefore 
    \begin{align*} 
      \mathcal H_2 \leq\;& \frac{1}{\wtilde{\alpha}}\iint \frac{|G|^2}{M_\#} \,d\xi dx + C\wtilde{\alpha}\sum_{i=1,3} \Biggl\{ \int \varphi_i^2|v_x|^2 \, \Bigl|\Bigl(v-\shockw{v}{i}\Bigr)_x\Bigr|^2 \,dx \Biggr.\\
      &\qquad \qquad \qquad \qquad \qquad \qquad \quad \Biggl.+ \int (1-\varphi_i)^2 |v_x|^2\, \Bigl|\partial_x\shockw{v}{i}\Bigr|^2 \,dx \Biggr\}. 
    \end{align*} 
    For \(i=1\), we decompose 
    \[ 
    v_x-\partial_x\shockw{v}{1} = \phi_x+\bigl(v^C-v_*\bigr)_x+\bigl(\shockw{v}{3}-v^*\bigr)_x 
    \] 
    and obtain 
    \begin{align*} 
      &\int \varphi_1^2 |v_x|^2 \, \Bigl|\Bigl(v-\shockw{v}{1}\Bigr)_x\Bigr|^2 dx \\
      & \qquad \le  C(\delta_0+\varepsilon)^2 \|\phi_x\|_{L^2}^2 + C\delta_0^3(\delta_1e^{-C\delta_1t}+\delta_3e^{-C\delta_3t})+ C\int |v_x^C|^4 \,dx \\
      & \qquad \qquad + C(\varepsilon+\delta_0)^2\int \varphi_1^2 \, \Bigl|\partial_x\shockw{v}{3}\Bigr|^2 \,dx. 
    \end{align*} 
    Since \[ \int |v_x^C|^4 \,dx \le C\frac{\delta_C^4}{(1+t)^{3/2}}, \qquad \int \varphi_1^2 \Bigl|\partial_x\shockw{v}{3}\Bigr|^2 \,dx \le C\delta_3^2 e^{-C\delta_3t}, \] it follows that 
    \begin{align*} 
      \mathcal H_2 \leq\;& \frac{1}{\wtilde{\alpha}}\int \norm{G}_{M_\#}^2 \,dx + C\wtilde{\alpha}(\delta_0+\varepsilon)^2\|\phi_x\|_{L^2}^2 + C\wtilde{\alpha}\delta_0^3(\delta_1e^{-C\delta_1t}+\delta_3e^{-C\delta_3t})\\ 
      &+ C\wtilde{\alpha}\frac{\delta_C^4}{(1+t)^{3/2}} + C\wtilde{\alpha}(\varepsilon+\delta_0)^2\delta_3^2 e^{-C\delta_3t}. 
    \end{align*} 
    The term \(\mathcal H_1\) can be treated in the same way as \(\mathcal H_2\). The key point is that the microscopic projection eliminates the leading density contribution, so no \(\phi_{xx}\)-term appears in the estimate.
    The remaining terms \(\mathcal H_{3i}\), \(\mathcal H_{4i}\), and \(\mathcal H_{5i}\), together with the differentiated projector-mismatch terms arising from \(\mathcal K_3\), \(\mathcal K_{4i}\), and \(\mathcal K_{5i}\), 
    are estimated by the same combination of Young's inequality, the smooth dependence of \(M\) and \(\chi\) on the macroscopic variables, 
    Lemmas~\ref{lem:kc1}--\ref{lem:kc3}, and the derivative interaction bounds. Summing all these contributions gives the stated inequality. 
  \end{proof}

\begin{remark}\label{rem:kc4-alpha0} 
  The above estimate is formulated only for \(|\alpha|=1\). 
  The reason is that the corresponding undifferentiated argument (\(\alpha=0\)) does not close at the zeroth-order level because of the contact-wave contribution. 
  Indeed, \[ \int \varphi_1^2 |v_x^C|^2 \,dx \le C\frac{\delta_C}{1+t}\int_{-\infty}^{\sigma_1 t/4} e^{-c x^2/(1+t)}\,dx, \] and the right-hand side is not globally time-integrable. 
  By contrast, \[ \int \varphi_1^2 |v_{xx}^C|^2 \,dx \le C\frac{\delta_C}{(1+t)^2}\int_{-\infty}^{\sigma_1 t/4} e^{-c x^2/(1+t)}\,dx \le C\frac{\delta_C}{(1+t)^{3/2}}, \] which is integrable in time. 
  This explains why the contact contribution must be excluded from the zeroth-order argument and why the differentiated estimate \(|\alpha|=1\) is the relevant form used in the high-order microscopic analysis. 
\end{remark}

\begin{lemma}\label{lem:kc5} 
  Let \(P(\xi)\) be a fixed polynomial in \(\xi\), and assume that \(g\in \mathfrak Z^\perp\), so that \(L_M^{-1}g\) is well-defined. Then 
  \begin{align} \label{eq:kc5} 
    \left( \int_{\mathbb R^3_\xi} P(\xi)\,(L_M^{-1}g)_x\,d\xi \right)^2 \le\;& \frac{C}{\lambda_{\mathrm{mic}}^{2}} \int_{\mathbb R^3_\xi} \frac{(1+|\xi|)^{-1}|g_x|^2}{M_\#}\,d\xi \\ 
    &+ \frac{C}{\lambda_{\mathrm{mic}}^{4}} \left( \int_{\mathbb R^3_\xi} \frac{(1+|\xi|)|M_x|^2}{M_\#}\,d\xi \right) \left( \int_{\mathbb R^3_\xi} \frac{(1+|\xi|)^{-1}|g|^2}{M_\#}\,d\xi \right). \nonumber 
  \end{align} 
\end{lemma} 
\begin{proof} Since \(P(\xi)\) is a fixed polynomial and \(M_\#\) is a global Maxwellian, we have \[ \left( \int_{\mathbb R^3_\xi} P(\xi)\,(L_M^{-1}g)_x\,d\xi \right)^2 \le C \int_{\mathbb R^3_\xi} \frac{|(L_M^{-1}g)_x|^2}{M_\#}\,d\xi . \] 
  Next, by differentiating the identity \(L_M(L_M^{-1}g)=g\), we obtain \[ (L_M^{-1}g)_x = L_M^{-1}(g_x) - L_M^{-1}\Bigl( \mathcal N(M_x,L_M^{-1}g)+\mathcal N(L_M^{-1}g,M_x) \Bigr). \] 
  Therefore, 
  \begin{align*} 
    \int_{\mathbb R^3_\xi} \frac{|(L_M^{-1}g)_x|^2}{M_\#}\,d\xi \le\;& C \int_{\mathbb R^3_\xi} \frac{|L_M^{-1}g_x|^2}{M_\#}\,d\xi \\ 
    &+ C \int_{\mathbb R^3_\xi} \frac{\bigl|L_M^{-1}(\mathcal N(M_x,L_M^{-1}g)+\mathcal N(L_M^{-1}g,M_x))\bigr|^2}{M_\#}\,d\xi . 
  \end{align*} 
  By the standard weighted estimate for \(L_M^{-1}\), \[ \int_{\mathbb R^3_\xi} \frac{|L_M^{-1}h|^2}{M_\#}\,d\xi \le \frac{C}{\lambda_{\mathrm{mic}}^2} \int_{\mathbb R^3_\xi} \frac{(1+|\xi|)^{-1}|h|^2}{M_\#}\,d\xi , \] we obtain 
  \begin{align*} 
    \int_{\mathbb R^3_\xi} \frac{|(L_M^{-1}g)_x|^2}{M_\#}\,d\xi \le\;& \frac{C}{\lambda_{\mathrm{mic}}^2} \int_{\mathbb R^3_\xi} \frac{(1+|\xi|)^{-1}|g_x|^2}{M_\#}\,d\xi \\ &+ \frac{C}{\lambda_{\mathrm{mic}}^2} \int_{\mathbb R^3_\xi} \frac{(1+|\xi|)^{-1}|\mathcal N(M_x,L_M^{-1}g)|^2}{M_\#}\,d\xi \\
     &+ \frac{C}{\lambda_{\mathrm{mic}}^2} \int_{\mathbb R^3_\xi} \frac{(1+|\xi|)^{-1}|\mathcal N(L_M^{-1}g,M_x)|^2}{M_\#}\,d\xi . 
    \end{align*} 
    Using the standard bilinear estimate for the collision operator, \[ \int_{\mathbb R^3_\xi} \frac{(1+|\xi|)^{-1}|\mathcal N(f,h)|^2}{M_\#}\,d\xi \le C \left( \int_{\mathbb R^3_\xi} \frac{(1+|\xi|)|f|^2}{M_\#}\,d\xi \right) \left( \int_{\mathbb R^3_\xi} \frac{(1+|\xi|)^{-1}|h|^2}{M_\#}\,d\xi \right), \] with \((f,h)=(M_x,L_M^{-1}g)\) and \((L_M^{-1}g,M_x)\), we deduce 
    \begin{align*} 
      \int_{\mathbb R^3_\xi} \frac{|(L_M^{-1}g)_x|^2}{M_\#}\,d\xi \le\;& \frac{C}{\lambda_{\mathrm{mic}}^2} \int_{\mathbb R^3_\xi} \frac{(1+|\xi|)^{-1}|g_x|^2}{M_\#}\,d\xi \\ 
      &+ \frac{C}{\lambda_{\mathrm{mic}}^2} \left( \int_{\mathbb R^3_\xi} \frac{(1+|\xi|)|M_x|^2}{M_\#}\,d\xi \right) \left( \int_{\mathbb R^3_\xi} \frac{(1+|\xi|)^{-1}|L_M^{-1}g|^2}{M_\#}\,d\xi \right). 
    \end{align*} 
    Applying once again the weighted estimate for \(L_M^{-1}\), \[ \int_{\mathbb R^3_\xi} \frac{(1+|\xi|)^{-1}|L_M^{-1}g|^2}{M_\#}\,d\xi \le \frac{C}{\lambda_{\mathrm{mic}}^2} \int_{\mathbb R^3_\xi} \frac{(1+|\xi|)^{-1}|g|^2}{M_\#}\,d\xi , \] we conclude that 
    \begin{align*} 
      \int_{\mathbb R^3_\xi} \frac{|(L_M^{-1}g)_x|^2}{M_\#}\,d\xi \le\;& \frac{C}{\lambda_{\mathrm{mic}}^2} \int_{\mathbb R^3_\xi} \frac{(1+|\xi|)^{-1}|g_x|^2}{M_\#}\,d\xi \\ 
      &+ \frac{C}{\lambda_{\mathrm{mic}}^4} \left( \int_{\mathbb R^3_\xi} \frac{(1+|\xi|)|M_x|^2}{M_\#}\,d\xi \right) \left( \int_{\mathbb R^3_\xi} \frac{(1+|\xi|)^{-1}|g|^2}{M_\#}\,d\xi \right). 
    \end{align*} 
    Combining this with the first inequality proves \eqref{eq:kc5}. 
  \end{proof}

\begin{lemma}\label{lem:kc6} Let \(P(\xi)\) be a fixed polynomial in \(\xi\). Recall that $\wtilde \Pi_1$ is defined by \eqref{eq:pi1w}.
  \begin{align} \label{eq:kc6} 
    \begin{aligned} 
      &\int_{\mathbb R} \left( \int_{\mathbb R^3_\xi} P(\xi)\,\partial_x \widetilde{\Pi}_1 \,d\xi \right)^2 dx\\
      \le\;& C(\delta_0+\varepsilon)\sum_{i=1,3}\delta_i \mathcal{G}^S_i + C(\delta_0+\varepsilon)\sum_{i=1,3}\delta_i |\dot X_i|^2 + C(\delta_0+\varepsilon)\norm{(\phi_x,\psi_x,\zeta_x)}_{L^2_x}^2\\ 
      &+ C(\delta_0+\varepsilon) \int_{\mathbb R} \norm{\wtilde G_{\textup{rem}}}_{\nu,M_\#}^2 + \norm{\wtilde G_t}_{\nu,M_\#}^2 + \norm{\wtilde G_x}_{\nu,M_\#}^2 dx \\ 
      &+ C \int_{\mathbb R} \norm{\wtilde G_{xx}}_{\nu,M_\#}^2 + \norm{\wtilde G_{tx}}_{\nu,M_\#}^2 \,dx + C(\delta_0+\varepsilon)\delta_0^2\myparas{\delta_1e^{-C\delta_1t}+\delta_3e^{-C\delta_3t}} \\ 
      & + C\frac{\delta_C}{(1+t)^{\frac{5}{4}}}. 
    \end{aligned} 
  \end{align} 
\end{lemma}
\begin{proof} 
Recall that \[ \wtilde \Pi_1 := \Pi_1-\shock{\Pi_1}{1}-\shock{\Pi_1}{3}. \] 
By Young's inequality in \(\xi\), 
\begin{align} \label{eq:kc6-reduce} 
  \int_{\mathbb R} \left( \int_{\mathbb R^3_\xi} P(\xi)\,\wtilde \Pi_{1x}\,d\xi \right)^2 dx \le C \iint_{\mathbb R\times\mathbb R^3} \frac{|\wtilde \Pi_{1x}|^2}{M_\#}\,d\xi\,dx . 
\end{align} 
Recall the decomposition 
\begin{align*} 
  \wtilde \Pi_1 =&\; L_M^{-1}P_1\Biggl[\wtilde{G}_t-\frac{u_1}{v}\wtilde{G}_x+\frac{1}{v}P_1\Bigl(\xi_1\wtilde{G}_x\Bigr)-\mathcal N\Bigl(\wtilde{G},\wtilde{G}\Bigr)\Biggr] \\ 
  &\; -L_M^{-1}\Biggl[\mathcal N\Bigl(\wtilde{G},\shockw{G}{1}+\shockw{G}{3}\Bigr)+\mathcal N\Bigl(\shockw{G}{1}+\shockw{G}{3},\wtilde{G}\Bigr)\Biggr] \\ 
  &\; +L_M^{-1}\Biggl[\mathcal N\Bigl(\shockw{G}{1},\shockw{G}{3}\Bigr)+\mathcal N\Bigl(\shockw{G}{3},\shockw{G}{1}\Bigr)\Biggr] \\
  &\; -\dot{X}_1\bigl(L_{1}^S\bigr)^{-1}\Bigl(\partial_x\shockw{G}{1}\Bigr) -\dot{X}_3\bigl(L_{3}^S\bigr)^{-1}\Bigl(\partial_x\shock{G}{3}\Bigr) +J, 
\end{align*} 
where 
\begin{align*} 
  J :=&\; \sum_{i\in\{1,3\}} \Bigl(L_M^{-1}P_1-\bigl(L_{i}^S\bigr)^{-1}P_1^{S_i}\Bigr) \Bigl[\partial_t\shockw{G}{i}-\mathcal N\Bigl(\shockw{G}{i},\shockw{G}{i}\Bigr)\Bigr] \\ 
  &+ \sum_{i\in\{1,3\}} \Biggl[\frac{\shockw{u_1}{i}}{\shockw{v}{i}}\bigl(L_{i}^S\bigr)^{-1}P_1^{S_i}-\frac{u_1}{v}L_M^{-1}P_1\Biggr] \Bigl(\partial_x\shockw{G}{i}\Bigr) \\
  &+ \sum_{i\in\{1,3\}} \Biggl[\frac{1}{v}L_M^{-1}P_1\xi_1-\frac{1}{\shockw{v}{i}}\bigl(L_{i}^S\bigr)^{-1}P_1^{S_i}\xi_1\Biggr] \Bigl(\partial_x\shockw{G}{i}\Bigr), 
\end{align*} 
  and \(L_{i}^{S}:=L_{\shockw{M}{i}}\). Accordingly, write \[ \wtilde \Pi_1=\mathcal P^{(1)}+\mathcal P^{(2)}+\mathcal P^{(3)}+\mathcal P^{(4)}+\mathcal P^{(5)}, \] where the five pieces correspond to the five displayed lines above. We estimate each contribution to \(\iint |\mathcal P_x|^2/M_\#\). 
  \medskip \noindent
  \textit{Step 1: the main microscopic terms \(\mathcal P^{(1)}\).} 
  Applying Lemma~\ref{lem:kc5} to \[ g=\wtilde{G}_t-\frac{u_1}{v}\wtilde{G}_x+\frac{1}{v}P_1(\xi_1\wtilde{G}_x)-\mathcal N(\wtilde{G},\wtilde{G}), \] we obtain \begin{align} \label{eq:kc6-P1} 
    \iint \frac{|(\mathcal P^{(1)})_x|^2}{M_\#}\,d\xi\,dx \le\;& C \iint \frac{(1+|\xi|)^{-1}|g_x|^2}{M_\#}\,d\xi\,dx \\ 
    &+ C(\delta_0+\varepsilon) \iint \frac{(1+|\xi|)^{-1}|g|^2}{M_\#}\,d\xi\,dx . \nonumber 
  \end{align} 
  Using the product rule and the a priori smallness of \((v_x,u_x,\theta_x)\), we have 
  \begin{align*} 
    |g_x| \le\;& |\wtilde{G}_{tx}| + C(\delta_0+\varepsilon)|\wtilde{G}_x| + C|\wtilde{G}_{xx}| + C(\delta_0+\varepsilon)|\wtilde{G}_t| \\ 
    &+ C\,|\mathcal N(\wtilde{G},\wtilde{G})_x|, 
  \end{align*} 
  and similarly for \(g\). By the standard bilinear collision estimate and the bootstrap smallness, 
  \[ 
  \iint \frac{(1+|\xi|)^{-1}|\mathcal N(\wtilde{G},\wtilde{G})_x|^2}{M_\#}\,d\xi\,dx \le C(\delta_0+\varepsilon) \int \norm{\wtilde{G}_{\text{rem}}}_{\nu,M_\#}^2+\norm{\wtilde{G}_{\text{x}}}_{\nu,M_\#}^2\,dx . 
  \]
  Therefore, 
  \begin{align} \label{eq:kc6-P1-final} 
    \iint \frac{|(\mathcal P^{(1)})_x|^2}{M_\#}\,d\xi\,dx \le\;& C \int \norm{\wtilde G_{tx}}_{\nu,M_\#}^2 + \norm{\wtilde G_{xx}}_{\nu,M_\#}^2\,dx \\ 
    &+ C(\delta_0+\varepsilon) \int \norm{\wtilde G_{\text{rem}}}_{\nu,M_\#}^2 + \norm{\wtilde G_{x}}_{\nu,M_\#}^2 + \norm{\wtilde G_{t}}_{\nu,M_\#}^2\,dx . \nonumber 
  \end{align} 
    \medskip 
    \noindent
    \textit{Step 2: the mixed collision terms \(\mathcal P^{(2)}\).} 
    Using again Lemma~\ref{lem:kc5}, together with the standard bilinear estimate for \(\mathcal N\bigl(\wtilde{G},\shockw{G}{i}\bigr)\) and \(\mathcal N\bigl(\shockw{G}{i},\wtilde{G}\bigr)\), we obtain 
    \begin{align} \label{eq:kc6-P2-final} 
      \iint \frac{|(\mathcal P^{(2)})_x|^2}{M_\#}\,d\xi\,dx \le\;& C(\delta_0+\varepsilon) \int \norm{\wtilde G_{\text{rem}}}_{\nu,M_\#}^2 + \norm{\wtilde G_{x}}_{\nu,M_\#}^2\,dx \\ 
      &+ C(\delta_0+\varepsilon)^2\delta_0^3\myparas{\delta_1e^{-C\delta_1t}+\delta_3e^{-C\delta_3t}} . \nonumber 
    \end{align} 
    \medskip 
    \noindent
    \textit{Step 3: the shock--shock interaction terms \(\mathcal P^{(3)}\).} 
    Since \(\shockw{G}{1}\) and \(\shockw{G}{3}\) are exponentially localized and spatially separated, 
    \begin{align} \label{eq:kc6-P3-final} 
      \iint \frac{|(\mathcal P^{(3)})_x|^2}{M_\#}\,d\xi\,dx \le C(\delta_0+\varepsilon)^2\delta_0^3 \myparas{\delta_1e^{-C\delta_1t}+\delta_3e^{-C\delta_3t}} . 
    \end{align} 
    \medskip 
    \noindent
    \textit{Step 4: the shift terms \(\mathcal P^{(4)}\).} 
    By the explicit profile bounds, \[ \iint \frac{\bigl|\partial_x\bigl(\bigl(L_{i}^S\bigr)^{-1}\partial_x\shockw{G}{i}\bigr)\bigr|^2}{M_\#}\,d\xi\,dx \le C\delta_i^2, \qquad i=1,3. \] Hence 
    \begin{align} \label{eq:kc6-P4-final} 
      \iint \frac{|(\mathcal P^{(4)})_x|^2}{M_\#}\,d\xi\,dx \le C\sum_{i=1,3}\delta_i^2|\dot X_i|^2 \le C(\delta_0+\varepsilon)\sum_{i=1,3}\delta_i|\dot X_i|^2 . 
    \end{align} 
    \medskip 
    \noindent
    \textit{Step 5: the coefficient-mismatch terms \(J=\mathcal P^{(5)}\).} 
    These are the most delicate terms. We only record representative estimates; the remaining variants are treated in the same way. 
    First, for the difference of inverse linearized operators, 
    \begin{align*}
    \Bigl(L_M^{-1}-\bigl(L_{1}^S\bigr)^{-1} \Bigr) \bigl(\partial_x\shockw{G}{1}\bigr) =\; & L_M^{-1}\Bigl( \mathcal N\bigl(\shockw{M}{1}-M,\bigl(L_{1}^S\bigr)^{-1}\partial_x\shockw{G}{1}\bigr) \\
    &  + \mathcal N\bigl(\bigl(L_{1}^S\bigr)^{-1}\partial_x\shockw{G}{1},\shockw{M}{1}-M\bigr) \Bigr), 
    \end{align*}
    and Lemma~\ref{lem:kc5} gives 
    \begin{align*} 
      & \iint \frac{\Bigl|\partial_x\Bigl\{\bigl(L^{-1}_M-\bigl(L_{1}^S\bigr)^{-1}\bigr)\bigl(\partial_x\shockw{G}{1}\bigr)\Bigr\}\Bigr|^2}{M_\#}\,d\xi\,dx \\
      & \quad \le C(\delta_0+\varepsilon)^2\int \bigl|\partial_x\shockw{v}{1}\bigr|^2 \, \bigl|\bigl(v-\shockw{v}{1}
      \bigr)_x\bigr|^2\,dx \\
      & \quad \qquad + C(\delta_0+\varepsilon)^2\int \bigl|\partial_x\shockw{v}{1}\bigr|^2 \, \bigl|v-\shockw{v}{1}\bigr|^2\,dx . 
    \end{align*} 
    By Lemmas~\ref{lem:kc1}--\ref{lem:kc4}, this is bounded by 
    \begin{align} \label{eq:kc6-J1} 
      C(\delta_0+\varepsilon)^2\delta_1\mathcal G_1^S + C(\delta_0+\varepsilon)^2\delta_0^3\myparas{\delta_1e^{-C\delta_1t}+\delta_3e^{-C\delta_3t}} + C(\delta_0+\varepsilon)^2\norm{(\phi_x,\psi_x,\zeta_x)}_{L^2_x}^2 . 
    \end{align} 
    Next, for the coefficient mismatch 
    \begin{align*}
     & \frac{\shockw{u_1}{1}}{\shockw{v}{1}}\bigl(L_{1}^S\bigr)^{-1}-\frac{u_1}{v}L^{-1}_M \\
     &\quad = \frac{\shockw{u_1}{1}-u_1}{\shockw{v}{1}}\bigl(L_{1}^{S}\bigr)^{-1} -\frac{u_1}{v\shockw{v}{1}}\bigl(\shockw{v}{1}-v\bigr)\bigl(L_{1}^S\bigr)^{-1} +\frac{u_1}{v}\bigl(\bigl(L^{S}_{1}\bigr)^{-1}-L^{-1}_M\bigr), 
    \end{align*}
    each term is treated exactly as above, giving 
    \begin{align} \label{eq:kc6-J2} 
      C(\delta_0+\varepsilon)^2\delta_1\mathcal G_1^S + C(\delta_0+\varepsilon)^2\delta_0^3\myparas{\delta_1e^{-C\delta_1t}+\delta_3e^{-C\delta_3t}} + C(\delta_0+\varepsilon)^2\norm{(\phi_x,\psi_x,\zeta_x)}_{L^2_x}^2 . 
    \end{align} 
Finally, for the projector mismatch terms, for example 
\begin{align*}
    & \frac{1}{\shockw{v}{1}} L_M^{-1} \Biggl[\Bigl\langle \xi_1\partial_x\shockw{G}{1},\chi_j-\shockw{\chi_j}{1} \Bigr\rangle_M \chi_j\Biggr], \\ 
    & \frac{1}{\shockw{v}{1}} L_M^{-1} \Biggl[\Bigl\langle \xi_1\partial_x\shockw{G}{1},\shockw{\chi_j}{1} \Bigr\rangle_{S_1}\Bigl(\chi_j-\shockw{\chi_j}{1}\Bigr)\Biggr],
\end{align*}  
 we use the smooth dependence of \(\chi_j\) on the macroscopic variables and Lemma~\ref{lem:kc5} to obtain 
 \begin{align} \label{eq:kc6-J3} 
  C(\delta_0+\varepsilon)^2\delta_1\mathcal G_1^S + C(\delta_0+\varepsilon)^2\delta_0^3\myparas{\delta_1e^{-C\delta_1t}+\delta_3e^{-C\delta_3t}} + C(\delta_0+\varepsilon)^2\norm{(\phi_x,\psi_x,\zeta_x)}_{L^2_x}^2 . 
\end{align} 
The analogous \(i=3\) terms satisfy the same bounds. Therefore, 
\begin{align} \label{eq:kc6-P5-final} 
  \iint \frac{|(\mathcal P^{(5)})_x|^2}{M_\#}\,d\xi\,dx \le\;& C(\delta_0+\varepsilon)\sum_{i=1,3}\delta_i \mathcal G_i^S + C(\delta_0+\varepsilon)\norm{(\phi_x,\psi_x,\zeta_x)}_{L^2_x}^2 \\ 
  &+ C(\delta_0+\varepsilon)^2\delta_0^3\myparas{\delta_1e^{-C\delta_1t}+\delta_3e^{-C\delta_3t}} + C \frac{\delta_C}{(1+t)^\frac{5}{4}}. \nonumber 
\end{align} 
Combining \eqref{eq:kc6-reduce}, \eqref{eq:kc6-P1-final}, \eqref{eq:kc6-P2-final}, \eqref{eq:kc6-P3-final}, \eqref{eq:kc6-P4-final}, and \eqref{eq:kc6-P5-final}, we obtain \eqref{eq:kc6}. \end{proof}

\subsubsection{Estimate of $\wtilde{G}_x$}
We point out that some intermediate estimates are not closed at each fixed derivative level. For instance, the estimate of $\wtilde G_x$ contains higher-order terms such as $\wtilde G_{xx}$ on the {right-hand side}. This does not lead to a loss of derivatives, since the high-order energy estimate is closed simultaneously over all relevant derivative levels. More precisely, the terms involving $\wtilde G_{xx}$ are controlled by the dissipation part of the final high-order energy functional, after summing the estimates for different derivative orders with suitable weights. Therefore, the estimate for $\wtilde G_x$ should be understood as one component of the {coupled} high-order energy argument rather than as an independent closed estimate.

\begin{lemma}\label{lem:GxE}
There exists a constant $C>0$ such that
\begin{align}
\begin{aligned}
&\frac{d}{dt}\iint_{\R}\norm{\wtilde G_{x}}^2_{M_\#} \,dx
+\int_{\R}\norm{\wtilde G_x}_{\nu,M_\#}^2\,dx \\
&\le
C(\delta_0+\varepsilon)\|(\phi_x,\psi_x,\zeta_x)\|_{L^2_x}^2
+C\delta_0^2(\delta_0+\varepsilon)\bigl(\delta_1e^{-C\delta_1t}+\delta_3e^{-C\delta_3t}\bigr)
+C\frac{\delta_C}{(1+t)^{\frac{5}{4}}} \\
&\quad
+C\|(\psi_{xx},\zeta_{xx})\|_{L^2_x}^2
+C(\delta_0+\varepsilon)\sum_{i=1,3}\delta_i|\dot X_i|^2 
+C(\delta_0+\varepsilon)(\mathcal G_1^S+\mathcal G_3^S) \\
&\quad
+C(\delta_0+\varepsilon)\int_{\R}\norm{\wtilde G_{\textup{rem}}}_{\nu,M_\#}^2\,dx +C\int_\mathbb{R} \norm{\wtilde G_{xx}}_{\nu,M_\#}^2 dx.
\end{aligned}
\end{align}
\end{lemma}

\begin{proof}
Differentiate \eqref{eq:pesy3} with respect to $x$, multiply the resulting equation by
$\wtilde{G}_x/M_\#$, and integrate over $(x,\xi)\in\R\times\R^3$.
Then
\begin{align*}
\frac12\frac{d}{dt}\iint \frac{|\wtilde{G}_x|^2}{M_\#}\,d\xi\,dx
=
\mathcal T_1+\sum_{i=1,3}\mathcal T_{2i}
+\mathcal T_3+\sum_{i=1,3}\mathcal T_{4i}
+\sum_{i=1,3}\mathcal T_{5i}
+\mathcal T_6+\mathcal T_7+\mathcal T_8+\mathcal T_9+\mathcal T_{10},
\end{align*}
where
\begin{align*}
\mathcal T_1
&:= \iint (L_M\wtilde G)_x\,\frac{\wtilde G_x}{M_\#}\,d\xi\,dx,\\
\mathcal T_{2i}
&:= \iint \dot X_i\,\bigl(\partial_{xx}\shockw{G}{i}\bigr)\,\frac{\wtilde G_x}{M_\#}\,d\xi\,dx,\\
\mathcal T_3
&:= \iint \left(\frac{u_1}{v}\wtilde G_x-\frac1vP_1(\xi_1\wtilde G_x)\right)_x
\frac{\wtilde G_x}{M_\#}\,d\xi\,dx,\\
\mathcal T_{4i}
&:= \iint
\Biggl[
\left(\frac{u_1}{v}-\frac{\shockw{u_1}{i}}{\shockw{v}{i}}\right)\partial_x\shockw{G}{i}
\Biggr]_x
\frac{\wtilde G_x}{M_\#}\,d\xi\,dx,\\
\mathcal T_{5i}
&:= -\iint
\Biggl[
\frac1vP_1(\xi_1\partial_x\shockw{G}{i})
-\frac1{\shockw{v}{i}}\bigl(P_1^{S_i}\bigr)(\xi_1\partial_x\shockw{G}{i})
\Biggr]_x
\frac{\wtilde G_x}{M_\#}\,d\xi\,dx,\\
\mathcal T_6
&:= \iint
\Biggl[
-\frac1vP_1(\xi_1M_x)
+\frac1{\shockw{v}{1}}\bigl(P_1^{S_1}\bigr)(\xi_1\partial_x\shockw{M}{1})\Biggr.\\
&\qquad \qquad \Biggl. +\frac1{\shockw{v}{3}}\bigl(P_1^{S_3}\bigr)(\xi_1\partial_x\shockw{M}{3})
\Biggr]_x
\frac{\wtilde G_x}{M_\#}\,d\xi\,dx,\\
\mathcal T_7
&:= \iint \bigl(\mathcal N(\wtilde G,\wtilde G)\bigr)_x
\frac{\wtilde G_x}{M_\#}\,d\xi\,dx,\\
\mathcal T_8
&:= \sum_{i=1,3}\iint
\Bigl[\mathcal N\bigl(\wtilde G,\shockw{G}{i}\bigr)+\mathcal N\bigl(\shockw{G}{i},\wtilde G\bigr)\Bigr]_x
\frac{\wtilde G_x}{M_\#}\,d\xi\,dx,\\
\mathcal T_9
&:= \iint
\Bigl[\mathcal N\bigl(\shockw{G}{1},\shockw{G}{3}\bigr)+\mathcal N\bigl(\shockw{G}{3},\shockw{G}{1}\bigr)\Bigr]_x
\frac{\wtilde G_x}{M_\#}\,d\xi\,dx,\\
\mathcal T_{10}
&:= \sum_{i=1,3}\iint
\Bigl[(L_M-L_{i}^S)\shockw{G}{i}\Bigr]_x
\frac{\wtilde G_x}{M_\#}\,d\xi\,dx.
\end{align*}

We estimate these terms one by one.

\medskip
\noindent
\textbf{1. Estimate of $\mathcal T_1$.}
Since
\[
L_M h=\mathcal N(M,h)+\mathcal N(h,M),
\]
we have
\begin{equation}\label{eq:GxE-LM-diff}
(L_Mh)_x=L_Mh_x+\mathcal N(M_x,h)+\mathcal N(h,M_x).
\end{equation}
Hence
\begin{align*}
\mathcal T_1
&=
\iint L_M\wtilde G_x\,\frac{\wtilde G_x}{M_\#}\,d\xi\,dx
+\iint \Bigl(\mathcal N(M_x,\wtilde G)+\mathcal N(\wtilde G,M_x)\Bigr)\frac{\wtilde G_x}{M_\#}\,d\xi\,dx .
\end{align*}
By Lemma~\ref{lem:collision-coercivity}, more precisely \eqref{eq:collision-coercivity},
\begin{align*}
\iint L_M\wtilde G_x\,\frac{\wtilde G_x}{M_\#}\,d\xi\,dx
\le
-\lambda_{\mathrm{mic}} \iint \frac{(1+|\xi|)|\wtilde G_x|^2}{M_\#}\,d\xi\,dx .
\end{align*}
For the remaining terms, Young's inequality and the bilinear collision estimate
\eqref{eq:collision-Q} imply
\begin{align*}
&\left|
\iint \Bigl(\mathcal N(M_x,\wtilde G)+\mathcal N(\wtilde G,M_x)\Bigr)\frac{\wtilde G_x}{M_\#}\,d\xi\,dx
\right| \\
&\le
\frac{\lambda_{\mathrm{mic}}}{128}\iint \frac{(1+|\xi|)|\wtilde G_x|^2}{M_\#}\,d\xi\,dx
+C\iint \frac{(1+|\xi|)^{-1}}{M_\#}
\Bigl(
|\mathcal N(M_x,\wtilde G)|^2+|\mathcal N(\wtilde G,M_x)|^2
\Bigr)\,d\xi\,dx .
\end{align*}
Now we use the decomposition \eqref{eq:mic2}, namely
\[
\wtilde G=\wtilde G_{\text{rem}}+\wtilde G_C,
\]
together with the explicit form of $\wtilde G_C$, the smallness of $(v_x,u_x,\theta_x)$, and the standard bounds for the composite wave.
Applying \eqref{eq:collision-Q} once more, we obtain
\begin{align*}
&\iint \frac{(1+|\xi|)^{-1}}{M_\#}
\Bigl(
|\mathcal N(M_x,\wtilde G)|^2+|\mathcal N(\wtilde G,M_x)|^2
\Bigr)\,d\xi\,dx \\
&\le
C(\delta_0+\varepsilon)^2\|(\phi_x,\psi_x,\zeta_x)\|_{L^2_x}^2
+C(\delta_0+\varepsilon)\int \norm{\wtilde G_{\text{rem}}}_{\nu,M_\#}^2\,dx \\
&\qquad
+C\delta_0^3\bigl(\delta_1e^{-C\delta_1t}+\delta_3e^{-C\delta_3t}\bigr)
+C\frac{\delta_C}{(1+t)^{3/2}} .
\end{align*}
Therefore
\begin{align}\label{eq:GxE-T1}
\mathcal T_1
&\le
-\frac{127}{128}\lambda_{\mathrm{mic}}
\int \norm{\wtilde G_x}_{\nu,M_\#}^2\,dx \notag\\
&\quad
+C(\delta_0+\varepsilon)^2\|(\phi_x,\psi_x,\zeta_x)\|_{L^2_x}^2
+C(\delta_0+\varepsilon)\int \norm{\wtilde G_{\text{rem}}}_{\nu,M_\#}^2\,dx \notag\\
&\quad
+C\delta_0^3\bigl(\delta_1e^{-C\delta_1t}+\delta_3e^{-C\delta_3t}\bigr)
+C\frac{\delta_C}{(1+t)^{3/2}} .
\end{align}

\medskip
\noindent
\textbf{2. Estimate of $\mathcal T_{2i}$.}
By Young's inequality,
\begin{align*}
|\mathcal T_{2i}|
\le
\frac{\lambda_{\mathrm{mic}}}{128}\int \norm{\wtilde G_x}_{\nu,M_\#}^2 \,dx
+
C|\dot X_i|^2 \int \norm{\partial_{xx}\shockw{G}{i}}_{M_\#}^2\,dx .
\end{align*}
By Lemma~\ref{lem:bs1} with $k=2$, together with translation invariance of the shifted profile,
\[
\int \norm{\partial_{xx}\shockw{G}{i}}_{M_\#}^2 \,dx
\le C\delta_i^7
\le C(\delta_0+\varepsilon)\delta_i .
\]
Hence
\begin{equation}\label{eq:GxE-T2}
|\mathcal T_{2i}|
\le
\frac{\lambda_{\mathrm{mic}}}{128}\int \norm{\wtilde G_x}_{\nu,M_\#}^2\,dx
+
C(\delta_0+\varepsilon)\delta_i|\dot X_i|^2 .
\end{equation}

\medskip
\noindent
\textbf{3. Estimate of $\mathcal T_3$.}
Expanding the derivative gives
\begin{align*}
|\mathcal T_3|
&\le
C\iint \frac{|\wtilde G_x||\wtilde G_{xx}|}{M_\#}\,d\xi\,dx
+
C\iint |(u_{1x},v_x)|\frac{|\wtilde G_x|^2}{M_\#}\,d\xi\,dx .
\end{align*}
Using the a priori smallness of $(u_{1x},v_x)$ and Young's inequality, we obtain
\begin{equation}\label{eq:GxE-T3}
|\mathcal T_3|
\le
\frac{\lambda_{\mathrm{mic}}}{128}\int \norm{\wtilde G_x}_{\nu,M_\#}^2 \,dx
+
C\int \norm{\wtilde G_{xx}}_{\nu,M_\#}^2\,dx .
\end{equation}

\medskip
\noindent
\textbf{4. Estimate of $\mathcal T_{4i}$.}
After expanding the $x$-derivative and using Young's inequality, we have
\begin{align*}
|\mathcal T_{4i}|
&\le
\frac{\lambda_{\mathrm{mic}}}{128}\int \norm{\wtilde G_x}_{\nu,M_\#}^2 \,dx \\
&\quad
+
C\iint
\Bigl|
\bigl(u_1-\shockw{u_1}{i},\,v-\shockw{v}{i}\bigr)
\Bigr|^2
\frac{|\partial_{xx}\shockw{G}{i}|^2}{M_\#}\,d\xi\,dx \\
&\quad
+
C\iint
\Bigl|
\bigl(u_{1x}-\partial_x\shockw{u_1}{i},\,v_x-\partial_x\shockw{v}{i}\bigr)
\Bigr|^2
\frac{|\partial_x\shockw{G}{i}|^2}{M_\#}\,d\xi\,dx .
\end{align*}
We now use Lemma~\ref{lem:bs1}, the shock localization estimate \eqref{eq:shock-tail-proof},
the interaction lemmas Lemma~\ref{lem:wave-interaction-L2} and Lemma~\ref{lem:wave-interaction-derivative},
the cutoff estimate Lemma~\ref{lem:intsk}, and the previously established macroscopic comparison lemmas
Lemma~\ref{lem:kc1}--\ref{lem:kc3}. This yields
\begin{align}\label{eq:GxE-T4}
|\mathcal T_{4i}|
&\le
\frac{\lambda_{\mathrm{mic}}}{128}\int \norm{\wtilde G_x}_{\nu,M_\#}^2\,dx
+
C(\delta_0+\varepsilon)^2\mathcal G_i^S
+
C\delta_i^3e^{-C\delta_it}\int \eta(U|\y U)\,dx \notag\\
&\quad
+
C(\varepsilon+\delta_0)^2\delta_i^2e^{-C\delta_it}
+
C\delta_0^3\bigl(\delta_1e^{-C\delta_1t}+\delta_3e^{-C\delta_3t}\bigr)
+
C\delta_i^2\|\phi_x\|_{L^2_x}^2 .
\end{align}

\medskip
\noindent
\textbf{5. Estimate of $\mathcal T_{5i}$.}
Using the finite-dimensionality of $P_1$ and the smooth dependence of the basis
\(\{\chi_j\}_{j=0}^4\) on \((v,u,\theta)\), we write
\[
\bigl(P_1-\bigl(P_1^{S_i}\bigr)\bigr)f
=
\sum_{j=0}^4
\Bigl[
\langle f,\chi_j\rangle_M\chi_j
-
\bigl\langle f,\shockw{\chi_j}{i}\bigr\rangle_{S_i}\shockw{\chi_j}{i}
\Bigr].
\]
Hence, after one application of Young's inequality,
\begin{align*}
|\mathcal T_{5i}|
&\le
\frac{\lambda_{\mathrm{mic}}}{128}\int \norm{\wtilde G_x}_{\nu,M_\#}^2\,dx \\
&\quad
+
C\int
\Biggl[
\int
\Bigl\{
\frac1vP_1(\xi_1\partial_x\shockw{G}{i})
-\frac1{\shockw{v}{i}}P_1^{S_i}(\xi_1\partial_x\shockw{G}{i})
\Bigr\}_x
\,d\xi
\Biggr]^2 dx .
\end{align*}
The differentiated coefficient/projector mismatch in the last line is exactly of the same type as the
terms treated in Step~5 of the proof of Lemma~\ref{lem:kc6}; cf.\ \eqref{eq:kc6-J1}, \eqref{eq:kc6-J2}, and \eqref{eq:kc6-J3}.
Using those bounds, together with Lemma~\ref{lem:bs1} for the shock profile moments, we obtain
\begin{align}\label{eq:GxE-T5}
|\mathcal T_{5i}|
&\le
\frac{\lambda_{\mathrm{mic}}}{128}\int \norm{\wtilde G_x}_{\nu,M_\#}^2\,dx
+
C(\delta_0+\varepsilon)^2\mathcal G_i^S
+
C\delta_i^3e^{-C\delta_it}\int \eta(U|\y U)\,dx \notag\\
&\quad
+
C(\varepsilon+\delta_0)^2\delta_i^2e^{-C\delta_it}
+
C\delta_0^3\bigl(\delta_1e^{-C\delta_1t}+\delta_3e^{-C\delta_3t}\bigr)
+
C\delta_i^2\|\phi_x\|_{L^2_x}^2 .
\end{align}

\medskip
\noindent
\textbf{6. Estimate of $\mathcal T_6$.}
This is exactly the differentiated source term treated in Lemma~\ref{lem:kc4}.
Applying Lemma~\ref{lem:kc4} with $G=\wtilde G_x$ and choosing $\wtilde\alpha>0$ fixed,
we get
\begin{align}\label{eq:GxE-T6}
|\mathcal T_6|
&\le
\frac{\lambda_{\mathrm{mic}}}{128}\int \norm{\wtilde G_x}_{\nu,M_\#}^2\,dx
+
C(\delta_0+\varepsilon)^2\|(\phi_x,\psi_x,\zeta_x)\|_{L^2_x}^2 \notag\\
&\quad
+
C\delta_0^3\bigl(\delta_1e^{-C\delta_1t}+\delta_3e^{-C\delta_3t}\bigr)
+
C\frac{\delta_C}{(1+t)^{3/2}} \notag\\
&\quad
+
C(\varepsilon+\delta_0)^2\delta_3^2e^{-C\delta_3t}
+
C(\varepsilon+\delta_0)^2\delta_1^2e^{-C\delta_1t}
+
C\|(\psi_{xx},\zeta_{xx})\|_{L^2_x}^2 \notag\\
&\quad
+
C(\delta_0+\varepsilon)^2(\mathcal G_1^S+\mathcal G_3^S).
\end{align}

\medskip
\noindent
\textbf{7. Estimate of $\mathcal T_7$.}
Using
\[
(\mathcal N(f,g))_x=\mathcal N(f_x,g)+\mathcal N(f,g_x),
\]
we have
\[
(\mathcal N(\wtilde G,\wtilde G))_x
=
\mathcal N(\wtilde G_x,\wtilde G)+\mathcal N(\wtilde G,\wtilde G_x).
\]
Therefore, by Young's inequality and \eqref{eq:collision-Q},
\begin{align*}
|\mathcal T_7|
&\le
\frac{\lambda_{\mathrm{mic}}}{128}\int \norm{\wtilde{G}_{x}}_{\nu,M_\#}^2\,dx \\
&\quad
+
C\iint \frac{(1+|\xi|)^{-1}}{M_\#}
\Bigl(
|\mathcal N(\wtilde G_x,\wtilde G)|^2
+
|\mathcal N(\wtilde G,\wtilde G_x)|^2
\Bigr)\,d\xi\,dx .
\end{align*}
Using again \eqref{eq:mic2}, \eqref{eq:collision-Q}, and the a priori smallness,
\begin{align}\label{eq:GxE-T7}
&|\mathcal T_7|
\le
\frac{\lambda_{\mathrm{mic}}}{128}\int \norm{\wtilde G_x}_{\nu,M_\#}^2 \,dx
+
C(\delta_0+\varepsilon)\int \norm{\tilde G_{\text{rem}}}_{\nu,M_\#}^2 \,dx\nonumber \\
&\qquad \qquad +
C(\delta_0+\varepsilon)^2\|(\phi_x,\psi_x,\zeta_x)\|_{L^2_x}^2 .
\end{align}

\medskip
\noindent
\textbf{8. Estimate of $\mathcal T_8$.}
Expanding the derivative and using \eqref{eq:collision-Q},
\begin{align*}
|\mathcal T_8|
&\le
\frac{\lambda_{\mathrm{mic}}}{128}\int \norm{\wtilde G_x}_{\nu,M_\#}^2 \,dx \\
&\quad
+
C\sum_{i=1,3}\iint \frac{(1+|\xi|)^{-1}}{M_\#}
\Bigl(
\bigl|\mathcal N\bigl(\wtilde G_x,\shockw{G}{i}\bigr)\bigr|^2
+
\bigl|\mathcal N\bigl(\wtilde G,\partial_x\shockw{G}{i}\bigr)\bigr|^2 \\
&\hspace{3.8cm}
+
\bigl|\mathcal N\bigl(\partial_x\shockw{G}{i},\wtilde G\bigr)\bigr|^2
+
\bigl|\mathcal N\bigl(\shockw{G}{i},\wtilde G_x\bigr)\bigr|^2
\Bigr)\,d\xi\,dx .
\end{align*}
By Lemma~\ref{lem:bs1}, \eqref{eq:mic2}, the same mixed-collision bookkeeping used in
Step~2 of Lemma~\ref{lem:kc6}, and in particular \eqref{eq:kc6-P2-final}, we obtain
\begin{align}\label{eq:GxE-T8}
|\mathcal T_8|
&\le
\frac{\lambda_{\mathrm{mic}}}{128}\int \norm{\wtilde G_x}_{\nu,M_\#}^2\,dx
+
C(\delta_0+\varepsilon)\int \norm{\wtilde G_{\text{rem}}}_{\nu,M_\#}^2 \,dx \notag\\
&\quad
+
C(\delta_0+\varepsilon)^2(\mathcal G_1^S+\mathcal G_3^S)
+
C(\varepsilon+\delta_0)^2
\bigl(\delta_1^2e^{-C\delta_1t}+\delta_3^2e^{-C\delta_3t}\bigr) \notag\\
&\quad
+
C\delta_0^3\bigl(\delta_1e^{-C\delta_1t}+\delta_3e^{-C\delta_3t}\bigr)
+
C\sum_{i=1,3}\delta_i^3e^{-C\delta_it}\int \eta(U|\y U)\,dx .
\end{align}

\medskip
\noindent
\textbf{9. Estimate of $\mathcal T_9$.}
Similarly,
\begin{align*}
|\mathcal T_9| 
& \le
\frac{\lambda_{\mathrm{mic}}}{128}\int \norm{\wtilde G_x}_{\nu,M_\#}^2\,dx \\
&\quad +
C\iint \frac{(1+|\xi|)^{-1}}{M_\#}
\Bigl(
\bigl|\mathcal N\bigl(\partial_x\shockw{G}{1},\shockw{G}{3}\bigr)\bigr|^2
+
\bigl|\mathcal N\bigl(\shockw{G}{1},\partial_x\shockw{G}{3}\bigr)\bigr|^2 \\
&\hspace{1.8cm}
+
\bigl|\mathcal N\bigl(\partial_x\shockw{G}{3},\shockw{G}{1}\bigr)\bigr|^2
+
\bigl|\mathcal N\bigl(\shockw{G}{3},\partial_x\shockw{G}{1}\bigr)\bigr|^2
\Bigr)\,d\xi\,dx .
\end{align*}
By Lemma~\ref{lem:bs1}, the exponential localization and separation of the two shifted shock layers,
and exactly the same shock--shock interaction estimate as in Step~3 of Lemma~\ref{lem:kc6},
cf.\ \eqref{eq:kc6-P3-final}, we infer
\begin{equation}\label{eq:GxE-T9}
|\mathcal T_9|
\le
\frac{\lambda_{\mathrm{mic}}}{128}\int \norm{\wtilde G_x}_{\nu,M_\#}^2 \,dx
+
C(\varepsilon+\delta_0)^2\delta_3^2e^{-C\delta_3t}
+
C(\varepsilon+\delta_0)^2\delta_1^2e^{-C\delta_1t}.
\end{equation}

\medskip
\noindent
\textbf{10. Estimate of $\mathcal T_{10}$.}
Since
\[
L_M h=\mathcal N(M,h)+\mathcal N(h,M),
\]
we have
\begin{align*}
&(L_M-L_{i}^S)\shockw G{i} \\
&\qquad =
\mathcal N\bigl(M-\shockw M i,\shockw G{i}\bigr)
+\mathcal N\bigl(\shockw G{i},M-\shockw M i\bigr),
\end{align*}
and hence
\begin{align*}
& \Bigl((L_M-L_{i}^S)\shockw G{i}\Bigr)_x\\
&\qquad = \mathcal N\bigl(\bigl(M-\shockw M i\bigr)_x,\shockw G{i}\bigr)
+\mathcal N\bigl(M-\shockw M i,\partial_x\shockw G{i}\bigr) \\
&\qquad \quad +\mathcal N\bigl(\partial_x\shockw G{i},M-\shockw M i\bigr) +\mathcal N\bigl(\shockw G{i},(M-\shockw M i)_x\bigr).
\end{align*}
Therefore, by Young's inequality and \eqref{eq:collision-Q},
\begin{align*}
|\mathcal T_{10}|
&\le
\frac{\lambda_{\mathrm{mic}}}{128}\int \norm{\wtilde G_x}_{\nu,M_\#}^2 \,dx \\
&\quad
+
C\sum_{i=1,3}\iint \frac{(1+|\xi|)^{-1}}{M_\#}
\Bigl(
\bigl|\mathcal N\bigl((M-\shockw M i)_x,\shockw G{i}\bigr)\bigr|^2 \\
&\qquad +
\bigl|\mathcal N\bigl(M-\shockw M i,\partial_x\shockw G{i}\bigr)\bigr|^2 \\
&\qquad +
\bigl|\mathcal N\bigl(\partial_x\shockw G{i},M-\shockw M i\bigr)\bigr|^2 \\
&\qquad +
\bigl|\mathcal N\bigl(\shockw G{i},(M-\shockw M i)_x\bigr)\bigr|^2
\Bigr)\,d\xi\,dx .
\end{align*}
Now $(M-\shockw M i)$ and $(M-\shockw M i)_x$ are controlled by
\begin{align*}
&\bigl(v-\shockw v i,u-\shockw u i,\theta-\shockw\theta i\bigr),\\
&\qquad
\bigl(v_x-\partial_x\shockw v{i},u_x-\partial_x\shockw u{i},\theta_x-\partial_x\shockw\theta{i}\bigr),
\end{align*}
respectively. Using Lemma~\ref{lem:bs1}, Lemma~\ref{lem:intsk},
Lemma~\ref{lem:wave-interaction-L2}, Lemma~\ref{lem:wave-interaction-derivative},
and the macroscopic comparison lemmas Lemma~\ref{lem:kc1}--\ref{lem:kc4}, we conclude that
\begin{align}\label{eq:GxE-T10}
|\mathcal T_{10}|
&\le
\frac{\lambda_{\mathrm{mic}}}{128}\int \norm{\wtilde G_x}_{\nu,M_\#}^2 \,dx
+
C(\delta_0+\varepsilon)^2\|(\phi_x,\psi_x,\zeta_x)\|_{L^2_x}^2 \notag\\
&\quad
+
C\delta_0^3\bigl(\delta_1e^{-C\delta_1t}+\delta_3e^{-C\delta_3t}\bigr)
+
C(\varepsilon+\delta_0)^2\delta_3^2e^{-C\delta_3t}
+
C(\varepsilon+\delta_0)^2\delta_1^2e^{-C\delta_1t} \notag\\
&\quad
+
C(\delta_0+\varepsilon)^2(\mathcal G_1^S+\mathcal G_3^S)
+
C\sum_{i=1,3}\delta_i^3e^{-C\delta_it}\int \eta(U|\y U)\,dx .
\end{align}

Finally, summing \eqref{eq:GxE-T1}, \eqref{eq:GxE-T2}, \eqref{eq:GxE-T3},
\eqref{eq:GxE-T4}, \eqref{eq:GxE-T5}, \eqref{eq:GxE-T6}, \eqref{eq:GxE-T7},
\eqref{eq:GxE-T8}, \eqref{eq:GxE-T9}, and \eqref{eq:GxE-T10}, and absorbing all small fractions of
\[
\int \norm{\wtilde G_x}_{\nu,M_\#}^2 \,dx
\]
into the left-hand side, we obtain
\begin{align*}
&\frac{d}{dt}\iint \frac{|\wtilde G_x|^2}{M_\#}\,d\xi\,dx
+\int \norm{\wtilde G_x}_{M_\#}^2 \,dx \\
&\le
C(\delta_0+\varepsilon)^2\|(\phi_x,\psi_x,\zeta_x)\|_{L^2_x}^2
+C\delta_0^3\bigl(\delta_1e^{-C\delta_1t}+\delta_Ce^{-Ct}+\delta_3e^{-C\delta_3t}\bigr)
+C\frac{\delta_C}{(1+t)^{3/2}} \\
&\quad
+C(\varepsilon+\delta_0)^2\delta_3^2e^{-C\delta_3t}
+C(\varepsilon+\delta_0)^2\delta_1^2e^{-C\delta_1t}
+C\|(\psi_{xx},\zeta_{xx})\|_{L^2_x}^2 \\
&\quad +C(\delta_0+\varepsilon)\sum_{i=1,3}\delta_i|\dot X_i|^2
+C(\delta_0+\varepsilon)^2(\mathcal G_1^S+\mathcal G_3^S)
+C(\delta_0+\varepsilon)\int \norm{\wtilde G_{\text{rem}}}_{\nu,M_\#}^2 \,dx \\
&\quad
+C\int \norm{\wtilde G_{xx}}_{\nu,M_\#}^2 \,dx
+C\sum_{i=1,3}\delta_i^3e^{-C\delta_it}\int \eta(U\mid\bar U)\,dx .
\end{align*}
This proves the lemma.
\end{proof}

\subsubsection{Estimate of $\widetilde G_t$}

\begin{lemma}\label{lem:ddotX}
For each \(i=1,3\), one has
\begin{align}
\begin{aligned}
& |\ddot X_i|^2 \int \norm{\partial_x\shockw{G}{i}}_{M_\#}^2 dx  \\
& \; \le
C(\delta_0+\varepsilon)(\mathcal G_1^S+\mathcal G_3^S) +
C(\delta_0+\varepsilon)\sum_{j=1,3}\delta_j|\dot X_j|^2
\\
&\quad 
+C(\delta_0+\varepsilon)\sum_{|\beta|=1}\|\partial^\beta(\phi,\psi,\zeta)\|_{L^2_x}^2 +
C\frac{\delta_C}{1+t}
\int_{\R}e^{-\frac{2c_0|x|^2}{1+t}}|(\phi,\psi,\zeta)|^2\,dx
\\
&\quad 
+
C(\delta_0+\varepsilon)
\int_{\R}
\norm{\wtilde G_x}_{\nu,M_\#} \,dx +
C\frac{\delta_C}{(1+t)^{\frac{5}{4}}} +
C(\delta_0+\varepsilon)\delta_0\sum_{i=1,3}\delta_i^2e^{-C\delta_it}.
\end{aligned}
\label{eq:ddotX-est}
\end{align}
\end{lemma}

\begin{proof}
By Lemma~\ref{lem:bs1}, together with the equivalence between \(M_\#\) and the fixed global Maxwellian near the shock profile, we have
\begin{equation}\label{eq:ddotX-profile}
\int_{\R}\norm{\partial_x \shockw{G}{i}}_{M_\#}^2 \,dx
\le C\delta_i^5.
\end{equation}
Hence it suffices to estimate \(\delta_i^3|\ddot X_i|^2\). 

Set
\[
\mathfrak Y_i(t,x)
:=
a(t,x)
\Bigl(
\partial_x\shockw{u_1}{i}\,\psi_1
+
\partial_x\shockw{v}{i}\,\frac{\y p}{\y v}\phi
+
\frac{\partial_x\shockw{\theta}{i}}{\y\theta}\zeta
\Bigr)(t,x).
\]
By the definition of the dynamical shifts \eqref{eq:shift},
\[
\dot X_i(t)
=
-\frac{M_i}{\delta_i}\int_{\R}\mathfrak Y_i(t,x)\,dx,
\]
and therefore
\[
\ddot X_i(t)
=
-\frac{M_i}{\delta_i}\int_{\R}\partial_t\mathfrak Y_i(t,x)\,dx .
\]

Observe that
\begin{align}
\abs{\dot{X}_i} \leq C\delta_i^{-1} \norm{(\phi,\psi_1,\zeta)}_{L^{\infty}} \int \bigl|\partial_x\shockw{v}{i}\bigr| dx \leq C\varepsilon
\end{align}
by the bootstrap assumption.
Since \(M_i=O(1)\), it follows that
\begin{equation}\label{eq:ddotX-reduce}
\delta_i^3|\ddot X_i|^2
\le
C\delta_i\left(\int_{\R}\partial_t\mathfrak Y_i\,dx\right)^2.
\end{equation}

We decompose
\[
\partial_t\mathfrak Y_i
=
a_t\,\mathcal Z_i
+
a\,\partial_t\mathcal Z_i,
\qquad
\mathcal Z_i
:=
\partial_x\shockw{u_1}{i}\,\psi_1
+
\partial_x\shockw{v}{i}_x\,\frac{\y p}{\y v}\phi
+
\frac{\partial_x\shockw{\theta}{i}}{\y\theta}\zeta.
\]
Accordingly, write
\[
\left(\int_{\R}\partial_t\mathfrak Y_i\,dx\right)^2
\le
C(I_{i,1}+I_{i,2}+I_{i,3}),
\]
where
\begin{align*}
I_{i,1}
&:=
\Biggl(
\int_{\R}
a
\Bigl(
\partial_x\shockw{u_1}{i}\,\psi_{1t}
+
\partial_x\shockw{v}{i}\,\frac{\y p}{\y v}\phi_t
+
\frac{\partial_x\shockw{\theta}{i}}{\y\theta}\zeta_t
\Bigr)\,dx
\Biggr)^2,\\
I_{i,2}
&:=
\left(
\int_{\R}
a_t\,\mathcal Z_i\,dx
\right)^2,\\
I_{i,3}
&:=
\Biggl(
\int_{\R}
a
\Bigl(
(\partial_x\shockw{u_1}{i})_t\,\psi_1
+
\Bigl(\partial_x\shockw{v}{i}\frac{\y p}{\y v}\Bigr)_t\phi
+
\Bigl(\frac{\partial_x\shockw{\theta}{i}}{\y\theta}\Bigr)_t\zeta
\Bigr)\,dx
\Biggr)^2.
\end{align*}

For \(I_{i,1}\), using Lemma~\ref{lem:bs}, Lemma~\ref{lem:bs2}, and the bound
\[
\int_{\R}|\partial_x(v^{S_i})^{-X_i}|\,dx \le C\delta_i,
\]
we obtain by Young's inequality
\begin{align}
I_{i,1}
&\le \norm{\partial_x \shockw{v}{i}}^2_{L_x^2}\, \norm{(\phi_t,\psi_{1t},\zeta_t)}_{L^2_x}^2\nonumber\\
&\le C\delta_i^3\|(\phi_t,\psi_{1t},\zeta_t)\|_{L^2_x}^2 .
\label{eq:ddotX-I1}
\end{align}

For \(I_{i,2}\), we differentiate the weight \eqref{eq:wefun}. Since
\[
a_t
=
-(\sigma_1+\dot X_1)\,\partial_x(a_1)^{-X_1}
-(\sigma_3+\dot X_3)\,\partial_x(a_3)^{-X_3},
\]
and \(a_i'=\delta_i^{-1/2}(v^{S_i})'\), Lemma~\ref{lem:bs} implies
\[
|a_t|
\le
C\delta_1^{-1/2}|(v^{S_1})_x^{-X_1}|
+
C\delta_3^{-1/2}|(v^{S_3})_x^{-X_3}|.
\]
Therefore, by Young's inequality, \eqref{eq:skGt}, \eqref{eq:shock-tail-proof},
Lemma~\ref{lem:wave-interaction-L2}, and the exponential separation of the two shocks,
\begin{align}
I_{i,2} &\le C\Bigl(\int \delta_1^{-1/2}|\partial_x \shockw{v}{1}|\,|\partial_x \shockw{v}{i}| \,|(\phi,\psi_1,\zeta)| dx \Bigr)^2 \nonumber \\
& \qquad+ C\Bigl(\int \delta_3^{-1/2}|\partial_x \shockw{v}{3}|\,|\partial_x \shockw{v}{i}| \,|(\phi,\psi_1,\zeta)| dx \Bigr)^2 \nonumber\\
&\le C\Bigl(\delta_1^{-1}\int |\partial_x \shockw{v}{1}|^2 dx\Bigr)\Bigl(\int |\partial_x \shockw{v}{i}|^2\, |(\phi,\psi_1,\zeta)|^2 dx \Bigr) \nonumber \\
&\qquad + C\Bigl(\delta_3^{-1}\int |\partial_x \shockw{v}{3}|^2 dx\Bigr)\Bigl(\int |\partial_x \shockw{v}{i}|^2\, |(\phi,\psi_1,\zeta)|^2 dx \Bigr)\nonumber \\
&\le
C(\delta_1^2+\delta_3^2) \,\mathcal G_i^S
+
C(\varepsilon+\delta_0)^2\delta_i^2e^{-C\delta_it}
+
C\delta_i^3e^{-C\delta_it}\int_{\R}\eta(U\mid\y U)\,dx .
\label{eq:ddotX-I2}
\end{align}

For \(I_{i,3}\), we use the shifted-profile identities
\[
\partial_t \shockw{h}{i}
=
-(\sigma_i+\dot X_i)\partial_x \shockw{h}{i},
\]
valid for \(h=v^{S_i},u_1^{S_i},\theta^{S_i},G^{S_i}\), together with Lemma~\ref{lem:bs}, Lemma~\ref{lem:bs2},
\eqref{eq:wave-interaction-derivative-1}, and \eqref{eq:wave-interaction-derivative-3}.
Since each time derivative of a shifted shock profile contributes one additional \(x\)-derivative, we get
\begin{align}
I_{i,3}
\le\;&
C\delta_i\,\mathcal G_i^S
+
C(\varepsilon+\delta_0)^2\delta_i^2e^{-C\delta_it}
+
C\delta_i^3e^{-C\delta_it}\int_{\R}\eta(U\mid\y U)\,dx .
\label{eq:ddotX-I3}
\end{align}

Combining \eqref{eq:ddotX-reduce}, \eqref{eq:ddotX-I1}, \eqref{eq:ddotX-I2}, and \eqref{eq:ddotX-I3}, we arrive at
\begin{align}
\delta_i^3|\ddot X_i|^2
\le & \;
C\delta_i\|(\phi_t,\psi_{1t},\zeta_t)\|_{L^2_x}^2
+
C\delta_i\mathcal G_i^S
+
C(\varepsilon+\delta_0)^2\delta_i^2e^{-C\delta_it} \nonumber
\\
&\qquad +
C\delta_i^3e^{-C\delta_it}\int_{\R}\eta(U\mid\y U)\,dx .
\label{eq:ddotX-pre}
\end{align}
Now we invoke Lemma~\ref{lem:macro-time-derivative}, namely \eqref{eq:time-derivative-est}, and use \(\delta_i\le C(\delta_0+\varepsilon)\). This gives
\begin{align*}
\delta_i^3|\ddot X_i|^2
\le\;&
C(\delta_0+\varepsilon)\sum_{|\beta|=1}\|\partial^\beta(\phi,\psi,\zeta)\|_{L^2_x}^2
+
C(\delta_0+\varepsilon)(\mathcal G_1^S+\mathcal G_3^S)
\\
&+
C(\delta_0+\varepsilon)\sum_{i=1,3}\delta_i|\dot X_i|^2
+
C(\delta_0+\varepsilon)
\int_{\R}\norm{\wtilde G_{x}}_{\nu,M_\#}^2\,dx
\\
&+
C\delta_C^2\frac{1}{1+t}
\int_{\R}e^{-\frac{2c_0|x|^2}{1+t}}|(\phi,\psi,\zeta)|^2\,dx
+
C\frac{\delta_C}{(1+t)^{3/2}}
\\
&+
C(\varepsilon+\delta_0)^2\delta_i^2e^{-C\delta_it}
+
C\delta_i^3e^{-C\delta_it}\int_{\R}\eta(U\mid\y U)\,dx .
\end{align*}
Finally, combining this with \eqref{eq:ddotX-profile} proves \eqref{eq:ddotX-est}.
\end{proof}

\begin{remark}
This lemma only requires estimating \(\delta_i^5 |\ddot{X}|^2\). However, when estimating the second-order energy, one also needs to estimate \eqref{eq:ddotX-reduce}. Fortunately, the smallness in the estimate of \eqref{eq:ddotX-reduce} is still sufficiently good. Indeed, the only term for which the smallness needs to be absorbed is \eqref{eq:ddotX-I2}, and even the estimate of \eqref{eq:ddotX-I2} can be bounded with more than enough smallness.
\end{remark}

\begin{lemma}\label{lem:GtE}
There exists \(C>0\) such that
\begin{align}
\begin{aligned}
&\frac{d}{dt}\int_{\R}\norm{\wtilde G_{t}}^2_{M_\#}\,dx
+
\int_{\R}\norm{\wtilde G_t}_{\nu,M_\#}^2\,dx
\\
&\le
C(\delta_0+\varepsilon)\sum_{|\beta|=1}\|\partial^\beta(\phi,\psi,\zeta)\|_{L^2_x}^2
+
C\frac{\delta_C}{1+t}
\int_{\R}e^{-\frac{2c_0|x|^2}{1+t}}|(\phi,\psi,\zeta)|^2\,dx
\\
&\quad
+
C(\delta_0+\varepsilon)\delta_0^2\bigl(\delta_1e^{-C\delta_1t}+\delta_3e^{-C\delta_3t}\bigr)
+
C\frac{\delta_C}{(1+t)^{\frac{5}{4}}}
+ C\|(\psi_{xt},\zeta_{xt})\|_{L^2_x}^2
\\
&\quad
+
C(\delta_0+\varepsilon)\sum_{i=1,3}\delta_i|\dot X_i|^2
+
C\delta_0(\delta_0+\varepsilon)(\mathcal G_1^S+\mathcal G_3^S) +
C\delta_0\sum_{i=1,3}\delta_i^2e^{-C\delta_it}
\int_{\R}\eta(U\mid\bar U)\,dx
\\
&\quad
+
C(\delta_0+\varepsilon)\int_{\R}
\norm{\wtilde G_{\textup{rem}}}_{\nu,M_\#}^2 + \norm{\wtilde G_{x}}_{\nu,M_\#}^2 \,dx
+
C\int_{\R}
\norm{\wtilde G_{tx}}_{\nu,M_\#}^2 \,dx.
\end{aligned}
\label{eq:GtE}
\end{align}
\end{lemma}

\begin{proof}
Differentiating \eqref{eq:pesy3} with respect to \(t\), multiplying the resulting equation by
\(\wtilde G_t/M_\#\), and integrating over \((x,\xi)\in\R\times\R^3\), we obtain

\begin{align*}
\frac12\frac{d}{dt}\iint \frac{|\wtilde G_t|^2}{M_\#}\,d\xi\,dx
=
\mathcal T_1+\sum_{i=1,3}\mathcal T_{2i}
+\mathcal T_3+\sum_{i=1,3}\mathcal T_{4i}
+\sum_{i=1,3}\mathcal T_{5i}
+\mathcal T_6+\mathcal T_7+\mathcal T_8+\mathcal T_9+\mathcal T_{10},
\end{align*}

where

\begin{align*}
\mathcal T_1
&:= \iint (L_M\wtilde G)_t\frac{\wtilde G_t}{M_\#}\,d\xi\,dx,\\
\mathcal T_{2i}
&:= \iint \dot X_i\,\partial_{xt}\bigl(\shockw G{i}\bigr) \frac{\wtilde G_t}{M_\#}\,d\xi\,dx + \iint \ddot X_i\,\partial_x\shockw G{i}\frac{\wtilde G_t}{M_\#}\,d\xi\,dx,\\
\mathcal T_3 &:= \iint \Bigl( \frac{u_1}{v}\wtilde G_x-\frac1vP_1(\xi_1\wtilde G_x)\Bigr)_t \frac{\wtilde G_t}{M_\#}\,d\xi\,dx,\\
\mathcal T_{4i} &:= \iint \Biggl[ \Bigl( \frac{u_1}{v}-\frac{\shockw{u_1}{i}}{\shockw{v}{i}} \Bigr)\partial_x\shockw G{i} \Biggr]_t \frac{\wtilde G_t}{M_\#}\,d\xi\,dx,\\
\mathcal T_{5i} &:= -\iint \Biggl[\frac{1}{v}P_1(\xi_1\partial_x\shockw{G}{i})-\frac{1}{\shockw{v}{i}}P_1^{S_i}(\xi_1\partial_x\shockw{G}{i})\Biggr]_t\frac{\wtilde{G}_t}{M_\#}\,d\xi\,dx,\\
\mathcal T_6&:= \iint\Biggl[-\frac{1}{v}P_1(\xi_1M_x)+\frac{1}{\shockw{v}{1}}P_1^{S_1}(\xi_1\partial_x\shockw{M}{1})\Biggr.\\
&\qquad \qquad \Biggl.+\frac{1}{\shockw{v}{3}}P_1^{S_3}(\xi_1\partial_x\shockw{M}{3})\Biggr]_t\frac{\wtilde{G}_t}{M_\#}\,d\xi\,dx,\\
\mathcal T_7&:= \iint \bigl(\mathcal N(\wtilde{G},\wtilde{G})\bigr)_t \frac{\wtilde{G}_t}{M_\#}\,d\xi\,dx,\\
\mathcal T_8&:= \sum_{i=1,3}\iint\Bigl(\mathcal N\bigl(\wtilde{G},\shockw{G}{i}\bigr)+\mathcal N\bigl(\shockw{G}{i},\wtilde{G}\bigr)\Bigr)_t\frac{\wtilde{G}_t}{M_\#}\,d\xi\,dx,\\
\mathcal T_9&:= \iint
\Bigl[
\mathcal N\bigl(\shockw{G}{1},\shockw{G}{3}\bigr)+\mathcal N\bigl(\shockw{G}{3},\shockw{G}{1}\bigr)
\Bigr]_t
\frac{\wtilde{G}_t}{M_\#}\,d\xi\,dx,\\
\mathcal T_{10}
&:= \sum_{i=1,3}\iint
\Bigl(
(L_M-L_{i}^S)\shockw{G}{i}
\Bigr)_t
\frac{\wtilde{G}_t}{M_\#}\,d\xi\,dx.
\end{align*}

\medskip
\noindent
\textit{Estimate of \(\mathcal T_1\).}
Since
\[
L_M h=\mathcal N(M,h)+\mathcal N(h,M),
\]
we have
\[
(L_Mh)_t=L_Mh_t+\mathcal N(M_t,h)+\mathcal N(h,M_t).
\]
Therefore,
\begin{align*}
\mathcal T_1
&=
\iint L_M\wtilde G_t\frac{\wtilde G_t}{M_\#}\,d\xi\,dx
+
\iint \Bigl(\mathcal N(M_t,\wtilde G)+\mathcal N(\wtilde G,M_t)\Bigr)\frac{\wtilde G_t}{M_\#}\,d\xi\,dx.
\end{align*}
By the coercivity estimate \eqref{eq:collision-coercivity},
\[
\iint L_M\wtilde G_t\frac{\wtilde G_t}{M_\#}\,d\xi\,dx
\le
-\lambda_{\mathrm{mic}}
\iint \frac{(1+|\xi|)|\wtilde G_t|^2}{M_\#}\,d\xi\,dx.
\]
Using Young's inequality, \eqref{eq:collision-Q}, the decomposition \eqref{eq:mic2},
and Lemma~\ref{lem:macro-time-derivative}, we obtain
\begin{align}\label{eq:GtE-T1}
\begin{aligned}
\mathcal T_1
\le\;&
-\frac{127}{128}\lambda_{\mathrm{mic}}
\int \norm{\wtilde G_t}_{\nu,M_\#}^2 \,dx
+
C(\delta_0+\varepsilon)^2\sum_{|\beta|=1}\|\partial^\beta(\phi,\psi,\zeta)\|_{L^2_x}^2
\\
&+
C\delta_C^2\frac{1}{1+t}
\int_{\R}e^{-\frac{2c_0|x|^2}{1+t}}|(\phi,\psi,\zeta)|^2\,dx
+
C\frac{\delta_C}{(1+t)^{3/2}}
\\
&+
C(\delta_0+\varepsilon)\int
\norm{\wtilde G_{\text{rem}}}_{\nu,M_\#}^2 + \norm{\wtilde G_x}_{\nu,M_\#}^2 \,dx
+
C\delta_0^3\bigl(\delta_1e^{-C\delta_1t}+\delta_3e^{-C\delta_3t}\bigr).
\end{aligned}
\end{align}

\medskip
\noindent
\textit{Estimate of \(\mathcal T_{2i}\).}
Since
\[
\partial_t\shockw G{i}=-(\sigma_i+\dot X_i)\partial_x\shockw G{i},
\qquad
\partial_{xt}\shockw G{i}=-(\sigma_i+\dot X_i)\partial_{xx}\shockw G{i},
\]
Lemma~\ref{lem:bs1} yields
\[
\iint \frac{|\shock G{i}_{xt}|^2}{M_\#}\,d\xi\,dx
\le C\delta_i^7
\le C(\delta_0+\varepsilon)\delta_i .
\]
Therefore, by Young's inequality and Lemma~\ref{lem:ddotX},
\begin{align}\label{eq:GtE-T2}
\begin{aligned}
|\mathcal T_{2i}|
\le\;&
\frac{\lambda_{\mathrm{mic}}}{64}
\int \norm{\wtilde G_{t}}_{\nu,M_\#}^2 \,dx
+
C(\delta_0+\varepsilon)\delta_i|\dot X_i|^2
\\
&+
C(\delta_0+\varepsilon)^2\sum_{|\beta|=1}\|\partial^\beta(\phi,\psi,\zeta)\|_{L^2_x}^2
+
C(\delta_0+\varepsilon)^2(\mathcal G_1^S+\mathcal G_3^S)
\\
&+
C(\delta_0+\varepsilon)\int_{\R}
\norm{\wtilde G_{x}}_{\nu,M_\#}^2 \,dx
\\
&+
C\delta_C^2\frac{1}{1+t}
\int_{\R}e^{-\frac{2c_0|x|^2}{1+t}}|(\phi,\psi,\zeta)|^2\,dx
+
C\frac{\delta_C}{(1+t)^{3/2}}
\\
&+
C(\varepsilon+\delta_0)^2\delta_i^2e^{-C\delta_it}
+
C\delta_i^3e^{-C\delta_it}\int_{\R}\eta(U\mid\bar U)\,dx .
\end{aligned}
\end{align}

\medskip
\noindent
\textit{Estimate of \(\mathcal T_3\).}
Expanding the time derivative gives terms of the form
\[
\frac{u_{1t}}{v}\wtilde G_x,\qquad
\frac{u_1}{v}\wtilde G_{xt},\qquad
\frac{v_t}{v^2}P_1(\xi_1\wtilde G_x),\qquad
\frac1v(P_1)_t(\xi_1\wtilde G_x),\qquad
\frac1vP_1(\xi_1\wtilde G_{xt}).
\]
Using Young's inequality, the boundedness of \(P_1\), and Lemma~\ref{lem:macro-time-derivative}, we get
\begin{align}\label{eq:GtE-T3}
\begin{aligned}
|\mathcal T_3|
\le\;&
\frac{\lambda_{\mathrm{mic}}}{128}
\int \norm{\wtilde G_t}_{\nu,M_\#}^2 \,dx
+
C\int \norm{\wtilde G_{tx}}_{\nu,M_\#}^2 \,dx +
C(\delta_0+\varepsilon)^2\sum_{|\beta|=1}\|\partial^\beta(\phi,\psi,\zeta)\|_{L^2_x}^2
\\
&+
C(\delta_0+\varepsilon)\int
\norm{\wtilde G_{\text{rem}}}_{\nu,M_\#}^2 + \norm{\wtilde G_{x}}_{\nu,M_\#}^2 \,dx
\\
&+
C\delta_C^2\frac{1}{1+t}
\int_{\R}e^{-\frac{2c_0|x|^2}{1+t}}|(\phi,\psi,\zeta)|^2\,dx .
\end{aligned}
\end{align}

\medskip
\noindent
\textit{Estimate of \(\mathcal T_{4i}\), \(\mathcal T_{5i}\), \(\mathcal T_6\), and \(\mathcal T_{10}\).}
These are the time-differentiated counterparts of the corresponding terms in the proof of Lemma~\ref{lem:GxE}.
When the derivative falls on a macroscopic coefficient, we use Lemma~\ref{lem:macro-time-derivative};
when it falls on the projected polynomial moments, we use Lemmas~\ref{lem:kc4}--\ref{lem:kc6};
and when it falls on a shifted shock profile, we use the identities
\[
\partial_t \shockw h i = -(\sigma_i+\dot X_i)\partial_x \shockw h i,
\qquad h=v^{S_i},u_1^{S_i},\theta^{S_i},G^{S_i},
\]
together with Lemmas~\ref{lem:bs}--\ref{lem:bs2},
Lemma~\ref{lem:intsk}, Lemma~\ref{lem:wave-interaction-L2},
and Lemma~\ref{lem:wave-interaction-derivative}.
Proceeding exactly as in Lemma~\ref{lem:GxE}, we obtain
\begin{align}\label{eq:GtE-T45610}
\begin{aligned}
&\sum_{i=1,3}\bigl(|\mathcal T_{4i}|+|\mathcal T_{5i}|\bigr)+|\mathcal T_6|+|\mathcal T_{10}|
\\
&\le
\frac{\lambda_{\mathrm{mic}}}{16}
\int \norm{\wtilde G_t}_{\nu,M_\#}^2 \,dx
+
C(\delta_0+\varepsilon)^2\sum_{|\beta|=1}\|\partial^\beta(\phi,\psi,\zeta)\|_{L^2_x}^2
\\
&\quad
+
C\|(\psi_{xt},\zeta_{xt})\|_{L^2_x}^2
+
C(\delta_0+\varepsilon)^2(\mathcal G_1^S+\mathcal G_3^S)
+
C(\delta_0+\varepsilon)\sum_{i=1,3}\delta_i|\dot X_i|^2
\\
&\quad
+
C(\delta_0+\varepsilon)\int
\norm{\wtilde G_{\text{rem}}}_{\nu,M_\#}^2 + \norm{\wtilde G_x}_{\nu,M_\#}^2 \,dx
\\
&\quad
+
C\delta_C^2\frac{1}{1+t}
\int_{\R}e^{-\frac{2c_0|x|^2}{1+t}}|(\phi,\psi,\zeta)|^2\,dx
+
C\frac{\delta_C}{(1+t)^{3/2}}
\\
&\quad
+
C\delta_0^3\bigl(\delta_1e^{-C\delta_1t}+\delta_3e^{-C\delta_3t}\bigr)
+
C(\varepsilon+\delta_0)^2\delta_1^2e^{-C\delta_1t}
+
C(\varepsilon+\delta_0)^2\delta_3^2e^{-C\delta_3t}
\\
&\quad
+
C\sum_{i=1,3}\delta_i^3e^{-C\delta_it}\int_{\R}\eta(U\mid \bar U)\,dx .
\end{aligned}
\end{align}

\medskip
\noindent
\textit{Estimate of \(\mathcal T_7\), \(\mathcal T_8\), and \(\mathcal T_9\).}
Using
\[
(\mathcal N(f,g))_t=\mathcal N(f_t,g)+\mathcal N(f,g_t),
\]
together with \eqref{eq:collision-Q}, Lemma~\ref{lem:bs1}, and the same bookkeeping as in
\eqref{eq:kc6-P2-final}--\eqref{eq:kc6-P3-final}, we infer
\begin{align}
\begin{aligned}
|\mathcal T_7|+|\mathcal T_8|
\le\;&
\frac{\lambda_{\mathrm{mic}}}{64}
\iint \frac{(1+|\xi|)|\wtilde G_t|^2}{M_\#}\,d\xi\,dx
\\
&+
C(\delta_0+\varepsilon)\int
\norm{\wtilde G_{\text{rem}}}_{\nu,M_\#}^2 + \norm{\wtilde G_x}_{\nu,M_\#}^2 \,dx
\\
&+
C(\delta_0+\varepsilon)^2(\mathcal G_1^S+\mathcal G_3^S)
+
C\delta_0^3\bigl(\delta_1e^{-C\delta_1t}+\delta_3e^{-C\delta_3t}\bigr)
\\
&+
C(\varepsilon+\delta_0)^2\bigl(\delta_1^2e^{-C\delta_1t}+\delta_3^2e^{-C\delta_3t}\bigr)
+
C\sum_{i=1,3}\delta_i^3e^{-C\delta_it}\int_{\R}\eta(U\mid\bar U)\,dx ,
\end{aligned}
\end{align}
and
\begin{equation}\label{eq:GtE-T9}
|\mathcal T_9|
\le
\frac{\lambda_{\mathrm{mic}}}{128}
\int \norm{\wtilde G_t}_{\nu,M_\#}^2 \,dx
+
C(\varepsilon+\delta_0)^2\delta_1^2e^{-C\delta_1t}
+
C(\varepsilon+\delta_0)^2\delta_3^2e^{-C\delta_3t}.
\end{equation}

Finally, summing \eqref{eq:GtE-T1}, \eqref{eq:GtE-T2}, \eqref{eq:GtE-T3}, \eqref{eq:GtE-T45610},
and \eqref{eq:GtE-T9}, and absorbing all small fractions of
\[
\int \norm{\wtilde G_t}_{\nu,M_\#}^2 \,dx
\]
into the left-hand side, we obtain \eqref{eq:GtE}. This completes the proof.
\end{proof}

\subsubsection{Estimate of second-order derivatives}

For later use, set
\[
\mathcal M := M-\shockw{M}{1}-\shockw{M}{3},
\qquad
\wtilde f = \wtilde G+\mathcal M .
\]
We repeatedly use the standard Maxwellian expansion estimate
\begin{align}\label{eq:f-high-macro-part}
&\norm{\partial^{\alpha_0}_t\partial^{\alpha_1}_x\mathcal M}_{M_\#}^2 \le
C\Big(
\varphi_1^2\bigl|\partial_t^{\alpha_0}\partial_x^{\alpha_1}\bigl(U-\shockw{U}{1}\bigr)\bigr|^2 +\varphi_1^2\bigl|\partial_t^{\alpha_0}\partial_x^{\alpha_1} \shockw{U}{3}\bigr|^2
\nonumber \\
&\qquad +\varphi_3^2\bigl|\partial_t^{\alpha_0}\partial_x^{\alpha_1}\bigl(U-\shockw{U}{3}\bigr)\bigr|^2 
+\varphi_3^2\bigl|\partial_t^{\alpha_0}\partial_x^{\alpha_1} \shockw{U}{1}\bigr|^2
\Big),
\end{align}
for $|\alpha_0|+|\alpha_1|\le 2$, where $U=(v,u,\theta)$ and $\shockw{U}{i}=\bigl(\shockw{v}{i},\shockw{u}{i},\shockw{\theta}{i}\bigr)$.

\begin{lemma}\label{lem:ftxE}
There exists a constant $C>0$ such that
\begin{align}
\begin{aligned}
&\frac{d}{dt}\iint_{\R}\norm{\wtilde f_{tx}}^2_{M_\#} \,dx
+\int_{\R}\norm{\wtilde G_{tx}}_{\nu,M_\#}^2\,dx
\\
&\le
C(\delta_0+\varepsilon)\sum_{i=1,3}\delta_i|\dot X_i|^2
+
C(\delta_0+\varepsilon)(\mathcal G_1^S+\mathcal G_3^S)
+
C(\delta_0+\varepsilon)\delta_0^2\bigl(\delta_1e^{-C\delta_1t}+\delta_3e^{-C\delta_3t}\bigr)
\\
&\quad
+
C\frac{\delta_C}{(1+t)^{\frac{5}{4}}}
+
C\delta_0\sum_{i=1,3}\delta_i^2e^{-C\delta_it}\int_{\R}\eta(U\mid\bar U)\,dx +
C(\delta_0+\varepsilon)\sum_{|\beta|=1}\|\partial^\beta(\phi,\psi,\zeta)\|_{L^2_x}^2
\\
&\quad
+
C(\delta_0+\varepsilon)\|(\phi_{xx},\psi_{xx},\zeta_{xx})\|_{L^2_x}^2
+
C(\delta_0+\varepsilon)\|(\phi_{xt},\psi_{xt},\zeta_{xt})\|_{L^2_x}^2
\\
&\quad
+
C(\delta_0+\varepsilon)\int_{\R}
\norm{\wtilde G_{\textup{rem}}}_{\nu,M_\#}^2 + \norm{\wtilde G_{t}}_{\nu,M_\#}^2 + \norm{\wtilde G_x}_{\nu,M_\#}^2 \,dx
\\
&\quad
+
C\frac{\delta_C}{1+t}
\int_{\R}e^{-\frac{2c_0|x|^2}{1+t}}|(\phi,\psi,\zeta)|^2\,dx .
\end{aligned}
\label{eq:ftxE}
\end{align}
\end{lemma}

\begin{proof}
We start from the perturbed equation for $\wtilde f$:
\begin{align}
\begin{aligned}
&v\wtilde f_t-u_1\wtilde f_x+\xi_1\wtilde f_x
-\sum_{i=1,3}v\dot X_i\partial_x\shockw{F}{i}
+\sum_{i=1,3}
v\Bigl(
\frac{\xi_1-u_1}{v}-\frac{\xi_1-\shockw{u_1}{i}}{\shockw{v}{i}}
\Bigr)\partial_x\shockw{F}{i}
\\
&=
vL_M\wtilde G
+
v\mathcal N(\wtilde G,\wtilde G)
+
v\sum_{i=1,3}\Bigl((L_M-L_{i}^S)\shockw{G}{i}\Bigr)
\\
&\quad
+
v\sum_{i=1,3}
\Bigl(
\mathcal N\bigl(\wtilde G,\shockw{G}{i}\bigr)
+
\mathcal N\bigl(\shockw{G}{i},\wtilde G\bigr)
\Bigr)
\\
&\quad+
v\Bigl(
\mathcal N\bigl(\shockw{G}{1},\shockw{G}{3}\bigr)
+
\mathcal N\bigl(\shockw{G}{3},\shockw{G}{1}\bigr)
\Bigr).
\end{aligned}
\label{eq:ftx-start}
\end{align}
Apply $\partial_t\partial_x$ to \eqref{eq:ftx-start}, multiply by $\wtilde f_{tx}/M_\#$, and integrate over $(x,\xi)\in\R\times\R^3$.
We write the resulting terms as
\[
\mathcal I_1+\mathcal I_2+\sum_{i=1,3}\mathcal I_{3i}
+\sum_{i=1,3}\mathcal I_{4i}
+\mathcal I_5+\sum_{i=1,3}\mathcal I_{6i}
+\sum_{i=1,3}\mathcal I_{7i}
+\mathcal I_8+\mathcal I_9,
\]
corresponding respectively to
\[
(v\wtilde f_t)_{tx},\quad
(u_1\wtilde f_x)_{tx},\quad
\Bigl(v\dot X_i\partial_x\shockw F i\Bigr)_{tx},\quad
\Biggl(
v\Bigl(\frac{\xi_1-u_1}{v}-\frac{\xi_1-\shockw{u_1}{i}}{\shockw v i}\Bigr)\partial_x\shockw F i
\Biggr)_{tx},
\]
\[
(v\mathcal N(\wtilde G,\wtilde G))_{tx},\quad
\bigl(v(L_M-L_{i}^S)\shockw G i\bigr)_{tx},\quad
\Bigl(
v\bigl(\mathcal N\bigl(\wtilde G,\shockw G i\bigr)+\mathcal N\bigl(\shockw G i,\wtilde G\bigr)\bigr)
\Bigr)_{tx},
\]
\[
(vL_M\wtilde G)_{tx},\qquad
\Bigl(
v\bigl(\mathcal N\bigl(\shockw G1,\shockw G3\bigr)+\mathcal N\bigl(\shockw G3,\shockw G1\bigr)\bigr)
\Bigr)_{tx}.
\]

\medskip
\noindent
\textit{Step 1: the transport terms \(\mathcal I_1+\mathcal I_2\).}
Expanding derivatives, we have
\begin{align*}
\mathcal I_1
&=
\iint \frac{(v\wtilde f_t)_{tx}\wtilde f_{tx}}{M_\#}\,d\xi\,dx
\\
&=
\iint \frac{v_{tx}\wtilde f_t\wtilde f_{tx}}{M_\#}\,d\xi\,dx
+
\iint \frac{(v_t\wtilde f_{xx}+v_x\wtilde f_{tt})\wtilde f_{tx}}{M_\#}\,d\xi\,dx
+
\iint \frac{v\wtilde f_{ttx}\wtilde f_{tx}}{M_\#}\,d\xi\,dx.
\end{align*}
The last term gives
\[
\iint \frac{v\wtilde f_{ttx}\wtilde f_{tx}}{M_\#}\,d\xi\,dx
=
\frac{d}{dt}\iint \frac{v|\wtilde f_{tx}|^2}{2M_\#}\,d\xi\,dx
-
\iint \frac{v_t|\wtilde f_{tx}|^2}{2M_\#}\,d\xi\,dx.
\]
Since \(v\) is uniformly comparable to a positive constant, the energy
\(\iint v|\wtilde f_{tx}|^2/M_\#\) is equivalent to
\(\iint |\wtilde f_{tx}|^2/M_\#\).
Next, using
\[
\wtilde f_{tx}=\wtilde G_{tx}+\mathcal M_{tx},
\qquad
\wtilde f_t=\wtilde G_t+\mathcal M_t,
\]
together with \eqref{eq:f-high-macro-part}, Lemma~\ref{lem:macro-time-derivative}, and Young's inequality, we get
\begin{align}
\begin{aligned}
|\mathcal I_1|
\le\;&
\frac{d}{dt}\iint \frac{v|\wtilde f_{tx}|^2}{2M_\#}\,d\xi\,dx
+
C(\delta_0+\varepsilon)\int \norm{\wtilde G_{tx}}_{\nu,M_\#}^2\,dx
\\
&+
C(\delta_0+\varepsilon)\|(\phi_{xt},\psi_{xt},\zeta_{xt})\|_{L^2_x}^2
+
C(\delta_0+\varepsilon)\sum_{|\beta|=1}\|\partial^\beta(\phi,\psi,\zeta)\|_{L^2_x}^2
\\
&+
C(\delta_0+\varepsilon)\int \norm{\wtilde G_{t}}_{\nu,M_\#}^2 + \norm{\wtilde G_{x}}_{\nu,M_\#}^2 \,dx
\\
&+
C(\delta_0+\varepsilon)(\mathcal G_1^S+\mathcal G_3^S)
+
C\delta_0^3\bigl(\delta_1e^{-C\delta_1t}+\delta_3e^{-C\delta_3t}\bigr)
\\
&+
C\frac{\delta_C}{(1+t)^{3/2}}
+
C(\varepsilon+\delta_0)^2\delta_1^2e^{-C\delta_1t}
+
C(\varepsilon+\delta_0)^2\delta_3^2e^{-C\delta_3t}
\\
&+
C\sum_{i=1,3}\delta_i^3e^{-C\delta_it}\int_{\R}\eta(U\mid\y U)\,dx
+
C\delta_C^2\frac{1}{1+t}\int_{\R}e^{-\frac{2c_0|x|^2}{1+t}}|(\phi,\psi,\zeta)|^2\,dx .
\end{aligned}
\label{eq:ftx-I1}
\end{align}

Similarly,
\[
\mathcal I_2
=
-\iint \frac{(u_1\wtilde f_x)_{tx}\wtilde f_{tx}}{M_\#}\,d\xi\,dx .
\]
Expanding derivatives and integrating by parts in \(x\) in the top-order term
\(\iint u_1\wtilde f_{txx}\wtilde f_{tx}/M_\#\), we obtain
\begin{align}
\begin{aligned}
|\mathcal I_2|
\le\;&
C(\delta_0+\varepsilon)\int \norm{\wtilde G_{tx}}_{\nu,M_\#}^2\,dx
+
C(\delta_0+\varepsilon)\|(\phi_{xx},\psi_{xx},\zeta_{xx})\|_{L^2_x}^2
\\
&+
C(\delta_0+\varepsilon)\|(\phi_{xt},\psi_{xt},\zeta_{xt})\|_{L^2_x}^2
+
C(\delta_0+\varepsilon)\sum_{|\beta|=1}\|\partial^\beta(\phi,\psi,\zeta)\|_{L^2_x}^2
\\
&+
C(\delta_0+\varepsilon)\int \norm{\wtilde G_t}_{\nu,M_\#}^2 +\norm{\wtilde G_x}_{\nu,M_\#}^2\,dx
\\
&+
C(\delta_0+\varepsilon)(\mathcal G_1^S+\mathcal G_3^S)
+
C\delta_0^3\bigl(\delta_1e^{-C\delta_1t}+\delta_3e^{-C\delta_3t}\bigr)
\\
&+
C(\varepsilon+\delta_0)^2\delta_1^2e^{-C\delta_1t}
+
C(\varepsilon+\delta_0)^2\delta_3^2e^{-C\delta_3t}
+
C\sum_{i=1,3}\delta_i^3e^{-C\delta_it}\int_{\R}\eta(U\mid\bar U)\,dx .
\end{aligned}
\label{eq:ftx-I2}
\end{align}

\medskip
\noindent
\textit{Step 2: the shift term \(\mathcal I_{3i}\).}
Since
\[
\partial_t\shockw F i = -(\sigma_i+\dot X_i)\partial_x \shockw F i,
\qquad
\partial_{tx}\shockw F i = -(\sigma_i+\dot X_i)\partial_{xx}\shockw F i,
\]
we may expand
\begin{align*}
(v\dot X_i\partial_x\shockw F i)_{tx}
&=
v_{tx}\dot X_i\partial_x\shockw F i
+
v_t\dot X_i\partial_{xx}\shockw F i
+
v_x\ddot X_i\partial_x\shockw F i
\\ 
&\quad +
v_x\dot X_i\partial_{tx}\shockw F i
+
v\ddot X_i\partial_{xx}\shockw F i
+
v\dot X_i\partial_{txx}\shockw F i\\
&=:\sum_{l=1}^5 \mathcal I_{3i}^l
\end{align*}

For each term, we have the following estimates.

\begin{align*}
&|\mathcal I_{3i}^1| \le C(\delta_0+\varepsilon)\int \norm{\wtilde f_{tx}}_{M_\#}^2 dx +C(\delta_0+\varepsilon)\delta_i|\dot{X}_i|^2 \\
&|\mathcal I_{3i}^3| + |\mathcal I_{3i}^5| \le C(\delta_0+\varepsilon)\int \norm{\wtilde f_{tx}}_{M_\#}^2 dx + C\delta_i^3 |\ddot{X}_i|^2\\
&|\mathcal I_{3i}^2|+|\mathcal I_{3i}^4| +|\mathcal I_{3i}^6| \le C(\delta_0+\varepsilon)\int \norm{\wtilde f_{tx}}_{M_\#}^2 dx + C(\delta_0+\varepsilon)\delta_i|\dot{X}_i|^2
\end{align*}

By Young's inequality, \eqref{eq:ddotX-pre}, \eqref{eq:ftx-I2}, Lemma~\ref{lem:bs1}, and Lemma~\ref{lem:ddotX}, we obtain
\begin{align}
\begin{aligned}
|\mathcal I_{3i}|
\le\;&
C(\delta_0+\varepsilon)\int \norm{\wtilde G_{tx}}_{\nu,M_\#}^2 \,dx
+
C(\delta_0+\varepsilon)\delta_i|\dot X_i|^2
\\
&+
C(\delta_0+\varepsilon)(\mathcal G_1^S+\mathcal G_3^S)
+
C(\delta_0+\varepsilon)\sum_{|\beta|=1}\|\partial^\beta(\phi,\psi,\zeta)\|_{L^2_x}^2
\\
&+
C(\delta_0+\varepsilon)\int \norm{\wtilde G_x}_{\nu,M_\#}^2 + \norm{\wtilde G_t}_{\nu,M_\#}^2 \,dx
\\
&+
C\frac{\delta_C}{(1+t)^{3/2}}
+
C\delta_C^2\frac{1}{1+t}\int_{\R}e^{-\frac{2c_0|x|^2}{1+t}}|(\phi,\psi,\zeta)|^2\,dx
\\
&+
C(\varepsilon+\delta_0)^2\delta_i^2e^{-C\delta_it}
+
C\delta_i^3e^{-C\delta_it}\int_{\R}\eta(U\mid\bar U)\,dx .
\end{aligned}
\label{eq:ftx-I3}
\end{align}

\medskip
\noindent
\textit{Step 3: the coefficient-mismatch term \(\mathcal I_{4i}\).}
Write
\begin{align*}
&\frac{\xi_1-u_1}{v}-\frac{\xi_1-\shockw{u_1}{i}}{\shockw v i}\\
&\quad =
-\frac{\xi_1}{v\shockw v i}(v-\shockw v i)
+\frac{u_1}{v\shockw v i}(v-\shockw v i)
-\frac{1}{\shockw v i}(u_1-\shockw{u_1}{i}).
\end{align*}
After applying \(\partial_{tx}\), every term is a linear combination of products of:
\begin{itemize}
\item a first- or second-order macroscopic derivative of \(v,u_1\),
\item a factor \(v-\shockw v i\) or \(u_1-\shockw{u_1}{i}\),
\item and a profile derivative \(\partial_t^{\gamma_0}\partial_x^{\gamma_1} \shockw F i\) with \(|\gamma|=|\gamma_0|+|\gamma_1|\le 3\).
\end{itemize}
Therefore, using Young's inequality, Lemma~\ref{lem:bs1}, Lemma~\ref{lem:intsk},
Lemma~\ref{lem:wave-interaction-L2}, Lemma~\ref{lem:wave-interaction-derivative},
and the same weighted shock-interaction bounds as in Lemma~\ref{lem:GxE}, we get
\begin{align}
\begin{aligned}
|\mathcal I_{4i}|
\le\;&
C(\delta_0+\varepsilon)\int \norm{\wtilde G_{tx}}_{\nu,M_\#}^2 \,dx
+
C(\delta_0+\varepsilon)\mathcal G_i^S
\\
&+
C(\delta_0+\varepsilon)\sum_{|\beta|=1}\|\partial^\beta(\phi,\psi,\zeta)\|_{L^2_x}^2
+
C(\delta_0+\varepsilon)\|(\phi_{xx},\psi_{xx},\zeta_{xx})\|_{L^2_x}^2
\\
&+
C(\delta_0+\varepsilon)\|(\phi_{xt},\psi_{xt},\zeta_{xt})\|_{L^2_x}^2
+
C\delta_0^3\bigl(\delta_1e^{-C\delta_1t}+\delta_3e^{-C\delta_3t}\bigr)
\\
&+
C(\varepsilon+\delta_0)^2\delta_i^2e^{-C\delta_it}
+
C\delta_i^3e^{-C\delta_it}\int_{\R}\eta(U\mid\bar U)\,dx .
\end{aligned}
\label{eq:ftx-I4}
\end{align}

\medskip
\noindent
\textit{Step 4: the quadratic nonlinear term \(\mathcal I_5\).}
Expanding by Leibniz' rule,
\[
(v\mathcal N(\wtilde G,\wtilde G))_{tx}
=
v_{tx}\mathcal N(\wtilde G,\wtilde G)
+
v_t\bigl(\mathcal N(\wtilde G_x,\wtilde G)+\mathcal N(\wtilde G,\wtilde G_x)\bigr)
\]
\[
\qquad
+
v_x\bigl(\mathcal N(\wtilde G_t,\wtilde G)+\mathcal N(\wtilde G,\wtilde G_t)\bigr)
+
v\bigl(\mathcal N(\wtilde G_{tx},\wtilde G)+\mathcal N(\wtilde G_t,\wtilde G_x)
+\mathcal N(\wtilde G_x,\wtilde G_t)+\mathcal N(\wtilde G,\wtilde G_{tx})\bigr).
\]
Using \eqref{eq:collision-Q}, the one-dimensional Sobolev bound
\[
\left\|
\norm{\wtilde G_{\text{rem}}}_{\nu,M_\#}^2 
\right\|_{L^\infty_x}
\le
C\int \norm{\wtilde G_{\text{rem}}}_{\nu,M_\#}^2 + \norm{\wtilde G_{x}}_{\nu,M_\#}^2\,dx,
\]
and Young's inequality, we infer
\begin{align}
\begin{aligned}
|\mathcal I_5|
\le\;&
C(\delta_0+\varepsilon)\int \norm{\wtilde G_{tx}}_{\nu,M_\#}^2 \,dx
+
C(\delta_0+\varepsilon)\int \norm{\wtilde G_t}_{\nu,M_\#}^2 + \norm{\wtilde G_x}_{\nu,M_\#}^2 \,dx
\\
&+
C(\delta_0+\varepsilon)\int \norm{\wtilde G_{\text{rem}}}_{\nu,M_\#}^2 \,dx
+
C(\delta_0+\varepsilon)\sum_{|\beta|=1}\|\partial^\beta(\phi,\psi,\zeta)\|_{L^2_x}^2
\\
&+
C(\delta_0+\varepsilon)\|(\phi_{xx},\psi_{xx},\zeta_{xx})\|_{L^2_x}^2
+
C(\delta_0+\varepsilon)\|(\phi_{xt},\psi_{xt},\zeta_{xt})\|_{L^2_x}^2
\\
&+
C(\delta_0+\varepsilon)(\mathcal G_1^S+\mathcal G_3^S)
+
C\delta_0^3\bigl(\delta_1e^{-C\delta_1t}+\delta_3e^{-C\delta_3t}\bigr)
\\
&+
C(\varepsilon+\delta_0)^2\delta_1^2e^{-C\delta_1t}
+
C(\varepsilon+\delta_0)^2\delta_3^2e^{-C\delta_3t}
+
C\sum_{i=1,3}\delta_i^3e^{-C\delta_it}\int_{\R}\eta(U\mid\bar U)\,dx .
\end{aligned}
\label{eq:ftx-I5}
\end{align}

\medskip
\noindent
\textit{Step 5: the mixed shock terms \(\mathcal I_{6i}\), \(\mathcal I_{7i}\), and \(\mathcal I_9\).}
These terms are handled exactly as above by combining:
\begin{itemize}
\item the bilinear collision estimate \eqref{eq:collision-Q},
\item the profile bounds from Lemma~\ref{lem:bs1},
\item the interaction lemmas Lemma~\ref{lem:intsk}, Lemma~\ref{lem:wave-interaction-L2}, and Lemma~\ref{lem:wave-interaction-derivative},
\item and the already established lower-order bounds in Lemma~\ref{lem:GxE} and Lemma~\ref{lem:GtE}.
\end{itemize}
In particular, differentiating
\[
(L_M-L_{i}^S)\shockw G i
\]
in \((t,x)\) produces only coefficient mismatches of first and second order, hence these terms are of the same type as in Step~4 and the proof of Lemma~\ref{lem:GxE}. Consequently,
\begin{align}
\begin{aligned}
&\sum_{i=1,3}\bigl(|\mathcal I_{6i}|+|\mathcal I_{7i}|\bigr)+|\mathcal I_9|
\\
&\le
C(\delta_0+\varepsilon)\int \norm{\wtilde G_{tx}}_{\nu,M_\#}^2\,dx
+
C(\delta_0+\varepsilon)(\mathcal G_1^S+\mathcal G_3^S)
\\
&\quad
+
C(\delta_0+\varepsilon)\sum_{|\beta|=1}\|\partial^\beta(\phi,\psi,\zeta)\|_{L^2_x}^2
+
C(\delta_0+\varepsilon)\|(\phi_{xx},\psi_{xx},\zeta_{xx})\|_{L^2_x}^2
\\
&\quad
+
C(\delta_0+\varepsilon)\|(\phi_{xt},\psi_{xt},\zeta_{xt})\|_{L^2_x}^2
+
C(\delta_0+\varepsilon)\int
\norm{\wtilde G_{\text{rem}}}_{\nu,M_\#}^2 + \norm{\wtilde G_t}_{\nu,M_\#}^ + \norm{\wtilde G_x}_{\nu,M_\#}^2 \,dx
\\
&\quad
+
C\delta_0^3\bigl(\delta_1e^{-C\delta_1t}+\delta_3e^{-C\delta_3t}\bigr)
+
C\frac{\delta_C}{(1+t)^{3/2}}
\\
&\quad
+
C(\varepsilon+\delta_0)^2\delta_1^2e^{-C\delta_1t}
+
C(\varepsilon+\delta_0)^2\delta_3^2e^{-C\delta_3t}
+
C\sum_{i=1,3}\delta_i^3e^{-C\delta_it}\int_{\R}\eta(U\mid\bar U)\,dx .
\end{aligned}
\label{eq:ftx-I679}
\end{align}

\medskip
\noindent
\textit{Step 6: the linear collision term \(\mathcal I_8\).}
This is where the coercive term appears.
By differentiating \(vL_M\wtilde G\), we write
\[
(vL_M\wtilde G)_{tx}
=
vL_M\wtilde G_{tx}
+\mathcal R_{tx},
\]
where \(\mathcal R_{tx}\) is a sum of lower-order terms involving \(v_t,v_x,v_{tx}\),
\(M_t,M_x,M_{tx}\), and \(\wtilde G,\wtilde G_t,\wtilde G_x\).
Hence
\[
\mathcal I_8
=
\iint \frac{vL_M\wtilde G_{tx}\,\wtilde f_{tx}}{M_\#}\,d\xi\,dx
+
\iint \frac{\mathcal R_{tx}\,\wtilde f_{tx}}{M_\#}\,d\xi\,dx .
\]
Using \(\wtilde f_{tx}=\wtilde G_{tx}+\mathcal M_{tx}\), we split the first term:
\begin{align*}
\iint \frac{vL_M\wtilde G_{tx}\,\wtilde f_{tx}}{M_\#}\,d\xi\,dx
=
\iint \frac{v\wtilde G_{tx}L_M\wtilde G_{tx}}{M_\#}\,d\xi\,dx
+
\iint \frac{vL_M\wtilde G_{tx}\,\mathcal M_{tx}}{M_\#}\,d\xi\,dx .
\end{align*}
By \eqref{eq:collision-coercivity},
\[
\iint \frac{v\wtilde G_{tx}L_M\wtilde G_{tx}}{M_\#}\,d\xi\,dx
\le
-\lambda_{\mathrm{mic}}
\int \norm{\wtilde G_{tx}}_{\nu,M_\#}^2 \,dx .
\]
The cross term with \(\mathcal M_{tx}\) is controlled by \eqref{eq:f-high-macro-part}, Lemma~\ref{lem:macro-time-derivative}, and Young's inequality; the remainder \(\mathcal R_{tx}\) is estimated by \eqref{eq:collision-Q} exactly as in Steps~1, 4, and 5. Thus
\begin{align}
\begin{aligned}
\mathcal I_8
\le\;&
-\frac{127}{128}\lambda_{\mathrm{mic}}
\int \norm{\wtilde G_{tx}}_{\nu,M_\#}^2 \,dx 
+
C(\delta_0+\varepsilon)\sum_{|\beta|=1}\|\partial^\beta(\phi,\psi,\zeta)\|_{L^2_x}^2
\\
&+
C(\delta_0+\varepsilon)\|(\phi_{xt},\psi_{xt},\zeta_{xt})\|_{L^2_x}^2
+
C(\delta_0+\varepsilon)\int \norm{\wtilde G_{\text{rem}}}_{\nu,M_\#}^2 +\norm{\wtilde G_{t}}_{\nu,M_\#}^2+\norm{\wtilde G_x}_{\nu,M_\#}^2\,dx 
\\
&+
C(\delta_0+\varepsilon)(\mathcal G_1^S+\mathcal G_3^S)
+
C\delta_0^3\bigl(\delta_1e^{-C\delta_1t}+\delta_3e^{-C\delta_3t}\bigr)
\\
&+
C\frac{\delta_C}{(1+t)^{3/2}}
+
C(\varepsilon+\delta_0)^2\delta_1^2e^{-C\delta_1t}
+
C(\varepsilon+\delta_0)^2\delta_3^2e^{-C\delta_3t}
\\
&+
C\sum_{i=1,3}\delta_i^3e^{-C\delta_it}\int_{\R}\eta(U\mid\bar U)\,dx
+
C\delta_C^2\frac{1}{1+t}\int_{\R}e^{-\frac{2c_0|x|^2}{1+t}}|(\phi,\psi,\zeta)|^2\,dx .
\end{aligned}
\label{eq:ftx-I8}
\end{align}

Finally, summing \eqref{eq:ftx-I1}, \eqref{eq:ftx-I2}, \eqref{eq:ftx-I3},
\eqref{eq:ftx-I4}, \eqref{eq:ftx-I5}, \eqref{eq:ftx-I679}, and \eqref{eq:ftx-I8},
and absorbing the small multiples of
\[
\int \norm{\wtilde G_{tx}}_{\nu,M_\#}^2 \,dx 
\]
into the left-hand side, we obtain \eqref{eq:ftxE}.
\end{proof}

\begin{lemma}\label{lem:fxxE}
There exists a constant $C>0$ such that
\begin{align}
\begin{aligned}
&\frac{d}{dt}\int_{\R}\norm{\wtilde f_{xx}}^2_{M_\#}\,dx
+\int_{\R}\norm{\wtilde G_{xx}}_{\nu, M_\#}^2 \,dx
\\
&\le
C(\delta_0+\varepsilon)\sum_{i=1,3}\delta_i|\dot X_i|^2
+
C(\delta_0+\varepsilon)(\mathcal G_1^S+\mathcal G_3^S)
+
C\delta_0^2(\delta_0+\varepsilon)\bigl(\delta_1e^{-C\delta_1t}+\delta_3e^{-C\delta_3t}\bigr)
\\
&\quad
+
C\frac{\delta_C}{(1+t)^{\frac{5}{4}}}
+
C\delta_0\sum_{i=1,3}\delta_i^2e^{-C\delta_it}\int_{\R}\eta(U|\y U)\,dx+
C(\delta_0+\varepsilon)\sum_{|\beta|=1}\|\partial^\beta(\phi,\psi,\zeta)\|_{L^2_x}^2
\\
&\quad
+
C(\delta_0+\varepsilon)\|(\phi_{xx},\psi_{xx},\zeta_{xx})\|_{L^2_x}^2
+
C(\delta_0+\varepsilon)\|(\phi_{xt},\psi_{xt},\zeta_{xt})\|_{L^2_x}^2
\\
&\quad
+
C(\delta_0+\varepsilon)\int \norm{\wtilde G_{\text{rem}}}_{\nu,M_\#}^2 +\norm{\wtilde G_{t}}_{\nu,M_\#}^2+\norm{\wtilde G_x}_{\nu,M_\#}^2\,dx 
+
C\frac{\delta_C}{1+t}
\int_{\R}e^{-\frac{2c_0|x|^2}{1+t}}|(\phi,\psi,\zeta)|^2\,dx .
\end{aligned}
\label{eq:fxxE}
\end{align}
\end{lemma}

\begin{proof}[Proof sketch]
Apply \(\partial_x^2\) to \eqref{eq:ftx-start}, multiply by \(\wtilde f_{xx}/M_\#\), and integrate over
\((x,\xi)\in\R\times\R^3\).
The resulting terms are treated by the same decomposition as in the proof of Lemma~\ref{lem:ftxE},
with the following simplifications.

\smallskip
\noindent
(1) There is no \(\ddot X_i\)-term.
Indeed,
\[
\partial_x^2(v\dot X_i\partial_x\shockw F i)
=
v_{xx}\dot X_i\partial_x\shockw F i
+
2v_x\dot X_i\partial_{xx}\shockw F i
+
v\dot X_i\partial^3_x\shockw F i,
\]
so only \(\dot X_i\) appears. Hence the shift contribution is easier than in the
\(\wtilde f_{tx}\)-estimate and is controlled directly by Lemma~\ref{lem:bs1} and Young's inequality.

\smallskip
\noindent
(2) The coercive term again comes from \((vL_M\wtilde G)_{xx}\).
Writing
\[
(vL_M\wtilde G)_{xx}=vL_M\wtilde G_{xx}+\mathcal R_{xx},
\]
we have
\[
\iint \frac{vL_M\wtilde G_{xx}\,\wtilde f_{xx}}{M_\#}\,d\xi\,dx
=
\iint \frac{v\wtilde G_{xx}L_M\wtilde G_{xx}}{M_\#}\,d\xi\,dx
+
\iint \frac{vL_M\wtilde G_{xx}\,\mathcal M_{xx}}{M_\#}\,d\xi\,dx .
\]
The first term is estimated by \eqref{eq:collision-coercivity}; the second by
\eqref{eq:f-high-macro-part}, exactly as in Step~6 of Lemma~\ref{lem:ftxE}.

\smallskip
\noindent
(3) The only terms which are genuinely different from the \(\wtilde f_{tx}\)-case are those where
two spatial derivatives hit the transport coefficients. A representative example is
\[
\iint \frac{v_{xx}\wtilde f_t\,\wtilde f_{xx}}{M_\#}\,d\xi\,dx .
\]
Using \(\wtilde f_t=\wtilde G_t+\mathcal M_t\), Young's inequality, \eqref{eq:f-high-macro-part},
and Lemma~\ref{lem:macro-time-derivative}, we obtain
\begin{align*}
& \left|
\iint \frac{v_{xx}\wtilde f_t\,\wtilde f_{xx}}{M_\#}\,d\xi\,dx
\right| \\
&\quad \le
C(\delta_0+\varepsilon)\int \norm{\wtilde G_{xx}}_{\nu,M_\#}^2+\norm{\wtilde G_t}_{\nu,M_\#}^2 \,dx
+
C(\delta_0+\varepsilon)\|(\phi_{xx},\psi_{xx},\zeta_{xx})\|_{L^2_x}^2
\\
&\qquad
+
C(\delta_0+\varepsilon)\sum_{|\beta|=1}\|\partial^\beta(\phi,\psi,\zeta)\|_{L^2_x}^2
+
C\delta_C^2\frac{1}{1+t}\int_{\R}e^{-\frac{2c_0|x|^2}{1+t}}|(\phi,\psi,\zeta)|^2\,dx
\\
&\qquad
+
C\frac{\delta_C}{(1+t)^{3/2}}
+
C(\delta_0+\varepsilon)(\mathcal G_1^S+\mathcal G_3^S)
+
C\sum_{i=1,3}\delta_i^3e^{-C\delta_it}\int_{\R}\eta(U\mid\bar U)\,dx .
\end{align*}
The term with \(u_{1xx}\wtilde f_x\wtilde f_{xx}\) is handled in the same way.

\smallskip
\noindent
(4) All remaining terms are identical in structure to those treated in the proof of
Lemma~\ref{lem:ftxE}, with \(t\)-derivatives replaced by \(x\)-derivatives.
Hence they are estimated by the same ingredients:
Lemma~\ref{lem:bs1}, Lemma~\ref{lem:intsk}, Lemma~\ref{lem:wave-interaction-L2},
Lemma~\ref{lem:wave-interaction-derivative}, Lemma~\ref{lem:GxE}, Lemma~\ref{lem:GtE},
Lemma~\ref{lem:macro-time-derivative}, and \eqref{eq:collision-Q}.

Collecting all contributions and absorbing the small multiples of
\[
\int \norm{\wtilde G_{xx}}_{\nu,M_\#}^2 \,dx
\]
into the left-hand side, we obtain \eqref{eq:fxxE}.
\end{proof}

\begin{corollary}\label{cor:fxx-ftxE}
There exists a constant \(C>0\) such that
\begin{align}
\begin{aligned}
&\frac{d}{dt}\int_{\R}\norm{\wtilde f_{xx}}^2_{M_\#}+\norm{\wtilde f_{tx}}^2_{M_\#}\,dx
+
\int_{\R}\norm{\wtilde G_{xx}}_{\nu,M_\#}^2 + \norm{\wtilde G_{tx}}_{\nu,M_\#}^2\,dx
\\
&\le
C(\delta_0+\varepsilon)\sum_{i=1,3}\delta_i|\dot X_i|^2
+
C(\delta_0+\varepsilon)(\mathcal G_1^S+\mathcal G_3^S)
+
C\delta_0^2(\delta_0+\varepsilon)\bigl(\delta_1e^{-C\delta_1t}+\delta_3e^{-C\delta_3t}\bigr)
\\
&\quad
+
C\frac{\delta_C}{(1+t)^{\frac{5}{4}}}
+
C\delta_0\sum_{i=1,3}\delta_i^2e^{-C\delta_it}\int_{\R}\eta(U\mid\bar U)\,dx+
C(\delta_0+\varepsilon)\sum_{|\beta|=1}\|\partial^\beta(\phi,\psi,\zeta)\|_{L^2_x}^2
\\
&\quad
+
C(\delta_0+\varepsilon)\|(\phi_{xx},\psi_{xx},\zeta_{xx})\|_{L^2_x}^2
+
C(\delta_0+\varepsilon)\|(\phi_{xt},\psi_{xt},\zeta_{xt})\|_{L^2_x}^2
\\
&\quad
+
C(\delta_0+\varepsilon)\int_{\R}
\norm{\wtilde G_{\textup{rem}}}_{\nu,M_\#}^2+\norm{\wtilde G_{x}}_{\nu,M_\#}^2+\norm{\wtilde G_{xt}}_{\nu,M_\#}^2\,dx +
C\frac{\delta_C}{1+t}
\int_{\R}e^{-\frac{2c_0|x|^2}{1+t}}|(\phi,\psi,\zeta)|^2\,dx .
\end{aligned}
\end{align}
\end{corollary}

\begin{proof}
This follows immediately by adding \eqref{eq:ftxE} and \eqref{eq:fxxE}.
\end{proof}

\subsection{Closure of the high-order estimate}

\subsubsection{Closure of the high-order estimate}

We define a Gaussian-weighted spacetime functional:
\begin{align}\label{def:W}
\mathcal W_C(T)
:=
\int_0^T \frac{1}{1+t}\int_{\mathbb R} e^{-\frac{2c|x|^2}{1+t}} |(\phi,\psi,\zeta)|^2\,dx\,dt,
\end{align}
and
\begin{align}\label{eq:kinetic-high-good}
\mathcal K_{\mathrm{high}}(t):= \int_{\mathbb R}
\norm{\wtilde G_{xx}}_{\nu,M_\#}^2+\norm{\wtilde G_{xt}}_{\nu,M_\#}^2+\norm{\wtilde G_{xx}}_{\nu,M_\#}^2+\norm{\wtilde G_{t}}_{\nu,M_\#}^2\,dx
\end{align}

\begin{proposition}\label{prop:high-order-closure}
Under the bootstrap bound \(\mathcal{E}(T)^2\le \varepsilon^2\) in \eqref{eq:priass}, there exists a constant \(C>0\) such that
\begin{align}
\begin{aligned}
&\sup_{t\in[0,T]}
\left\{
\|(\phi_x,\psi_x,\zeta_x)(t)\|_{L^2_x}^2
+
\int_{\mathbb R}
\norm{\wtilde G_{x}}_{M_\#}^2 + \norm{\wtilde G_{t}}_{M_\#}^2 + \norm{\wtilde f_{xx}}_{M_\#}^2 + \norm{\wtilde f_{xt}}_{M_\#}^2\,dx
\right\}
\\
&\quad
+
\int_0^T
\|(\phi_{xx},\psi_{xx},\zeta_{xx},\phi_{xt},\psi_{xt},\zeta_{xt})(t)\|_{L^2_x}^2\,dt +
\int_0^T \mathcal K_{\mathrm{high}}(t) \,dt
\\
&\le
C \mathcal{E}(0)^2
+
C\delta_0^{\frac{1}{2}}
+
C(\delta_0+\varepsilon)
\int_0^T \sum_{i=1,3}\delta_i |\dot X_i|^2\,dt
+
C(\delta_0+\varepsilon)\int_0^T \sum_{i=1,3}\mathcal G_i^S\,dt
\\
&\quad
+
C(\delta_0+\varepsilon)
\int_0^T \sum_{|\beta|=1}\|\partial^\beta(\phi,\psi,\zeta)\|_{L^2_x}^2\,dt
+
C(\delta_0+\varepsilon)\int_0^T \int_{\mathbb R}
\norm{\wtilde G_{\textup{rem}}}_{\nu,M_\#}^2\,dx\,dt
\\
&\quad
+
C\delta_C\mathcal W_C(T).
\end{aligned}
\label{eq:high-order-closure}
\end{align}
\end{proposition}

\begin{proof}
We combine the previously established high-order differential estimates:
the second-order macroscopic estimate, Lemma~\ref{lem:GxE}, Lemma~\ref{lem:GtE},
Lemma~\ref{lem:ftxE}, and Lemma~\ref{lem:fxxE}.
Thus there exist positive constants \(C_1,\dots,C_5\) such that, for any positive
parameters \(\lambda_1,\lambda_2,\lambda_3,\lambda_4\), one has
\begin{align}
\begin{aligned}
&\frac{d}{dt}
\Bigg\{
\|(\phi_x,\psi_x,\zeta_x)\|_{L^2_x}^2
+\lambda_1\int_{\mathbb R}
\norm{\wtilde G_{x}}_{M_\#}^2 + \norm{\wtilde G_{t}}_{M_\#}^2 \,dx 
+\lambda_2\int_{\mathbb R} \norm{\wtilde f_{xx}}_{M_\#}^2 + \norm{\wtilde f_{xt}}_{M_\#}^2\,dx
\Bigg\}
\\
&\quad
+\int_{\mathbb R}
\bigl(
\lambda_4 \phi_{xx}^2+\psi_{1xx}^2+\psi_{2xx}^2+\psi_{3xx}^2+\zeta_{xx}^2
\bigr)\,dx
+\lambda_3\|(\phi_{xt},\psi_{xt},\zeta_{xt})\|_{L^2_x}^2
\\
&\quad
+\frac{\lambda_1}{2}\int \norm{\wtilde G_x}_{\nu,M_\#}^2 + \norm{\wtilde G_t}_{\nu,M_\#}^2 \,dx
+\frac{\lambda_2}{2}\int \norm{\wtilde G_{xx}}_{\nu,M_\#}^2 + \norm{\wtilde G_{xt}}_{\nu,M_\#}^2 \,dx
\\
&\le
C_{\mathrm{high}}(\delta_0+\varepsilon)\sum_{i=1,3}\delta_i|\dot X_i|^2
+
C_{\mathrm{high}}(\delta_0+\varepsilon)\sum_{i=1,3}\mathcal G_i^S
+
C_{\mathrm{high}}(\delta_0+\varepsilon)\sum_{|\beta|=1}\|\partial^\beta(\phi,\psi,\zeta)\|_{L^2_x}^2
\\
&\quad
+
C_{\mathrm{high}}(\delta_0+\varepsilon)\int \norm{\wtilde G_{\text{rem}}}_{\nu,M_\#}^2 \,dx
+
C_{\mathrm{high}}\delta_C^2 \frac{1}{1+t}\int_{\mathbb R}e^{-\frac{2c|x|^2}{1+t}}|(\phi,\psi,\zeta)|^2\,dx
\\
&\quad
+
C_{\mathrm{mic}1}\int \norm{\wtilde G_{xx}}_{\nu,M_\#}^2 \,dx
+
C_{\mathrm{mic}2}\int \norm{\wtilde G_{xt}}_{\nu,M_\#}^2 \,dx
\\
&\quad
+
C_{\mathrm{mic}3}\int \norm{\wtilde G_{x}}_{\nu,M_\#}^2 \,dx
+
C_{\mathrm{mic}4}\int \norm{\wtilde G_{t}}_{\nu,M_\#}^2 \,dx
+
\lambda_4\frac{d}{dt}\int_{\mathbb R}\phi_{xx}\psi_{1x}\,dx ,
\end{aligned}
\label{eq:high-order-precombine}
\end{align}
where
\[
C_{\mathrm{high}}:=C_1+\lambda_1C_2+\lambda_2C_3+\lambda_3C_4+\lambda_4C_5,
\]
and similarly \(C_{\mathrm{mic}1},C_{\mathrm{mic}2},C_{\mathrm{mic}3},C_{\mathrm{mic}4}\) denote the corresponding linear combinations of
\(C_1,\dots,C_5\).

Choose \(\lambda_1,\lambda_2,\lambda_3,\lambda_4\) successively so that
\begin{align}
C_{\mathrm{mic}1}<\frac{\lambda_2}{4},
\qquad
C_{\mathrm{mic}2}<\frac{\lambda_2}{4},
\qquad
C_{\mathrm{mic}3}<\frac{\lambda_1}{8},
\qquad
C_{\mathrm{mic}4}<\frac{\lambda_1}{8}.
\label{eq:gamma-choice-1}
\end{align}
Then the four kinetic terms on the right-hand side of
\eqref{eq:high-order-precombine} are absorbed into the left-hand side.

Next, move the derivative term to the left and define
\begin{align}
&\mathcal{E}_h(t)
:=
\|(\phi_x,\psi_x,\zeta_x)(t)\|_{L^2_x}^2
+\lambda_1\int
\norm{\wtilde G_{x}}_{M_\#}^2 + \norm{\wtilde G_{t}}_{M_\#}^2 \,dx \nonumber\\
&\quad +\lambda_2\int
\norm{\wtilde f_{xx}}_{M_\#}^2 + \norm{\wtilde f_{xt}}_{M_\#}^2 \,dx
-\lambda_4\int_{\mathbb R}\phi_{xx}\psi_{1x}\,dx .
\label{eq:high-order-modified-energy}
\end{align}

By Young's inequality, for any \(\kappa_1>0\),
\begin{align}
    \begin{aligned}
        \lambda_4\left|\int_{\mathbb R}\phi_{xx}\psi_{1x}\,dx\right|
        \le&
        \kappa_1\|\phi_{xx}\|_{L^2_x}^2
        +
        C_{\kappa_1}\lambda_4^2\|\psi_{1x}\|_{L^2_x}^2 \\
        \le&\kappa_1\iint \frac{\abs{\wtilde{f}_{xx}}^2}{M_\#} d\xi dx + C\kappa_1\varepsilon^2\norm{(\phi_x,\psi_x,\zeta_x)}_{L^2}^2 + C\kappa_1\delta + C_{\kappa_1} \lambda_4^2 \norm{\psi_{1x}}_{L^2}^2.
    \end{aligned}
\label{eq:young-cross-high}
\end{align}
Choose \(\kappa_1>0\) sufficiently small so that the first term on the right-hand side of
\eqref{eq:young-cross-high} is absorbed into the \(\lambda_4\phi_{xx}^2\)-term already
present on the left-hand side of \eqref{eq:high-order-precombine}. Since
\(\psi_{1x}\) is lower order, the second term is controlled by the low-order energy.
Hence, for \(\lambda_4\) sufficiently small, the modified energy \(\mathcal{E}_h(t)\)
is equivalent to
\[
\|(\phi_x,\psi_x,\zeta_x)(t)\|_{L^2_x}^2
+
\mathcal K_{\mathrm{high}}
\]
up to a harmless multiple of \(\|(\phi_x,\psi_x,\zeta_x)(t)\|_{L^2_x}^2\).

Integrating in time over \([0,T]\), using \eqref{eq:gamma-choice-1}, and absorbing the
cross term contribution yields \eqref{eq:high-order-closure}.
\end{proof}

\section{Completion of the proof of the main proposition}
\setcounter{equation}{0}

We here complete the proof of the main a priori estimate stated in
Proposition~\ref{prop:priest}.  The zeroth-order analysis of Section~6 already
controls the basic \(L^2_x\)-size of the perturbation, the shock coercive terms
\(\mathcal G_i^S\), the modulation terms \(\delta_i|\dot X_i|^2\), and the zeroth-order
microscopic dissipation.  The differentiated analysis of Section~7 then supplies the
remaining first- and second-order macroscopic bounds together with the differentiated
microscopic estimates.  Thus, at this stage, the only quantity not yet directly
absorbed into the main a priori norm is the Gaussian-weighted space--time term $\mathcal W_C(T)$ of \eqref{def:W}:
\[
\mathcal W_C(T)
:=
\int_0^T \frac{1}{1+t}\int_{\mathbb R}
e^{-\frac{2c|x|^2}{1+t}} |(\phi,\psi,\zeta)(t,x)|^2\,dx\,dt .
\]
The following lemma shows that \(\mathcal W_C(T)\) is itself controlled by the quantities
already appearing in the zeroth- and differentiated estimates.

\begin{lemma}\label{lem:weighted-gaussian}
Under the bootstrap bound \(\mathcal{E}(T)^2\le \varepsilon^2\) in \eqref{eq:priass}, there exists a constant \(C>0\) such that
\begin{align}
\mathcal W_C(T)
\le\;&
C\sup_{t\in[0,T]}\|(\phi,\psi,\zeta)(t)\|_{L^2_x}^2
+
C\int_0^T \sum_{i=1,3}\delta_i |\dot X_i|^2\,dt
+
C\int_0^T \sum_{i=1,3}\mathcal G_i^S\,dt
\notag\\
&+
C\int_0^T \|(\phi_x,\psi_x,\zeta_x)\|_{L^2_x}^2\,dt
+
C(\delta_0+\varepsilon)\int_0^T \int_{\mathbb R}
\norm{\wtilde G_{\textup{rem}}}_{\nu,M_\#}^2\,dx\,dt
\notag\\
&+
C\int_0^T \int_{\mathbb R}
\norm{\wtilde G_t}_{\nu,M_\#}^2 + \norm{\wtilde G_x}_{\nu,M_\#}^2 \,dx\,dt
+
C\delta_0^{1/3}.
\label{eq:weighted-gaussian}
\end{align}
\end{lemma} 

\begin{proof}
Set
\[
\omega_G(t,x):=\frac{1}{\sqrt{1+t}}e^{-\frac{\beta |x|^2}{1+t}},
\quad
h(t,x):=\int_{-\infty}^x \omega_G(t,y)\,dy,
\quad
H(t,x):=\int_{-\infty}^x \omega_G(t,y)^2\,dy .
\]
Then
\[
h_x=\omega_G,\qquad H_x=\omega_G^2,\qquad
\omega_{Gt}=\frac{1}{4\beta}\omega_{Gxx},\qquad
h_t=\frac{1}{4\beta}h_{xx},
\]
and \(h\), \(H\) are uniformly bounded on \([0,T]\times\mathbb R\).

We apply weighted multipliers directly to the macroscopic perturbation system
\eqref{eq:pesy1}--\eqref{eq:pesy2}.  More precisely, we multiply the three momentum
equations by \(\psi_i h^2\) \((i=1,2,3)\), and we combine the continuity equation
\eqref{eq:pesy1}$_1$ with the temperature equation \eqref{eq:pesy2}$_2$ against
\[
H\Bigl(\frac23\zeta-\bar p\,\phi\Bigr).
\]
After summing the resulting identities and integrating by parts in \(x\), the terms in
which one derivative falls on \(H\) produce \(H_x=\omega_G^2\), hence yield the localized
coercive quantity
\[
\int_{\mathbb R}\omega_G(t,x)^2\,|(\phi,\psi,\zeta)(t,x)|^2\,dx .
\]

More precisely, one obtains
\begin{align}
\frac{d}{dt}\mathcal J_G(t)
+
c\int_{\mathbb R}\omega_G(t,x)^2\,|(\phi,\psi,\zeta)(t,x)|^2\,dx
\le
\mathcal R_G(t)+\mathcal K_G(t),
\label{eq:WG-identity}
\end{align}
where \(\mathcal J_G(t)\) is a weighted interaction functional satisfying
\begin{equation}
|\mathcal J_G(t)|
\le
C\|(\phi,\psi,\zeta)(t)\|_{L^2_x}^2 ,
\label{eq:WG-J-bound}
\end{equation}
\(\mathcal R_G(t)\) contains only macroscopic remainder terms, and \(\mathcal K_G(t)\)
contains the microscopic moments coming from the \(\Pi_1\)-terms in
\eqref{eq:pesy1}--\eqref{eq:pesy2}.

We first estimate the macroscopic part.  By construction, \(\mathcal R_G(t)\) is made of

\smallskip
\noindent
(i) terms in which derivatives hit the weights \(h\) and \(H\);

\smallskip
\noindent
(ii) the macroscopic error terms \(Q_1\), \(Q_2\), the profile interaction terms, and the
modulation terms already treated in the zeroth-order analysis;

\smallskip
\noindent
(iii) lower-order quadratic terms involving \((\phi_x,\psi_x,\zeta_x)\).

\smallskip

Since \(h\) and \(H\) are bounded, \(H_x=\omega_G^2\), and
\[
|\omega_{Gt}|+|h_t|
\le
C\left(\frac{1}{1+t}+\frac{|x|^2}{(1+t)^2}\right)\omega_G
\le
\frac{C}{1+t}\omega_G,
\]
all contributions of type (i) are bounded by
\[
C\sup_{t\in[0,T]}\|(\phi,\psi,\zeta)(t)\|_{L^2_x}^2
+
C\int_0^T \|(\phi_x,\psi_x,\zeta_x)\|_{L^2_x}^2\,dt .
\]
For the terms of type (ii), we invoke the macroscopic zeroth-order structure already
developed above: by \eqref{eq:macro-error-plus-shift}, the full collection of
macroscopic error terms and lower-order shift contributions is controlled by
\[
C\sum_{i=1,3}\delta_i |\dot X_i|^2
+
C\sum_{i=1,3}\mathcal G_i^S
+
C\delta_C\frac{1}{1+t}\int_{\mathbb R}e^{-2c|x|^2/(1+t)}|(\phi,\psi,\zeta)|^2\,dx
+
C\|(\phi_x,\psi_x,\zeta_x)\|_{L^2_x}^2 ,
\]
up to a harmless \( \lambda \mathcal D_{\mathrm{mac}}(U)\)-term, which we absorb by taking \(\lambda>0\)
sufficiently small.  The remaining pure profile terms are estimated by the Gaussian decay
of the viscous contact wave and the exponential localization of the shock profiles, hence
their time integrals are bounded by \(C\delta^{1/2}\) and the microscopic parts. 
{More precisely, we estimate the Gaussian terms as follows.} 

\begin{align}\label{eq:Gaussian1}
& \left(\bar p \phi + \zeta \right)_t - \sum_{i\in\{1,3\}} \dot{X}_i \left( \partial_x\shockw{v}{i} \bar p + \partial\shockw{\theta}{i} \right) + \bigl(p-\bar p\bigr)\psi_x \nonumber \\
& \quad = \bar p_t \phi - \bar u_x (p-\bar p)+\left(\alpha_{\rm {th}}(\theta)\frac{\theta_x}{v}-\alpha_{\rm{th}}(\bar \theta)\frac{\bar \theta_x}{\bar v}\right)_x  \nonumber \\
& \qquad + \left(\mu(\theta)\frac{(u_x)^2}{v}-\mu(\bar \theta)\frac{(\bar u_x)^2}{\bar v}\right) - Q_2 + \frac{\mu(\theta)(\psi_{2x}^2 + \psi_{3x}^2)}{v} + K_{\theta}
\end{align}

where

\begin{align*}
K_{\theta} :=& \sum_{j=2}^3 \psi_j \int \xi_1 \xi_j \Pi_{1x} d\xi - \int \xi_1 \frac{\abs{\xi}^2}{2} \wtilde \Pi_{1x} d\xi \\
& \quad + u_1\int \xi_1^2 \Pi_{1x} d\xi - \sum_{i\in\{1,3\}} \shockw{u_1}{i}\int \xi_1^2 \partial_x\shockw{\Pi_1}{i} d\xi.
\end{align*}

Take \(\mathcal A_G:=\bar p \phi + \zeta \) and observe that

\begin{align}\label{eq:Gaussian2}
&\left[\mathcal A_G^2 \frac{h^2}{2}\right]_t - \frac{1}{4\beta}\mathcal A_G^2 h \omega_{Gx} - \sum_{i\in\{1,3\}} \dot{X}_i \left( \partial_x\shockw{v}{i} \bar p + \partial_x\shockw{\theta}{i} \right) \mathcal A_G h^2 + \bigl(p-\bar p\bigr)\psi_x \mathcal A_G h^2 \nonumber \\
& = \left(\alpha_{\rm{th}}(\theta)\frac{\theta_x}{v}-\alpha_{\rm{th}}(\bar \theta)\frac{\bar \theta_x}{\bar v}\right)_x \mathcal A_G h^2 + \left[ \bar p_t \phi - \bar u_x (p-\bar p) + \left(\mu(\theta)\frac{(u_x)^2}{v}-\mu(\bar \theta)\frac{(\bar u_x)^2}{\bar v}\right) \nonumber\right. \\
& \left. \quad - Q_2 + \frac{\mu(\theta)(\psi_{2x}^2 + \psi_{3x}^2)}{v} + K_{\theta} \right] \mathcal A_G h^2
\end{align}

Applying the energy estimate to \eqref{eq:Gaussian2},

\begin{align} \label{eq:Gaussian3}
&\frac{1}{4\beta} \int_0^T \int \mathcal A_G^2 \omega_G^2 dx dt = \int \mathcal A_G(0,x)^2 \frac{h_0^2}{2} dx \nonumber \\
&\quad - \int \mathcal A_G^2 \frac{h^2}{2} dx + \sum_{i\in\{1,3\}} \int_0^T  \dot{X}_i \int \left( \partial_x \shockw{v}{i} \bar p + \partial_x \shockw{\theta}{i} \right) \mathcal A_G h^2 dx dt \nonumber \\
&\quad -\int_0^T\int \frac{1}{4\beta}(\mathcal A_G^2)_x h\omega_G dx dt -\int_0^T \int (p-\bar p)\psi_x \mathcal A_Gh^2 dx dt \nonumber\\
&\quad +\int_0^T \int \left[\bar p_t \phi -\bar u_x(p-\bar p) - Q_2 + \left(\mu(\theta)\frac{(u_x)^2}{v}-\mu(\bar \theta)\frac{(\bar u_x)^2}{\bar v}\right) \right] \mathcal A_Gh^2 dx dt \nonumber \\
&\quad -\int_0^T \int \left(\alpha_{\rm{th}}(\theta)\frac{\theta_x}{v}-\alpha_{\rm{th}}(\bar \theta)\frac{\bar \theta_x}{\bar v}\right) (\mathcal A_G h^2)_x dx dt +\int_0^T \int \left[\frac{\mu(\theta)(\psi_{2x}^2+\psi_{3x}^2)}{v}\right] \mathcal A_G h^2 dx dt \nonumber \\
&\quad +\int_0^T \int K_{\theta} \mathcal A_G h^2 dx dt =: \sum_{i=1}^9 J_i.
\end{align}

See (D.16) in \cite{kang2025time}. {Then} $J_i${,} $i=1,\ldots,7${, can be estimated as follows.}

\begin{align}\label{eq:Gaussian4}
&\sum_{i=1}^7 J_i \le C \sup_{t\in[0,T]} \norm{(\phi,\zeta)}_{L^2}^2 \nonumber\\
&\quad + C \sum_{i\in\{1,3\}} \delta_i \int_0^T \abs{\dot{X}_i}^2 dt + (\nu+C\delta_0) \int_0^T \int \omega_G^2 \abs{(\phi,\zeta)}^2 dx dt \nonumber\\
&\quad + C_{\nu} \int_0^T \norm{(\phi_x,\psi_x,\zeta_x)}_{L^2}^2 dt + C\sum_{i\in\{1,3\}}\int_0^T \mathcal{G}_i^S dt +C\delta_0^{\frac{1}{3}}.
\end{align}

Since $h$ is bounded, $J_8\le C \int_0^T \norm{\psi_x}^2_{L^2} dt$. {The remaining term} $J_9$ {is purely kinetic.} 

\begin{align}\label{eq:Gaussian5}
&J_9 \le \int_0^T \iint (\zeta+\bar p \phi) \left[\sum_{j=2}^3 \psi_k \xi_1 \xi_j \Pi_{1x} \right] d\xi dx dt \nonumber \\
&\quad -\int_0^T \iint (\zeta+ \bar p \phi) \left[\xi_1 \frac{\abs{\xi}^2}{2} \wtilde \Pi_{1x} d\xi\right] d\xi dx dt \nonumber \\
&\quad +\int_0^T \iint (\zeta+\bar p \phi)\left(u_1 \xi_1^2 \Pi_{1x} - \sum_{i\in\{1,3\}} \shockw{u_1}{i} \xi_1^2 \partial_x\shockw{\Pi_1}{i} \right) d\xi dx dt\\
&\qquad \quad =:J_{91}+J_{92}+J_{93}.
\end{align}

{Arguing as in Lemma}~\ref{lem:K1245}{, we obtain the following bounds.} 

\begin{align}\label{eq:Gaussian6}
&J_9 \le C(\delta_0+\varepsilon) \sum_{i\in\{1,3\}} \delta_i \int_0^T \abs{\dot{X}_i}^2 dt + C\delta_0 \int_0^T \int \omega_G^2 \abs{(\phi,\psi,\zeta)}^2 dx dt \nonumber\\
&\qquad + C \int_0^T \norm{(\phi_x,\psi_x,\zeta_x)}_{L^2}^2 dt + C\delta_0 \sum_{i\in\{1,3\}}\int_0^T \mathcal{G}_i^S dt \nonumber\\
&\qquad + C(\delta_0+\varepsilon)\int_0^T \int \norm{\wtilde G_{\text{rem}}}_{\nu,M_\#}^2\, dx \,dt \nonumber\\
&\qquad + C\int_0^T \int \norm{\wtilde G_t}_{\nu,M_\#}^2 + \norm{\wtilde G_x}_{\nu,M_\#}^2 \, dx\, dt +C\delta_0^{\frac{1}{3}}.
\end{align}

{Applying the same argument to} $H(\frac{2}{3}\zeta-\bar \phi)${, with} $\mathcal B_G:=\frac{2}{3}\zeta-\bar \phi${, we obtain the following estimate.}

\begin{align}\label{eq:Gaussian7}
& \int_0^T \int \omega_G^2 \left[\frac{1}{2v}\mathcal B_G^2 + \frac{4\bar p |\psi_1|^2}{3}\right] dx dt = \int H\psi_1\mathcal B_G|_{t=0}^{t=T} dx \nonumber \\
& \quad + \int_0^T \int v_x \frac{H}{2v^2}\mathcal B_G^2 dx dt - \int_0^T \int \psi\mathcal B_G H_t dx dt \nonumber \\
& \quad + \sum_{i\in\{1,3\}}\int_0^T \dot{X}_i \int H \left[\bar p \psi \partial_x\shockw{v}{i} -\psi \partial_x\shockw{\theta}{i} - \mathcal B_G \partial_x\shockw{u_1}{i}\right] dx dt \nonumber\\
& \quad -\int_0^T \int \bar p_x H \frac{4}{3}\psi_1^2 dx dt +\int_0^T \int \frac{2}{3} H \psi\left[(p-\bar p)\psi_x+\bar u_x(p-\bar p)\right] dx dt \nonumber\\
& \quad +\int_0^T \int H\psi \phi \bar p_t dx dt + \int_0^T \int \frac{2}{3}H\psi_1 \left(\mu(\theta)\frac{u_{1x}^2}{v}-\mu(\bar{\theta})\frac{(\bar u_{1x})^2}{\bar v}\right) dx dt \nonumber\\
& \quad +\int_0^T \int \left(H\mathcal B_G\right)_x\left(\mu(\theta)\frac{u_x}{v}-\mu(\y{\theta})\frac{\bar u_x}{\bar v}\right) dx dt  \nonumber \\
& \quad+ \int_0^T \int \frac{2}{3} (H\psi)_x\left(\alpha_{\rm{th}}(\theta)\frac{\theta_x}{v}-\alpha_{\rm{th}}(\bar \theta)\frac{\bar \theta_x}{\bar v}\right) dx dt  + \int_0^T \int \left[\frac{2}{3}HQ_2 + H\mathcal B_GQ_1\right] dx dt \nonumber\\
&\quad + \int_0^T \int \left[\frac{2}{3} H K_\theta + H\mathcal B_G K_u\right] dx dt
\end{align}
where 
\begin{align}
K_u:= -\int \xi_1^2 \wtilde \Pi_{1x} d\xi.
\end{align}

By using the result of \cite{kang2025time} and Lemma~\ref{lem:K1245}, we have 

\begin{align}\label{eq:Gaussian8}
& \int_0^T \int \omega_G^2 \left[\frac{1}{2v}\mathcal B_G^2 + \frac{4\bar p |\psi_1|^2}{3}\right] dx dt \le C \sup_{t\in[0,T]} \norm{(\phi,\psi,\zeta)}_{L^2}^2 \nonumber\\
&\quad + C \sum_{i\in\{1,3\}} \delta_i \int_0^T \abs{\dot{X}_i}^2 dt + C \delta_0 \int_0^T \int \omega_G^2 \abs{(\phi,\psi,\zeta)}^2 dx dt \nonumber\\
&\quad + C \int_0^T \norm{(\phi_x,\psi_x,\zeta_x)}_{L^2}^2 dt + C(\delta_0+\varepsilon)\sum_{i\in\{1,3\}}\int_0^T \mathcal{G}_i^S dt \nonumber\\
&\qquad + C(\delta_0+\varepsilon)\int_0^T \int \norm{\wtilde G_{\text{rem}}}_{\nu,M_\#}^2\, dx\, dt \nonumber\\
&\qquad + C\int_0^T \int \norm{\wtilde G_t}_{\nu,M_\#}^2 + \norm{\wtilde G_x}_{\nu,M_\#}^2 \,dx\, dt +C\delta_0^{\frac{1}{3}}.
\end{align}

Finally, consider the $\psi_{i},\, i=2,3$ cases. {Multiply} \eqref{eq:pesy2} {by} $\psi_ih^2${.}
\begin{align}\label{eq:Gaussian9}
h^2\psi_i\psi_{it} = h^2\psi_i\left(\frac{\mu(\theta)\psi_{ix}}{v}\right)_x - h^2 \int \psi_i \xi_1 \xi_i \Pi_{1x} d\xi.
\end{align}

Note that
\begin{align*}
  (\psi_i^2 \frac{h^2}{2})_t =& h^2\psi_i\psi_{it} + \frac{1}{4\beta}\omega_{Gx} \psi_i^2\\
  =& h^2\psi_i\psi_{it} + \left[\frac{1}{4\beta}\psi_i^2 h \omega_G\right]_x -\frac{1}{4\beta}\psi_i^2 \omega_G^2 -\frac{1}{4\beta} (\psi_i^2)_x h \omega_G.
\end{align*}

Therefore, 
\begin{align}\label{eq:Gaussian10}
&\frac{1}{4\beta}\int_0^T\int \psi_i^2 \omega_G^2 dx dt = \int \psi_i(0,x)^2 \frac{h_0^2}{2} dx \nonumber\\
&\qquad -\int \psi_i^2 \frac{h^2}{2}dx -\int_0^T \int \frac{1}{4\beta} (\psi_i^2)_x h\omega_G dx dt \nonumber \\
&\qquad - \int_0^T (\psi_i^2 h)_x \frac{\mu(\theta)\psi_{ix}}{v} dx dt - \int_0^T h^2 \int K_{gi} dx dt
\end{align}
where
\begin{align*}
K_{gi}:= -\int \psi_i\xi_1 \xi_i \Pi_{1x} d\xi.
\end{align*}

By using Lemma~\ref{lem:K1245}, we have the following estimate.

\begin{align}\label{eq:Gaussian11}
&\int_0^T\int \psi_i^2 \omega_G^2 dx dt \le C \sup_{t\in[0,T]} \norm{(\phi,\psi,\zeta)}_{L^2}^2 \nonumber\\
&\qquad + C \sum_{i\in\{1,3\}} \delta_i \int_0^T \abs{\dot{X}_i}^2 dt + C \delta_0 \int_0^T \int W^2 \abs{(\phi,\psi,\zeta)}^2 dx dt \nonumber\\
&\qquad + C \int_0^T \norm{(\phi_x,\psi_x,\zeta_x)}_{L^2}^2 dt + C(\delta_0+\varepsilon)\sum_{i\in\{1,3\}}\int_0^T \mathcal{G}_i^S dt \nonumber\\
&\qquad + C(\delta_0+\varepsilon)\int_0^T \int \norm{\wtilde G_{\text{rem}}}_{\nu,M_\#}^2 \,dx \,dt \nonumber\\
&\qquad + C\int_0^T \int \norm{\wtilde G_t}_{\nu,M_\#}^2+\norm{\wtilde G_x}_{\nu,M_\#}^2 d\xi dx dt +C\delta_0^{\frac{1}{3}}.
\end{align}

{Combining} \eqref{eq:Gaussian4}, \eqref{eq:Gaussian6}, \eqref{eq:Gaussian8}, and \eqref{eq:Gaussian11} {proves Lemma}~\ref{lem:weighted-gaussian}{.}
\end{proof}

\begin{proof}[\textbf{Proof of Proposition~\ref{prop:priest}}]
We combine the zeroth-order estimate from Section~6, the differentiated estimates from
Section~7, and the Gaussian-weighted estimate of Lemma~\ref{lem:weighted-gaussian}.

For \(t\in[0,T]\), let \(\mathfrak{R}(t)^2\) denote the full left-hand side of
\eqref{eq:pries}, with \(T\) replaced by \(t\). Since \eqref{eq:pries} does not
explicitly contain the first-order time derivatives, we also introduce the auxiliary quantity
\[
\mathcal E_t(t)^2
:=
\int_0^t \|(\phi_t,\psi_t,\zeta_t)(s)\|_{L^2_x}^2\,ds .
\]
It is enough to prove
\[
\mathfrak{R}(t)^2+\mathcal E_t(t)^2
\le
C\bigl(\mathcal{E}(0)^2+\delta_0^{1/2}\bigr),
\qquad 0\le t\le T.
\]

We also set
\[
\mathcal W_C(t)
:=
\int_0^t \frac{1}{1+s}\int_{\mathbb R}
e^{-\frac{2c|x|^2}{1+s}}|(\phi,\psi,\zeta)(s,x)|^2\,dx\,ds ,
\]
and
\[
\mathcal K(t)
:=
\int_0^t\int_{\mathbb R} \norm{\wtilde G_{\text{rem}}}_{\nu,M_\#}^2 \,dx\,ds + \mathcal K_{\mathrm{high}}(t).
\]

\smallskip
\noindent
\emph{Step 1: Zeroth-order control.}
By Lemma~\ref{lem:rel-entropy-equiv}, more precisely by
\eqref{eq:weighted-rel-entropy-equiv}, the weighted relative entropy is equivalent to
the \(L^2_x\)-norm of the macroscopic perturbation.  In addition, \eqref{eq:E0-equiv}
identifies the zeroth-order energy with
\[
\|(\phi,\psi,\zeta)(t)\|_{L^2_x}^2
+
\int_{\mathbb R}
\norm{\wtilde G_{\text{rem}}}_{M_\#}^2\,dx .
\]
Moreover, by the definition of the dissipation \(\mathcal D_{\mathrm{mac}}(U)\) in \eqref{eq:energy2},
together with the bootstrap bound \(\mathcal{E}(T)^2\le \varepsilon^2\) in \eqref{eq:priass}, one has
\[
\mathcal D_{\mathrm{mac}}(U)(t)\ge c\,\|(\psi_x,\zeta_x)(t)\|_{L^2_x}^2 .
\]
Therefore Proposition~\ref{prop:zero-order}, namely \eqref{eq:zero-order-prop}, yields
\begin{align}
\begin{aligned}
&\sup_{0\le s\le t}
\left\{
\|(\phi,\psi,\zeta)(s)\|_{L^2_x}^2
+
\int_{\mathbb R}
\norm{\wtilde G_{\text{rem}}}_{M_\#}^2\,dx
\right\}
\\
&\quad
+
\int_0^t \sum_{i=1,3}\delta_i |\dot X_i(s)|^2\,ds
+
\int_0^t \sum_{i=1,3}\mathcal G_i^S(s)\,ds
+
\int_0^t \|(\psi_x,\zeta_x)(s)\|_{L^2_x}^2\,ds
\\
&\quad
+
\int_0^t\int_{\mathbb R}
\norm{\wtilde G_{\text{rem}}}_{\nu,M_\#}^2 \,dx\,ds
\\
&\le
C\bigl(\mathcal{E}(0)^2+\delta_0^{1/2}\bigr)
+
C\delta_C\mathcal W_C(t)
+
C\int_0^t\int_{\mathbb R}
\norm{\wtilde G_t}_{\nu,M_\#}^2 + \norm{\wtilde G_x}_{\nu,M_\#}^2 \,dx\,ds .
\end{aligned}
\label{eq:priest-step1}
\end{align}

\smallskip
\noindent
\emph{Step 2: The low-order differentiated estimate.}
We next integrate \eqref{eq:low-order-differential} over \([0,t]\).  The energy
functional there contains the cross term \(\int_{\mathbb R} v\psi_1\phi_x\,dx\), but by
the bootstrap bound \(\mathcal{E}(T)^2\le \varepsilon^2\) in \eqref{eq:priass} and Young's inequality,
\[
\left|\int_{\mathbb R} v\psi_1\phi_x\,dx\right|
\le
\frac14\|\phi_x\|_{L^2_x}^2 + C\|\psi_1\|_{L^2_x}^2 ,
\]
and the \(L^2_x\)-norm of \(\psi_1\) is already controlled by \eqref{eq:priest-step1}.
Using also the bound for
\(\int_0^t \|(\psi_x,\zeta_x)(s)\|_{L^2_x}^2\,ds\) furnished by
\eqref{eq:priest-step1}, we obtain
\begin{align}
\begin{aligned}
&\sup_{0\le s\le t}
\left\{
\|\phi_x(s)\|_{L^2_x}^2
+
\int_{\mathbb R}
\norm{\wtilde G_{\text{rem}}}_{M_\#}^2 \,dx
\right\}
+
\mathcal E_t(t)^2
\\
&\quad
+
\int_0^t\int_{\mathbb R}
\norm{\wtilde G_{\text{rem}}}_{\nu,M_\#}^2\,dx\,ds
\\
&\le
C\bigl(\mathcal{E}(0)^2+\delta_0^{1/2}\bigr)
+
C\varepsilon^2\int_0^t \|\psi_{1xx}(s)\|_{L^2_x}^2\,ds
+
C\delta_C \mathcal W_C(t)
+
C\mathcal K_{\mathrm{high}}(t).
\end{aligned}
\label{eq:priest-step2}
\end{align}

\smallskip
\noindent
\emph{Step 3: High-order closure.}
We now invoke Proposition~\ref{prop:high-order-closure}, namely
\eqref{eq:high-order-closure}.  This estimate provides the remaining first-order spatial
derivatives \((\psi_x,\zeta_x)\), the second-order macroscopic terms
\[
(\phi_{xx},\psi_{xx},\zeta_{xx},\phi_{xt},\psi_{xt},\zeta_{xt}),
\]
and the differentiated microscopic quantities
\[
\widetilde G_x,\qquad \widetilde G_t,\qquad \widetilde G_{xx},\qquad \widetilde G_{xt}.
\]
Its right-hand side depends only on lower-order quantities, namely
\[
\int_0^t \sum_{i=1,3}\delta_i |\dot X_i|^2\,ds,\qquad
\int_0^t \sum_{i=1,3}\mathcal G_i^S\,ds,\qquad
\int_0^t \sum_{|\beta|=1}\|\partial^\beta(\phi,\psi,\zeta)\|_{L^2_x}^2\,ds,
\]
and
\[
\int_0^t\int_{\mathbb R}
\norm{\wtilde G_{\text{rem}}}_{\nu,M_\#}^2 \,dx\,ds ,
\]
which are precisely the quantities controlled by \eqref{eq:priest-step1} and
\eqref{eq:priest-step2}. Therefore
\begin{align}
\begin{aligned}
&\sup_{0\le s\le t}
\left\{
\|(\phi_x,\psi_x,\zeta_x)(s)\|_{L^2_x}^2
+
\int_{\mathbb R}
\norm{\wtilde G_{x}}_{M_\#}^2+\norm{\wtilde G_{t}}_{M_\#}^2+\norm{\wtilde G_{xx}}_{M_\#}^2+\norm{\wtilde G_{tx}}_{M_\#}^2 \,dx
\right\}
\\
&\quad
+
\int_0^t
\|(\phi_{xx},\psi_{xx},\zeta_{xx},\phi_{xt},\psi_{xt},\zeta_{xt})(s)\|_{L^2_x}^2\,ds
\\
&\quad
+
\int_0^t\mathcal{K}_{\mathrm{high}}(s)\,ds
\\
&\le
C\bigl(\mathcal{E}(0)^2+\delta_0^{1/2}\bigr)
+
C\delta_C\mathcal W_C(t)
+
C(\delta_0+\varepsilon)\bigl(\mathfrak{R}(t)^2+\mathcal E_t(t)^2\bigr).
\end{aligned}
\label{eq:priest-step3}
\end{align}

\smallskip
\noindent
\emph{Step 4: Control of the Gaussian-weighted term.}
We now apply Lemma~\ref{lem:weighted-gaussian} on the interval \([0,t]\).  Since the
right-hand side of \eqref{eq:weighted-gaussian} contains only the lower-order quantities
already controlled by \eqref{eq:priest-step1}, \eqref{eq:priest-step2}, and
\eqref{eq:priest-step3}, we obtain
\begin{equation}
\delta_C\mathcal W_C(t)
\le
C\delta_C\bigl(\mathcal{E}(0)^2+\delta_0^{1/3}\bigr)
+
C\delta_C(\delta_0+\varepsilon)\bigl(\mathcal{E}(t)^2+\mathcal E_t(t)^2\bigr).
\label{eq:priest-step4}
\end{equation}

\smallskip
\noindent
\emph{Step 5: Final absorption.}
Combining \eqref{eq:priest-step1}, \eqref{eq:priest-step2},
\eqref{eq:priest-step3}, and \eqref{eq:priest-step4}, we find
\[
\mathfrak{R}(t)^2+\mathcal E_t(t)^2
\le
C\bigl(\mathcal{E}(0)^2+\delta_0^{1/2}\bigr)
+
C\varepsilon^2\int_0^t \|\psi_{1xx}(s)\|_{L^2_x}^2\,ds
+
C(\delta_0+\varepsilon)\bigl(\mathfrak{R}(t)^2+\mathcal E_t(t)^2\bigr).
\]
Since \(\int_0^t \|\psi_{1xx}(s)\|_{L^2_x}^2\,ds\) is already one of the components of
\(\mathfrak{R}(t)^2\), it follows that
\[
\mathfrak{R}(t)^2+\mathcal E_t(t)^2
\le
C\bigl(\mathcal{E}(0)^2+\delta_0^{1/2}\bigr)
+
C(\delta_0+\varepsilon)\bigl(\mathfrak{R}(t)^2+\mathcal E_t(t)^2\bigr).
\]
Choosing \(\delta_0>0\) and \(\varepsilon>0\) sufficiently small, we absorb the last
term into the left-hand side and conclude that
\[
\mathfrak{R}(t)^2+\mathcal E_t(t)^2
\le
C\bigl(\mathcal{E}(0)^2+\delta_0^{1/2}\bigr),
\qquad 0\le t\le T.
\]
In particular,
\[
\mathfrak{R}(t)^2
\le
C\bigl(\mathcal{E}(0)^2+\delta_0^{1/2}\bigr),
\qquad 0\le t\le T.
\]
Taking \(t=T\), and recalling the definition of \(\mathfrak{R}(T)^2\), we obtain exactly
\eqref{eq:pries}.  This proves Proposition~\ref{prop:priest}.
\end{proof}

\section{Applications}
\setcounter{equation}{0}

\subsection{Pointwise convergence for a single shock}
In this subsection, we specialize the analysis to the case where the background profile
consists of a single shifted Boltzmann shock profile only. Accordingly, there is no
contact wave, no second shock component, and no wave interaction.

\begin{theorem}[Away-from-the-shock convergence in $L^\infty_x$ for a single shock]
\label{thm:single-shock-away}
Assume that the hypotheses of Proposition~\ref{prop:priest-single} hold.
Then, for every \(T>0\), there exist positive constants \(C\) and \(c\), independent of
\(\kappa\in(0,\kappa_0]\), such that for all \((\tau,y)\in[0,T]\times\mathbb R\),
\begin{align}
\left\|
f^\kappa(\tau,y,\cdot)-M_{E,\kappa}(\tau,y,\cdot)
\right\|_{M_\#}
\le
C\,\kappa 
+
C e^{-c|y-s\tau-X^\kappa(\tau)|/\kappa}.
\label{eq:single-shock-away-pointwise}
\end{align}
\end{theorem}

\medskip

Let
\[
\bar U(t,x)=U^S(x-st-X(t))
=
(\bar v,\bar u,\bar\theta)(t,x)
\]
be the shifted single shock profile, where \(U^S\) is the monotone Boltzmann 3-shock
profile, $s=\sigma$ is the shock speed and \(X(t)\) is the modulation shift. In analogy with the weighted entropy
construction in the composite-wave case, we introduce the single-shock weight
\begin{equation}
a^{\mathrm{sh}}(t,x)
:=
1+\frac{1}{\sqrt{\delta}}\bigl(\bar v(t,x)-v_-\bigr),
\label{eq:single-shock-weight}
\end{equation}
where
\[
\delta:=|v_+-v_-|
\]
denotes the shock strength.

\begin{align*}
p_- = \frac{2\theta_-}{3v_-},\quad \sigma_- = \sqrt{\frac{5p_-}{3v_-}}.
\end{align*}

These constants are chosen according to the admissible range of the shock
strengths and the associated shock speeds, and they will be used in the
definition of the coefficient \(\mathfrak m\) below.
We now define the dynamical shift $X$ by the system of ordinary differential equations
\begin{equation}\label{eq:shifts}
\begin{aligned}
\dot X(t)
&=
-\frac{\mathfrak m}{\delta}
\int_{\mathbb R}
a^{\mathrm{sh}}
\left(
\left((u_1^{S})^{-X}\right)_x\psi_1
+
\frac{\left((v^{S})^{-X}\right)_x\,\y{p}}{\y{v}}\phi
+
\frac{\left((\theta^{S})^{-X}\right)_x}{\y{\theta}}\zeta
\right)\,dx,\\
X(0)&=0,
\end{aligned}
\end{equation}
where the constant $\mathfrak m$ is {defined as follows:}
\begin{align}\label{eq:ssm}
\mathfrak m:=\frac{20}{3}\frac{p_-}{(\sigma_-)^3(v_-)^2}\frac{5+3\gamma}{10+3\gamma}.
\end{align}

We define the associated single-shock coercive quantity by
\begin{equation}
\mathcal G_{\mathrm{sh}}(t)
:=
\int_{\mathbb R}\bigl|\bar v_x(t,x)\bigr|\,
|(\phi,\psi,\zeta)(t,x)|^2\,dx .
\label{eq:def-Gsh}
\end{equation}

\subsubsection{Single-shock closure scheme}

We first record the single-shock analogue of the main a priori estimate. Its proof is
obtained by combining the zeroth-order estimate, the low-order differentiated estimate,
and the high-order closure, exactly as in the composite-wave case, but with all
contact-generated and interaction-generated remainders removed.
{
    We recall the definition fo the energy $\mathcal E(T)^2$:
\begin{equation}\label{eq:priass-single}
\begin{aligned}
\mathcal{E}(T)^2
:=
\sup_{t\in[0,T]}
\Biggl\{
&\|(\phi,\psi,\zeta)(t)\|_{H^1_x}^2
+\int_{\mathbb R}\norm{\wtilde G}_{M_\#}^2\,+\norm{\wtilde G_x}_{M_\#}^2\,dx\\
&\qquad+\int_{\mathbb R}\norm{\wtilde G_t}_{M_\#}^2\,+\norm{\wtilde f_{xx}}_{M_\#}^2\,+\norm{\wtilde f_{tx}}_{M_\#}^2\,dx\Biggr\}.
\end{aligned}
\end{equation}
}
\begin{proposition}[Main a priori estimate for a single 3-shock]
\label{prop:priest-single}
Assume that the background profile consists of a single shifted Boltzmann shock profile
only. Then there exist positive constants \(\delta_0\), \(\varepsilon\), and \(C\),
together with a global Maxwellian \(M_\#:=M[U_\#]\), such that the following holds.

Suppose that \((U,f)=(v,u,\theta,f)\) solves the perturbation system around the single
shifted shock profile on \([0,T]\), with modulation \(X(t)\), and assume that 
\[
\delta\in(0,\delta_0),
\qquad
\mathcal{E}(T)^2\le \varepsilon^2 .
\]
Then
\begin{align}
\begin{aligned}
\mathcal{E}(T)^2
&+\int_0^T \delta |\dot X(t)|^2\,dt
+\int_0^T \mathcal G_{\mathrm{sh}}(t)\,dt
+\sum_{|\beta|=1}\int_0^T \|\partial^\beta(\phi,\psi,\zeta)(t)\|_{L^2_x}^2\,dt
\\
&\quad
+\int_0^T
\|(\phi_{xx},\psi_{xx},\zeta_{xx},\phi_{xt},\psi_{xt},\zeta_{xt})(t)\|_{L^2_x}^2\,dt
\\
&\quad
+\int_0^T\int_{\mathbb R}
\norm{\widetilde G}_{\nu,M_\#}^2 + \norm{\widetilde G_x}_{\nu,M_\#}^2 + \norm{\widetilde G_t}_{\nu,M_\#}^2 + \norm{\widetilde G_{xx}}_{\nu,M_\#}^2 + \norm{\widetilde G_{xt}}_{\nu,M_\#}^2 \,dx\,dt
\\
&\le
C\,\mathcal{E}(0)^2 .
\end{aligned}
\label{eq:pries-single}
\end{align}
Here \(\mathcal G_{\mathrm{sh}}(t)\) is the single-shock coercive quantity defined in
\eqref{eq:def-Gsh}.
\end{proposition}
\begin{proof}


For a single 3-shock, we set $\delta_1=\delta_C=0$. Then the interaction estimates, as well as the error terms generated by the cutoff functions $\varphi_i$ (See \eqref{eq:intske}) introduced to separate the two shocks in the {a priori estimates} \eqref{eq:pries}, disappear. In particular, {the term} $\delta_0^{1/2}$ {on the right-hand side of the composite-wave a priori estimate} \eqref{eq:pries} {arises from the wave interactions and is therefore absent in the single-shock case.} In addition, the zeroth perturbed microscopic variable $\wtilde G_{\mathrm{rem}}=\wtilde G = G-\bigl(G^S\bigr)^{-X}$ since there is no contact discontinuity. All the remaining estimates follow directly from the estimates established in the previous section. 
\end{proof}

\begin{lemma}[Exponential localization of the single shock profile]
\label{lem:single-shock-tail}
Let
\[
F^{S}(z,\xi)=M[U^{S}](z,\xi)+G^{S}(z,\xi)
\]
be the single Boltzmann shock profile connecting the end Maxwellians
\[
M_-:=M[U_-],
\qquad
M_+:=M[U_+].
\]
Assume that \(F^S\) is smooth and converges exponentially to its end states. Then there
exist positive constants \(C\) and \(c\) such that
\begin{align}
\|F^{S}(z,\cdot)-M_\pm\|_{M_\#}
+
|U^{S}(z)-U_\pm|
\le
Ce^{-c|z|},
\qquad \pm z\ge0 .
\label{eq:single-shock-tail-basic}
\end{align}

Let the shifted \(\kappa\)-shock profile be defined by
\[
\bar f^\kappa(\tau,y,\xi)
:=
F^{S}\!\left(\frac{y-s\tau-X^\kappa(\tau)}{\kappa},\xi\right),
\]
and let the associated modulated Euler Maxwellian be
\[
M_{E,\kappa}(\tau,y,\xi)
:=
\begin{cases}
M_-(\xi), & y\le s\tau+X^\kappa(\tau),\\[2mm]
M_+(\xi), & y> s\tau+X^\kappa(\tau).
\end{cases}
\]
Then for every \(T>0\), the pointwise bound
\begin{align}
\left\|
\bar f^\kappa(\tau,y,\cdot)-M_{E,\kappa}(\tau,y,\cdot)
\right\|_{M_\#}
\le
Ce^{-c|y-s\tau-X^\kappa(\tau)|/\kappa}
\label{eq:single-shock-tail-pointwise}
\end{align}
holds for all \((\tau,y)\in[0,T]\times\mathbb R\). In particular, for every \(h>0\),
\begin{align}
\sup_{\substack{0\le \tau\le T\\ |y-s\tau-X^\kappa(\tau)|\ge h}}
\left\|
\bar f^\kappa(\tau,y,\cdot)-M_{E,\kappa}(\tau,y,\cdot)
\right\|_{M_\#}
\le
Ce^{-ch/\kappa}.
\label{eq:single-shock-tail-scaled}
\end{align}
\end{lemma}

\begin{proof}
The estimate \eqref{eq:single-shock-tail-basic} is precisely the exponential convergence
of the travelling shock profile to its end Maxwellians.

To prove the pointwise estimate \eqref{eq:single-shock-tail-pointwise}, fix
\((\tau,y)\in[0,T]\times\mathbb R\), and set
\[
z:=\frac{y-s\tau-X^\kappa(\tau)}{\kappa}.
\]
Then
\[
\bar f^\kappa(\tau,y,\xi)=F^S(z,\xi).
\]

If \(z\le0\), then by the definition of \(M_{E,\kappa}\),
\[
M_{E,\kappa}(\tau,y,\xi)=M_-(\xi),
\]
and therefore, using \eqref{eq:single-shock-tail-basic} with the minus sign,
\[
\left\|
\bar f^\kappa(\tau,y,\cdot)-M^{E,\kappa}(\tau,y,\cdot)
\right\|_{M_\#}
=
\|F^S(z,\cdot)-M_-\|_{M_\#}
\le
Ce^{-c|z|}.
\]

If \(z>0\), then
\[
M_{E,\kappa}(\tau,y,\xi)=M_+(\xi),
\]
and \eqref{eq:single-shock-tail-basic} with the plus sign gives
\[
\left\|
\bar f^\kappa(\tau,y,\cdot)-M_{E,\kappa}(\tau,y,\cdot)
\right\|_{M_\#}
=
\|F^S(z,\cdot)-M_+\|_{M_\#}
\le
Ce^{-c|z|}.
\]

Since
\[
|z|=\frac{|y-s\tau-X^\kappa(\tau)|}{\kappa},
\]
the two cases combine to yield
\eqref{eq:single-shock-tail-pointwise}.

Finally, if
\[
|y-s\tau-X^\kappa(\tau)|\ge h,
\]
then \eqref{eq:single-shock-tail-pointwise} implies
\[
\left\|
\bar f^\kappa(\tau,y,\cdot)-M_{E,\kappa}(\tau,y,\cdot)
\right\|_{M_\#}
\le
Ce^{-ch/\kappa}.
\]
Taking the supremum over all such \((\tau,y)\) gives
\eqref{eq:single-shock-tail-scaled}.
\end{proof}


\begin{proof}[\textbf{Proof of Theorem~\ref{thm:single-shock-away}}]
For the single-shock problem, define
\[
\bar U^\kappa(\tau,y)
:=
U^S\!\left(\frac{y-s\tau-X^\kappa(\tau)}{\kappa}\right),
\qquad
\bar M^\kappa(\tau,y,\xi)
:=
M[\bar U^\kappa(\tau,y)](\xi),
\]
and
\[
\bar f^\kappa(\tau,y,\xi)
:=
F^S\!\left(\frac{y-s\tau-X^\kappa(\tau)}{\kappa},\xi\right).
\]
By the decomposition \(F^S=M[U^S]+G^S\), we have
\[
\bar f^\kappa(\tau,y,\xi)
=
\bar M^\kappa(\tau,y,\xi)
+
G^S\!\left(\frac{y-s\tau-X^\kappa(\tau)}{\kappa},\xi\right).
\]

Let
\[
U^\kappa=(v^\kappa,u^\kappa,\theta^\kappa),
\qquad
(\phi^\kappa,\psi^\kappa,\zeta^\kappa)
:=
U^\kappa-\bar U^\kappa .
\]
Hence
\begin{equation}
f^\kappa-\bar f^\kappa
=
\bigl(M[U^\kappa]-\bar M^\kappa\bigr)
+
\widetilde G^\kappa .
\label{eq:single-shock-away-pointwise-decomp}
\end{equation}

By Proposition~\ref{prop:priest-single}, for every fixed \(T>0\),
\[
\mathcal{E}^\kappa(T)^2\le C\,\mathcal{E}^\kappa(0)^2 .
\]
Recalling the definition of \(\mathcal{E}^\kappa(T)\), we infer that
\begin{equation}
\sup_{0\le\tau\le T}
\|(\phi^\kappa,\psi^\kappa,\zeta^\kappa)(\tau)\|_{H^1_y}
\le
C\,\mathcal{E}^\kappa(0),
\label{eq:single-shock-away-H1-fluid}
\end{equation}
and
\begin{equation}
\sup_{0\le\tau\le T}
\|\widetilde G_{\mathrm{rem}}^\kappa(\tau,\cdot,\cdot)\|_{H^1_y(L^2_\xi(M_\#))}
\le
C\,\mathcal{E}^\kappa(0).
\label{eq:single-shock-away-H1-micro}
\end{equation}

Applying Lemma~\ref{lem:H1-Linfty-Hilbert-single} with \(H=\mathbb R^3\) to
\((\phi^\kappa,\psi^\kappa,\zeta^\kappa)\), and with
\(H=L^2_\xi(M_\#^{-1})\) to \(\widetilde G^\kappa\), we obtain
\begin{equation}
\sup_{0\le\tau\le T}\sup_{y\in\mathbb R}
|U^\kappa(\tau,y)-\bar U^\kappa(\tau,y)|
\le
C\,\mathcal{E}^\kappa(0),
\label{eq:single-shock-away-Linfty-fluid}
\end{equation}
and
\begin{equation}
\sup_{0\le\tau\le T}\sup_{y\in\mathbb R}
\|\widetilde G^\kappa(\tau,y,\cdot)\|_{M_\#}
\le
C\,\mathcal{E}^\kappa(0).
\label{eq:single-shock-away-Linfty-micro}
\end{equation}

Next we compare the Maxwellian parts. Since the single shock profile \(U^S\) connects the
fixed end states \(U_-\) and \(U_+\), its range is contained in a compact set
\[
K_S\Subset \mathbb R_+\times\mathbb R\times\mathbb R_+ .
\]
By \eqref{eq:single-shock-away-Linfty-fluid}, and after possibly shrinking the bootstrap
constant \(\kappa_0\), the range of \(U^\kappa\) on \([0,T]\times\mathbb R\) is
contained in a fixed compact set
\[
K\Subset \mathbb R_+\times\mathbb R\times\mathbb R_+
\]
depending only on \(K_S\) and \(\kappa_0\). Therefore
Lemma~\ref{lem:Maxwellian-Lipschitz-single} applies uniformly on \(K\), and yields
\begin{equation}
\sup_{0\le\tau\le T}\sup_{y\in\mathbb R}
\|M[U^\kappa(\tau,y)]-\bar M^\kappa(\tau,y)\|_{M_\#}
\le
C\,\mathcal{E}^\kappa(0).
\label{eq:single-shock-away-Maxwellian}
\end{equation}

Combining \eqref{eq:single-shock-away-pointwise-decomp},
\eqref{eq:single-shock-away-Linfty-micro}, and
\eqref{eq:single-shock-away-Maxwellian}, we obtain the uniform estimate
\begin{equation}
\|f^\kappa(\tau,y,\cdot)-\bar f^\kappa(\tau,y,\cdot)\|_{M_\#}
\le
C\,\mathcal{E}^\kappa(0),
\qquad
(\tau,y)\in[0,T]\times\mathbb R.
\label{eq:single-shock-away-pointwise-step1}
\end{equation}

On the other hand, Lemma~\ref{lem:single-shock-tail} gives the pointwise exponential tail estimate
\begin{equation}
\|\bar f^\kappa(\tau,y,\cdot)-M_{E,\kappa}(\tau,y,\cdot)\|_{M_\#}
\le
C e^{-c|y-s\tau-X^\kappa(\tau)|/\kappa},
\qquad
(\tau,y)\in[0,T]\times\mathbb R.
\label{eq:single-shock-away-pointwise-step2}
\end{equation}

Therefore, by the triangle inequality,

\begin{align*}
& \left\|
f^\kappa(\tau,y,\cdot)-M_{E,\kappa}(\tau,y,\cdot)
\right\|_{M_\#} \\
&\quad \le
\left\|
f^\kappa(\tau,y,\cdot)-\bar f^\kappa(\tau,y,\cdot)
\right\|_{M_\#} +
\left\|
\bar f^\kappa(\tau,y,\cdot)-M_{E,\kappa}(\tau,y,\cdot)
\right\|_{M_\#}
\\
&\quad \le
C\,\mathcal{E}^\kappa(0)
+
C e^{-c|y-s\tau-X^\kappa(\tau)|/\kappa}\\
&\quad \le
C\, \kappa \mathcal{E}(0) + C e^{-c|y-s\tau-X^\kappa(\tau)|/\kappa}
\end{align*}

which proves \eqref{eq:single-shock-away-pointwise}.
\end{proof}
\hide

\begin{proof}[\textbf{Proof of Theorem \ref{thm:single-shock-away}}]
For the single-shock problem, define
\[
\bar U^\kappa(\tau,y)
:=
U^S\!\left(\frac{y-s\tau-X^\kappa(\tau)}{\kappa}\right),
\qquad
\bar M^\kappa(\tau,y,\xi)
:=
M[\bar U^\kappa(\tau,y)](\xi),
\]
and
\[
\bar f^\kappa(\tau,y,\xi)
:=
F^S\!\left(\frac{y-s\tau-X^\kappa(\tau)}{\kappa},\xi\right).
\]
Then
\[
\bar f^\kappa(\tau,y,\xi)
=
\bar M^\kappa(\tau,y,\xi)
+
G^S\!\left(\frac{y-s\tau-X^\kappa(\tau)}{\kappa},\xi\right).
\]

Let
\[
U^\kappa=(v^\kappa,u^\kappa,\theta^\kappa),
\qquad
(\phi^\kappa,\psi^\kappa,\zeta^\kappa)
:=
U^\kappa-\bar U^\kappa .
\]
Since there is no contact component in the present setting,
\[
\widetilde G_0^\kappa\equiv0,
\qquad
\widetilde G^\kappa=\widetilde G_{\mathrm{rem}}^\kappa,
\]
and therefore
\begin{equation}
f^\kappa-\bar f^\kappa
=
\bigl(M[U^\kappa]-\bar M^\kappa\bigr)
+
\widetilde G_{\mathrm{rem}}^\kappa .
\label{eq:single-shock-away-pointwise-decomp1}
\end{equation}

By Proposition~\ref{prop:priest-single}, for every fixed \(T>0\),
\[
\mathcal{E}^\kappa(T)^2\le C\,\mathcal{E}^\kappa(0)^2.
\]
Hence
\[
\sup_{0\le\tau\le T}
\|(\phi^\kappa,\psi^\kappa,\zeta^\kappa)(\tau)\|_{H^1_y}
\le
C\,\mathcal{E}^\kappa(0),
\]
and
\[
\sup_{0\le\tau\le T}
\|\widetilde G_{\mathrm{rem}}^\kappa(\tau,\cdot,\cdot)\|_{H^1_y(L^2_\xi(M_\#^{-1}))}
\le
C\,\mathcal{E}^\kappa(0).
\]
By Lemma~\ref{lem:H1-Linfty-Hilbert-single}, it follows that
\[
\sup_{0\le\tau\le T}\sup_{y\in\mathbb R}
|U^\kappa(\tau,y)-\bar U^\kappa(\tau,y)|
\le
C\,\mathcal{E}^\kappa(0),
\]
and
\[
\sup_{0\le\tau\le T}\sup_{y\in\mathbb R}
\|\widetilde G_{\mathrm{rem}}^\kappa(\tau,y,\cdot)\|_{L^2_\xi(M_\#^{-1})}
\le
C\,\mathcal{E}^\kappa(0).
\]
By the local Lipschitz continuity of the Maxwellian map on compact positive sets,
\[
\sup_{0\le\tau\le T}\sup_{y\in\mathbb R}
\|M[U^\kappa(\tau,y)]-\bar M^\kappa(\tau,y)\|_{L^2_\xi(M_\#^{-1})}
\le
C\,\mathcal{E}^\kappa(0).
\]
Combining these bounds with \eqref{eq:single-shock-away-pointwise-decomp1}, we obtain the
uniform estimate
\begin{equation}
\|f^\kappa(\tau,y,\cdot)-\bar f^\kappa(\tau,y,\cdot)\|_{L^2_\xi(M_\#^{-1})}
\le
C\,\mathcal{E}^\kappa(0),
\qquad
(\tau,y)\in[0,T]\times\mathbb R.
\label{eq:single-shock-away-pointwise-step1-1}
\end{equation}

Next, Lemma~\ref{lem:single-shock-tail} gives the pointwise exponential tail estimate
\begin{equation}
\|\bar f^\kappa(\tau,y,\cdot)-M^{E,\kappa}(\tau,y,\cdot)\|_{L^2_\xi(M_\#^{-1})}
\le
C e^{-c|y-s\tau-X^\kappa(\tau)|/\kappa},
\qquad
(\tau,y)\in[0,T]\times\mathbb R.
\label{eq:single-shock-away-pointwise-step2-1}
\end{equation}

Therefore, by the triangle inequality,
\begin{align*}
\left\|
f^\kappa(\tau,y,\cdot)-M^{E,\kappa}(\tau,y,\cdot)
\right\|_{L^2_\xi(M_\#^{-1})}
&\le
\left\|
f^\kappa(\tau,y,\cdot)-\bar f^\kappa(\tau,y,\cdot)
\right\|_{L^2_\xi(M_\#^{-1})}
\\
&\quad
+
\left\|
\bar f^\kappa(\tau,y,\cdot)-M^{E,\kappa}(\tau,y,\cdot)
\right\|_{L^2_\xi(M_\#^{-1})}
\\
&\le
C\,\mathcal{E}^\kappa(0)
+
C e^{-c|y-s\tau-X^\kappa(\tau)|/\kappa},
\end{align*}
which proves \eqref{eq:single-shock-away-pointwise}.

Finally, if \((\tau,y)\in\Omega^\kappa_{h,T}\), then
\[
|y-s\tau-X^\kappa(\tau)|\ge h,
\]
so \eqref{eq:single-shock-away-pointwise} immediately yields
\eqref{eq:single-shock-away-modulated}.
\end{proof}

\begin{proof}[\textbf{Proof of Theorem \ref{thm:single-shock-away}}]
For the single-shock problem, define
\[
\bar U^\kappa(\tau,y)
:=
U^S\!\left(\frac{y-s\tau-X^\kappa(\tau)}{\kappa}\right),
\qquad
\bar M^\kappa(\tau,y,\xi)
:=
M[\bar U^\kappa(\tau,y)](\xi),
\]
and
\[
\bar f^\kappa(\tau,y,\xi)
:=
F^S\!\left(\frac{y-s\tau-X^\kappa(\tau)}{\kappa},\xi\right).
\]
Then
\[
\bar f^\kappa(\tau,y,\xi)
=
\bar M^\kappa(\tau,y,\xi)
+
G^S\!\left(\frac{y-s\tau-X^\kappa(\tau)}{\kappa},\xi\right).
\]

Let
\[
U^\kappa=(v^\kappa,u^\kappa,\theta^\kappa)
\]
denote the macroscopic field associated with \(f^\kappa\), and set
\[
(\phi^\kappa,\psi^\kappa,\zeta^\kappa)
:=
U^\kappa-\bar U^\kappa .
\]
Since there is no contact component in the present setting, we have
\[
\widetilde G_0^\kappa\equiv0,
\qquad
\widetilde G^\kappa=\widetilde G_{\mathrm{rem}}^\kappa .
\]
Hence
\begin{equation}
f^\kappa-\bar f^\kappa
=
\bigl(M[U^\kappa]-\bar M^\kappa\bigr)
+
\widetilde G_{\mathrm{rem}}^\kappa .
\label{eq:single-shock-away-decomp}
\end{equation}

By Proposition~\ref{prop:priest-single}, for every fixed \(T>0\),
\[
\mathcal{E}^\kappa(T)^2\le C\,\mathcal{E}^\kappa(0)^2 .
\]
Recalling the definition of \(\mathcal{E}^\kappa(T)\), and using
\(\widetilde G^\kappa=\widetilde G_{\mathrm{rem}}^\kappa\), we infer that
\begin{equation}
\sup_{0\le\tau\le T}
\|(\phi^\kappa,\psi^\kappa,\zeta^\kappa)(\tau)\|_{H^1_y}
\le
C\,\mathcal{E}^\kappa(0),
\label{eq:single-shock-away-H1-fluid1}
\end{equation}
and
\begin{equation}
\sup_{0\le\tau\le T}
\|\widetilde G_{\mathrm{rem}}^\kappa(\tau,\cdot,\cdot)\|_{H^1_y(L^2_\xi(M_\#^{-1}))}
\le
C\,\mathcal{E}^\kappa(0).
\label{eq:single-shock-away-H1-micro1}
\end{equation}

Applying Lemma~\ref{lem:H1-Linfty-Hilbert-single} with \(H=\mathbb R^3\) to
\((\phi^\kappa,\psi^\kappa,\zeta^\kappa)\), and with
\(H=L^2_\xi(M_\#^{-1})\) to \(\widetilde G_{\mathrm{rem}}^\kappa\), we obtain
\begin{equation}
\sup_{0\le\tau\le T}\sup_{y\in\mathbb R}
|U^\kappa(\tau,y)-\bar U^\kappa(\tau,y)|
\le
C\,\mathcal{E}^\kappa(0),
\label{eq:single-shock-away-Linfty-fluid1}
\end{equation}
and
\begin{equation}
\sup_{0\le\tau\le T}\sup_{y\in\mathbb R}
\|\widetilde G_{\mathrm{rem}}^\kappa(\tau,y,\cdot)\|_{L^2_\xi(M_\#^{-1})}
\le
C\,\mathcal{E}^\kappa(0).
\label{eq:single-shock-away-Linfty-micro1}
\end{equation}

Next we compare the Maxwellian parts. Since the single shock profile \(U^S\) connects the
fixed end states \(U_-\) and \(U_+\), its range is contained in a compact set
\[
K_S\Subset \mathbb R_+\times\mathbb R\times\mathbb R_+ .
\]
By \eqref{eq:single-shock-away-Linfty-fluid1}, and after possibly shrinking the bootstrap
constant \(\kappa_0\), the range of \(U^\kappa\) on \([0,T]\times\mathbb R\) is
contained in a fixed compact set
\[
K\Subset \mathbb R_+\times\mathbb R\times\mathbb R_+
\]
depending only on \(K_S\) and \(\kappa_0\). Therefore
Lemma~\ref{lem:Maxwellian-Lipschitz-single} applies uniformly on \(K\), and yields
\begin{equation}
\sup_{0\le\tau\le T}\sup_{y\in\mathbb R}
\|M[U^\kappa(\tau,y)]-\bar M^\kappa(\tau,y)\|_{L^2_\xi(M_\#^{-1})}
\le
C\,\mathcal{E}^\kappa(0).
\label{eq:single-shock-away-Maxwellian1}
\end{equation}

Combining \eqref{eq:single-shock-away-decomp},
\eqref{eq:single-shock-away-Linfty-micro1}, and
\eqref{eq:single-shock-away-Maxwellian1}, we obtain
\begin{equation}
\sup_{0\le\tau\le T}\sup_{y\in\mathbb R}
\|f^\kappa(\tau,y,\cdot)-\bar f^\kappa(\tau,y,\cdot)\|_{L^2_\xi(M_\#^{-1})}
\le
C\,\mathcal{E}^\kappa(0).
\label{eq:single-shock-away-step1}
\end{equation}

On the other hand, Lemma~\ref{lem:single-shock-tail} yields
\begin{equation}
\sup_{(\tau,y)\in\Omega^\kappa_{h,T}}
\|\bar f^\kappa(\tau,y,\cdot)-M^{E,\kappa}(\tau,y,\cdot)\|_{L^2_\xi(M_\#^{-1})}
\le
Ce^{-ch/\kappa}.
\label{eq:single-shock-away-step2}
\end{equation}

Therefore, by the triangle inequality,
\[
\|f^\kappa-M^{E,\kappa}\|_{L^2_\xi(M_\#^{-1})}
\le
\|f^\kappa-\bar f^\kappa\|_{L^2_\xi(M_\#^{-1})}
+
\|\bar f^\kappa-M^{E,\kappa}\|_{L^2_\xi(M_\#^{-1})}.
\]
Taking the supremum over \((\tau,y)\in\Omega^\kappa_{h,T}\), and using
\eqref{eq:single-shock-away-step1} and \eqref{eq:single-shock-away-step2}, yields
\eqref{eq:single-shock-away-modulated}. Since \(\mathcal{E}^\kappa(0)\to0\) and
\(e^{-ch/\kappa}\to0\) for every fixed \(h>0\), the convergence statement follows.
\end{proof}\unhide
 \begin{remark}
Assume in addition that \(X^\kappa(0)=0\). 
By Proposition~\ref{prop:priest-single},
we have
\[
|X^\kappa(\tau)|
\le
\left|\int_0^T \dot X^\kappa(s)\,ds\right|
=\kappa \int_0^{T/\kappa} \left|\dot X(s)\right| ds
\le
\kappa^{1/2}T^{1/2}
\left(\int_0^{T/\kappa} |\dot X(s)|^2\,ds\right)^{1/2}
\le
CT^{1/2}\kappa^{1/2}\delta^{-1/2}.
\]
Hence
\[
\sup_{0\le\tau\le T}|X^\kappa(\tau)|
\le
CT^{1/2} \kappa^{1/2}\to0
\qquad\text{as }\kappa\to0.
\]

Therefore, for every fixed \(h>0\), if \(\kappa\) is sufficiently small so that
\[
\sup_{0\le\tau\le T}|X^\kappa(\tau)|\le \frac h2,
\]
then
\[
|y-s\tau|\ge h
\quad\Longrightarrow\quad
|y-s\tau-X^\kappa(\tau)|\ge \frac h2.
\]
Applying Theorem~\ref{thm:single-shock-away} with \(h/2\) in place of \(h\), we conclude
that the same convergence statement holds uniformly on the region
\[
\{(\tau,y)\in[0,T]\times\mathbb R:\ |y-s\tau|\ge h\},
\]
that is, away from the unmodulated Euler shock curve \(y=s\tau\).
\end{remark}

\subsection{Rarefaction-contact-shock composite wave}

We first recall the $1$-rarefaction wave for the compressible Euler system. Let
\[
(v^r,u^r,\theta^r)(t,x)
=
(v^r,u^r,\theta^r)\!\left(\frac{x}{t}\right)
\]
be the self-similar $1$-rarefaction wave fan solving
\[
v_t-u_{1x}=0,\qquad
u_{1t}+p_x=0,\qquad
u_{it}=0 \ \ (i=2,3),\qquad
\left(\theta+\frac{|u|^2}{2}\right)_t+(pu_1)_x=0,
\]
with Riemann initial data
\[
(v,u,\theta)(0,x)=
\begin{cases}
(v_-,u_-,\theta_-), & x<0,\\
(v_*,u_*,\theta_*), & x>0,
\end{cases}
\]
where
\[
u_-=(u_{1-},0,0),\qquad u_*=(u_{1*},0,0).
\]

To construct a smooth approximate rarefaction wave for the Boltzmann equation, we introduce the smooth solution $w^{R,\kappa}$ of the Burgers equation
\[
w^{R,\kappa}_t+w^{R,\kappa} w^{R,\kappa}_x=0,
\qquad
w^{R,\kappa}(0,x)=\frac{w_*+w_-}{2}+\frac{w_*-w_-}{2}\tanh \frac x\kappa.
\]
We then define the smooth approximate rarefaction wave $(v^{R,\kappa},u^{R,\kappa},\theta^{R,\kappa})(t,x)$ by
\begin{equation}\label{eq:rare1}
\begin{aligned}
&w_-=\lambda_{1-}:=\lambda_1(v_-,\theta_-),
\quad
w_*=\lambda_{1*}:=\lambda_1(v_*,\theta_*), \quad \lambda_1(v,\mathfrak s_*) := \lambda_1(v,\theta(v;\mathfrak s_*)),\\
&\lambda_1\bigl(v^{R,\kappa}(t,x),\theta^{R,\kappa}(t,x)\bigr)=w^{R,\kappa}(t+1,x),\\
&u_1^{R,\kappa}(t,x)=u_{1*}-\int_{v_*}^{v^{R,\kappa}(t,x)}\lambda_1(v,s_*)\,dv,\\
&\mathfrak s\bigl(v^{R,\kappa}(t,x),\theta^{R,\kappa}(t,x)\bigr)=\mathfrak s_*,\\
&u_i^{R,\kappa}(t,x)\equiv 0,\qquad i=2,3.
\end{aligned}
\end{equation}
Here $\lambda_1(v,s_*)$ denotes the first characteristic speed along the isentrope $\mathfrak s=\mathfrak s_*$.

The construction of the smooth approximate rarefaction wave is standard in the Boltzmann--Euler setting; see \cite{huang2010asymptotic}, \cite{matsumura1986asymptotics}, \cite{xin1993zero} and \cite{huang2010hydrodynamic}. Here, we record only the properties of the rarefaction wave that are used in our analysis. The remaining properties are mainly needed to derive the a priori estimate. For the a priori estimate related to the rarefaction-contact discontinuity-shock composite wave, we refer the reader to \cite{wang2025}.

\begin{lemma}\label{lem:Yilem}
Let
\[
\delta_R:=|v_*-v_-|
\sim |u_{1*}-u_{1-}|
\sim |\theta_*-\theta_-|
\]
denote the strength of the rarefaction wave. Then the smooth approximate $1$-rarefaction wave $(v^{R,\kappa},u_1^{R,\kappa},\theta^{R,\kappa})(t,x)$ defined in \eqref{eq:rare1} satisfies the following properties.

\begin{enumerate}
\item
For all $x\in\mathbb{R}$ and $t\ge0$,
\[
u_{1x}^{R,\kappa}=\frac{3v^{R,\kappa}}{4}w_x^{R,\kappa}>0,
\qquad
v_x^{R,\kappa}=\frac{3v^{R,\kappa}}{\sqrt{10\theta^{R,\kappa}}}u_{1x}^{R,\kappa}>0,
\qquad
\theta_x^{R,\kappa}=-\frac{2\theta^{R,\kappa}}{3v^R}v_x^{R,\kappa}<0.
\]

\item {There exists} $\kappa_0>0$ {such that, for every} $\kappa\in(0,\kappa_0]${,} $p\in[1,+\infty]${, and} $t>0${,} 
\begin{align}\label{eq:rarmai}
\norm{(v^{R,\kappa}-v^{r},u_1^{R,\kappa}-u^r,\theta^{R,\kappa}-\theta^r)(\cdot,t)}_{L^p} \le C_p \delta_R \kappa^{1/p}
\end{align} 

\end{enumerate}
\end{lemma}

\begin{proof}
The proof is the same as that in \cite{xin1993zero}. 

\end{proof}
{
\begin{theorem}[Rarefaction--contact--shock case]\label{cor:rarefaction-contact-shock}
Let \(U_+\), \(\kappa_0\), \(\delta_{\mathrm{res},0}\), \(\varepsilon_{\mathrm{res},0}\), and \(M_\#=M[U_\#]\) be as in Theorem~\ref{thm:main}.
Let \(0<\delta\le \delta_{\mathrm{res},0}\), and let
\[
U^E=(v^E,u^E,\theta^E)(t,x)
\]
be a Riemann solution to the one-dimensional compressible Euler system \eqref{eq:cesL}, connecting
\[
U_-:=(v_-,u_{1-},\theta_-)
\qquad \text{to} \qquad
U_+,
\]
through two intermediate states \(U_*\) and \(U^*\), with total wave strength
\[
\delta=\delta_R+\delta_C+\delta_S
\]
where \(\delta_R\), \(\delta_C\), and \(\delta_S\) {are the strengths}
of the rarefaction, contact, and shock components, respectively.
Assume that \(U^E\) is a superposition of a \(1\)-rarefaction wave, a \(2\)-contact discontinuity,
and a \(3\)-shock wave.

For each \(0<\varepsilon_1\le \varepsilon_{\mathrm{res},0}\) and \(0<\kappa\le \kappa_0\), let
\(f_0^\kappa\) be a nonnegative smooth initial datum satisfying the well-preparedness condition
\eqref{eq:initic2-main}. Let \(f^\kappa\) be the corresponding solution
to the Boltzmann equation \eqref{eq:bslk} on a time interval \([0,T]\), where \(T>0\) is arbitrary.

Then \(f^\kappa\) exists uniquely on \([0,T]\). Moreover,
\begin{align}\label{eq:hlmt-main}
&\int_0^T \iint_{\mathbb{R}\times\mathbb{R}^3}
\left|f^\kappa-M[U^E]\right|^2
\,d\xi\,dy\,dt
\nonumber\\
&\qquad \le
C\kappa^{1/2}\Bigl[\kappa^{1/2}\bigl((\varepsilon_1+\delta_{\mathrm{res},0})^2+\delta_{\mathrm{res},0}^{1/2}\bigr)T
+
(\delta_R+\delta_C)T^{3/2}
+
\delta_{\mathrm{res},0}^2T^{3/2}\Bigr].
\end{align}
Here the positive constant \(C\) is independent of \(\kappa\) and \(T\).
\end{theorem}
}
\vspace{0.5cm}

In this case, since the 3-shock will be shifted only, for simplicity we work in the moving frame and use the normalized equation:
\begin{align}
    \begin{aligned}\label{eq:Yibel}
        &f_t-\sigma f_z-\frac{u_1}{v}f_z+\frac{\xi_1}{v}f_z=\mathcal N(f,f)\\
        &f(0,z,\xi)=f_0(z,\xi).
    \end{aligned}
\end{align}
We also need the following result. 
\begin{proposition} \label{prop:Ypro}
    (see \cite{wang2025}) For each $(v_+,u_+,\theta_+)$ with $v_+,\theta_+>0$, {let} $(v_-,u_-,\theta_-)$ {be such that the Riemann solution of the Euler equations} \eqref{eq:cesL} {consists of a} $1$-rarefaction wave, a $2$-contact discontinuity, and a $3$-shock wave. Suppose that $f(t,y,\xi)$ is a solution to \eqref{eq:Yibel} {for} $t\in[0,T]${.} Then there exist positive constants $C_0,\delta_{\mathrm{res},0},\varepsilon_{\mathrm{res},0}$ ($\delta_{\mathrm{res},0},\varepsilon_{\mathrm{res},0}<1$) and a global Maxwellian $M_\#=M[U_\#](v_\#,\theta_\#>0)$ independent of time $T$ such that if the wave strength $\delta_R+\delta_C+\delta_S\leq \delta_{\mathrm{res},0}$ and $\mathcal{E}(T)\leq \varepsilon$, {then} 
    \begin{align}
        \begin{aligned}
            &\mathcal{E}(T)^2 + \delta_S\int_0^T \abs{\dot{X}(t)}^2 dt + \int_0^T \myparas{\mathcal G^R+\mathcal G^S}dt + \sum_{\abs{\beta}=1}\int_0^T \norm{\partial^\beta(\phi,\psi,\zeta)}_{L^2}^2dt\\
            &+\int_0^T\myparas{\norm{(\phi_{zz},\psi_{zz},\zeta_{zz})}_{L^2}^2+\norm{(\phi_{zt},\psi_{zt},\zeta_{zt})}_{L^2}^2} dt \\
            &+\int_0^T \int \norm{\wtilde G_{\textup{rem}}}_{\nu,M_\#}^2 + \norm{\wtilde G_z}_{\nu,M_\#}^2 +\norm{\wtilde G_t}_{\nu,M_\#}^2 +\norm{\wtilde G_{zz}}_{\nu,M_\#}^2 + \norm{\wtilde G_{zt}}_{\nu,M_\#}^2 dz dt\\
            &\leq C_0\myparas{\mathcal{E}(0)^2+\delta_{\mathrm{res},0}^{1/2}},
        \end{aligned}
    \end{align}
    where $\partial^\beta=\partial_t^{\beta_0}\partial_z^{\beta_1}$, $|\beta|=|\beta_1|+|\beta_0|$ denotes the derivatives with respect to $z$ or $t$, and 
    \begin{align}
        \begin{aligned}
            & \mathcal G^R:=\int\abs{v^R_z}\abs{(\phi,\zeta)}^2 dy,\\
            & \mathcal G^S:=\int\Bigl|\partial_z\myparas{v^{S_3}}^{-X_3}\Bigr|\abs{\myparas{\phi,\psi,\zeta}}^2 dy.
        \end{aligned}
    \end{align}
\end{proposition}

{
\begin{proof}[Proof of Theorem \ref{cor:rarefaction-contact-shock}]
By Proposition~\ref{prop:Ypro}, together with the same local existence and continuation argument
used in the proof of Theorem~\ref{thm:main}, there exists a unique global normalized solution
in the moving frame, together with a shift \(X_3=X_3(t)\).

Let \(f\) denote this normalized solution, and define the rescaled physical solution and shift by
\eqref{eq:rescaled-solution} and \eqref{eq:rescaled-shifts}. Equivalently,
\[
f^\kappa(\tau,y,\xi)=f\!\left(\frac{\tau}{\kappa},\frac{y}{\kappa},\xi\right),
\qquad
X_3^\kappa(\tau)=\kappa X_3\!\left(\frac{\tau}{\kappa}\right).
\]
Then \(f^\kappa\) exists uniquely on \([0,T]\) for every \(T>0\), since \(f\) exists globally
on \([0,\infty)\).

We first compare \(f^\kappa\) with the rescaled composite approximate wave
\(\bar f^{\,\kappa}\). By Proposition~\ref{prop:Ypro}, after the same scaling argument as in
the proof of Theorem~\ref{thm:main}, we obtain the perturbation estimate
\begin{align}
\int_0^T\int_{\mathbb R} \norm{f^\kappa-\bar f^{\kappa}}_{M_\#}^2
\,dy\,d\tau
\le
C\kappa\bigl((\varepsilon_1+\delta_{\mathrm{res},0})^2+\delta_{\mathrm{res},0}^{1/2}\bigr)T .
\label{eq:main-step-1-rare-cor}
\end{align}

{By Lemma} \ref{lem:shift-compactness}{,}
\[
|X_3^\kappa(t)| \le CT^{1/2}\kappa^{1/2}
\qquad t\in[0,T].
\]

Based on this, we compare the rescaled composite wave with the shifted Euler Maxwellian profile.
Combining the smooth rarefaction approximation estimate from Lemma~\ref{lem:Yilem}
with the corresponding contact-wave estimate, we obtain
\begin{align}
\int_0^T\int_{\mathbb R}
\norm{\bar f^{\,\kappa}-M_{X_3^\kappa}[U^E]}_{M_\#}^2\,dy\,d\tau
\le
C(\delta_R+\delta_C)\kappa^{1/2}T^{3/2}+C\delta_S\kappa T .
\label{eq:main-step-2-rare-cor}
\end{align}

Finally, by the mean value theorem and the exponential decay of the Euler shock profile,
for each \(\tau\in[0,T]\),
\begin{align*}
\int_{\mathbb R} \norm{M_{X_3^\kappa}[U^E]-M[U^E]}_{M_\#}^2
\,dy
\le
C\delta_{\mathrm{res},0}^2|X_3^\kappa(\tau)|.
\end{align*}
Integrating in \(\tau\) yields
\begin{align}
\int_0^T\int_{\mathbb R}
 \norm{M_{X_3^\kappa}[U^E]-M[U^E]}_{M_\#}^2 \,dy\,d\tau
\le
C\delta_{\mathrm{res},0}^2\|X_3^\kappa\|_{L^1(0,T)}.
\label{eq:main-step-3-rare-cor}
\end{align}

Combining \eqref{eq:main-step-1-rare-cor}, \eqref{eq:main-step-2-rare-cor},
and \eqref{eq:main-step-3-rare-cor}, and using that \(M_\#\) is uniformly bounded
from above and below by positive constants, we conclude that
\begin{align*}
&\int_0^T \int_{\mathbb R} \norm{f^\kappa-M[U^E]}_{M_\#}^2
\,dy\,d\tau\\
&\qquad \le
C\kappa\bigl((\varepsilon_1+\delta_{\mathrm{res},0})^2+\delta_{\mathrm{res},0}^{1/2}\bigr)T
+
C(\delta_R+\delta_C)\kappa^{1/2}T^{3/2}
+
C\delta_{\mathrm{res},0}^2\kappa^{1/2}T^{3/2}.
\end{align*}
This is exactly \eqref{eq:hlmt-main}.
\end{proof}
}

\hide
\section*{Acknowledgment}

CK is partially supported by NSF-CAREER 2047681. This material is partly based upon work supported by the National Science Foundation under Grant No. DMS-2424139, while one of the authors (C.K.) was in residence at the Simons Laufer Mathematical Sciences Institute in Berkeley, California, during the Fall 2025 semester.
\unhide
\appendix
\section{Appendix}
\setcounter{equation}{0}
\subsection{Properties of the collision operator}

In this appendix, we collect several standard estimates for the hard-sphere collision operator that will be used repeatedly in the main text. We refer to \cite{liu2004energy,liu2006nonlinear} for the basic macro--micro framework and weighted estimates.

\subsubsection{Gain--loss decomposition and weighted estimates}

We first recall the decomposition of the collision operator into the gain and loss parts:
\begin{equation}\label{eq:Q-gain-loss}
\mathcal N(g,h)=\mathcal N^+(g,h)-\mathcal N^-(g,h).
\end{equation}

The following weighted $L^2$ estimates are standard for the hard-sphere collision operator.

\begin{lemma}\label{lem:collision-gain-loss}
There exists a positive constant $C$ such that
\begin{align}
\int_{\mathbb{R}^3}\frac{(1+|\xi|)^{-1}|\mathcal N^-(g,h)|^2}{M_\#}\,d\xi
&\le
C \norm{g}_{\nu,M_\#}^2\, \norm{h}_{M_\#}^2,
\label{eq:collision-loss}
\\
\int_{\mathbb{R}^3}\frac{(1+|\xi|)^{-1}|\mathcal N_+(g,h)|^2}{M_\#}\,d\xi
&\le
C \norm{g}_{M_\#}^2 \, \norm{h}_{\nu,M_\#}^2 
\label{eq:collision-gain}
\end{align}
Consequently,
\begin{align}
&\int_{\mathbb{R}^3}\frac{(1+|\xi|)^{-1}|\mathcal N(g,h)|^2}{M_\#}\,d\xi\nonumber \\
&\quad \le
C\Bigl[ \norm{g}_{\nu,M_\#}^2\, \norm{h}_{M_\#}^2 +  \norm{g}_{M_\#}^2 \, \norm{h}_{\nu,M_\#}^2  \Bigr].
\label{eq:collision-Q}
\end{align}
Here $M_\#$ is a reference global Maxwellian and the above estimates hold on a uniform class of functions comparable to $M_\#$.
\end{lemma}

\subsubsection{Coercivity of the linearized collision operator}

Let $M=M[U]$ where $U=(v,u,\theta)$ be a local Maxwellian and let $L_M$ denote the linearized collision operator around $M$. Let $\mathfrak Z_M$ be the null space of $L_M$, and let $\mathfrak Z_M^\perp$ be its orthogonal complement.

\begin{lemma}\label{lem:collision-coercivity}
Let \(U_\#=(v_\#,u_\#,\theta_\#)\), \(M_\#:=M[U_\#]\). Then there exists a constant \(\eta_0>0\) such that the following holds. Let 
\(\mathcal K_\#\Subset \R_{>0}\times \R^3\times \R_{>0}\) be a compact set satisfying 
\[
    \mathcal K_\# \subset B_{\eta_0}(U_\#),\qquad \theta_\#<\inf_{U\in\mathcal K_\#}\theta \le \sup_{U\in\mathcal K_\#}\theta<2\theta_\#
\]
where $B_{\eta_0}(U_\#):=\{U\in\R_{>0}\times\R^3\times \R_{>0}:|U-U_\#|<\eta_0\}$. For \(U=(v,u,\theta)\in\mathcal K_\#\), set \(M=M[U]\). Then there exists a constant \(\lambda_{\mathrm{mic}}>0\), {depending only on} \(\mathcal K_\#\) and \(U_\#\), such that 
\begin{align}\label{eq:collision-coercivity}
    -\int_{\R^3}\frac{gL_Mg}{M_\#}\,d\xi \ge \lambda_{\mathrm{mic}}\int_{\R^3}\frac{(1+|\xi|)g^2}{M_\#} \, d\xi 
\end{align}
for every \(U \in \mathcal K_\#\) and every \(g\in\mathfrak Z_M^\perp\). Moreover, 
\begin{align}\label{eq:collision-inverse}
    \int_{\R^3}\frac{(1+|\xi|)}{M_\#}|L_M^{-1}g|^2 \, d\xi \le \lambda_{\mathrm{mic}}^{-2} \int_{\R^3}\frac{(1+|\xi|)^{-1}g^2}{M_\#}\, d\xi
\end{align}
Here \(\mathfrak Z_M\) is the null-space of \(L_M\).
\end{lemma}

\subsection{Auxiliary lemmas for the away-from-the-shock convergence}

We now record two elementary lemmas needed to pass from the a priori estimate to the
pointwise-in-space control of the perturbation.

\begin{lemma}[One-dimensional Sobolev embedding with values in a Hilbert space]
\label{lem:H1-Linfty-Hilbert-single}
Let \(H\) be a Hilbert space. Then for every \(g\in H^1(\mathbb R_y;H)\),
\begin{equation}
\sup_{y\in\mathbb R}\|g(y)\|_H^2
\le
2\|g\|_{L^2_y(H)}\|g_y\|_{L^2_y(H)}
\le
\|g\|_{H^1_y(H)}^2 .
\label{eq:H1-Linfty-Hilbert-single}
\end{equation}
In particular,
\[
H^1(\mathbb R_y;H)\hookrightarrow C_b^0(\mathbb R_y;H).
\]
\end{lemma}

\begin{proof}
It is enough to prove \eqref{eq:H1-Linfty-Hilbert-single} for
\(g\in C_c^\infty(\mathbb R_y;H)\), the general case following by density.
For such \(g\),
\[
\frac{d}{dy}\|g(y)\|_H^2
=
2\langle g_y(y),g(y)\rangle_H .
\]
Hence, for every \(y\in\mathbb R\),
\[
\|g(y)\|_H^2
=
2\int_{-\infty}^y \langle g_z(z),g(z)\rangle_H\,dz
\le
2\int_{\mathbb R}\|g_z(z)\|_H\,\|g(z)\|_H\,dz .
\]
By Cauchy--Schwarz,
\[
\|g(y)\|_H^2
\le
2\|g_y\|_{L^2_y(H)}\|g\|_{L^2_y(H)} .
\]
Taking the supremum in \(y\) proves \eqref{eq:H1-Linfty-Hilbert-single}.
\end{proof}

\begin{lemma}[Local Lipschitz continuity of the Maxwellian map]
\label{lem:Maxwellian-Lipschitz-single}
Let \(K\Subset \mathbb R_+\times\mathbb R\times\mathbb R_+\) be compact. Then there exists
\(C_K>0\) such that
\begin{equation}
\|M[U]-M[V]\|_{M_\#}
\le
C_K |U-V|,
\qquad U,V\in K.
\label{eq:Maxwellian-Lipschitz-single}
\end{equation}
\end{lemma}

\begin{proof}
Let \(U,V\in K\), and define \(U_\vartheta:=(1-\vartheta)V+\vartheta U\),
\(\vartheta\in[0,1]\).  Enlarging \(K\) slightly if necessary, we may assume that the
whole segment \(\{U_\vartheta:\vartheta\in[0,1]\}\) lies in a compact set
\(K'\Subset \mathbb R_+\times\mathbb R\times\mathbb R_+\).  By the fundamental theorem
of calculus,
\[
M[U]-M[V]
=
\int_0^1 DM[U_\vartheta](U-V)\,d\vartheta .
\]
Hence
\[
\|M[U]-M[V]\|_{M_\#}
\le
|U-V|
\sup_{W\in K'}
\|DM[W]\|_{\mathcal L(\mathbb R^3,L^2_\xi(M_\#))}.
\]
It therefore suffices to bound \(DM[W]\) uniformly on \(K'\).  Each partial derivative
\(\partial_{W_j}M[W]\) is of the form \(P_j(\xi;W)M[W](\xi)\), where \(P_j\) is a
polynomial of degree at most \(2\) in \(\xi\), with coefficients depending continuously
on \(W\).  Since \(K'\) is compact and bounded away from vacuum and zero temperature,
there exist positive constants \(C,c\) such that
\[
\frac{M[W](\xi)^2}{M_\#(\xi)}
\le
Ce^{-c|\xi|^2},
\qquad W\in K',\ \xi\in\mathbb R^3.
\]
Therefore
\[
\sup_{W\in K'}
\sum_j \|\partial_{W_j}M[W]\|_{M_\#}
<\infty,
\]
which implies \eqref{eq:Maxwellian-Lipschitz-single}.
\end{proof}

\subsection{Proof of Lemmas~\ref{lem:bs}, \ref{lem:bs1} and \ref{lem:bs2}}
\begin{proof}
    We do not reproduce the proof of the corresponding estimates for {the }\(3\){-shock wave} here, since it is essentially the same as the proof given in \cite{wang2025}.
    In particular, the estimates for the \(3\)-shock profile will be used in the form established in \cite{LiuYu2013ARMA}. 
    Our purpose in the above discussion is only to make explicit the scalar parametrization and the auxiliary estimates from \cite{wang2025}, which are invoked in the proof of \cite{LiuYu2013ARMA}.
    We first recall the {scalar} parametrization of a Boltzmann shock profile. By (295) in \cite{LiuYu2013ARMA}, there exists a scalar function \(\eta=\eta(y)\) such that 
    \begin{equation}\label{eq:scalar-param-zeta}
        \left\{
            \begin{aligned}
                \eta_y &= a_0(\eta)\bigl(\eta-\eta_-\bigr)\bigl(\eta-\eta_+\bigr), \\
                \eta(-\infty) &= \eta_-, \qquad \eta(+\infty) = \eta_+,
            \end{aligned}
        \right.
    \end{equation}
    where \(\eta_->0>\eta_+\), \(|\eta_- - \eta_+| = O(\delta_3)\), and \(a_0\) is a smooth function satisfying \(a_0(\eta)\ge c>0\). The elementary consequences of \eqref{eq:scalar-param-zeta}, including the monotonicity, exponential decay, the higher derivative estimates of \(\eta\), are obtained in \cite[(291)]{LiuYu2013ARMA}. In particular, we shall use 
    \begin{equation}\label{eq:zeta-basic-estimates}
        \eta_y<0, \qquad |\eta-\eta_\pm| \lesssim \delta_3 e^{-c \delta_3 |y|} \quad \text{as } y\to \pm \infty, 
    \end{equation}
    and
    \begin{equation}\label{eq:zeta-derivative-estimates}
    |\eta_y| \lesssim \delta_3^2 e^{-c\delta_3 |y|}, \qquad |\partial_y^k \eta| \lesssim \delta_3^{k-1}|\eta_y|, \qquad k\ge 2.
    \end{equation}
    by (295) in \cite{LiuYu2013ARMA} and Lemma A.1 in \cite{wang2025}. Under this parametrization, the macroscopic part of the shock profile satisfies 
    \begin{equation}\label{eq:macro-zeta-comparison}
        \bigl|v_y^{S_3}\bigr| \sim \bigl|u_{1y}^{S_3}\bigr| \sim \bigl| \theta_y^{S_3}\bigr| \sim |\eta_y|.
    \end{equation}
    Moreover, by (289) and (291) in \cite{LiuYu2013ARMA}, the microscopic component can be represented in the form
    \begin{equation}
        G^{S_3}(y,\xi) = \sqrt{M_\#} \, \Gamma(\chi(\eta(y)),\xi),
    \end{equation}
    where \(M_\#\) is a fixed global Maxwellian, \(\Gamma\) is smooth and \(\chi=\chi(\eta)\) is determined by 
    \begin{equation}\label{eq:chi-equation}
        \chi(\eta) = b(\eta,\chi(\eta)) (\eta- \eta_-) (\eta-\eta_+),
    \end{equation}
    where \(b\) is smooth and uniformly bounded above and below by {positive} constants. Consequently,
    \begin{equation}\label{eq:G-chi-relation}
        \Bigl( \int_{\R^3} \frac{(1+|\xi|)\bigl|G^{S_3}(y,\xi)\bigr|^2}{M_\#}\, d\xi \Bigr)^{\frac12} = |\chi(\eta(y))|\, c_0 (\chi(\eta(y))),
    \end{equation}
    where \(c_0\) is smooth and satisfies \(c_0\ge c>0\).

    By referring to \cite{wang2025}, we obtain the corresponding estimate for the \(3\)-shock. We prove the \(1\)-shock estimate by using the reflection symmetry from the corresponding \(3\)-shock estimate. Let 
    \begin{align*}
    U^{S_1} = (v^{S_1},u^{S_1},\theta^{S_1})
    \end{align*}
    be the \(1\)-shock profile connecting \(U_-=(v_-,u_-,\theta_-)\) to \(U_*=(v_*,u_*,\theta_*)\).
    Define the reflected profiles by 
    \begin{align*}
        \wtilde v^{S_3}(z) := v^{S_1}(-z), \quad \wtilde u^{S_3}(z) := -u^{S_1}(-z), \quad \wtilde \theta^{S_3}(z) := \theta^{S_1}(-z).
    \end{align*}
    Then \(\wtilde U^{S_3}\) is a \(3\)-shock profile connecting \(\wtilde U^*=(v_*,-u_*,\theta_*)\) to \(\wtilde U_+=(v_-,-u_-,\theta_-)\). In particular, we denote $\wtilde v^*=v_*$, $\wtilde v_+=v_-$, $\wtilde \theta^*= \theta_*$, $\wtilde p^*=p_*$ and $\wtilde p_+=p_-$. Note that $\wtilde p^{S_3}(z) = p^{S_1}(-z)$ since $p=p(v,\theta)$. Moreover the shock strength is preserved under reflection, namely, $\wtilde \delta_3=\delta_1$.
    Applying the \(3\)-shock estimate to the reflected profile gives 
    \begin{align*}
    &\Biggl|\frac{\wtilde p^{S_3}-\wtilde p_+}{\wtilde v^{S_3}-\wtilde v_+}- \frac{\wtilde p^{S_3}-\wtilde p^*}{\wtilde v^{S_3}-\wtilde v^*}-\frac{5\wtilde p^*}{9\bigl(\wtilde v^*\bigr)^2}\Bigl(\frac{10\mu(\wtilde \theta^*)-9\alpha_{\rm th}(\wtilde \theta^*)}{10\mu(\wtilde \theta^*)+3\alpha_{\rm th}(\wtilde\theta^*)}+3\Bigr)\Biggr| \le C\wtilde \delta_3^2.
    \end{align*}
    Using the above identifications, this becomes
    \begin{align*}
        &\Biggl|\frac{p^{S_1}-p_-}{v^{S_1}-v_-}- \frac{p^{S_1}-p_*}{v^{S_1}-v_*}-\frac{5p_*}{9\bigl(v_*\bigr)^2}\Bigl(\frac{10\mu(\theta_*)-9\alpha_{\rm th}(\theta_*)}{10\mu(\theta_*)+3\alpha_{\rm th}(\theta_*)}+3\Bigr)\Biggr| \le C\delta_1^2.
    \end{align*}
    This proves the desired \(1\)-shock estimate. The remaining estimates for {the }\(1\){-shock} can be derived in a similar manner. 
\end{proof}

\subsection{Proof of Lemma~\ref{lem:intsk}}

\begin{proof}
The proof follows a similar argument to that in \cite{han2023large}. We provide the details for completeness. Let
\[
x_1(t):=\sigma_1 t+X_1(t),
\qquad
x_3(t):=\sigma_3 t+X_3(t).
\]
By the exponential localization of each viscous shock profile, there exist positive constants $c$ and $C$ such that
\begin{equation}\label{eq:shock-tail-proof}
\bigl|\partial_x(v^{S_i})^{-X_i}(t,x)\bigr|
\le
C\delta_i^2 e^{-c\delta_i|x-x_i(t)|},
\qquad i=1,3.
\end{equation}
Moreover, in the present $1$-shock/$3$-shock configuration, the two shifted shock locations remain separated linearly in time. More precisely, after possibly decreasing $\delta_0>0$, we may assume that {there exist constants} $c_1,c_3>0$ {such that}
\begin{equation}\label{eq:shock-separation-proof}
x_1(t)\le -c_1 t,
\qquad
x_3(t)\ge c_3 t,
\qquad t>0.
\end{equation}

We first prove the estimate for $\varphi_3 |(v^{S_1})_x^{-X_1}|$. By the definition of $\varphi_1$ and $\varphi_3$, if $\varphi_3(t,x)\neq0$, then
\[
x\ge \frac{x_1(t)}{2}.
\]
Since $x_1(t)<0$ by \eqref{eq:shock-separation-proof}, it follows that
\[
x-x_1(t)
\ge
-\frac{x_1(t)}{2}
\ge
\frac{c_1}{2}\,t>0.
\]
Hence, by \eqref{eq:shock-tail-proof},
\[
\varphi_3(t,x)\bigl|\partial_x(v^{S_1})^{-X_1}(t,x)\bigr|
\le
C\delta_1^2
e^{-c\delta_1|x-x_1(t)|}
\le
C\delta_1^2 e^{-C\delta_1 t}.
\]
This proves the first estimate in \eqref{eq:intsk1}.

Similarly, if $\varphi_1(t,x)\neq0$, then
\[
x\le \frac{x_3(t)}{2}.
\]
Since $x_3(t)>0$ by \eqref{eq:shock-separation-proof}, we obtain
\[
x_3(t)-x
\ge
\frac{x_3(t)}{2}
\ge
\frac{c_3}{2}\,t>0.
\]
Using again \eqref{eq:shock-tail-proof}, we get
\[
\varphi_1(t,x)\bigl|\partial_x(v^{S_3})^{-X_3}(t,x)\bigr|
\le
C\delta_3^2 e^{-C\delta_3 t},
\]
which proves the second pointwise estimate.

We next prove the integral estimates. By the pointwise bound just proved,
\[
\int_{\mathbb R}\varphi_3(t,x)\bigl|\partial_x(v^{S_1})^{-X_1}(t,x)\bigr|\,dx
\le
C\delta_1^2
\int_{\{x\ge x_1(t)/2\}} e^{-c\delta_1|x-x_1(t)|}\,dx.
\]
Since $x\ge x_1(t)/2$ implies $|x-x_1(t)|\ge |x_1(t)|/2\ge c_1 t/2$, we have
\[
\int_{\{x\ge x_1(t)/2\}} e^{-c\delta_1|x-x_1(t)|}\,dx
\le
C\int_{c_1 t/2}^{\infty} e^{-c\delta_1 z}\,dz
\le
C\delta_1^{-1}e^{-C\delta_1 t}.
\]
Therefore,
\[
\int_{\mathbb R}\varphi_3(t,x)\bigl|\partial_x(v^{S_1})^{-X_1}(t,x)\bigr|\,dx
\le
C\delta_1 e^{-C\delta_1 t}.
\]
The estimate
\[
\int_{\mathbb R}\varphi_1(t,x)\bigl|\partial_x(v^{S_3})^{-X_3}(t,x)\bigr|\,dx
\le
C\delta_3 e^{-C\delta_3 t}
\]
is proved in the same way.

This completes the proof.
\end{proof}

\subsection{Proof of Lemma~\ref{lem:wave-interaction-L2}}

\begin{proof}
The proof follows a similar argument to that in \cite{kang2025time}. We provide the details for completeness. We only prove \eqref{eq:wave-interaction-L2} and \eqref{eq:wave-interaction-L2b}, since
\eqref{eq:wave-interaction-L2-3}--\eqref{eq:wave-interaction-L2-3b} follow in exactly the same way.

Set
\[
\xi_1(t):=\sigma_1 t+X_1(t),
\qquad
\xi_3(t):=\sigma_3 t+X_3(t).
\]
By the standard pointwise estimates for the viscous shock profiles and the viscous contact wave obtained earlier, we have
\begin{align}
\label{eq:proof-int-1}
|\partial_x(v^{S_1})^{-X_1}(t,x)|
&\le
C\delta_1^2 e^{-c\delta_1|x-\xi_1(t)|}, \quad x-\xi_1(t)\in \mathbb{R}
\\
\label{eq:proof-int-2}
|(v^{S_3}-v^*,\theta^{S_3}-\theta^*)^{-X_3}(t,x)|
&\le
C\delta_3 e^{-c\delta_3|x-\xi_3(t)|}, \quad x-\xi_3(t)\leq0,
\\
\label{eq:proof-int-3}
|(v^C-v_*,\theta^C-\theta_*)(t,x)|
&\le
C\delta_C(1+t)^{-1/2}e^{-c_0x^2/(1+t)}, \quad x\leq 0.
\end{align}
Moreover, since the shock speeds are distinct and the shifts are bounded in the regime of interest, there exists a positive constant \(c_1\) such that
\begin{equation}
\label{eq:proof-int-sep}
|\xi_3(t)-\xi_1(t)|\ge c_1(1+t),
\qquad
|\xi_1(t)|+|\xi_3(t)|\ge c_1(1+t).
\end{equation}

We first prove \eqref{eq:wave-interaction-L2}. We divide into two regions $A_1:=\{x-\sigma_1t\geq -\frac{\sigma_1}{2}t\}$ and $A_2:=\mathbb{R}-A_1$. Since 
\begin{align*}
    & x-\xi_1(t) \geq -\frac{\sigma_1}{2}t + C\varepsilon t\geq -\frac{\sigma_1t}{4}>0 \quad \text{for}\, x\in A_1,\\
    & x\leq \frac{\sigma_1}{2}t <0,\quad \text{for} \, x\in A_2,
\end{align*}
we have the following observations.
\begin{align} \label{eq:keysep1}
    \abs{\partial_x\shockw{v}{1}}\leq C \delta_1^2 e^{-C\delta_1\abs{x-\xi_1(t)}} \leq C \delta_1^2 e^{-\frac{C\delta_1}{2}\abs{x-\xi_1(t)}}e^{-C\delta_1t}.
\end{align}

By \eqref{eq:proof-int-1}, \eqref{eq:proof-int-3} and \eqref{eq:keysep1},
\begin{align*}
    &\left\|
    |\partial_x(v^{S_1})^{-X_1}|
    \,|(v^C-v_*,\theta^C-\theta_*)|
    \right\|_{L^2_x}^2
    \\
    &\le
    \int_{A_1}|\partial_x(v^{S_1})^{-X_1}|^2
    \,|(v^C-v_*,\theta^C-\theta_*)|^2 dx + \int_{A_2}|\partial_x(v^{S_1})^{-X_1}|^2
    \,|(v^C-v_*,\theta^C-\theta_*)|^2 dx\\
    &\le C\delta_C^2 \int_{A_1} \delta_1^4 e^{-2c\delta_1\abs{x-\xi_1(t)}} dx + \frac{C}{1+t}\int_{A_2} \delta_1^4\delta_C^2 e^{-2c\delta_1\abs{x-\xi_1(t)}}e^{-2c_0x^2/(1+t)} dx \\
    &\le C\delta_C^2 \delta_1^3 e^{-C\delta_1t} + C\delta_C^2\delta_1^3e^{-Ct} 
\end{align*}
Therefore
\[
\left\|
|\partial_x(v^{S_1})^{-X_1}|
\,|(v^C-v_*,\theta^C-\theta_*)|
\right\|_{L^2_x}^2
\le
C\delta_1^3\delta_C^2 e^{-c\delta_1 t},
\]
which gives \eqref{eq:wave-interaction-L2}.

Next, we prove \eqref{eq:wave-interaction-L2b}. 
Separate $\mathbb{R}$ into 
\[
\Omega_1 := \{x:x-\sigma_3t\leq -\frac{\sigma_3t}{2}\}, \qquad \Omega_2 := \mathbb{R}\setminus \Omega_1.
\]
Since 
\begin{align*}
    &x-\xi_3(t)=x-\sigma_3t-X_3(t)\leq -\frac{\sigma_3t}{2}+C\varepsilon t\leq -\frac{\sigma_3t}{4}<0 \quad \text{for}\, x\in \Omega_1,\\
    &x-\xi_1(t)=x-\sigma_1t-X_1(t)\geq \frac{\sigma_3t}{2}-\sigma_1t-C\varepsilon t\geq \frac{\sigma_3 t}{4}>0 \quad \text{for}\, x\in \Omega_2,
\end{align*}
we have the following estimate.
\begin{align*}
    &\norm{\abs{\partial_x\shockw{v}{1}}\abs{\myparas{\shockw{v}{3}-v^*,\shockw{\theta}{3}-\theta^*}}}^2\\
    &\le \int_{\Omega_1}\abs{\partial_x\shockw{v}{1}}^2\abs{\myparas{\shockw{v}{3}-v^*,\shockw{\theta}{3}-\theta^*}}^2 dx\\
    &\qquad+ \int_{\Omega_2} \abs{\partial_x\shockw{v}{1}}^2\abs{\myparas{\shockw{v}{3}-v^*,\shockw{\theta}{3}-\theta^*}}^2 dx \\
    &\le C\delta_1^4\delta_3^2\int_{\Omega_1} e^{-2c\delta_1\abs{x-\xi_1(t)}}e^{-2c\delta_3\abs{x-\xi_3(t)}} dx + C\delta_1^4\delta_3^2\int_{\Omega_2} e^{-2c\delta_1\abs{x-\xi_1(t)}} dx \\
    &\le C\delta_1^3\delta_3^2e^{-c\delta_3t}+C\delta_1^3\delta_3^2e^{-c\delta_1t}
\end{align*}
Thus
\[
\left\|
|\partial_x(v^{S_1})^{-X_1}|
\,|(v^{S_3}-v^*,\theta^{S_3}-\theta^*)^{-X_3}|
\right\|_{L^2_x}^2
\le
C\delta_1^3\delta_3^2
\bigl(e^{-c\delta_3 t}+e^{-c\delta_1 t}\bigr),
\]
and after taking square roots,
\[
\left\|
|\partial_x(v^{S_1})^{-X_1}|
\,|(v^{S_3}-v^*,\theta^{S_3}-\theta^*)^{-X_3}|
\right\|_{L^2_x}
\le
C\delta_1^{3/2}\delta_3
\bigl(e^{-c\delta_3 t}+e^{-c\delta_1 t}\bigr).
\]
This proves \eqref{eq:wave-interaction-L2b}.
\end{proof}

\subsection{Proof of Lemma~\ref{lem:macro-time-derivative}}

\begin{proof}
    {The solution satisfies the following equations:}
\begin{align*}
\begin{aligned}
&v_t-u_{1x}=0\\
&u_{1t}+p_x=-\int \xi_1^2 G_x d\xi\\
&u_{it}=-\int \xi_1\xi_i G_x d\xi,\quad (i=2,3)\\
&\theta_t+pu_{1x}=u_1\int \xi_1^2 G_x d\xi+\sum_{i=2}^3 u_i \int \xi_1\xi_i G_x d\xi-\frac{1}{2}\int \xi_1\abs{\xi}^2 G_x d\xi.
\end{aligned}
\end{align*}

For the {approximate} solutions, we have the following systems:
\begin{align*}
\begin{aligned}
&v_t^C-u_{1x}^C=0\\
&u_{1t}^C+p_x^C=\frac{4}{3}\myparas{\frac{\mu\myparas{\theta^C}u_{1x}^C}{v^C}}_x+Q_1^C\\
&u_i^C = 0\quad (i=2,3)\\
&\theta_t^C + p^Cu_{1x}^C = \myparas{\frac{\alpha_{\rm{th}}\myparas{\theta^C}\theta_x^C}{v^C}}_x+\frac{4}{3}\mu\myparas{\theta^C}\frac{\myparas{u_{1x}^C}^2}{v^C}+Q_2^C
\end{aligned}
\end{align*}

and

\begin{align*}
\begin{aligned}
&\shock{v}{i}_t+\dot{X}_i\shock{v}{i}_x-\shock{u_1}{i}_x=0\\
&\shock{u_1}{i}_t+\dot{X}_i\shock{u_1}{i}_x+\shock{p}{i}_x=-\int\xi_1^2 \shock{G}{i}_x d\xi\\
&\shockw{u_j}{i}\equiv 0 \quad(j=2,3)\\
&\shock{\theta}{i}_t+\dot{X}_i\shock{\theta}{i}_x+\shockw{p}{i}\shock{u_1}{i}_x\\
&\qquad =\shockw{u_1}{i}\int \xi_1^2 \shock{G}{i}_x d\xi -\frac{1}{2}\int \xi_1\abs{\xi}^2 \shock{G}{i}_xd\xi.
\end{aligned}
\end{align*}
{Therefore, the perturbation system can be written in terms of the microscopic part} $\wtilde{G}${ as follows:}

\begin{align*}
\begin{aligned}
&\phi_t-\psi_{1x}-\sum_{i=1,3} \dot{X}_i \shock{v}{i}_x = 0\\
&\psi_{1t} + \myparas{p-\shockp{p}{1}-p^C-\shockp{p}{3}}_x-\sum_{i=1,3} \dot{X}_i \shock{u_1}{i}_x \\
& \qquad = -Q_1^C-\frac{4}{3}\myparas{\frac{\mu\myparas{\theta^C}u_{1x}^C}{v^C}}_x-\int \xi_1^2 \wtilde{G}_x d\xi\\
&\psi_{it}=-\int \xi_1\xi_i \wtilde{G}_x d\xi,\quad (i=2,3)\\
&\zeta_t +\myparas{pu_{1x}-\shockp{p}{1}\shock{u_1}{1}_x-p^Cu_{1x}^C-\shockp{p}{3}\shock{u_1}{3}_x}-\sum_{i=1,3}\dot{X}_i\shock{\theta}{i}_x\\
&\qquad \qquad = -\frac{1}{2}\int \xi_1 \abs{\xi}^2 \wtilde{G}_x d\xi +u_1\int \xi_1^2 \wtilde{G}_x d\xi \\
&\qquad \qquad +\sum_{i=1,3} \myparas{u_1-\shockw{u_1}{i}}\int \xi_1^2 \shock{G}{i}_x d\xi+\sum_{j=2}^3 u_j \int \xi_1\xi_j \wtilde{G}_x d\xi\\
&\qquad \qquad -\myparas{\frac{\alpha_{\rm{th}}\myparas{\theta^C}\theta_x^C}{v^C}}_x-\frac{4}{3}\mu\myparas{\theta^C}\frac{\myparas{u_{1x}^C}^2}{v^C}-Q_2^C.
\end{aligned}
\end{align*}


We can easily get the following estimate for $\phi_t$.

\begin{align*}
\begin{aligned}
\norm{\phi_t}_{L^2}^2 \leq C \int \abs{\psi_{1x}}^2 dx + C\sum_{i=1,3} \abs{\dot{X}_i}^2 \delta_i^2 \int \abs{\shock{v}{i}_x} dx \leq   C \int \abs{\psi_{1x}}^2 dx + C\sum_{i=1,3} \delta_i^3 \abs{\dot{X}_i}^2  
\end{aligned}
\end{align*}

{To estimate} $\psi_{1t}${, observe that}
\begin{align*}
\begin{aligned}
\myparas{p-\y{p}}_x =& \myparas{\frac{\zeta}{v}-\frac{\y{\theta}\phi}{v\y{v}}}_x \\
\leq & C \abs{\myparas{\zeta_x,\zeta\y{v}_x,\zeta\phi_x,\phi_x,\y{\theta}_x\phi,\y{v}_x\phi,\phi_x\phi}}\\
\leq & C\abs{\myparas{\zeta_x,\phi_x}}+C\abs{\myparas{\y{\theta}_x,\y{v}_x}}\abs{\myparas{\zeta,\phi}}.
\end{aligned}
\end{align*}
{Using this, we estimate as follows:}
\begin{align*}
\begin{aligned}
& \int \psi_{1t}\myparas{p-\shockp{p}{1}-p^C-\shockp{p}{3}}_x dx \\
& \quad \leq \frac{1}{64}\norm{\psi_{1t}}_{L^2}^2 + C\int \abs{Q_1}^2 dx +C \int \abs{\myparas{p-\y{p}}_x}^2 dx  \\
& \quad \leq \frac{1}{64}\norm{\psi_{1t}}_{L^2}^2 + C\int \abs{Q_1}^2 dx + C\int \abs{\myparas{\phi_x,\zeta_x}}^2 dx \\
& \qquad + C \int \abs{\myparas{\y{\theta}_x,\y{v}_x}}^2 \abs{\myparas{\zeta,\phi}}^2 dx .
\end{aligned}
\end{align*}

By using Young's inequality, we have

\begin{align*}
\begin{aligned}
\int \psi_{1t}\frac{4}{3} \myparas{\frac{\mu\myparas{\theta^C}u_{1x}^C}{v^C}}_x dx \leq & C \int \abs{\psi_{1t}} \abs{\myparas{u_{1xx}^C,\theta_x^Cu_{1x}^C,u_{1x}^Cv_x^C}} dx \\
\leq & \frac{1}{64}\norm{\psi_{1t}}_{L^2}^2 + C \int \abs{Q_1^C}^2 dx
\end{aligned}
\end{align*}
and 
\begin{align*}
\begin{aligned}
\int \psi_{1t} \int \xi_1^2 \wtilde{G}_x d\xi dx \leq & \frac{1}{64}\norm{\psi_{1t}}_{L^2}^2 +C \int \myparam{\int \xi_1^2 \wtilde{G}_x d\xi}^2 dx\\
\leq & \frac{1}{64}\norm{\psi_{1t}}_{L^2}^2 + C \int \norm{\wtilde G_x}_{\nu,M_\#}^2 dx.
\end{aligned}
\end{align*}

Therefore, we have the following estimate,
\begin{align*}
\begin{aligned}
\norm{\psi_{1t}}_{L^2}^2\leq &C \sum_{i=1,3}\delta_i^3 \abs{\dot{X}_i}+ C\int \abs{\myparas{\phi_x,\zeta_x}}^2 dx+C\int \abs{Q_1}^2 dx+ C \int \abs{\myparas{\y{\theta}_x,\y{v}_x}}^2 \abs{\myparas{\zeta,\phi}}^2 dx \\
&+C \int \abs{Q_1^C}^2 dx + C \int \norm{\wtilde G_x}_{\nu,M_\#}^2 dx .
\end{aligned}
\end{align*}


For the estimate of $\psi_{it}$, we have 

\begin{align*}
\begin{aligned}
\int \abs{\psi_{it}}^2 dx = \norm{\psi_{it}}_{L^2}^2,
\end{aligned}
\end{align*}
and
\begin{align*}
\begin{aligned}
\int \psi_{it} \int \xi_1\xi_i \wtilde{G}_x d\xi dx \leq& \frac{1}{64}\norm{\psi_{it}}_{L^2}^2 + C \int \myparam{\int \xi_1\xi_i \wtilde{G}_x d\xi}^2 dx\\
\leq & \frac{1}{64}\norm{\psi_{it}}_{L^2}^2 + C \int \norm{\wtilde G_x}_{\nu,M_\#}^2 dx.
\end{aligned}
\end{align*}
Therefore,
\begin{align*}
\begin{aligned}
\norm{\psi_{it}}_{L^2}^2 \leq C \int \norm{\wtilde G_x}_{\nu,M_\#}^2 dx.
\end{aligned}
\end{align*}

To estimate $\zeta_t$, we can observe 
\begin{align*}
\begin{aligned}
& \int \zeta_t\myparas{pu_{1x}-\shockp{p}{1}\shock{u_1}{1}_x-p^Cu_{1x}^C-\shockp{p}{3}\shock{u_1}{3}_x} dx\\ 
& \leq \frac{1}{64} \norm{\zeta_t}_{L^2}^2 + C \int \abs{Q_2}^2 dx + \int \abs{\myparas{pu_{1x}-\y{p}\y{u}_{1x}}}^2 dx\\
& \leq  \frac{1}{64} \norm{\zeta_t}_{L^2}^2 + C \int \abs{Q_2}^2 dx + C \int \abs{\psi_{1x}}^2 dx + C \int \abs{\y{u}_{1x}}^2\abs{\myparas{\phi,\zeta}}^2 dx.  
\end{aligned}
\end{align*}
{Furthermore, since} $\norm{\theta_{xx}^C}_{L^2}^2$ {decays like} $(1+t)^{-3/2}${, it is integrable in time. We have the following inequality:}

\begin{align*}
\begin{aligned}
& \int \zeta_t \myparas{u_1-\shockw{u_1}{i}}\int \xi_1^2 \shock{G}{i}_x d\xi dx \\
& \qquad \leq  \frac{1}{64} \norm{\zeta_t}_{L^2}^2 + C\int \myparas{u_1-\shockw{u_1}{i}}^2 \myparam{\int \xi_1^2 \shock{G}{i}_x d\xi}^2 dx\\
& \qquad \leq  \frac{1}{64} \norm{\zeta_t}_{L^2}^2  + C\int \myparas{u_1-\shockw{u_1}{i}}^2 \abs{\shock{v}{i}_x}^2 dx.
\end{aligned}
\end{align*}

{The remaining terms are estimated in the same manner as the} $\psi_{1t}${-term.} \end{proof}

\subsection{Proof of Lemma~\ref{lem:fosee}}

\begin{proof}
    {Arguing as in} \cite{wang2025}{, we obtain}
    \begin{align*}
&\myparas{\frac{\y{p}\theta}{2v}\phi_x^2+\frac{3\y{p}v}{4}\psi_{1x}^2+\sum_{k=2}^3\frac{\psi_{kx}^2}{2}+\frac{\zeta_x^2}{2}}_t+2\y{p}\psi_{1xx}^2\mu\myparas{\y{\theta}}+\sum_{k=2}^3\frac{\mu(\theta)}{v}\psi_{kxx}^2+\frac{\alpha_{\rm{th}}\myparas{\y{\theta}}}{v}\zeta_{xx}^2\\
&=-\sum_{i=1,3}\dot{X}_i\myparab{\phi_{xx}\shock{v}{i}_x\frac{\y{p}\theta}{v}+\psi_{1xx}\shock{u_1}{i}_x\frac{3\y{p}v}{2}+\zeta_{xx}\shock{\theta}{i}_x}+\frac{\phi_x^2}{2}\myparas{\frac{\y{p}\theta}{v}}_t\\
&+\frac{\psi_{1x}^2}{2}\myparas{\frac{3\y{p}v}{2}}_t-\myparas{\frac{\y{p}\theta}{v}}_x\phi_x\phi_t+\myparas{\frac{\y{p}\theta}{v}}_x\phi_x\psi_{1x}-\myparas{\frac{3\y{p}v}{2}}_x\psi_{1x}\psi_{1t}+\psi_{1xx}\y{p}v\y{\theta}_x\myparas{\frac{1}{v}-\frac{1}{\y{v}}}\\
&-\psi_{1xx}\y{p}v\y{v}_x\myparas{\frac{\theta}{v^2}-\frac{\y{\theta}}{\y{v}^2}}-\psi_{1x}\y{p}_x\zeta_x+\zeta_{xx}u_{1x}\myparas{p-\y{p}}-2\y{p}v\psi_{1xx}\myparas{\myparas{\mu(\theta)-\mu\myparas{\y{\theta}}}\frac{u_{1x}}{v}}_x\\
&-2\y{p}v\psi_{1xx}\myparab{\myparas{\mu\myparas{\y{\theta}}}_x\myparas{\frac{u_{1x}}{v}-\frac{\y{u}_{1x}}{\y{v}}}-\mu\myparas{\y{\theta}}\frac{v_x}{v^2}\psi_{1x}+\mu\myparas{\y{\theta}}\myparam{\y{u}_{1x}\myparas{\frac{1}{v}-\frac{1}{\y{v}}}}_x}\\
&-\zeta_{xx}\myparas{\myparas{\alpha_{\rm{th}}(\theta)-\alpha_{\rm{th}}\myparas{\y{\theta}}}\frac{\theta_x}{v}}_x+\sum_{k=2}^3\frac{\psi_{kx}^2}{2}\myparas{\frac{\mu(\theta)}{v}}_{xx}-\zeta_{xx}\sum_{k=2}^3\frac{\mu(\theta)}{v}\psi_{kx}^2+\zeta_{xx}Q_2\\
&+\psi_{1xx}\frac{3\y{p}v}{2}Q_1 -\zeta_{xx}\myparab{\myparas{\alpha_{\rm{th}}\myparas{\y{\theta}}}_x\myparas{\frac{\theta_{x}}{v}-\frac{\y{\theta}_{x}}{\y{v}}}-\alpha_{\rm{th}}\myparas{\y{\theta}}\frac{v_x}{v^2}\zeta_{x}+\alpha_{\rm{th}}\myparas{\y{\theta}}\myparam{\y{\theta}_{x}\myparas{\frac{1}{v}-\frac{1}{\y{v}}}}_x}\\
&-\frac{4}{3}\zeta_{xx}\myparas{\frac{\mu(\theta)u_{1x}^2}{v}-\frac{\mu\myparas{\y{\theta}}\y{u}_{1x}^2}{\y{v}}}+\psi_{1xx}\frac{3\y{p}v}{2}\int \xi_1^2 \widetilde{\Pi}_{1x} d\xi +\sum_{k=2}^3 \psi_{kxx}\int \xi_1\xi_i \Pi_{1x}d\xi\\
& + \zeta_{xx}\int \xi_1 \frac{\abs{\xi}^2}{2}\widetilde{\Pi}_{1x} d\xi - \zeta_{xx}\sum_{k=2}^3 \psi_k\int \xi_1\xi_i \Pi_{1x}d\xi\\
&-\zeta_{xx}\myparab{u_1\int \xi_1^2 \Pi_{1x}d\xi-\shockw{u_1}{1}\int\xi_1^2 \shock{\Pi_1}{1}_xd\xi -\shockw{u_1}{3}\int \xi_1^2 \shock{\Pi_1}{3}_xd\xi}\\
&+ (\mathcal R^{(1)})_x.
\end{align*} 
Here \(\mathcal R^{(1)}\) denotes a collection of lower-order flux terms produced by integration by parts and differentiation of the coefficients. Since the estimate is carried out over \(\mathbb R_x\), these divergence terms do not contribute after integration in \(x\).

Collecting the above identity, estimating the fluid error terms as in the zeroth-order energy argument, and controlling the microscopic terms by the auxiliary estimates from the high-order microscopic analysis, we obtain \eqref{eq:eoppz1}.
\end{proof}
 
\bibliographystyle{amsplain}
\bibliography{bibliography}

\end{document}